\documentclass[11pt,letterpaper,twoside,openany]{book}
\usepackage[utf8]{inputenc}

\usepackage[a4paper,left=1.5in,right=1.5in,footskip=0.15in,headheight=0.0in,tmargin=.5in,bmargin=.5in]{geometry}

\usepackage[colorlinks]{hyperref}
\usepackage[postdot]{glossaries-extra}

\usepackage{yourStyleFile}

\usepackage{stackengine}
\usepackage{epigraph}
\usepackage{chngcntr}

\usepackage[absolute,overlay]{textpos}

\usepackage{mathpazo, courier}
\usepackage{mathtools,blkarray}
\usepackage{mathtools}
\usepackage{changepage}
\usepackage{titling}
\usepackage{marvosym}
\usepackage{multirow}
\usepackage{lipsum}
\usepackage{listings}
\usepackage{adjustbox}
\usepackage{multicol}
\usepackage{tikz-cd}
\usetikzlibrary{positioning}
\tikzset{main node/.style={circle,fill=blue!20,draw,minimum size=0.1cm,inner sep=0pt}}
\usetikzlibrary{calc}
\usepackage{graphicx}
\usepackage{amssymb}
\usepackage{amsfonts}
\usepackage{qtree}
\usepackage{tipa}
\usepackage{amsthm}
\usepackage{amsmath}
\allowdisplaybreaks
\usepackage{color}
\usepackage{enumerate}
\usepackage{doi}
\usepackage{mathdots}
\usepackage{xcolor}
\usepackage{csquotes}
\usepackage[all,cmtip]{xy}

\usepackage{etoolbox}
\usepackage{longtable}
\usepackage{prerex}

\definecolor{seeblau}{RGB}{0, 169, 224}
\setulcolor{seeblau}
\setul{0.1ex}{0.2ex}

\newtheorem*{corolario*}{Corollary}
\newtheorem*{proposicion*}{Proposition}
\newtheorem*{teorema*}{Theorem}
\newtheorem{teorema}{Theorem}[section]

\newtheorem{lema}[teorema]{Lemma}

\newtheorem{proposicion}[teorema]{Proposition}
\newtheorem{corolario}[teorema]{Corollary}
\newtheorem{porisma}[teorema]{Porism}
\newtheorem{hecho}[teorema]{Fact}

\newtheorem{particularizacion}[teorema]{Particularization}

\newtheorem{simplificacion}[teorema]{Simplification}
\newtheorem{ampliacion}[teorema]{Ampliation}
\theoremstyle{definition}
\newtheorem{estrategia}[teorema]{Strategy}
\newtheorem{procedimiento}[teorema]{Procedure}
\newtheorem{recordatorio}[teorema]{Reminder}
\newtheorem{conjetura}[teorema]{Conjecture}
\newtheorem{comentario}[teorema]{Comment}

\newtheorem{definicion}[teorema]{Definition}
\newtheorem{cuestion}[teorema]{Question}
\newtheorem{limitacion}[teorema]{Limitation}
\newtheorem{warning}[teorema]{Warning}

\newtheorem{ramificacion}[teorema]{Ramification}
\newtheorem{convencion}[teorema]{Convention}

\newtheorem{construccion}[teorema]{Construction}
\newtheorem{ejemplo}[teorema]{Example}

\newtheorem{experimento}[teorema]{Experiment}

\newtheorem{objetivo}[teorema]{Objective}

\newtheorem{mejora}[teorema]{Improvement}
\newtheorem{ejercicio}[teorema]{Exercise}

\newtheorem{observacion}[teorema]{Observation}
\newtheorem{remark}[teorema]{Remark}
\newtheorem{computacion}[teorema]{Computation}

\newtheorem{notacion}[teorema]{Notation}
\newtheorem{problem}[teorema]{Problem}
\newtheorem{reminder}[teorema]{Reminder}
\newtheorem{metafora}{Metaphor}
\newtheorem{roadmap}{Roadmap}

\newtheorem{step}{Step}

\DeclareMathOperator{\mirror}{mirror}
\DeclareMathOperator{\innr}{innr}

\DeclareMathOperator{\bou}{bou}
\DeclareMathOperator{\rel}{rel}
\DeclareMathOperator{\asy}{asy}

\DeclareMathOperator{\E}{E}

\DeclareMathOperator{\tot}{tot}

\DeclareMathOperator{\mult}{multi}

\DeclareMathOperator{\Des}{Des}
\DeclareMathOperator{\des}{des}

\DeclareMathOperator{\exc}{exc}

\DeclareMathOperator{\Sym}{Sym}

\DeclareMathOperator{\uni}{uni}

\DeclareMathOperator{\suma}{sum}

\DeclareMathOperator{\rcs}{rcs}

\DeclareMathOperator{\Ima}{Im}

\DeclareMathOperator{\tr}{tr}

\DeclareMathOperator{\Her}{Her}

\DeclareMathOperator{\Rea}{Re}

\DeclareMathOperator{\hur}{hur}

\DeclareMathOperator{\card}{card}

\DeclareMathOperator{\linspan}{linspan}

\DeclareMathOperator{\id}{id}

\DeclareMathOperator{\Int}{int}

\DeclareMathOperator{\un}{uni}

\DeclareMathOperator{\diag}{diag}

\DeclareMathOperator{\sym}{sym}

\DeclareMathOperator{\eigval}{eigval}

\DeclareMathOperator{\bd}{bd}
\DeclareMathOperator{\cl}{cl}

\DeclareMathOperator{\bi}{bi}
\DeclareMathOperator{\vol}{vol}

\DeclareMathOperator{\rt}{rt}

\DeclareMathOperator{\coeff}{coeff}
\DeclareMathOperator{\shortcoeff}{C}
\DeclareMathOperator{\shortdeg}{D}

\DeclareMathOperator{\rec}{rec}
\DeclareMathOperator{\trun}{trun}

\newcommand{\comm}[1]{}
\newcommand{\rotA}{{\mathpalette\rotaA\relax}}
\newcommand{\rotaA}[2]{\rotatebox[origin=c]{180}{$A$}}

\newcommand{\rotaar}[2]{\rotatebox[origin=c]{180}{$r$}}

\newcommand{\rotae}[2]{\rotatebox[origin=c]{180}{$e$}}

\newcommand{\rotaE}[2]{\rotatebox[origin=c]{180}{$E$}}

\newcommand{\rotaM}[2]{\rotatebox[origin=c]{180}{$M$}}

\title{\bf Your title}

\author{Your name}
\date{University of Konstanz}

\makenoidxglossaries

\GlsXtrEnablePreLocationTag{Page: }{Pages: }

\newglossaryentry{brackn}{%
name=\ensuremath{[n]},
description={The interval of integers up to $n\in\mathbb{N}$, i.e., $\{1,\dots,n\}\subseteq\mathbb{N}$}
}

\newglossaryentry{ptilde}{%
name=\ensuremath{\Tilde{p}},
description={For $p\in\mathbb{R}[\mathbf{x}]$ a polynomial, $p+y p^{(1)}\in\mathbb{R}[\mathbf{x},y]$}
}

\newglossaryentry{MLMPDR}{%
name=MLMPDR,
description={Shorthand for the term \textit{monic linear (symmetric) matrix polynomial determinantal representation}. We add \textit{of size $s$} if such size $s\in\mathbb{N}$ is made explicit as MLMPDR($s$)}
}

\newglossaryentry{LMP}{%
name=LMP,
description={Shorthand for the term \textit{linear (symmetric) matrix polynomial}. We add \text{of size $s$} if such size $s\in\mathbb{N}$ is made explicit as LMP($s$)}
}

\newglossaryentry{MLMP}{%
name=MLMP,
description={Shorthand for the term \textit{monic linear (symmetric) matrix polynomial}. We add \text{of size $s$} if such size $s\in\mathbb{N}$ is made explicit as MLMP($s$)}
}

\newglossaryentry{GLC}{%
name=GLC,
description={Shorthand for the \textit{generalized Lax conjecture}}
}

\newglossaryentry{RZ}{%
name=RZ,
description={Shorthand for the property of being \textit{real-zero} for a polynomial $p\in\mathbb{R}$}
}

\newglossaryentry{renegpk}{%
name=\ensuremath{p^{(k)}},
description={$k$-th Renegar derivative of $p\in\mathbb{R}[\mathbf{x}]$}
}

\newglossaryentry{homop}{%
name=\ensuremath{p^{h}},
description={Homogenization of $p\in\mathbb{R}[\mathbf{x}]$ using the additional variable $x_{0}$}
}

\newglossaryentry{bda}{%
name=\ensuremath{\bd(A)},
description={Topological border of $A\subseteq X$ for $(X,\tau)$ a topological space}
}

\newglossaryentry{padirection}{%
name=\ensuremath{p_{A}},
description={$p(\sum_{i=1}^{l}x_{i}a_{i})$ for a finite tuple of directions $A=(a_{1},\dots,a_{l})$ with $a_{i}\in\mathbb{R}^{n}$ for all $i\in[l]$ and $p\in\mathbb{R}[\mathbf{x}]$}
}

\newglossaryentry{sadirection}{%
name=\ensuremath{S_{a}(p)},
description={$\{\lambda\in\mathbb{R}\mid\lambda a\in S(p)\}$ for a direction $a\in\mathbb{R}^{n}$ and a polynomial $p\in\mathbb{R}[\mathbf{x}]$}
}

\newglossaryentry{phat}{%
name=\ensuremath{\hat{p}},
description={For $p\in\mathbb{R}[\mathbf{x}]$ a polynomial with $d=\deg(p)$, $\sum_{k=0}^{d}\frac{1}{k!}y^{k}p^{(k)}\in\mathbb{R}[\mathbf{x},y]$}
}

\newglossaryentry{overlinep}{%
name=\ensuremath{\overline{p}},
description={For $p\in\mathbb{R}[\mathbf{x}]$ a polynomial, $\sum_{k=0}^{2}\frac{1}{k!}y^{k}p^{(k)}\in\mathbb{R}[\mathbf{x},y]$}
}

\newglossaryentry{coefmp}{%
name=\ensuremath{\shortcoeff(m,p)},
description={Coefficient of the monomial $m$ in the polynomial $p\in\mathbb{R}[\mathbf{x}]$}
}

\newglossaryentry{degp}{%
name=\ensuremath{\shortdeg(p)},
description={A shorter notation more adequate for writing computations for the degree of the polynomial $p\in\mathbb{R}[\mathbf{x}]$}
}

\newglossaryentry{overaa}{%
name=\ensuremath{\overline{a}},
description={$(a,0)\in\mathbb{R}^{n+1}$ for $a\in\mathbb{R}^{n}$}
}

\newglossaryentry{overa}{%
name=\ensuremath{\overline{A}},
description={$(\overline{a_{1}},\dots,\overline{a_{l}},(0,\dots,0,1))\in(\mathbb{R}^{n}\times\{0\})^{l}\times(\{0\}^{n}\times\{1\})$ for a tuple of vectors $A=(a_{1},\dots,a_{l})\in(\mathbb{R}^{n})^{l}$ with $a_{i}\in\mathbb{R}^{n}$}
}

\newglossaryentry{rrz}{%
name=\ensuremath{\mathbb{R}_{\mbox{RZ}}[\mathbf{x}]},
description={Set of real-zero polynomials $p\in\mathbb{R}[\mathbf{x}]$ verifying $p(0)=1$ and $\deg(p)>0$}
}

\newglossaryentry{mni}{%
name=\ensuremath{M_{n,i;(j_{1},\dots,j_{m}),(i_{1},\dots,i_{k})}},
description={Mold monomial matrix given by the outer product of the vector $(1, x_{j_{1}}, \cdots, x_{j_{m}}, x_{i}x_{i_{1}}, \cdots, x_{i}x_{i_{k}})$ with $k,m,i,i_{a},j_{b}\in[n]$ for all $a\in[k], b\in[m]$ and $i_{a}\neq i_{a'}$ whenever $a\neq a'$ and $j_{b}\neq j_{b'}$ whenever $b\neq b'$. We denote it simply $M_{n,i}$ when we take $(j_{1},\dots,j_{m})=(i_{1},\dots,i_{k})=(1,\dots,n)$}
}

\newglossaryentry{ftg}{%
name=\ensuremath{f^{T}g},
description={For $F$ a field, indexed scalar product over the common inedexing set $I$ of $f\colon A\subseteq I\to F$ and $g\colon B\subseteq I\to F$}
}

\newglossaryentry{mn1}{%
name=\ensuremath{M_{n,\leq1}},
description={Mold monomial matrix of the relaxation given by the outer product $(1, x_{1}, \cdots, x_{n})^{T}(1, x_{1}, \cdots, x_{n})$ of the vector $(1, x_{1}, \cdots, x_{n})$ of monomials of degrees up to $1$ in the monomial basis}
}

\newglossaryentry{mpd}{%
name=\ensuremath{M_{p,d}},
description={Pencil associated to $p$ with respect to the virtual degree $d\in\mathbb{N}$. We write $M_{p}:=M_{p,\deg(p)}$}
}

\newglossaryentry{spd}{%
name=\ensuremath{S_{d}(p)},
description={Spectrahedron associated to $p$ with respect to the virtual degree $d\in\mathbb{N}$. We write $S(p):=S_{\deg(p)}(p)$}
}

\newglossaryentry{circ}{%
name=\ensuremath{L\circledcirc M},
description={Tensor in $B^{s_{1}\times\cdots\times s_{k}}$ obtained from applying the molding map $L\colon A\to B$ to each entry of the mold tensor $M\in A^{s_{1}\times\cdots\times s_{k}}$}
}

\newglossaryentry{lpd}{%
name=\ensuremath{L_{p,d}},
description={$L$-form associated to $p\in\mathbb{R}[\mathbf{x}]$ with virtual degree $d$. We write $L_{p}:=L_{p,\deg(p)}$}
}

\newglossaryentry{trundp}{%
name=\ensuremath{\trun_{d}(p)},
description={Truncation of the polynomial $p\in\mathbb{R}[\mathbf{x}]$ at degree $d\in\mathbb{N}$}
}

\newglossaryentry{expp}{%
name=\ensuremath{\exp(p)},
description={Exponential of the polynomial $p\in\mathbb{R}[\mathbf{x}]$}
}

\newglossaryentry{logp}{%
name=\ensuremath{\log(p)},
description={Logarithm of the polynomial $p\in\mathbb{R}[\mathbf{x}]$}
}

\newglossaryentry{rcs}{%
name=\ensuremath{\rcs(p)},
description={the rigidly convex set of the real-zero polynomial $p\in\mathbb{R}[\mathbf{x}]$}
}

\newglossaryentry{hur}{%
name=\ensuremath{\hur_{\alpha}(A_{1},\dots,A_{n})},
description={$\alpha$-Hurwitz product of $A_{i}\in\mathbb{C}^{d\times d}$ for $i\in[n]$ and $\alpha\in\mathbb{N}^{n}_{0}$}
}

\newglossaryentry{ma}{%
name=\ensuremath{m_{A}},
description={The only minimum degree monic monomial in the variables strictly appearing in the polynomial $A(\mathbf{x})\in\mathbb{R}[\mathbf{x}]$ that produces a polynomial after performing the multiplication $m_{A}A(\frac{1}{x_{1}},\dots,\frac{1}{x_{n}})$}
}

\newglossaryentry{rec}{%
name=\ensuremath{\rec(p)},
description={Reciprocal polynomial of the polynomial $p$. It can be applied to univariate or multivariate polynomials}
}

\newglossaryentry{mirrorx}{%
name=\ensuremath{\mirror(\mathbf{x})},
description={For $\mathbf{x}=(x_{1},\dots,x_{n})$ this gives $(x_{n},\dots,x_{1})$}
}

\newglossaryentry{aatob}{%
name=\ensuremath{A^{a\to b}},
description={The subset of permutations $\sigma\in A\subseteq\mathfrak{S}_{n}$ sending $a$ to $\sigma(a)=b$}
}

\newglossaryentry{asyA}{%
name=\ensuremath{\asy_{A}},
description={Asymptotic exponential growth map for the sequence of polynomials $A:=\{A_{n}\}_{n=0}^{\infty}$ with each polynomial $A_{n}\in\mathbb{R}[x_{1},\dots,x_{n}]$. Its domain can either be the space of real sequences $\mathbb{R}^\mathbb{N}$ or the space $S:=\{w\in(\bigcup_{n}\mathbb{R}^{n})^\mathbb{N}\mid w_{i}\in\mathbb{R}^{i}\}$ of vector sequences of index size}
}

\newglossaryentry{anw}{%
name=\ensuremath{A_{n,w}},
description={Weighted $n$-th multivariate Eulerian polynomials with weight $w\in\mathbb{R}^{n+1}$}
}

\newglossaryentry{snw}{%
name=\ensuremath{\mathfrak{S}_{n}^{w}},
description={The set of weighted permutations in $\mathfrak{S}_{n}$ with weight $w\colon[n]\to\mathbb{R}$}
}

\newglossaryentry{rela}{%
name=\ensuremath{\rel_{A}},
description={General relaxation map of the sequence of polynomials $A:=\{A_{n}\}_{n=0}^{\infty}$ with each polynomial $A_{n}\in\mathbb{R}[x_{1},\dots,x_{n}]$}
}

\newglossaryentry{boua}{%
name=\ensuremath{\bou_{A}},
description={General bound map of the sequence of polynomials $A:=\{A_{n}\}$ with each polynomial $A_{n}\in\mathbb{R}[x_{1},\dots,x_{n}]$}
}

\newglossaryentry{chijk}{%
name=\ensuremath{\chi_{j}^{(k)}},
description={$\frac{E(k+1,j-1)E(k+1,j+1)}{E(k+1,j)^2}-\frac{E(k+1,j-2)E(k+1,j+2)}{E(k+1,j)^2}+\frac{E(k+1,j-3)E(k+1,j+3)}{E(k+1,j)^2}-\cdots$}
}

\newglossaryentry{ujks}{%
name=\ensuremath{u_{j}^{(k,s)}},
description={$\binom{k+2}{j}\left(\frac{s+1-j}{s+1}\right)^{k+1}$}
}

\newglossaryentry{uni}{%
name=\ensuremath{\uni(n)},
description={Best possible bound obtained through the application of the relaxation to the univariate Eulerian polynomial $A_{n}$}
}

\newglossaryentry{multv}{%
name=\ensuremath{\mult_{v(n)}(n)},
description={Bound obtained through the application of the relaxation to the multivariate Eulerian polynomial $A_{n}$ and the linearization of the resulting LMI through the vector $v(n)$}
}

\newglossaryentry{volA}{%
name=\ensuremath{\vol(A)},
description={The volume of the Lebesgue-measurable set $A\subseteq\mathbb{R}^{n}$}
}

\newglossaryentry{mn}{%
name=\ensuremath{M_{n}},
description={LMP given by the application of the relaxation to the $n$-th RZ multivariate Eulerian polynomial $A_{n}(\mathbf{x})$}
}

\newglossaryentry{enk}{%
name=\ensuremath{E(n,k)},
description={Eulerian number for the number of permutations in $\mathfrak{S}_{n}$ having exactly $k$ descents}
}

\newglossaryentry{l}{%
name=\ensuremath{f\sim g},
description={For the sequences $f$ and $g$, have the limit $\lim_{n\to\infty}\frac{f(n)}{g(n)}=1$}
}

\newglossaryentry{m}{%
name=\ensuremath{f\lesssim g},
description={For the sequences $f$ and $g$, there exist another sequence $h$ verifying $f\sim h\leq g$}
}

\newglossaryentry{qin}{%
name=\ensuremath{q_{i}^{(n)}},
description={The $i$-th root of the $n$-th univariate Eulerian polynomial counting from the left to the right}
}

\newglossaryentry{desxy}{%
name=\ensuremath{\Des_{X,Y}},
description={Set of adequate descents of a permutation in $\mathfrak{S}_{n+1}$ for the pair of sets $X,Y\subseteq\mathbb{N}$}
}

\newglossaryentry{Rns}{%
name=\ensuremath{R(n,S)},
description={Set of permutations in $\mathfrak{S}_{n+1}$ that descend \textit{exactly} at $S\subseteq[n+1]$. When $n$ is fixed or clear by the context we simply denote these numbers $R(S)$ for any subset $S\subseteq[n+1]$. See also Remark \ref{rshortening}.}
}

\newglossaryentry{XS}{%
name=\ensuremath{(X)(\mathfrak{S}_{n+1})},
description={Set of permutations in $\mathfrak{S}_{n+1}$ whose descet top set is included in the set $X$}
}

\newglossaryentry{pnsxy}{%
name=\ensuremath{P_{n,s}^{X,Y}},
description={Number of permutations $\sigma\in\mathfrak{S}_{n}$ with exactly $s$ adequate descents for the pair $(X,Y)$. When $Y=\mathbb{N}$ we omit it and write $P_{n,s}^{X}$}
}

\newglossaryentry{S}{%
name=\ensuremath{\mathfrak{S}},
description={Union of all the sets of permutations on ordered sets of the form $[n]$ for $n\in\mathbb{N}$}
}

\newglossaryentry{des}{%
name=\ensuremath{\des},
description={Descent counting function}
}

\newglossaryentry{an}{%
name=\ensuremath{A_{n}},
description={The $n$-th Eulerian polynomial either univariate or multivariate depending on the dimension of the argument and the context in the text}
}

\newglossaryentry{en}{%
name=\ensuremath{E_{n}},
description={The $n$-th Eulerian-Frobenius polynomial either univariate or multivariate depending on the dimension of the argument and the context in the text}
}

\newglossaryentry{H}{%
name=\ensuremath{\mathcal{H}},
description={Open upper half-plane $\{z\in\mathbb{C}\mid\Ima(z)>0\}$}
}

\newglossaryentry{dt}{%
name=\ensuremath{\mathcal{DT}(\sigma)},
description={Descent top set of the permutation $\sigma\in\mathfrak{S}$}
}

\newglossaryentry{at}{%
name=\ensuremath{\mathcal{AT}(\sigma)},
description={Ascent top set of the permutation $\sigma\in\mathfrak{S}$}
}

\newglossaryentry{sample1}{name={sample1},description={first example}}
\newglossaryentry{sample2}{name={sample2},description={second example}}

\begin{document}



\newgeometry{hmarginratio=1:1,bottom=1cm}


\thispagestyle{empty}
\begin{center}
  \begin{sffamily}
    \begin{bfseries}
      \begin{LARGE}
	\mbox{\ul{A Method for Establishing Asymptotically Accurate}}\\%
 \mbox{\ul{Bounds for Extreme Roots of Eulerian Polynomials}}\\%
 \mbox{\ul{Using Polynomial Stability Preservers}}\\%
	\vspace{15mm}
      \end{LARGE}
       \end{bfseries}
      \begin{Large}
	\Large
   Dissertation zur Erlangung des akademischen Grades
eines Doktors der Naturwissenschaften (Dr. rer. nat.) \\
      \end{Large}

      \begin{large}
	\vspace{12mm}
	vorgelegt von \\[0.6\baselineskip]
	Alejandro Gonz\'alez Nevado\\
	\vspace{12mm}
	an der\\[\baselineskip]
	{\includegraphics[width=0.5\textwidth]{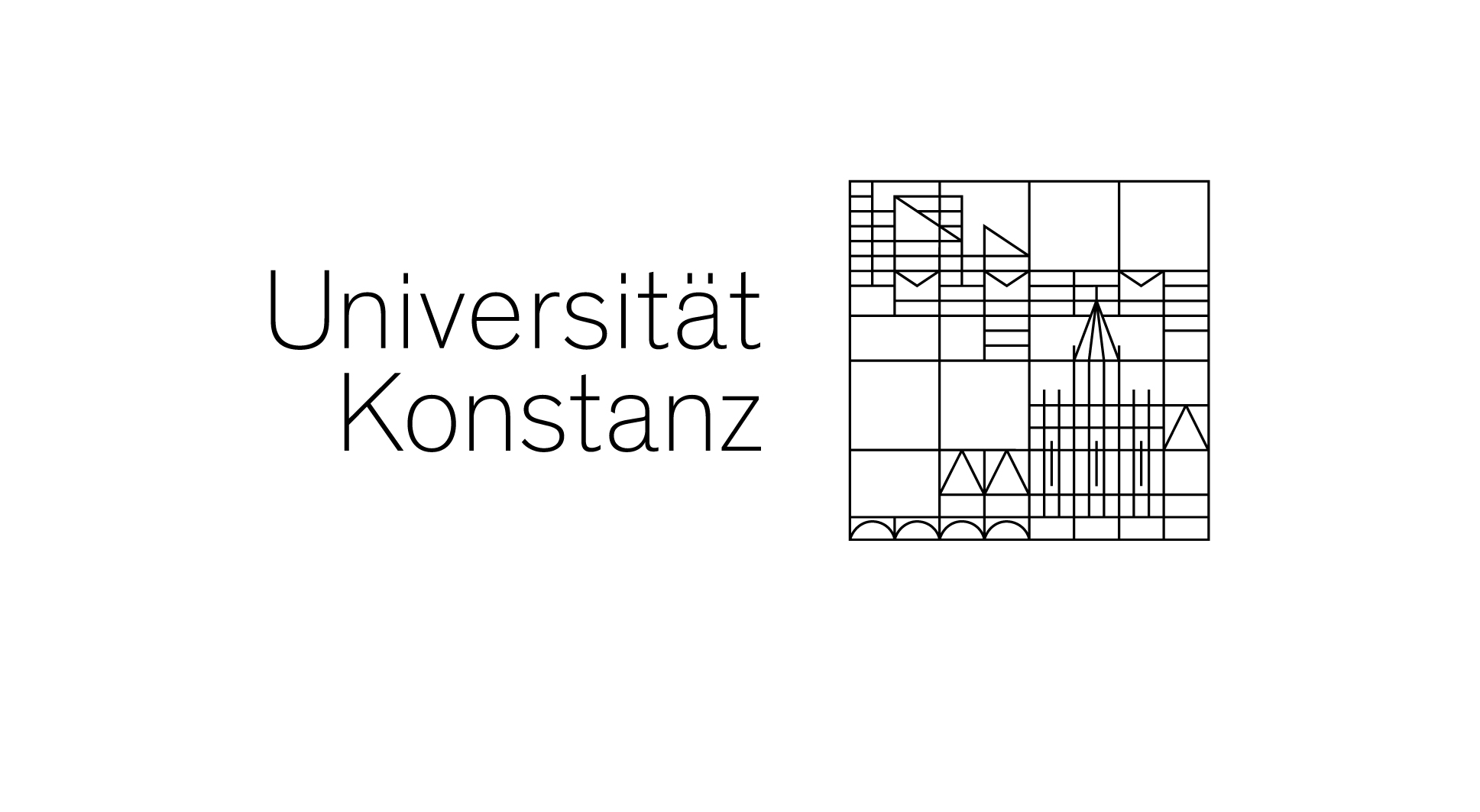}}\\
	\vspace{12mm}
	Mathematisch-Naturwissenschaftliche Sektion\\[0.6\baselineskip]
	Fachbereich f\"ur Mathematik und Statistik\\
      \end{large}

  \end{sffamily}

\end{center}

\begin{textblock*}{120mm}(20mm,235mm) 
\large\mdseries\sffamily
  {
    \noindent
    1. Referent: Prof. Dr. Markus Schweighofer \\

    \noindent
    2. Referent: Prof. Dr. Cordian Riener
  }

      \null
  \vspace{4mm}
    \large\mdseries\sffamily
      \noindent
    Tag der mündlichen Prüfung: 12. September 2025 (Konstanz)
\end{textblock*}


\restoregeometry





	
\pagestyle{fancy}
\fancyhead{}
\fancyhead[RO,LE]{\thepage}{}
\fancyfoot{}

	\pagenumbering{roman}
    
\renewcommand*{\glossaryname}{List of Symbols}

\printnoidxglossary[sort=use]
\clearpage
 \setcounter{chapter}{-1}

 \setlength{\emergencystretch}{3em}%

 \chapter*{Formalities on Writing a Dissertation}

\setlength{\emergencystretch}{3em}%

This short and unnumbered chapter serves the double purpose of summarizing the objectives of this thesis and of giving the corresponding thanks to all the people that made this thesis possible. Thus, we have four sections to fulfill here.
The first two are the summary (in English) and its corresponding translation (in German).

The next section is used to thank and acknowledge all the people that actively made possible this research. I speak here about \textit{making possible} in a very ample sense. Hence, this position include the people that gave me ideas, the people that supported me when these ideas did not work, the people that congratulated me and felt happy for me when they did actually work, the people that accompanied me through this journey (or part of it) by just being there and the people that made possible my survival and existence until the point in which I was able to declare that this thesis was \textit{finally} finished.

Finally, I use the last section of this chapter to acknowledge and thank the work of all the mathematicians that came before me and that made this work possible, one way or another, without actively knowing me nor the objectives of this research, just by existing at some point in history and making work and progress in their fields, serving thus either as foundations for the work that I develop here or as inspiration for my younger self for falling unconditionally in love with this beautiful area of knowledge. That is, I want to thank also here some of all these mathematicians whose theorems, definitions and insights set me in the place and position of pursuing mathematics as the most and highest noble art a human being can aspire to command, lead and advance.

\section*{Summary}

We develop tools to study, understand and bound extreme roots of multivariate real zero polynomials globally. This is done through the use of a relaxation that approximates their rigidly convex sets. This relaxation can easily be constructed using the degree $3$ truncation of the polynomial and it produces in this way a spectrahedron whose computation is relatively easy and whose size is relatively small and depending solely on the number of variables of the polynomial. As we know that, in order to be able to produce in general spectrahedral representations of rigidly convex sets, it is necessary to build matrices of very big size, we try, analyze and experiment with several constructions that could increase the size of the matrices of the relaxation. These constructions are based principally in two main approaches: adding information about higher degree monomials or non-trivially increasing the number of variables of the original polynomial. We explore these two construction first in a general setting and see that it is necessary to particularize to certain families of polynomials in order to make them work. In particular, we are able to prove that increasing the number of variables improves the behavior of the relaxation along the diagonal in the case of Eulerian polynomials. We see that applying the relaxation to multivariate Eulerian polynomials and then looking at the univariate polynomials injected in their diagonals produces an exponential asymptotic improvement in the bounds provided. We compare these bounds with other bounds that have appeared previously in the literature and refine these previous bounds in order to study how close do the bounds provided by the relaxation are to the actual roots of the Eulerian polynomials. As the univariate Eulerian polynomials present symmetries that helped us to understand better the structure of their roots, we also study generalizations of these notions of symmetry to multivariate Eulerian polynomials. These generalizations of symmetry shed new light on the combinatorial object these polynomials encode. Finally, the combination of all these techniques coming from real algebraic geometry, the theory of stability preservers, the numerical methods for root estimations and the study of symmetries of polynomials encoding combinatorial objects that we use in this thesis along the path suggested by Eulerian polynomials suggests to us that there is much more happening under the surface. This additional knowledge would affect and increase our understanding of many more families of polynomials emerging as generating polynomials associated to combinatorial objects, this is, in the same way as Eulerian polynomials emerge. Thus, the insights provided here by the study of Eulerian polynomials allow us to propose a further reaching mathematical program aimed at studying and delving deeper into the relations, connections and phenomena first devised and exemplified by the work presented here on Eulerian polynomials. We name this program the Mindelsee Program and find that it encourages connecting paths that could open bridges between many different areas of mathematics promoting thus the study of new families of (mainly real zero) polynomials and of new techniques aimed at improving our abilities to perform the task of understanding better the overall structure of their roots in depth.

\section*{Zusammenfassung}

Wir entwickeln Werkzeuge, um extreme Nullstellen multivariater RZ-Polynome global zu untersuchen, zu verstehen und zu beschr\"anken. Dies geschieht durch die Verwendung einer Relaxierung, die ihre starr konvexe Menge approximiert. Diese Relaxierung kann einfach mithilfe der Grad-$3$-Trunkierung des Polynoms konstruiert werden und erzeugt auf diese Weise ein Spektraeder, dessen Berechnung relativ einfach ist und dessen Größe vergleichsweise klein bleibt und ausschließlich von der Anzahl der Variablen des Polynoms abhängt. Da wir wissen, dass man im Allgemeinen sehr große Matrizen ben\"otigt, um eine starr konvexe Menge als Spektraeder darzustellen (falls m\"oglich), testen, analysieren und experimentieren wir mit mehreren Konstruktionen, die die Größe der Matrizen der Relaxierung erhöhen könnten. Diese Konstruktionen basieren hauptsächlich auf zwei Ansätzen: das Hinzufügen von Informationen über Monome höheren Grades oder der nicht-trivialen Erhöhung der Anzahl der Variablen des ursprünglichen Polynoms. Wir erforschen diese beiden Konstruktionen zunächst in einem allgemeinen Rahmen und stellen fest, dass es notwendig ist, sie auf bestimmte Familien von Polynomen zu spezialisieren, um sie wirksam zu machen. Insbesondere gelingt es uns zu beweisen, dass eine Erhöhung der Anzahl der Variablen das Verhalten der Relaxierung entlang der Diagonalen im Fall der eulerschen Polynome verbessert. Wir sehen, dass die Anwendung der Relaxierung auf die multivariaten eulerschen Polynome und die anschließende Betrachtung entlang ihrer Diagonalen als univariate Polynome eine exponentielle asymptotische Verbesserung der bereitgestellten Schranken bewirkt. Wir vergleichen diese Schranken mit anderen Schranken, die bereits zuvor in der Literatur erschienen sind, und verfeinern diese früheren Schranken weiter, um zu untersuchen, wie nah die durch die Relaxierung gelieferten Schranken an den tatsächlichen Nullstellen der eulerschen Polynome liegen. Da die univariaten eulerschen Polynome Symmetrien aufweisen, die uns geholfen haben, die Struktur ihrer Nullstellen besser zu verstehen, untersuchen wir auch Verallgemeinerungen dieser Begriffe von Symmetrie bei multivariaten eulerschen Polynomen. Diese Verallgemeinerungen von Symmetrie werfen neues Licht auf das kombinatorische Objekt, das diese Polynome kodieren. Schließlich legt die Kombination all dieser in dieser Dissertation verwendeten Techniken aus der reellen algebraischen Geometrie, der Theorie stabilitätserhaltender Operatoren, der numerischen Methoden zur Nullstellenabschätzung und der Untersuchung von Symmetrien von Polynomen, die kombinatorische Objekte kodieren, entlang des von eulerschen Polynomen vorgeschlagenen Weges nahe, dass hinter den Kulissen weitaus mehr geschieht. Dieses zusätzliche Wissen würde unser Verständnis vieler weiterer Familien von Polynomen erweitern, die als erzeugende Polynome kombinatorischer Objekte auftreten, das heißt, auf dieselbe Weise wie eulersche Polynome entstehen. Somit erlaubt uns die durch die Untersuchung der eulerschen Polynome gewonnene Einsicht, ein weiterreichendes mathematisches Programm vorzuschlagen, das darauf abzielt, die hier erstmals erkannten und exemplifizierten Beziehungen, Verbindungen und Phänomene eingehender zu untersuchen und zu vertiefen. Wir nennen dieses Programm das Mindelsee Programm und stellen fest, dass es Wege verbindet, die Brücken zwischen vielen verschiedenen Bereichen der Mathematik schlagen könnten, und somit das Studium neuer Familien von (hauptsächlich RZ-)Polynomen sowie neuer Techniken zur Verbesserung unserer Fähigkeiten zur eingehenden Analyse der Gesamtstruktur ihrer Nullstellen fördert.

\section*{Acknowledgements to those that made this possible}

I will proceed here in chronological order so that I can avoid, as far as possible, forgetting someone along the way.
I could clearly not be here if my parents (Rosa and Salvador) had used a condom that night of, I guess, mid April 1994.
Besides my conception and beyond that moment, my parents infused into me the spirit of curiosity, creativity and love for science and knowledge that made me, since I was very small, look into the direction of the purest knowledge endeavours I could think of. That ample sight, in a process that can only happen through the refinement in tastes that only aging can provide to us, made me evolve from liking biology (yes, and also geology) to chemistry, then from chemistry to physics and, finally, from physics to the pureness of the eternal truths that only the mathematical thought can provide and produce. Thus, they were the main people responsible for me following my current track. Also, inside my family, I cannot forget to mention my brother (Salvi), who always made me look further and beyond the mere surface of things into more
creative and, at the same time, strangely deep realms of reality. I think that I could not thank my brother without mentioning that he is the only person I know able to transform, like if he was an alchemist of thoughts, the \textit{concerning} into \textit{insightful} in unexpectedly comically hilarious ways while providing perspectives to views that were not there before and solving problems that never existed before he oversaw a solution. My brother created three small people that also showed to me the power of connecting creativity and playfulness, both in life and in science. I think I learned from my nephews (Salvita, Jairo and Angel) more than what they learned from me.

Going beyond my blood family, I have to mention \textit{the family we choose}, that is, my longtime friends. I have to thank here my group of friends from high school (Andrés, Quique, Fernando, Álvaro, Antonio, Edu, Jose and Pedro). They are always there to have an interesting conversation with (or to have crazy parties when we coincide in Málaga). But those with whom I have spoken the most about this thesis and the processes of researching its topics and writing it are my mathematician friends and two of the most clever persons I grew up with (Ricardo and Felix). Ricardo has always been there when I needed to have a rest from researching or when I just needed to walk while talking about some results or processes that frustrated me. He always provides fresh ideas and interesting points of view. And, even more important than that, he is able to do that in insuperable funny ways. If there is someone I can talk about every topic and joke about everything with and expect to be understood, that person is Ricardo with his vast knowledge of, basically, \textit{everything}. I cannot forget mentioning Felix, who always pushed me into surpassing myself and getting things done. He also always served as, not just a friend, but as an inspiration towards different ways of looking at mathematics (and its applications).

Along our journey towards intellectual maturity, we have many teachers and professors that set us in the way to be what we become. In this respect, I feel that I have to thank Pedro and Carlos from my high school (IES Ciudad de Melilla), who showed to me the beauty of real mathematics when the boring mathematics that I had to do in high school only was not able to do so. This grew into my university years, where many professors served as inspiration into mathematical research and its beauty. Among these, the main that come now to my mind now are Aniceto Murillo Mas, Miguel Ángel Gómez Lozano, Amable García Martín, Cándido Martín González, Mercedes Siles Molina, Francisco Javier Turiel Sandín, Daniel Girela Álvarez, Antonio Ángel Viruel Arbaizar, Nieves Álamo Antúnez, Gonzalo Aranda Pino, Antonio Díaz Ramos, Carlos María Parés Madroñal, Diego Gallardo Gómez, Carmen del Castillo Vázquez, María del Carmen Morcillo Aixelá, Santiago Marín Malave, Jose Ramón Brox López and Antonio Fernández López. I am sure I am forgetting many.

When I finished my master in Malaga, I moved to Germany thanks to Markus and the POEMA network. Here I also met marvelous mathematicians that showed to me new realms of mathematics I did not know so well before. This is where my research career started and, therefore, these are also the people that appeared during the time I was thinking about the results that I will present in this thesis.

The first people I met in Konstanz were Markus and his Ph.D. student Alexander when I came for the first interview. That day, I was also able to meet other Ph.D. students, like Julian Vill, Thorsten Mayer and Christoph Schulze. These were my first contacts with the mathematical community in Konstanz. When I finally became a POEMA, I also met all the POEMA staff that helped me to navigate this research project. I can only be infinitely grateful towards the whole POEMA network. Specially, within this network, I thank the people with whom I mainly worked: Markus and the two professors of my secondment Giorgio in Firenze and Mohab in Paris.

During the time that I took to complete this dissertation, Markus was always a helpful and encouraging supervisor who helped me to put my ideas down to Earth and polish them so other people could understand me better. I will always be grateful for his trust on me and the freedom he gave to me. His support and enthusiasm about the topic encouraged me to continue developing this project until the state in which it is presented here now.

We have, unfortunately, lived through very interesting times during the time that I needed to complete this dissertation. There was a global pandemic and a war in Europe. Even a volcano and a weird behavior of the Earth's magnetic field happened during this time. These events irremediably affected the development of this research. I have to be very grateful, especially during the pandemic, to Verena and her parents (Simone and Markus) and grandparents (Ute and Kurt), who emotionally supported me during the hardest parts of the pandemic when the cities became empty and going back to Spain was almost impossible due to the draconian travel restrictions we had to endure. They took me in their home as one more and I did not feel so lonely in the middle of the isolation thanks to that. Thus, I will always be grateful to Verena, who accompanied me during the first steps of this project.

At the end, everything changed. The pandemic was over and Konstanz recovered its normal activity. It was therefore not until the very end of this project that I could finally fully enjoy all the things that this beautiful city at the shores of the Bodensee offers. During this last period of this project and without pandemic restrictions anymore, I finally could made local friends. I am grateful that I could met all these beautiful people that pushed and encouraged me in the last stages of this project. I want to thank my officemate and Markus' Ph.D. student David Sawall, with whom I have had many interesting and helpful discussions (not always about mathematics \textit{stricto sensu}) during the last months of writing. He also helped me with the German translation of the summary that you could read above. I want to thank Alma and the Adtejandro group (I did not choose the name) for the beautiful moments I lived with them this past year. Elena, Kalle and Lazkin were there to celebrate with me the victories of Real Madrid in Champions League and the Spanish National Team in the Euro 2024. I enjoyed the walks with Maja at the end of summer and the talks with Selma. I enjoyed the night walks with my friends Jakob and Rufat as much as our time at the shore of the Bodensee. And, finally, I want to specially thank Anna, who has been possibly the greatest support I had at the very end of this project and who has given me times in which I could forget about all the worries and tasks associated to researching and therefore focus into enjoying the \textbf{beauty} of immediate life. She is the most clever person I found in my last year of this project and always provides me with new points of view about many things I feel I can only talk about with her. But we have to remember to \textbf{actually eat something} from time to time.

All in all, as every human endeavor, this project would not have been possible without all the people mentioned here and without all that other people whose names I do not know or I forgot along the way but that made sustainable that I could dedicate myself to mathematics during this time. For that reason, I want to thank them all and wish that the small piece of mathematics that they helped, in some sense, to create ends up touching positively in their lives. Thanks again to all of you!

\section*{Dedication to those who inspired me}

In the section above I thanked all the people that, more or less, were actively present during my time in this project or before. But there were people even before this immediate time that contributed to me wanting to pursue mathematics. These are the giants over which I could climb to look at the horizon. I think I would not be the mathematician I am today without reading the books of Carlos Ivorra that first introduced me into realms of mathematics unexplored for a child when I was still far from the university. I would not see mathematics the same way without having read the thoughts of Gauss, Euler and Grothendieck about this beautiful intellectual endeavor. I might not be so passionate, persevering and resilient about pursuing pure knowledge without having learned about the passion that Galois, Abel and Cantor put into a field that, at first, looked down on them and their discoveries. My philosophy towards pursuing the pure beauty of knowledge would neither be there without the many philosophers that pushed me into believing that pure knowledge, just by itself, constitutes the most honorable endeavor the human mind can pursue (Thales, Pythagoras, Socrates, Plato, Aristotle, Euclid, Suárez, Kant, Fichte, Hegel, Nietzsche and S{\'a}nchez Meca). I would have never known about the deep connections between my two favorite human endeavors without Jes{\'u}s Moster{\'i}n and his translation to Spanish of G{\"o}del full works. I also learned a lot from the dedication that Riemann, Poincar{\'e}, Hilbert, Dedekind, Dirichlet, Weierstrass or Cauchy put into the field. And our modern world would be impossible without Turing, of course. I left many more I did not thank because they do not come to my mind currently, but they should, of course, all be here. All the people that contributed to the development of this beautiful discipline have a space dedicated solely to them in my heart, as they should have in the hearts of all human beings that, without knowing it, benefit from the world they devoted their lives to construct. It had to be them the ones with whom this formal thanking chapter is closed before driving ourselves into the mathematics of this thesis.

\setlength{\emergencystretch}{0em}%
 \newpage
 \newgeometry{a4paper,left=1in,right=1in,footskip=0.15in,headheight=2.0in,tmargin=2in,bmargin=1in}
\begin{center}
\usetikzlibrary{fit}
\thispagestyle{empty}
\setcounter{diagheight}{50}

\begin{center}\subsection*{Graph of conceptual (weak) dependencies of chapters and parts}\end{center}
\begin{chart}
\reqhalfcourse 45,45:{}{Chapter \ref{ChPreamble}}{}
\reqhalfcourse 45,35:{}{Chapter \ref{ChRelaxation}}{}
\reqhalfcourse 25,25:{}{Chapter \ref{ChLimi}}{}
\reqhalfcourse 45,25:{}{Chapter \ref{4}}{}
\reqhalfcourse 45,15:{}{Chapter \ref{ChPrimer}}{}
\reqhalfcourse 65,15:{}{Chapter \ref{ChCombi}}{}
\reqhalfcourse 45,5:{}{Chapter \ref{ChRoot}}{}
\reqhalfcourse 65,5:{}{Chapter \ref{ChWeight}}{}
\reqhalfcourse 45,-5:{}{Chapter \ref{ChCounting}}{}
\reqhalfcourse 25,-15:{}{Chapter \ref{ChApp}}{}
\reqhalfcourse 45,-15:{}{Chapter \ref{ChExplosion}}{}
\reqhalfcourse 65,-15:{}{Chapter \ref{ChComparisons}}{}
\reqhalfcourse 45,-25:{}{Chapter \ref{ChDLGCombi}}{}
\reqhalfcourse 60,-25:{}{Chapter \ref{ChDLGMeth}}{}
\reqhalfcourse 25,-25:{}{Chapter \ref{ChStructured}}{}
\reqhalfcourse 45,-35:{}{Chapter \ref{conclusion}}{}

\prereq 45,45,45,35:
\prereq 45,35,25,25:
\prereq 45,35,45,25:
\prereq 45,15,65,15:
\prereq 45,15,45,5:
\prereq 45,5,65,5:
\prereq 45,5,45,-5:
\prereq 45,-5,25,-15:
\prereq 25,-15,45,-15:
\prereq 45,-15,65,-15:
\prereq 65,-15,60,-25:
\prereq 60,-25,45,-25:
\prereq 45,-25,25,-25:
\prereq 25,-25,45,-35:
\prereqc 65,15,45,-35;-200:
\prereqc 65,5,45,-35;-30:

\coreq 65,5,65,15:
\coreq 25,-15,25,-25:
\coreq 45,25,45,15:

\begin{pgfonlayer}{courses}
\draw[dashed] ([shift={(1mm,2mm)}]x45y35.north east) rectangle ([shift={(-1mm,-1mm)}]x25y25.south west);
\end{pgfonlayer}

\begin{pgfonlayer}{courses}
\draw[dashed] ([shift={(1mm,2mm)}]x45y15.north east) rectangle ([shift={(-1mm,-1mm)}]x45y-5.south west);
\end{pgfonlayer}

\begin{pgfonlayer}{courses}
\draw[dashed] ([shift={(1mm,1mm)}]x65y15.north east) rectangle ([shift={(-2mm,-1mm)}]x65y5.south west);
\end{pgfonlayer}

\begin{pgfonlayer}{courses}
\draw[dashed] ([shift={(-1mm,-1mm)}]x25y-15.south west) rectangle ([shift={(1mm,2mm)}]x65y-15.north east);
\end{pgfonlayer}

\begin{pgfonlayer}{courses}
\draw[dashed] ([shift={(-1mm,-1mm)}]x25y-25.south west) rectangle ([shift={(1mm,2mm)}]x60y-25.north east);
\end{pgfonlayer}
\end{chart}
\end{center}
\restoregeometry

\tableofcontents
\clearpage

\listoffigures
\clearpage

\listoftables
\clearpage


\pagestyle{fancy}
\fancyhead[RO,LE]{\thepage}{}
\fancyfoot{}
	\fancyhead[RE]{\small \rightmark}
\fancyhead[LO]{\small \leftmark}
	\pagenumbering{arabic}
\chapter[Preamble]{Preamble to Understand the Architecture of this Dissertation}\label{ChPreamble}

\section{A mathological foreword}
Paul Halmos introduced in \cite{63d34b69-f100-336d-9993-1bf80b759b0d} the term `mathology' to describe pure mathematics as an art so he could split it from the applied `mathophysics' as these two appear connected in an almost indivisible way in the atom of knowledge that we decided to call collectively `mathematics'. Here we claim to write a mathological foreword to this dissertation because in these initial words we do not pretend to go technical and into the mathematical details that we will address afterwards. Contrary to that, in this very first section of this work, we plan to use a metaphorical language in order to describe and advance in a very high and abstract level what we wanted to accomplish here. High level and abstract, but not in the usual pure mathematical sense. We want to write a story about how the characters that appear in this work relate to each other. This will let us draw a picture in simple terms that will allow the readers to navigate this document with a conceptual representation in their minds while advancing through the next pages. That is, before giving the map of the lands we are set to explore here properly, we will proceed \`{a} la Tolkien in \textit{The Silmarillion}, and firstly give to the reader a cosmogony of how our music developed into the five themes that compose this work. The map we promised will come, but, \textit{in the beginning was the word}.

\begin{metafora}[Relaxation as a hammer]
The main object of our path here is the relaxation introduced by Schweighofer in \cite{main}. During our journey, the relaxation acts like this tool searching for a nail because we first had the relaxation and then we spend much of our time thinking about its applications, the nails.
\end{metafora}

We began our journey with a hammer and started searching for nails. The \textit{holy nail} we began pursuing was the Generalized Lax Conjecture (GLC, from now on) \cite[Section 6]{helton2007linear}. However, our attempts to attack it with our hammer mostly failed, leaving us with some frustration but, at the same time, with a deep knowledge about the properties and abilities of our hammer. In particular, during our journey we learned that a successful attack would require our hammer to be fortified

\begin{remark}[Fortifications of the abilities of the hammer]
We learned that, in order to improve our relaxation, we need to develop techniques allowing us to collect more information than what the original relaxation actually carried. For this, we had to explore the paths of the relaxation and these explorations showed us that the best bet to improve the strength of the hammer was to augment the number of variables of our polynomials keeping some fundamental property we require: \gls{RZ}-ness. See again \cite{helton2007linear}. This observation initiated our trip into the methods to \textit{gracefully} extend the variables of a RZ polynomial keeping the property of being RZ. The most famous and simpler one of these methods stems form a construction in \cite{Nuij1968ANO} that we will use quite often in the future and that allows us to add one variable to our original polynomial.
\end{remark}

Extending the relaxation passed before through a phase in which we tried to directly increase the number of coefficients of the original polynomial that we use in the relaxation. It turns out that there is not a clear path of how to do that and, in fact, our incomplete explorations in that direction tend to point towards non-commutative liftings. Unfortunately, we could not pursue here these constructions because they would have led us very far from our objectives. This is why we ended up doing variable extensions. The first one of these extensions uses interlacing and Renegar derivatives, see \cite{fisk2006polynomials} and \cite[Section 4]{Renegar2006HyperbolicPA}.

\begin{remark}[Using several hammers to pin one nail]
Several properties allow us to take several polynomials and put them together into one big multivariate polynomial having many more variables. This was the original objective of the Amalgamation Conjecture \cite[Section 8.1]{main}. However, these attempts also failed in high degree or high number of variables, showing us that the whole idea of constructing these extensions is much more subtle than we originally thought. When the polynomials verify certain order theoretic properties with respect to their roots, it is in fact possible to build something similar to an amalgam. In that sense, several polynomials can be used to mix several hammers capable of attacking the same nail, the original polynomial, e.g., \cite[Proposition 2.7]{wagner2011multivariate}. We explored that path and verified that, if we want to succeed in that direction, more elaborated constructions will be needed. The usual construction coming to our mind did not perturb at all the relaxation.
\end{remark}

Indeed, the usual construction using the Renegar derivative \cite[Definition 2.4.2]{habnetzer} does not improve the relaxation. It does not even change it. We prove this limitation in Chapter \ref{4}. And it is this limitation what forces us to expand our explorations into deeper directions down the path of multivariability. For exploring these new ways, we have to narrow our sight to particular families of polynomials. We have to do that because the new extensions we try will work in different ways depending on the family to which we apply them. This forces us to narrow the scope and the methods in order to reach in a bike, by capillarity, areas of the forest that we could not explore with a car.

\begin{metafora}[The forest, the car and the bike]
Within the forest of variable extensions we encounter paths that cannot be traversed with a car\footnote{My professor of algebraic topology (F. J. Turiel) introduced us homology in the following way ``Homotopy is nice. It carries a lot of information about the space you are studying, but it is a bus and, often, you know that you cannot explore certain streets with a bus, so you need a bike." I borrow here from him that metaphor that was fixed in my mind many years ago.}. Before, we tried to fit our hammer in the car and search for its corresponding nail traversing big roads in the form of methods that allow us to increase the number of variables of any RZ polynomial. We saw that, at least so far, we do not really have access to such very general methods. For this reason, we decided to pack our hammer in our backpack and explore the forest on a bike. This bike allows us to explore narrow paths that connect with unexplored regions of the forest. The bike is the sequence of Eulerian polynomials \cite{mezo2019combinatorics,petersen2015eulerian} and the path that it allows us to explore is the realm of its multivariate generalizations \cite{visontai2013stable,haglund2012stable}.
\end{metafora}

Injecting the univariate Eulerian polynomials into the diagonal of its multivariate counterparts, we are able to show that we can increase the accuracy of the relaxation when the number of variables grows. In order to do this, we need to explore the combinatorial properties of the coefficients of these polynomials. In particular, we need to find formulas for these coefficients that eventually allow us to effectively compute the relaxation. With a computationally efficient version of the relaxation we are able to play in order to prove that the accuracy in the diagonal augments when the number of variables increases. This shows, for the first time, that increasing the number of variables has a positive net effect in the power of the relaxation to approximate the roots of the original polynomial.

\begin{metafora}[Linearization as shortcuts]
The forest of multivariate generalizations we traverse here has at some points some paths that are difficult to cross. The relaxation becomes too large to be properly handled in the presence of many variables. In these cases, we need to take shortcuts in the form of linearizations of our problems through numerical guesses of approximated generalized eigenvectors for the matrices (and generalized eigenvalue problems) that will pop up due to the application of the relaxation. 
\end{metafora}

Finally, we will become interested in the bike itself and on how it compares with other bikes. This will prompt us to explore new symmetries within the multivariate Eulerian polynomials and to compare the effectivity of applying the relaxation to other different methods. We will ultimately also become interested in the possibility of combining and mixing different methods like when one changes bikes (or wheels) when the texture of paths changes.

\begin{remark}[Symmetries, combinations and beyond]
These last explorations will open doors that we cannot close here. In this regard, the last part of this thesis shows many different paths to pursue beyond the limits that we imposed ourselves here in order to keep the dissertation in reasonably manageable boundaries.
\end{remark}

We let these metaphors serve us in our task of understanding the rest of this work. In the next section, we will highlight a sketch of the main peaks and tracks of our exploration path. So, as Thor would do, we take Mj\"{o}lnir with us into the unknown realm we will venture soon.

\section{Prolegomena to any attempt to read this work}

Here we will offer a strategy to read and understand the spirit of this work through a set of roadmaps. These will guide us from the very beginning of the relaxation to the ultimate borders of the research we have conducted here.

\begin{roadmap}[From the relaxation to its limitations]
We will begin introducing several results about the relaxation developed in \cite{main}. We will focus our attention on these results in \cite{main} that will be helpful and relevant for our endeavour here. After these first bricks are clear, we will work on the limitations we find when we try to improve and extend the relaxation in a general manner. That is, when we try to apply extensions that seem like they could work for all RZ polynomials.
\end{roadmap}

\setlength{\emergencystretch}{3em}%
This describes our work in Part \ref{I}. The limitations discovered there during the previous and initial roadmap will guide us towards narrower and more restricted venues. The main one of these venues is directly related and linked to a family of polynomials connected with counting descent statistics on permutations.
\setlength{\emergencystretch}{0em}%

\begin{roadmap}[Eulerian polynomials]
This is just the tip of the iceberg. In part \ref{II}, we concentrate our efforts on some famous, well-known and extensively studied families of polynomials in order to know how far the relaxation can reach. This leads us directly towards the study of Eulerian polynomials and their multivariate generalizations. We center our attention in the most known sequence of these generalizations and use them in order to improve the output of the relaxation. This interest will end up also guiding us towards a better understanding of the coefficients of these polynomials in Part \ref{V} because this information will be greatly needed in order to effectively build and compute the relaxation.
\end{roadmap}

Once we know enough about the Eulerian polynomials and we have built the relaxation, we have to study how good the bounds provided by the relaxation are. We do this in Part \ref{III} relating the bounds obtained in this way with previously known bounds in the literature. Luckily, the great interest that exists for these polynomials means that there have been many attempts to find bounds for their roots. This is precisely why we chose to concentrate in these for the first application of the relaxation in this direction of root bounding: because they provide a well-known framework in which we can easily make comparisons in order to establish winners and losers among the different methods available in the literature.

\begin{roadmap}[Comparisons with the relaxation]
In Part \ref{III}, we study the relaxation and the bounds it provides. This study is itself tricky and requires us to take care on the number and complexity of the computations necessary. In particular, several tricks and numerical experiments are needed in order to come up with proofs of actual improvement when the number of variables increases. Later, in order to perform other comparisons with different techniques, we have to deal with problems related to higher order asymptotic expansions and equations, polynomial stability and numeric stability (which we remark here that are completely different topics in spite of their similar names!).
\end{roadmap}

These comparisons irremediably lead us to the work of Sobolev in \cite{sobolev2006selected}. Refining what Sobolev did there and combining it with the relaxation, we finally obtain the best bound for these extreme roots in this work. However, these intermediate works of refinement that we have to do and a closer look to the methods used by Sobolev in connection to our setting open many new questions about doors leading to paths that we could not venture to explore here. It is clear for us that these explorations definitely broaden the scope of our framework here in ways that this thesis can in no way fully contain or even just hint at.

\begin{roadmap}[Sobolev half-bounds]
The work of Sobolev in \cite{sobolev2006selected} attacks the problem of understanding Eulerian polynomials from a completely different perspective. Sobolev is interested in these polynomials by analytic and numeric reasons instead of combinatorial ones. Surprisingly, Sobolev, completely ignoring that other point of view, gives the best account of the roots of these polynomials in the literature. However, a closer look at its work reveals that he proves more than what he claims. This is how \textit{understanding Sobolev better than he understood himself}\footnote{This is intended as a joke in the form of a reference to the famous view and quote of Kant that it is \textit{not at all unusual} to understand an author better than he understood himself \cite[Chapter 1]{strauss2013leo}, \cite[Footnote 6]{seebohm2002phenomenological}.} we could improve upon his methods and, eventually, combine them with the relaxation to provide the best inner bound for these extreme roots in this work.
\end{roadmap}

Finally, we understood the importance of knowing more about these polynomials. That idea moved us towards exploring its symmetries, which eventually urged us to introduce several generalizations of well-known concepts in the theory of univariate polynomials in Part \ref{V}. In particular, we extended notions of symmetry usual for univariate polynomials (and verified by the univariate Eulerian polynomials in particular) to the multivariate setting and thus we could verify that the version we use for their multivariate generalization also verify these generalized forms of symmetry.

\begin{roadmap}[Symmetry beyond one variable]
The palindromicity of the univariate Eulerian polynomial is a key property that helps us relate inner bounds for the smallest root to outer bounds for the biggest root. We are able to find a generalization of this property present in multivariate Eulerian polynomials. That property of the coefficients of these polynomials has a direct translation in terms of the underlying combinatorial objects. We also explore the consequences in that combinatorial setting.
\end{roadmap}

All these roadmaps, explored in each part of this work, lead us towards a conclusion proposing a path for new explorations in Part \ref{VI}. In particular, the whole spirit of this work is to establish the sketch of a framework suitable for exploring all these concepts in much more depth that we could do at this first general sight. This encourages us to propose in Chapter \ref{conclusion} a program aimed at developing a full theoretical framework connecting our real algebraic geometric tools with explorations in combinatorics through numerical lenses. This approach tries to put to work together several disparate areas of mathematical research into a task that could allow us to explore wide and profound connections between these disciplines. The explanation of (the importance of exploiting and developing further) this intersection of realms and branches of mathematics is the topic of the next section.

\section{Introduction to the intersection of subjects meeting here}

In order to understand the depth and width of the connections exploited within this work, it is important to first lay down the topics that interact at the mathematical junction we are exploring here. First and foremost, the relaxation was initially conceived as a way to attack the \gls{GLC}, which is a central question in convex or conic optimization whose ultimate objects belong to the realm of real algebraic geometry. The relaxation is at the same time a linear object that is built over a moment matrix through linear maps emerging from the series expansion of extensions to power series of the $\log$ function. So just in the ideation and definition of the relaxation we already have all these topics playing a role.

\begin{remark}[RAG and optimization]
The relaxation originally emerged as a way to attack the GLC, which sits at the juncture of convex real algebraic geometry, optimization and linear algebra. In the definition of the relaxation the theory of series expansion plays also a crucial role.
\end{remark}

Beyond these, when we try to extend the relaxation in order increase the information it can carry, we have to deal with trace inequalities that ultimately suggest an unexplored path into non-commutative algebra. The other way we study with the hope of extending the relaxation involves the theory of polynomial interlacers. Deeply far in this theory of interlacers, a seed for the theory of polynomial stability lays already latent.

\begin{remark}[Trace inequalities, non-commutative, interlacers and polynomial stability]
Our attempts to amplify the reach of the \gls{LMP} of the relaxation involve dealing with several trace inequalities that hint at a connection with noncommutative algebra and trace algebras. A different path to the extension of the relaxation involves the theory of interlacing polynomials and, with this, also some bits of the theory of polynomial stability.
\end{remark}

When we see the limitations of these general methods, we have to turn our attention towards some concrete and particularly well suited, studied and understood families of real-rooted and RZ polynomials. In our case, we decide to explore the well-known Eulerian polynomials, which connect us and our explorations with the combinatorial setting of the theory of statistics on permutations. Going multivariate in this setting requires thus both a sufficient understanding of the theory of stability preservers and of the nuances of the combinatorial theory of permutations over ordered sets.

\begin{remark}[Eulerian polynomials and statistics on permutations over ordered sets]
The definition we use for Eulerian polynomials connects us directly with the combinatorial setting through its relation to counting descent (top) statistics on permutations. These also involve simple notions of order and hint therefore at other generalizations of Eulerian polynomials defined over other differently (partially) ordered sets.
\end{remark}

Once these polynomials become the center of our attention, the relaxation begins to produce root bounds that we compare and study through the theory of asymptotic growth. In order for the relaxation to produce bounds that we can effectively analyze we have to turn to the theory of generalized eigenvalues in order to linearize our LMPs through numerical guesses for their corresponding approximated generalized eigenvalues.

\begin{remark}[Root bounds, asymptotic expansions and numerics of generalized eigenvalue problems]
The relaxation outputs a LMP. Dealing with it requires us to do some numeric work in order to simplify the expressions and objects we obtain so that we can actually work with them and effectively compute the bounds. These bounds need to be compared between them as they come in families. In order to perform these comparisons, we need to use the theory of asymptotic expansions so that we can actually understand the growth of the bounds we obtain.
\end{remark}

In this sense, the number of computations necessary plays a role and thus we graze the theory of effective computation, although we never really enter in it properly here. The numerical analysis of the guessed vector leads us towards the fields of numerics and numerical stability, although these topics will take more prominence afterwards.

\begin{remark}[(Approximated) computational complexity and numerical stability]
When we analyze the output of the relaxation, we get narrowly interested in how many computations we had to perform before obtaining such output. At some points, reducing these computations is essential. Additionally, it is also important to control and look at the problems that could emerge from problems with the numerical stability of the numerical methods at play because some quantities we have to deal with might grow very fast. This growth could cause problems that we have to be able to understand, isolate, prevent and solve. 
\end{remark}

The true and big irruption of numerics in the scene has to do with the study that Sobolev makes of these polynomials as helpful tools in the development of quadrature formulas. Thus, as his approach is analytic and numeric, when looking at his work, these two branches will acquire particular relevance. Particularly, numerics play a crucial role when we analyze the ultimate method that Sobolev uses in order to approximate these roots: the DLG method. This method is well-known in the field and it is also known that it has problems of numerical stability. The rest of the analytic tools used by Sobolev set us lightly into an analytic framework that we just touch briefly. This analytic framework could become even more relevant if we had decided to investigate the relation between these roots and the statistics of coefficients of polynomials, where deep connections between algebra, probability, analysis and statistics arise while trying to transport information from the coefficients to the roots and the other way around. A deeply famous theorem in this regard is due to Harper \cite{harper1967stirling}.

\begin{remark}[Numerics and distant bits of probability and analysis]
The approach of Sobolev to the study of the roots of the univariate Eulerian polynomials sets our discussion far from our initial combinatorial point of view deep into a different realm of numerics and analysis where Eulerian polynomials actually erupted for the first time in Euler's work \cite{euler}. In these areas, we can find deep knowledge about these polynomials totally disconnected from the literature from the combinatorial counterpart we began with. Also, the presence of analysis in this theory could actually be increased if we explored connections between roots and coefficients in terms of statistics \cite{pitman1997probabilistic}; something interesting that we do not explore in detail here.
\end{remark}

Finally, as the palindromicity of the univariate Eulerian polynomials was useful for us, we briefly ventured ourselves into the realm of symmetries and patterns in polynomials. We could prove a result about the multivariate generalization that ultimately connected with the combinatorics underlying at the nucleus of their construction.

\begin{remark}[Symmetries and patterns in polynomials]
The study of Eulerian polynomials invokes questions about patterns and symmetries in polynomials. These questions are prevalent in the recent literature and Eulerian polynomials (and other related polynomial families) are usually nice examples where some of these patterns arise naturally. We found ways to generalize some of these patterns beyond the univariate setting. This also allowed us to study how that symmetry can give us back some information about the underlying combinatorial object through the path that connects these polynomials to statistics on permutations.
\end{remark}

Thus we see that many different topics come together at this juncture and therefore, although lightly because this work cannot possibly go too deep into all the details around the connections between these topics, we have a long list of prerequisites. This list, as we said, is long but we do not require a deep understanding of all these topics beyond the ability to understand and know what we are talking about at each moment. Making these clear is the objective of the next important section.

\section{Mathematical prerequisites necessary before continuing}

The picture of the connections in the section above makes clear the main topics that will have prominence here. In the definition of the relaxation, it will suffice to understand linear matrix polynomials, matrix inequalities and basic linear algebra. Some convex real algebraic geometry will be helpful, but it is enough with a superficial knowledge of the areas around the GLC. For the extensions of the relaxation, trace inequalities and polynomial interlacing will be the main themes. The part about Eulerian polynomials requires some understanding of the combinatorics around statistics of permutations over order sets. For the comparison of the bounds, the main topic will require knowledge about some rudiments of asymptotic analysis and scales for asymptotic expansions. Some numerical linear algebra will be helpful to understand the numerical stability problems and the attempts at approximating generalized eigenvectors. When Sobolev enters the game, numerics of approximations of roots of polynomials will take some prominence together with some simple methods and tools associated to these areas, like the DLG method itself. Finally, for the last part, a light understanding of polynomial patterns paired together with some additional knowledge about statistics on permutations should do the job. With a light understanding of these topics, we can ensure a smooth reading throughout this text.

\section{Structure of the contents}

The structure of this thesis is simple. It is divided in this preamble, five parts in the main body and a conclusion. This preamble serves as a smooth introduction into the topic, themes and problems that we will develop here. Part \ref{I} introduces the relaxation and deals with our initial attempts at extending it in general, showing therefore the limitations of these extensions. In Part \ref{II} we introduce Eulerian polynomials, study their properties and roots and compute their coefficients and the corresponding $L$-forms that will allow us to write down the relaxations associated to these polynomials. Part \ref{III} deals with the comparison of the bounds obtained by other means and the bounds produced by the relaxation. In this part, we also have to develop techniques to compute approximations and linearizations for the LMP associated to the relaxation in order to be able to study the corresponding bounds produced by the relaxation. Part \ref{IV} deepens into Sobolev's approach to these polynomials and his attempts at providing approximations and bounds for their roots. Here, we also review and deepen the techniques, tools and results in Sobolev's papers and refine them in order to obtain even better results. Finally, we combine his methods with the relaxation in order to see how such combination provides even better bounds. We end the main body with Part \ref{V}, where we study symmetries on multivariate Eulerian polynomials expanding results and concepts well known in the theory of univariate polynomials to the multivariate case. We finish this dissertation with a conclusion that attempts to establish a look beyond our work here: in this conclusion, we propose the \textit{Mindelsee Program} as a thoroughly expanded extension and completion of the initial work we did here in a way that, building on the first blocks we lay down in this work, we can build a whole theory about the interrelations and connections that we could only superficially devise and analyze here. 

\section{Main results and a look beyond}

On Part \ref{I}, the main results are those related to the properties of the relaxation and the limitations of the possible ways of extending it. These are the following.

\begin{proposicion}[PSD-ness of extended initial matrix]
\label{propomatrix-Pre}
Applying $L_{p,e}$ for any integer $e\geq d:=\deg(p)$ to all the entries of any MMM of the form given by Equation \ref{ext1} for any RZ polynomial $p\in\mathbb{R}_{\mbox{RZ}}[\mathbf{x}]$ produces a PSD matrix.
\end{proposicion}

\begin{corolario}[Relaxation commutes with restriction for truncated extension]
\label{tildein-Pre}
Let $p\in\mathbb{R}[x]$ be a bivariate RZ polynomial and denote now $s:=p+yp^{(1)}\in\mathbb{R}[x,y]$. Then we have the identity $S(p)=\{a\in\mathbb{R}^{2}\mid (a,0)\in S(s)\}$.
\end{corolario}

On Part \ref{II}, results about Eulerian polynomials pop up. The main ones are the following.

\begin{corolario}[RZ-ness of dehomogenization]
$A_{n}(\mathbf{x},\mathbf{1})$ is RZ.
\end{corolario}

\begin{corolario}[Cardinal of sets with exact descent top set]
\label{-Pre}
Fix $s=|X|$ and $\{x_{1}<\cdots<x_{s}\}=X\subseteq\gls{brackn}$ and $\{y_{1}<\cdots<y_{n-s}\}=Y\subseteq[n]$ with $X\cup Y=[n]$. For subsets $S\subseteq Y$ name the ordered chain of elements obtained through the union $X\cup S=\{x_{S,1}<\cdots<x_{S,s+|S|}\}$. Thus, going through the complement, we have that $|R(n-1,X)|=$
\begin{gather}\label{coroR1-Pre}
\sum_{S\subseteq[n]\smallsetminus{X}}(-1)^{|S|}(n-|X\cup S|)!\prod_{i=1}^{s+|S|}(x_{S,i}-i).
\end{gather} Similarly, we can express this number in terms of deletions in the initial set as $|R(n-1,X)|=$
\begin{gather}\label{coroR2-Pre}
\sum_{J\subseteq X}(-1)^{|X\smallsetminus J|}\alpha(J)\hat{!}.
\end{gather}
\end{corolario}

On Part \ref{III}, the comparisons between bounds begin and therefore we center our attention in theorems about asymptotic expansions and differences. The corresponding main results are the followoing.

\begin{teorema}[Improved estimated and exact asymptotics]
\label{refinementStanley-Pre}
Fix $n,k\in\mathbb{N}$ with $k\leq n$ and decompose the $n$-th univariate Eulerian polynomial $A_{n}(x)=\prod_{i=1}^{n}(x+q_{i}^{(n)})$ so that $0<q_{n}^{(n)}\leq\cdots\leq q_{1}^{(n)}.$ Then we can estimate the growth of the absolute value of the $k$-th root of the $n$-th univariate Eulerian polynomial $A_{n}$ as $$\lim_{n\to\infty}\frac{q^{(n)}_{k}}{\left(\frac{k+1}{k}\right)^{n+1}}=1.$$
\end{teorema}

\begin{proposicion}[Beating univariate]\label{multundiff-Pre}
Call $\un,\mult_{v}\colon\mathbb{N}\to\mathbb{R}_{>0}$ the sequence of bounds for the absolute value of the smallest (leftmost) root of the univariate Eulerian polynomial $A_{n}$ obtained after applying the relaxation to the univariate $A_{n}$ and the multivariate $A_{n}$ Eulerian polynomials respectively, where $\mult_{v}$ is obtained after linearizing using the vector $v$ and $\un$ is the optimal the actual root of the determinant of the relaxation. Remember also that $v$ depends on $y$. Then there exists a sequence $y\colon\mathbb{N}\to\mathbb{R}$ such that the difference $\mult_{v}-\un\sim\frac{1}{2}\left(\frac{3}{4}\right)^{n}+o(\left(\frac{3}{4}\right)^{n})$ is positive. 
\end{proposicion}

On Part \ref{IV}, we review and rewrite in depth the work of Sobolev. Taking a higher level of detail we can improve his claims.

\begin{proposicion}
The estimate for the absolute value of the leftmost root of the Eulerian polynomial obtained from approximating using Vi{\`e}te's formula for the biggest after the first iteration of the DLG method is, up to its forth growth term, $2^{k+1}-\left(\frac{3}{2}\right)^{k+1}-\frac{1}{2}\left(\frac{9}{8}\right)^{k+1}-\frac{k}{2}+o(k).$
\end{proposicion}

\begin{corolario}
Let $q_{1}^{(k)}$ be the leftmost root of the univariate Eulerian polynomial of degree $k$. Then we have the inequalities \begin{gather*}\frac{2^{k+1}-(k+2)+\sqrt{2^{3 + k} - 4 \cdot 3^{1 + k} + 4^{1 + k} - k \left(2 - 2^{2 + k} + k\right)}}{2}\leq|q_{1}^{(k)}|\\\leq\frac{2}{k}\bigg{(}\frac{2^{k+1}-(k+2)}{2}+\frac{k-1}{2}\\\sqrt{\frac{-4 \left(2^{k}-1\right)^2 + \left(4^{1 + k} -2 + 2^{2 + k} - 2 \cdot 3^{1 + k}\right) k}{k-1}
}\bigg{)}.\end{gather*}
\end{corolario}

\begin{proposicion}
The majorant asymptotic growth first differs from the raw estimate at its fourth growth terms. In particular the majorant four first growth terms are $2^{k+1}-\left(\frac{3}{2}\right)^{k+1}-\frac{1}{2}\left(\frac{9}{8}\right)^{k+1}-\frac{1}{2k}\left(\frac{9}{8}\right)^{k+1}+o(\frac{1}{k}\left(\frac{9}{8}\right)^{k})$.
\end{proposicion}

\begin{proposicion}
Sobolev's refined minorant up until its fourth growth term grows like $2^{k+1}-\left(\frac{3}{2}\right)^{k+1}-\left(\frac{9}{8}\right)^{k+1}-2\left(\frac{27}{32}\right)^{k+1}+o(\left(\frac{27}{32}\right)^{k}).$
\end{proposicion}

We fix the strong improvement showed by these computations in the following important result.

\begin{teorema}
The leftmost root of the Eulerian polynomial admits the bound $|q_{1}^{k}|>2^{k+1}-\left(\frac{3}{2}\right)^{k+1}-\left(\frac{9}{8}\right)^{k+1}+1+o(1).$
\end{teorema}

Finally, on Part \ref{V}, we study and generalize some symmetries of multivariate polynomials. Such analysis produces the next results.

\begin{proposicion}[Globality of the relaxation]
Let $A=\{A_{n}\}_{n=0}^{\infty}$ be a sequence of multivariate polynomials and $w=\{w_{n}\}_{n=0}^{\infty}\in S$ a sequence of directions choosing a sequence of univariate polynomials $$A_{w}=\{A_{n}(w_{n1}x_{1},\dots,w_{nn}x_{n})|_{x_{i}=t}\}$$ with all roots non-positive injected in the sequence $A$. Fix, moreover, any sequence of vectors $v=\{v_{n}\}_{n=0}^{\infty}$ and denote $R_{n}:=R_{n,0}+ x_{1}R_{n,1}+\cdots+x_{n}R_{n,n}$ the LMP of the relaxation applied to the corresponding polynomial $A_{n}$ in the sequence $A_{w}$ so that $R_{n,\tot}:=t(w_{1}R_{n,1}+\cdots+w_{n}R_{n,n})$ verifies $v_{n}^{T}R_{n,\tot}v_{n}>0$. Then we have the bounds $r_{n}>-\frac{v_{n}^{T}R_{n,0}v_{n}}{v_{n}^{T}R_{n,\tot}v_{n}}$, where $r=\{r_{n}\}_{n=0}^{\infty}$ is the sequence of the largest roots of the polynomials in $A_{w}$.
\end{proposicion}

\begin{teorema}[Eulerian mirrorreciprocity]
The $n$-th multivariate Eulerian polynomial $A_{n}(\mathbf{x})$ is mirrorreciprocal.
\end{teorema}

The combination of these results, their proofs, the whole story we tell here and the spirit of the general ideas that guide our research puts us at the end of this work in the position to propose the Mindelsee Program in the conclusion. With this brief tour around our main results in this thesis, we can finally venture now into the promised work itself. First, we need to introduce the relaxation (which is our task in the first chapter of the first part of this thesis) in order to study the limitations of our initial attempts at extending it (in the rest of the chapters of the first part).

\part{On the Definition of the Relaxation and its Limitations}\label{I}
\chapter{Relaxation of Rigidly Convex Sets}\label{ChRelaxation}

Here, the main reference is \cite{main}. We introduce the main definitions and results from this reference to keep the exposition self-contained. However, we will not include proofs in this chapter as these can be found in the cited reference. We also mention that, along this whole text, all the matrices are symmetric unless we explicitly say the opposite.

\section{Main definitions for real zero polynomials}

The main object we will center our attention to is the next generalization of being real-rooted for multivariate polynomials. These are the polynomials we are looking for and the ones to which our methods can be applied.

\begin{definicion}[Real zero polynomial]\cite[Definition 2.1]{main}\label{RZ}
Let $p\in\mathbb{R}[\mathbf{x}]$ be a polynomial. We say that $p$ is a \textit{real zero polynomial} (or a \textit{RZ polynomial} for short) if, for all directions $a\in\mathbb{R}^{n}$, the univariate restriction $p(ta)$ verifies that all its roots are real, i.e., if, for all $\lambda\in\mathbb{C}$, we have $$p(\lambda a)=0 \Rightarrow \lambda\in\mathbb{R}.$$
\end{definicion}

Observe that a RZ polynomial cannot have value $0$ at the origin, i.e., $p(0)\neq0$. This allows us to introduce the next subset of its domain $\mathbb{R}^{n}$ that will be of fundamental importance for us.

\begin{definicion}[Rigidly convex set]\cite[Definition 2.11]{main}\label{RCS}
Let $p\in\mathbb{R}[\mathbf{x}]$ be a RZ polynomial and call $V\subseteq\mathbb{R}^{n}$ its set of zeros. We call the \textit{rigidly convex set} of $p$ (or its \textit{RCS} for short) the Euclidean closure of the connected component of the origin $\mathbf{0}$ in the set $\mathbb{R}^{n}\smallsetminus V$ and denote it $\rcs(p)$.
\end{definicion}

We can give another definition. It is common to see both in the literature.

\begin{remark}[Equivalent definition]
Let $p\in\mathbb{R}[\mathbf{x}]$ be a RZ polynomial. Then we can equivalently define $$\rcs(p):=\{a\in\mathbb{R}^{n}\mid\forall\lambda\in[0,1)\ p(\lambda a)\neq0\}.$$ 
\end{remark}

The study of RCSs leads to many questions that still remain open. One of these questions is the generalized Lax conjecture (GLC for short), which is only resolved in certain special cases. We explain in the next section how these kinds of question arise from looking for examples of these polynomials.

\section{Motivation for the necessity and usability of a relaxation}

The next question motivates why we search for matrix relaxations of RCSs. It all boils down to the most natural example of RZ polynomials we can imagine. One could wonder if there are polynomials that verify the conditions in Definition \ref{RZ}. Yes, there are and we can find them through the use of the spectral theorem for Hermitian matrices. 

\begin{proposicion}[Example of RZ polynomials]\cite[Proposition 2.6]{main}
Consider $A_{0},\dots,A_{n}\in\Her_{d}(\mathbb{C})$ complex hermitian matrices with $A_{0}$ PD and build $$p=\det(A_{0}+x_{1}A_{1}+\cdots+x_{n}A_{n})\in\mathbb{R}[\mathbf{x}].$$ Then $p$ is RZ.
\end{proposicion}

The fact that we found the most natural example using Hermitian matrices makes us wonder if these are actually all possible examples. For $n=2$, we can identify our RZ polynomials even better than that.

\begin{teorema}[Helton-Vinnikov]\cite[Corollary 2.8]{main}
Let $p\in\mathbb{R}[x_{1},x_{2}]_{d}$ be a degree $d$ RZ polynomial with $p(0)=1$. Then there exist hermitian matrices $A_{1},A_{2}\in\Sym_{d}(\mathbb{R})$ with $p=\det(I_{d}+x_{1}A_{1}+x_{2}A_{2})$.
\end{teorema}

It is natural to ask now what happens for more variables. A dimension count tells us immediately that it cannot be that easy. For this reason, the next conjecture requires the introduction of a cofactor that overcomes such obstruction. The cofactor is chosen in a way that respects the RCSs of the original polynomial.

\begin{conjetura}[Generalized Lax conjecture, GLC]\cite[Conjecture 8.1, Theorem 8.8]{main}
Let $p\in\mathbb{R}[\mathbf{x}]$ be a RZ polynomial with $p(0)=1$. Then there exist another RZ polynomial $q\in\mathbb{R}[\mathbf{x}]$ and symmetric matrices $A_{1},\dots,A_{n}\in\Sym_{d}(\mathbb{R})$ for some $d\in\mathbb{N}$ such that \begin{enumerate}
\item $pq=\det(I_{d}+x_{1}A_{1}+\cdots+x_{n}A_{n})$
\item $\rcs(p)\subseteq\rcs(q).$
\end{enumerate}
\end{conjetura}

This conjecture is widely open, although it is confirmed in some particular cases. Also, remarkably, dropping the containment condition (2) produces a known theorem. However, without that second condition, we lack real control over the definition of the corresponding RCS via these determinantal representations.

\begin{definicion}\cite[Definition 2.1.2]{habnetzer}
We say that a RZ polynomial $p\in\mathbb{R}[\mathbf{x}]$ is \textit{smooth} if for each $\mathbf{a}\in\mathbb{R}^{n}$ the univariate polynomial $p(t\mathbf{a})\in\mathbb{R}[t]$ has all its roots different, i.e., of multiplicity $1$, including possible roots at infinity.
\end{definicion}

\begin{teorema}\cite[Theorem 6]{kummer2017determinantal}
Let $p\in\mathbb{R}[\mathbf{x}]$ be a smooth RZ polynomial with $p(0)=1$. Then there exist another RZ polynomial $q\in\mathbb{R}[\mathbf{x}]$ and symmetric matrices $A_{1},\dots,A_{n}\in\Sym_{d}(\mathbb{R})$ for some $d\in\mathbb{N}$ such that $pq=\det(I_{d}+x_{1}A_{1}+\cdots+x_{n}A_{n})$.
\end{teorema}

Therefore, searching for a representation with a cofactor that respects the RCS of the original polynomial is the main conundrum towards a positive solution to the conjecture above. Finding these cofactors is in general not evident and, without more knowledge about the polynomial, this constitutes an insurmountable obstacle towards establishing the conjecture. This is why we introduce a relaxation that respects the RCS in the sense that the RCS of the relaxation contains the RCS of the original polynomial. In order to build this relaxation directly from the coefficients of our initial polynomial, we have to look first at logarithms and exponentials of power series, which is the objective of the next section.

\section{Extending the logarithm and the exponential}

The logarithm and the exponential maps admit definitions that can be extended to certain power series of polynomials. We see this now.

\begin{definicion}[Truncation, logarithm and exponential for power series]\cite[Definition 3.1, Definition 3.2]{main}
Let $p=\sum_{\alpha\in\mathbb{N}_{0}^{n}}a_{\alpha}x^{\alpha}\in\mathbb{R}[[\mathbf{x}]]$ be a real power series. Let $d\in\mathbb{N}$. We denote $$\gls{trundp}:=\sum_{\alpha\in\mathbb{N}_{0}^{n}\\|\alpha|\leq d}a_{\alpha}x^{\alpha}\in\mathbb{R}[\mathbf{x}]$$ the \textit{truncation of $p$ at degree $d$}. If $p(0)=0$, then we define $$\gls{expp}:=\sum_{k=0}^{\infty}\frac{p^{k}}{k!}\in\mathbb{R}[[\mathbf{x}]]$$ the \textit{exponential} of $p$. If $p(0)=1$, then we define $$\gls{logp}:=\sum_{k=1}^{\infty}(-1)^{k+1}\frac{(p-1)^{k}}{k}\in\mathbb{R}[[\mathbf{x}]]$$ the \textit{logarithm} of $p$.
\end{definicion}

Studying the truncation and seeing that the monomials in each degree stabilize after a finite number of steps, it is immediate that these functions are well defined. They are the inverse of each other and verify the usual properties of these maps with respect to the sum and product.

\begin{proposicion}[Properties of the exponential and the logarithm]\cite[Proposition 3.3]{main}
Write $$\mathbb{R}[[\mathbf{x}]]_{a\to b}:=\{p\in\mathbb{R}[[\mathbf{x}]]\mid p(a)=b\}.$$ Then it holds that the maps $$\exp\colon \mathbb{R}[[\mathbf{x}]]_{0\to0}\to \mathbb{R}[[\mathbf{x}]]_{0\to1}$$ and $$\log\colon \mathbb{R}[[\mathbf{x}]]_{0\to1}\to \mathbb{R}[[\mathbf{x}]]_{0\to0}$$ are inverse of each other. Moreover, as usual, $\exp(p+q)=\exp(p)\exp(q)$ and $\log(pq)=\log(p)+\log(q)$ for $p,q$ in the corresponding domains of definition.
\end{proposicion}

The introduction of the logarithm, in particular, allows us to define a linear form that will play a major role in the rest of this text. In order to define this form, we look at the coefficients of a power series generated as the logarithm of our initial polynomial.

\begin{definicion}[$L$-form]\cite[Proposition 3.4]{main}
Let $p\in\mathbb{R}[[\mathbf{x}]]$ be a real power series with $p(0)\neq0$ and $d\in\mathbb{N}.$ We define the linear form $\gls{lpd}$ on $\mathbb{R}[\mathbf{x}]$ by specifying it on the monomial basis of $\mathbb{R}[\mathbf{x}]$. We set $L_{p,d}(1)=d$ and define implicitly the rest of the values by requiring the identity of the formal power series $$-\log\frac{p(-x)}{p(0)}=\sum_{\alpha\in\mathbb{N}_{0}^{n}\\\alpha\neq0}\frac{1}{|\alpha|}\binom{|\alpha|}{\alpha}L_{p,d}(x^{\alpha})x^{\alpha}\in\mathbb{R}[[\mathbf{x}]].$$ When $p\in\mathbb{R}[\mathbf{x}]$ is a polynomial, we denote $L_{p}:=L_{p,\deg(p)}$ and call it the \textit{$L$-form associated to $p$}.
\end{definicion}

The properties of the logarithm allow us now to translate several identities to the just introduced linear forms. We also compute the few values that will be important for the rest of this work. In particular, in order to define the relaxation, we just need the values over monomials of degree up to three.

\begin{corolario}[Properties of the $L$ form and values in low degree]\cite[Example 3.4]{main}
Let $p,q\in\mathbb{R}[[\mathbf{x}]]$ be power series with $p(0)\neq0\neq q(0)$ and $d,e\in\mathbb{N}_{0}.$ Then $$L_{pq,d+e}=L_{p,d}+L_{q,e}.$$ Moreover, suppose that $p\in\mathbb{R}[[\mathbf{x}]]$ truncates to $$\trun_{3}(p)=1+\sum_{i\in[n]}a_{i}x_{i}+\sum_{i,j\in[n]\\i\leq j}a_{ij}x_{i}x_{j}+\sum_{i,j,k\in[n]\\i\leq j\leq k}a_{ijk}x_{i}x_{j}x_{k}.$$ Then we have that \begin{gather*}L_{p,d}(x_{i})=a_{i},\\L_{p,d}(x_{i}^{2})=-2a_{ii}+a_{i}^{2},\\L_{p,d}(x_{i}^{3})=3a_{iii}-3a_{i}a_{ii}+a_{i}^{3}\end{gather*} for all $i\in[n]$, while \begin{gather*}L_{p,d}(x_{i}x_{j})=-a_{ij}+a_{i}a_{j},\\L_{p,d}(x_{i}^{2}x_{j})=a_{iij}+a_{i}a_{ij}-a_{j}a_{ii}+a_{i}^{2}a_{j}\end{gather*} for all $i,j\in[n]$ with $i<j$ and $$L_{p,d}(x_{i}x_{j}x_{k})=\frac{1}{2}(a_{ijk}-a_{i}a_{jk}-a_{j}a_{ik}-a_{k}a_{ij}+2a_{i}a_{j}a_{k})$$ for all $i,j,k\in[n]$ with $i<j<k.$
\end{corolario}

The avid reader could ask why exactly we stop at degree $3$. The reason is very practical: we go to degree up to three for a very particular reason that will decisively affect the definition of the relaxation. Until such degree the values of the $L$-forms have a nice expression in the form of traces for determinantal polynomials. This is fundamental.

\begin{proposicion}[$L$-form and traces]\cite[Corollary 3.18]{main}\label{lformtraces}
Let $$p=\det(I_{d}+x_{1}A_{1}+\cdots+x_{n}A_{n})\in\mathbb{R}[\mathbf{x}]_{d}$$ be a (determinantal) polynomial of degree $d$ with $A_{i}\in\Her_{d}(\mathbb{C})$ for all $i\in[n]$. Then \begin{gather*}L_{p}(1)=\tr(I_{d}),\\L_{p}(x_{i})=\tr(A_{i}),\\L_{p}(x_{i}x_{j})=\tr(A_{i}A_{j})=\tr(A_{j}A_{i}) \mbox{\ and}\\ L_{p}(x_{i}x_{j}x_{k})=\Rea\tr(A_{\sigma(i)}A_{\sigma(j)}A_{\sigma(k)})\end{gather*} for any permutation $\sigma$ of the formal indices $i,j,k$ for all $i,j,k\in[n].$
\end{proposicion}

We see therefore that the $L$-form has a nice behaviour in terms of traces for the monomials of degree up to three. Drawing on these relations, we build the relaxation in the next section.

\section{Defining the relaxation}

Our relaxation has the form of a LMP whose initial matrix is PSD and whose entries are defined in terms of values of the $L$-form on monomials of degree up to three. For clarity on building this, we will need the concept of mold matrices and entrywise application of maps.

\begin{notacion}[Mold matrices and entrywise transformations]
We denote $\gls{circ}$ the tensor in $B^{s_{1}\times\cdots\times s_{k}}$ obtained after applying the map $L\colon A\to B$ to each entry of $M\in A^{s_{1}\times\cdots\times s_{k}}.$ The tensor $M$ is called the \textit{mold tensor} and the map $L$ is the \textit{molding map}.
\end{notacion}

In particular, the relaxation is built over the simplest matrix mold that contains all the variables. This matrix is the mold given by the symmetric moment matrix indexed (multiplicatively) by all the monomials in the variables $x$ up to degree $1$. Due to its importance for us, we give it a name.

\begin{notacion}
We denote $\gls{mn1}=(1, x_{1}, \cdots, x_{n})^{T}(1, x_{1}, \cdots, x_{n})=$ \begin{equation}\label{mainmoment}
\begin{pmatrix}
1 & x_{1} & \cdots & x_{n}\\
x_{1} & x_{1}^{2} & \cdots & x_{1}x_{n}\\
\vdots & \vdots & \ddots & \vdots\\
x_{n} & x_{1}x_{n} & \cdots & x_{n}^2
\end{pmatrix}
\end{equation} the mold matrix indexed by monomials of degree up to $1$.
\end{notacion}

We now know our mold matrix. The molding map will be the corresponding $L$-form associated to our polynomial.

\begin{definicion}[Relaxation]\cite[Definition 3.19]{main}\
Let $p\in\mathbb{R}[[\mathbf{x}]]$ be a power series with $p(0)\neq0$ and $d\in\mathbb{N}_{0}$ and consider the symmetric matrices $A_{0}=L_{p,d}\circledcirc M_{n,\leq1}$ and $A_{i}=L_{p,d}\circledcirc (x_{i}M_{n,\leq1})$ for all $i\in[n]$. We call the linear matrix polynomial $$\gls{mpd}:=A_{0}+\sum_{i=1}^{n}x_{i}A_{i}$$ the \textit{pencil associated to $p$ with respect to the virtual degree $d$} and $$\gls{spd}:=\{a\in\mathbb{R}^{n}\mid M_{p,d}(a) \mbox{\ is PSD}\}$$ the \textit{spectrahedron associated to $p$ with respect to the virtual degree $d$}. When $p$ is a polynomial, we can omit the reference to the virtual degree by taking simply $d=\deg(p).$ 
\end{definicion}

Now it is easy to understand why we care so much about degree up to three and not beyond. The relaxation does not use any information of $p$ lying beyond degree $3$. This means that $p$ and $\trun_{k}(p)$ will have the same relaxation for $k>2$. Moreover, assuming $p(0)=1,$ the entries of the relaxation are given by degree up to three polynomials in the coefficients of $\trun_{3}(p)$, as we saw in the computations of some values of the $L$-form above. In particular, this means that, knowing the polynomial, we have a straightforward way of explicitly writing down the relaxation just defined. We now need to understand better which are the key properties that make this relaxation interesting for us.

\section{Key properties of the relaxation}

After the last discussion above, we know that computing the relaxation is easy. Evaluating its associated quadratic form at each point is also easy thanks to the properties of the $L$-form.

\setlength{\emergencystretch}{3em}%
\begin{proposicion}[Evaluation of relaxation quadratic form through the $L$-form]\cite[Lemma 3.22]{main}
Fix $d\in\mathbb{N}_{0}$ and let $p\in\mathbb{R}[[x]]$ be a power series with $p(0)\neq0$ and consider a point $a\in\mathbb{R}^{n}$ and a vector $v\in\mathbb{R}^{n+1}.$ Then $$v^{T}M_{p,d}(a)v=L_{p,d}((v_{0}+\sum_{i=1}^{n}v_{i}x_{i})^{2}(1+\sum_{i=1}^{n}a_{i}x_{i})).$$
\end{proposicion}
\setlength{\emergencystretch}{0em}%

This property makes it easy to study how the relaxation changes when an orthogonal transformation is applied to the variables of the original power series.

\begin{proposicion}[Translation of linear transformation in the variables to the $L$-form]\cite[Proposition 3.23]{main}
Fix $d\in\mathbb{N}_{0}$ and let $p\in\mathbb{R}[[x]]$ be a power series with $p(0)\neq0$ and consider the (extended) orthogonal (block) matrix $\Tilde{U}:=\diag(1,U)$ with the block $U\in\in\mathbb{R}^{n\times n}$ orthogonal. Then $$M_{p(Ux),d}=\Tilde{U}^{T}M_{p(Ux),d}\Tilde{U}.$$
\end{proposicion}

This has the next immediate consequence.

\begin{corolario}[Consequences for positivity sets]\cite[Proposition 3.24]{main}
Fix $d\in\mathbb{N}_{0}$ and let $p\in\mathbb{R}[[x]]$ be a power series with $p(0)\neq0$ and $U\in\mathbb{R}^{n\times n}$ and orthogonal matrix. Then $$S_{d}(p(Ux))=\{U^{T}a\mid a\in S_{d}(p)\}.$$
\end{corolario}

The behaviour towards restrictions is also the expected one.

\begin{corolario}[Restrictions]\cite[Lemma 3.26]{main}
Fix $d,m,n\in\mathbb{N}_{0}$ with $m\leq n$, let $q\in\mathbb{R}[[x_{1},\dots,x_{n}]]$ be a power series with $q(0)\neq0$ and define $r=q|_{x_{m+1}=0,\dots,x_{n}=0}\in\mathbb{R}[[x_{1},\dots,x_{m}]].$ Then $L_{q,d}(p)=L_{r,d}(p)$ for all polynomials $p\in\mathbb{R}[x_{1},\dots,x_{m}]$ and $\{a\in\mathbb{R}^{m}\mid (a,\mathbf{0})\in S_{d}(q)\}\subseteq S_{d}(r).$
\end{corolario}

But the most important property of the relaxation is the fact that it is indeed a relaxation.

\begin{teorema}[Relaxation]\cite[Theorem 3.35]{main}
Let $p\in\mathbb{R}[x]$ be a RZ polynomial. Then $\rcs(p)\subseteq S(p).$
\end{teorema}

These nice properties of the relaxation push us into the direction of finding objects similar to it. The first way to find these objects is studying variations of the relaxation in some special cases. This is the topic of the next section that will put us in the line towards the analysis constituting the first part of this thesis.

\section{Variations of the relaxation}

There are two natural ways to proceed in order to increase the amount of information about the polynomial collected in the relaxation. We can \textit{meaningfully} increase the number of variables of our polynomial or we can extend the number of monomials used in the mold matrix. We will see a glimpse of both strategies before studying them in detail in the next chapters.

\subsection{Extending the variables}

When we want to extend the number of variables in our RZ polynomials, the first problem that we will find is that there is, in general, not an easy way to add nonzero monomials to our polynomial in a way that keeps RZ-ness. This could be attempted through amalgamations of RZ polynomials, but this approach does not work in general, as proved in \cite{david}. Here we will explore why these extensions could be important in order to establish the GLC through the use of our relaxation. For this, we consider the easiest case possible and see how the relaxation behaves when adding \textit{enough} variables.

\begin{proposicion}[Always available determinantal representation]\cite[Theorem 4.4]{quarez2012symmetric} Let $p\in\mathbb{R}[\mathbf{x}]$ be a polynomial of degree $d$ with $p(0)\neq0$. Then there exist symmetric matrices $A_{i}\in\Sym_{s}(\mathbb{R})$ such that $p(\mathbf{x})=c\det(A_{0}+\sum_{i=1}^{n}x_{i}A_{i})$ with $c\in\mathbb{R}.$ The degree $d$ of $p$ and the size $s$ of the LMP do not have to agree. We can indeed take $s=2\binom{n+\lfloor\frac{d}{2}\rfloor}{n}$.
\end{proposicion}

Clearly, if the polynomial $p$ verifies $p(0)\neq0$, then $A_{0}$ verifies $\det(A_{0})\neq0$ and therefore we can orthogonally diagonalize it in order to obtain a signature matrix $J$ using the spectral theorem. Once we have this determinantal representation, we can increase the number of variables in a fairly easy way. Our objective with this increase will be to obtain a set of matrices linearly spanning a particular kind of set.

\begin{definicion}[Perfect subsets]\cite[Definition 3.29]{main}
We say that a subset of Hermitian matrices $U\subseteq\Her_{d}(\mathbb{R})$ is \textit{perfect} if it satisfies $$\forall A\in U\colon((\forall M\in U\colon\tr(M^{2}A)\geq0)\implies A \mbox{\ is PSD}).$$
\end{definicion}

We therefore only consider extensions of variables having at least one determinantal representation whose matrices linearly span one of these sets.

\begin{definicion}[Perfection]
Fix $n,m\in\mathbb{N}$ with $n>m$. Given a polynomial $p\in\mathbb{R}[x_{1},\dots x_{m}]$, we say that $\overline{p}\in\mathbb{R}[x_{1},\dots x_{n}]$ is a \textit{perfection} of $p$ if $\overline{p}$ admits at least one determinantal representation generating a perfect subset and $\overline{p}(x_{1},\dots,x_{m},\mathbf{0})=p(x_{1},\dots,x_{m})$.
\end{definicion}

This allows us to write the next revealing result. This result tells us that increasing the number of variables in a meaningful way helps the relaxation to manage better the information of the polynomial.

\begin{proposicion}[Exactness in perfect sets]\cite[Proposition 3.33]{main}\
Consider the two polynomials $p\in\mathbb{R}[x_{1},\dots, x_{m}]$ and $\overline{p}\in\mathbb{R}[x_{1},\dots, x_{n}]$ a perfection of $p$ whose determinantal representation linearly spanning a perfect set has size $d$. Then $$\{a\in\mathbb{R}^{m}\mid (a,0)\in S_{d}(\overline{p})\subseteq\mathbb{R}^{n}\}=\rcs(p).$$ In particular, if $\overline{p}$ is RZ, then the initial matrix $A_{0}$ of the LMP defining its associated relaxation is PSD.
\end{proposicion}

\begin{proof}
The proof of \cite[Proposition 3.33]{main} can be immediately extended to cover this more general case.
\end{proof}

Thus, in the scenario where $\overline{p}$ is RZ (which implies that $p$ itself is so) the relaxation automatically gives another determinantal representation and a cofactor respecting the original RCS of $p$ as long as its initial matrix verifies that its determinant is nonzero.

\begin{corolario}[PSD-ness in monic sets]
If the RZ polynomial $$p=\det(I_{d}+\sum_{i=1}^{m}x_{i}A_{i})\in\mathbb{R}[x_{1},\dots,x_{m}]$$ is completed to the RZ polynomial $$\overline{p}=\det(I_{d}+\sum_{i=1}^{n}x_{i}A_{i})\in\mathbb{R}[x_{1},\dots,x_{n}]$$ such that the family of symmetric matrices $I_{d}, A_{1}, \dots, A_{n}$ is linearly independent and spans $\Sym_{d}(\mathbb{R})$, then $\det(M_{p,d})=q\overline{p}\neq0$ with $\rcs(\overline{p})\subseteq\rcs(q).$
\end{corolario}

Notice that the condition that asks that the family of symmetric matrices $I_{d}, A_{1}, \dots, A_{n}$ is linearly independent and spans $\Sym_{d}(\mathbb{R})$ implies that it has to be $n+1=\dim(\Sym_{d}(\mathbb{R})).$ The result above also implies immediately the next corollary. 

\begin{corolario}[Monic perfect sets provide RZ relaxations]
If $p\in\mathbb{R}[\mathbf{x}]$ admits a RZ perfection which admits a determinantal representation whose initial matrix is PD, then the relaxation is exact and it produces a cofactor respecting the original RCS.
\end{corolario}

All this shows us that increasing the number of variables should have in general a positive impact on the approximation that the relaxation gives to the original RCS. We will see more about this in the next parts of this work. We have to keep in mind that usually we start \textit{without} any clue about the determinantal representation of $p$ and that not all polynomials admit monic symmetric determinantal representations, as we said before. Thus, in general, finding these variable extensions is not an easy task. Our intention with this section was just to show how adding variables helps, although anything here gave any real clue about how this addition of variables should work in general when we lack a determinantal representation. In the next subsection, we explore another path to build a better relaxation: increasing the number of monomials that play a role in it thus allowing for the collection of more information about the original polynomial. In this sense, remember that the relaxation of $p$ equals that of $\trun_{3}(p)$. If we want to capture the behaviour of the polynomial, we clearly need to pack more information about it in our relaxation. The possible methods allowing us to successfully do this are what we study and explore in the next subsection. However, we will see that some obstructions appear immediately on our way.

\subsection{Extending the matrices}

Putting more monomials into the mix looking directly at the polynomial is problematic for several reasons. These reasons boil down to the noncommutativity of the product of matrices. Before seeing this, we introduce these products.

\begin{definicion}[Hurwitz product]\cite[Definition 3.16]{main}\ 
 Let $A_{i}\in\mathbb{C}^{d\times d}$ for $i\in[n]$ and $\alpha\in\mathbb{N}^{n}_{0}.$ We define the \textit{$\alpha$-Hurwitz product} of $A_{i}$ for $i\in[n]$ as \begin{gather*}\gls{hur}:=\sum_{f\colon[|\alpha|]\to[n]\ \forall i\in[n]\ \card(f^{-1}(i))=\alpha_{i}}A_{f(1)}\cdots A_{f(|\alpha|)}\in\mathbb{C}^{d\times d}.\end{gather*} 
\end{definicion}

Now we can see that the $L$-forms associated to determinantal polynomials come in the form of these products. We will see this next.

\begin{teorema}[Values of $L$-form as traces of Hurwitz products]\cite[Proposition 3.17]{main} Let $d\in\mathbb{N}$ and consider $A_{1},\dots,A_{n}\in\Her_{d}(\mathbb{C})$ so that we can build the RZ polynomial $p:=\det(I_{d}+\sum_{i=1}^{n}x_{i}A_{i})\in\mathbb{R}[\mathbf{x}].$ Then $$L_{p,d}(x^{\alpha})=\frac{1}{\binom{|\alpha|}{\alpha}}\tr(\hur_{\alpha}(A_{1},\dots,A_{n}))$$ for all $\alpha\in\mathbb{N}^{n}_{0}.$ 
\end{teorema}

Now, this allows us to see what is our problem when we try to extract traces of products of the matrices involving more than $3$ matrices using just the information about the coefficients of the polynomial. This, in particular, contrasts with what happened to traces of less than $4$ matrices and this is why we could establish these traces in Proposition \ref{lformtraces}.

\begin{remark}[$L$-form behaviour and noncommutativy problems for degree higher than three] Suppose $p=\det(I_{d}+\sum_{i=1}^{n}x_{i}A_{i})$ has a determinantal representation with $A_{i}\in\Sym_{d}(\mathbb{R})$ for $i\in[n].$ Up to degree $3$, the $L$-form allows us to separate the trace of the products of the matrices in the determinantal representation. This means that we have $L_{p}(x^{\alpha})=\tr(A^{\alpha})$ for $\alpha\in\mathbb{N}^{n}_{0}.$ This is no longer true for degree higher than $4$ because for example $\tr(A_{1}A_{2}A_{1}A_{2})\neq\tr(A_{1}A_{2}A_{2}A_{1})$ in general and therefore, without access to the determinantal representation, it is not clear what should the value of $L_{p}(x^{2}_{1}x^{2}_{2})$ be in order to capture that product of traces into the matrix defining the relaxation.
\end{remark}

The remark above will be important when we try to extend the number of monomials at play. It also points clearly in the direction of a necessity for lifting our polynomials in commutative variables into a \textit{nice} non-commutative version and then applying the $L$-form there. However, dealing with non-commutative variables lies beyond our objectives here. One could ask why it is important to extend the number of monomials with which we deal in the relaxation. The answer comes again in terms of traces.

\begin{remark}[Extending the monomials in terms of traces]
 Suppose again that $p=\det(I_{d}+\sum_{i=1}^{n}x_{i}A_{i})$ has a determinantal representation with $A_{i}\in\Her_{d}(\mathbb{R})$ for $i\in[n]$ and $(I_{d},A_{1},\dots,A_{n})$ linearly independent. We saw above that, if we increase its variables to form the polynomial extension $\det(I_{d}+\sum_{i=1}^{m}x_{i}A_{i})$ for $m\geq n$ such that $(I_{d},A_{1},\dots,A_{m})$ is a basis of $\Her_{d}(\mathbb{R})$, then applying the relaxation to $\overline{p}:=\det(I_{d}+\sum_{i=1}^{m}x_{i}A_{i})$ gives us a determinantal representation for the product $\overline{p}\cdot\overline{q}$ for some real zero polynomial $\overline{q}\in\mathbb{R}[x_{1},\dots,x_{m}]$ with $\rcs(\overline{p})\subseteq\rcs(\overline{p})$, which, setting the additional variables $x_{n+1},\dots,x_{m}$ to zero so we come back to the original subspace, gives us immediately a determinantal representation of the product $pq$ with $\rcs(p)\subseteq\rcs(q)$ with $q:=\overline{q}|_{(x_{n+1},\dots,x_{m})=\mathbf{0}}$. This is what the GLC requires. Now, it is not clear how to extend the variables without having already previously a determinantal representation. However, applying Burnside's theorem on matrix algebras (this is a well-known result widely used in many subfields of linear algebra, see its proof in different environments, e.g., in \cite{shapiro2014burnside},\cite{lomonosov2004simplest},\cite{freslon2023compact},\cite{horn2012matrix} or \cite{zhelobenko1973compact}), we know that two generic symmetric matrices $A,B\in\Sym_{d}(\mathbb{R})$ generate the whole matrix algebra $\mathbb{R}^{d\times d}$. Thus, the generic $p$ will have its matrices $A_{1}$ and $A_{2}$ generating the whole matrix algebra and therefore taking the symmetric matrices generated by $A_{1}$ and $A_{2}$ would give $\Sym_{d}(\mathbb{R})$. Using this, we can form a basis $\{B_{1},\dots,B_{m}\}$ of $\Sym_{d}(\mathbb{R})$ written only in terms of $A_{1}$ and $A_{2}$ and form the matrices $\tr(B_{i}A_{k}B_{j})$ with $k\in[n]\cup\{0\}$ with $A_{0}=I_{d}$. Now, if we could completely split these products through the $L$-form we could try to recover a determinantal representation just looking at the $L$-form generically without needing access to the determinantal representation.
\end{remark}

This means that understanding the $L$-forms better in terms of traces would bring us very close to solving the GLC. Unfortunately, this task is not easy. Another consequence of this is that, augmenting the monomials at play, we could hope to build better relaxations or even to recover our polynomial in special cases. However, doing this directly through the $L$-forms as we understand them so far has a problem that we will see in detail in the next section and that we describe next. This problem comes in the form of an obstruction for the general initial moment matrix generated when we expand too much the number of monomials considered.

\begin{remark}[Lost of PSD-ness and control]
In general, the main obstruction we will encounter is that the initial matrix of the relaxation stops being in general PSD when we consider more monomials. This is a problem because we search for determinantal representations where this matrix is PSD. Of course, this initial matrix can be made PSD if we \textit{artificially} multiply by some polynomials that we know that have PSD initial matrices. This can be done because of the properties of the relaxation. However, this does not really add any new information about the polynomial to the matrices in the LMP defining the relaxation. This strategy also does not rely in making the initial matrix bigger in order to carry such information. This artificial construction does not accomplish what we want because we want to look at matrix molds and not at arbitrary factors that do not actually help in the fundamental task of increasing the amount of information about the polynomial carried by the relaxation.
\end{remark}

Now, we are in a position to understand why the two kinds of extensions that we study next are important to consider. Studying certain manageable particularizations of these extensions is the topic of the rest of this part. In particular, we will see the limitations that we encounter when trying to follow these paths towards a useful extension in our way to obtain the kind of determinantal representations we want.
\chapter{Limitations of Immediate Matrix Extensions}\label{ChLimi}

Building on what we saw in the previous chapter, we study the main obstruction encountered when we try to extend the relaxation: the initial matrix should be PSD. In general, arbitrary extension will not respect this. However, we can find an extension that respects this, although we will end up seeing that it still presents some other problems. We present the results in this chapter as an example procedure of how to analyze and study these extensions. We could not find larger extensions overcoming the obstruction we study here. The main tool in this section will be tracial matrix inequalities. This is a huge field. We will mainly use the inequality in \cite[Theorem 1]{dawei} for $k=1$ (also in \cite[Theorem 2.1.12]{bikchentaevtrace}). In order to understand the complexity of this field, we direct the interested reader to the papers \cite{theobald1975inequality,shebrawi2013trace,lasserre1995trace,fang1994inequalities,petz1994survey,sutter2017multivariate,coope1994matrix}. Notice also the connection to the invariant theory of matrices in \cite{klep2018positive,PROCESI1976306}.

\section[The construction of logarithmic spectrahedral relaxations]{Revisiting the construction of logarithmic spectrahedral relaxations}

We pivot over the constructions of the previous chapter in order to build something carrying more information. Some new notation will be necessary in order to shorten our arguments.

\begin{notacion}[Linear combinations of the same matrices]
Let $A_{i}\in\Her_{d}(\mathbb{C})$ for $i\in[n]$ for some $n\in\mathbb{N}.$ For simplicity, in what follows we will denote $K:=u_{1}A_{1}+\cdots+u_{n}A_{n}$ and $K':=u'_{1}A_{1}+\cdots+u'_{n}A_{n}$ for the values of $n$ and $A_{i},u_{i},u'_{i}$ specified in each circumstance so there is no possible confusion. Moreover, we suppose wlog that all our RZ polynomials $p$ verify $p(0)=1$ and $\deg(p)>0$ and denote this set $\gls{rrz}.$
\end{notacion}

The mold that we used above is going to be extended in order to fit more monomials. The new mold will allow us to carry more information than the original mold of the relaxation allowed.

\begin{remark}[Mold of the relaxation]
The relaxation in the previous chapter was built over the mold $(1, x_{1}, \cdots, x_{n})^{T}(1, x_{1}, \cdots, x_{n})$. We will increase this mold introducing $n$ new monomials.
\end{remark}

We remind the reader of our entrywise application of maps to entries of tensors. We use tensors here because we will describe how the relaxation is constructed in a more general setting. In particular, we can interpret that the relaxation in \cite[Definition 3.19]{main} is constructed from a (symmetric) [(generalized)] moment $3$-tensor given as the tensor product of the vector of monomials $(1,x_{1},\dots,x_{n})$ with itself. This gives us a (symmetric) $3$-tensor filled with monomials up to degree $3$. Now we introduce some notation and define the $L$-polynomial determinants of such a symmetric tensor. The generality we require for this is the reason why we set our notation to be able to speak easily about entrywise applications of functions to the elements inside a matrix. We will see that this is helpful for our task. Now we remember how to denote subtensors of lower order inside a tensor. In order to keep the notation as simple as possible, these will be denoted just by fixing some indices.

\begin{definicion}[Sliced specializations of tensors]
Fix $k,s\in\mathbb{N}.$ Consider $A\in S^{n_{1}\times\cdots\times n_{k}}$ a $k$-tensor with elements given by $t_{i_{1},\dots,i_{k}}$ for $i_{j}\in[n_{j}]$ for $j\in[k]$. We denote the subtensor of order $k-s$ obtained from fixing $s$ indices in positions $(j_{1}<\dots<j_{s})$ to the respective numbers $(i_{j_{1}},\dots,i_{j_{s}})$ for $s,j_{l}\in[k]$ and $i_{j_{l}}\in[n_{j_{l}}]$ for $l\in[s]$ by $A_{(j_{1}\to i_{j_{1}},\dots,j_{s}\to i_{j_{s}})}$ and call it the \textit{$(j_{1}\to i_{j_{1}},\dots,j_{s}\to i_{j_{s}})$ slice of $A$}.
\end{definicion}

Now we can easily define the map in (the entries of) a $3$-tensor that gives the relaxation in \cite{main} as we have the notation necessary to manipulate elements entrywise and to refer to subtensors obtained by fixing some indices.

\begin{definicion}[From tensors to determinantal polynomials]
\label{detpol}
Let $R$ be a ring and $A\in S^{n\times n\times n}$ be a symmetric $3$-tensor with entries in the set $S$ and consider two maps $f\colon[n]\to R$ and $L\colon S\to R.$ Then we define the \textit{$L$-polynomial determinant of the $3$-tensor $A$ with markers $f$} as $\det(f,L,A):=\det(f(1)L\circledcirc A_{1\to 1}+\cdots+f(n)L\circledcirc A_{1\to n}).$
\end{definicion}

The symmetry of the $3$-tensor gives us some freedom at the time of slicing.

\begin{hecho}[Symmetry]
The index being fixed in every summand inside the determinant can be chosen to be any of the three available as a consequence of the symmetry of the $3$-tensor, i.e., $\det(f,L,A)=\det(f(1)L\circledcirc A_{i\to 1}+\cdots+f(n)L\circledcirc A_{i\to n})$ for $i\in\{1,2,3\}.$
\end{hecho}

It is easy to build the relaxation using this.

\begin{remark}[Construction of the relaxation through sliced specializations]
In particular, the relaxation in \cite[Definition 3.19]{main} is constructed fixing $A$ the obvious mold moment symmetric $3$-tensor built through taking the $3$-outer power of the vector $(1,x_{1},\dots,x_{n})$, $L$ any of the $L$-forms defined there corresponding to a power series $p\in\mathbb{R}[\mathbf{x}]$ and $f\colon\{1,\dots,n+1\}\to\mathbb{R}[\mathbf{x}]$ with $f(1)=1$ and $f(i)=x_{i-1}$ for $i\in\{2,\dots,n+1\}.$ We observe that a necessary condition for this operation to give an RZ polynomial is that the matrix $f(1)L\circledcirc A_{i\to 1}$ is PD. \cite[Theorem 3.35]{main} shows that in fact this condition is fulfilled in this case whenever we take $L=L_{p,e}$ with $p\in\mathbb{R}[\mathbf{x}]$ an RZ polynomial with $e\geq\deg(p).$ We will put the focus in what follows on this condition translated to other situations and objects under our study.
\end{remark}

We use moment matrices as molds to build other polynomials over them applying the $L$-forms $L_{p,e}$ entrywise together with other intermediate operations as in \cite[Definition 3.19]{main}. We will refer quite often to these objects, so we simplify their name.

\begin{notacion}[Mold moment matrices](cf. \cite[Subsection 2.2.2]{putinar2008emerging} and \cite[Subsection 3.2.1]{lasserre2009moments})
Each moment matrix in the next section that is supposed to be the initial matrix (at the origin) of an LMP generated by applying $L_{p,e}$ entrywise will be called a \textit{mold moment matrix} (or \textit{MMM}, for short).
\end{notacion}

We need mold big enough so they are able to carry information about many coefficients. There are many possibilities to explore, but just a few seem to overcome the most basic obstructions.

\begin{remark}[Limit amount of information]
The main problem of this relaxation is that it takes a limited amount of information about the polynomial $p$ which interests us at each moment. Here we investigate then precisely other directions that allow the collection of higher degree information given by using MMMs whose indexes run up to these higher degrees.
\end{remark}

There are of course very different possibilities of structures of MMMs that lead us to the objective of collecting the information given by the coefficients of the monomials of higher degree. These can be classified in several categories that we will name when we consider each candidate structure. This is the topic of the next sections. For this we first have to study the traces.

\section{Trace inequalities matter}

We begin with a lemma about an inequality between the two different traces of products of four matrices that are pairwise equal that will be useful in the proof of Proposition \ref{propomatrix}. In \cite[Theorem 1]{dawei} there is a generalization of the next result. We reproduce here only the short proof of the particular case $k=1.$

\begin{lema}[$(4,2)$-products trace inequality]\label{lemmatrace}
Let $A,B\in\Sym_{d}(\mathbb{R})$ be two real symmetric matrices. Then there are just two possible traces of $4$-products formed by two copies of each matrix and these are given by $\tr(A^2B^2)\geq\tr(ABAB),$ which gives equality if and only if $A$ and $B$ commute.
\end{lema}

\begin{proof}
The first part is immediate checking all the possible products and reducing using the properties of the trace. For the second part, we refer to \cite{dawei}, where we found the following straightforward argument. The matrix $AB-BA$ is skew-symmetric and therefore its square $(AB-BA)^{2}$ is negative semidefinite so we obtain $0\geq\tr((AB-BA)^{2})=2(\tr(A^2B^2)-\tr(ABAB)),$ which finishes the proof.
\end{proof}

This inequality allows us to establish another inequality in terms of the $L$-form. We need to connect the $L$-form with the trace.

\begin{lema}[$L$-form versus trace inequality]\label{lemmaineq}
Let $r(x_{1},x_{2})=\det(I_{d}+A_{1}x_{1}+A_{2}x_{2})\in\mathbb{R}_{\mbox{RZ}}[\mathbf{x}]$ be a bivariate RZ polynomial and fix an arbitrary integer $e\in\mathbb{N}_{0}.$ Then $L_{r,e}(x_{1}^2(u_{1}x_{1}+u_{2}x_{2})^{2})\geq\tr(A_{1}(u_{1}A_{1}+u_{2}A_{2})A_{1}(u_{1}A_{1}+u_{2}A_{2}))$ for each $(u_{1},u_{2})\in\mathbb{R}^{2}.$
\end{lema}

\begin{proof}
We only have to develop the products of matrices and express the values of $L_{r,e}$ over the corresponding products of variables in terms of the trace. For this we must have in mind the following list of values obtained applying \cite[Proposition 3.17]{main} to write the form $L_{r,e}$ in terms of the traces of the matrices in the representation of $r$. Here we have to take special care of monomials of degree four in two variables, which are not considered in \cite[Corollary 3.18]{main} because they do not behave as nice as the monomials of degree up to $3.$ Thus applying \cite[Proposition 3.17]{main} we observe that we have the following expressions in terms of the trace. \begin{enumerate}
    \item $$L_{r,e}(x_{1}^3x_{2})=4\binom{4}{3,1}^{-1}\tr(A_{1}^3 A_{2})=\tr(A_{1}^3 A_{2}).$$
    \item \begin{gather*}L_{r,e}(x_{1}^2x_{2}^2)=\binom{4}{2,2}^{-1}(4\tr(A_{1}^2 A_{2}^2)+2\tr(A_{1}A_{2}A_{1}A_{2}))\\=\frac{4}{6}\tr(A_{1}^2 A_{2}^2)+\frac{2}{6}\tr(A_{1}A_{2}A_{1}A_{2}).\end{gather*}
    \item $$L_{r,e}(x_{1}^4)=\binom{4}{4}^{-1}\tr(A_{1}^4)=\tr(A_{1}^4).$$
\end{enumerate}
This allows us to write the next chain of comparisons where we use Lemma \ref{lemmatrace} for the inequality
\begin{gather}
   L_{r,e}(x_{1}^2(u_{1}x_{1}+u_{2}x_{2})^2)=L_{r,e}(u_{1}^{2}x_{1}^4+2u_{1}u_{2}x_{1}^3x_{2} + u_{2}^{2}x_{1}^2x_{2}^2)=\\\tr(u_{1}^{2}A_{1}^4+2u_{1}u_{2}A_{1}^3A_{2}+\frac{4}{6}u_{2}^{2}A_{1}^2 A_{2}^2+\frac{2}{6}u_{2}^{2}A_{1}A_{2}A_{1}A_{2})\geq\\
\begin{gathered}\label{playineq}\tr(u_{1}^{2}A_{1}^4+2u_{1}u_{2}A_{1}^3A_{2}+u_{2}^{2}A_{1}A_{2}A_{1}A_{2})=\\\tr(A_{1}(u_{1}A_{1}+u_{2}A_{2})A_{1}(u_{1}A_{1}+u_{2}A_{2})),\end{gathered}
\end{gather} which finishes the proof of our result.
\end{proof}

Equation \ref{playineq} in the chain of comparisons of the proof above will be important in the future. For that reason, it has its own number. A similar lemma presenting these kinds of inequalities follows from the one above.

\begin{lema}[$L$-form versus trace inequality with more degree freedom]\label{lemmaineqalmost}
Fix $\mathbf{x}=(x_{1},x_{2})$ and let $r(\mathbf{x})=r(x_{1},x_{2})=\det(I_{d}+A_{1}x_{1}+A_{2}x_{2})=\det(I_{d}+x^{T}A)\in\mathbb{R}_{\mbox{RZ}}[x]_{d}$ be a bivariate RZ polynomial of degree $d:=\deg(r)$ and fix an arbitrary integer $d\leq e\in\mathbb{N}_{0}.$ Then \begin{equation*}L_{r,e}((1+u^{T}x+x_{1}u'^{T}x)^2)\geq\tr((I_{d}+u^{T}A+A_{1}(u'^{T}A))^2)\end{equation*} for each $u=(u_{1},u_{2})^{T}, u'=(u'_{1},u'_{2})^{T}\in\mathbb{R}^{2}$ and where $A=(A_{1},A_{2})\in(\Sym_{d}(\mathbb{R}))^{2}$ so we can denote $u^{T}A=u_{1}A_{1}+u_{2}A_{2}$ and $u'^{T}A=u'_{1}A_{1}+u'_{2}A_{2}.$
\end{lema}

\begin{proof}
We want to show\begin{gather}L_{r,e}((1+u^{T}x+x_{1}u'^{T}x)^2)=\\L_{r,e}((1+u^{T}x)^2+2x_{1}(1+u^{T}x)(u'^{T}x)+x_{1}^2(u'^{T}x)^2)=\\\begin{gathered}\label{intra1}L_{r,e}\bigg{(}(1+u_{1}x_{1}+u_{2}x_{2})^2+2x_{1}(1+u_{1}x_{1}+u_{2}x_{2})(u'_{1}x_{1}+u'_{2}x_{2})\\+x_{1}^2(u'_{1}x_{1}+u'_{2}x_{2})^2\bigg{)}\geq\tr((I_{d}+u^{T}A+A_{1}(u'^{T}A))^2).\end{gathered}\end{gather} Remembering \cite[Proposition 3.17 and Corollary 3.18]{main} and the developments in the proof of Lemma \ref{lemmaineq}, we see that the only problematic case in order to express the form $L_{r,e}$ in terms of traces is $L_{r,e}(x_{1}^2x_{2}^2)=\frac{4}{6}\tr(A_{1}^2 A_{2}^2)+\frac{2}{6}\tr(A_{1}A_{2}A_{1}A_{2})$ where the trace divides into two summands. Therefore, developing the first square on the LHS of Inequation \ref{intra1} and taking into account the fact that the last summand has problematic degree $4$ and so it deserves an special treatment, we have that the inequality that we must prove can be expanded and then simplified in terms of traces to \begin{gather*}e+\tr\bigg{(}2(u_{1}A_{1}+u_{2}A_{2})+(u_{1}A_{1}+u_{2}A_{2})^2+\\2A_{1}(I_{d}+u_{1}A_{1}+u_{2}A_{2})(u'_{1}A_{1}+u'_{2}A_{2})\bigg{)}+L_{r,e}(x_{1}^2(u'_{1}x_{1}+u'_{2}x_{2})^2)\geq\\e+\tr\bigg{(}2(u_{1}A_{1}+u_{2}A_{2})+(u_{1}A_{1}+u_{2}A_{2})^2+\\2A_{1}(I_{d}+u_{1}A_{1}+u_{2}A_{2})(u'_{1}A_{1}+u'_{2}A_{2})\bigg{)}+\\\tr(A_{1}(u'_{1}A_{1}+u'_{2}A_{2})A_{1}(u'_{1}A_{1}+u'_{2}A_{2})),\end{gather*} where we used Lemma \ref{lemmaineq} to deal at once with the otherwise problematic \begin{gather*}
   L_{r,e}(x_{1}^2(u'_{1}x_{1}+u'_{2}x_{2})^2)=L_{r,e}(u'_{1}{}^{2}x_{1}^{4}+2u'_{1}u'_{2}x_{1}^{3}x_{2})+L_{r,e}(u'_{2}{}^{2}x_{1}^{2}x_{2}^2)=\\\tr(u'_{1}{}^{2}A_{1}^{4}+2u'_{1}u'_{2}A_{1}^{3}A_{2}+\frac{4}{6}u'{}_{2}^{2}A_{1}^{2} A_{2}^{2}+\frac{2}{6}u'_{2}{}^{2}A_{1}A_{2}A_{1}A_{2}).
\end{gather*}
Finally, we can see that the the last expression in the previous inequality provides an easy path towards the inequality that we are searching for so we have at once
\begin{gather*}e+\tr(2K+K^2+2A_{1}(I_{d}+K)K'+A_{1}K'A_{1}K')\geq\\ d+\tr(2K+K^2+2A_{1}(I_{d}+K)K'+A_{1}K'A_{1}K')=\\\tr(I_{d}+2K+K^2+2A_{1}(I_{d}+K)K'+A_{1}K'A_{1}K')=\\\tr((I_{d}+K)^2+2(I_{d}+K)A_{1}K'+A_{1}K'A_{1}K')=\tr(((I_{d}+K)+A_{1}K')^2),\end{gather*} where we denoted $K=u_{1}A_{1}+u_{2}A_{2}$ and $K'=u'_{1}A_{1}+u'_{2}A_{2}$ and used the facts that $e\geq d$ and several identities between traces of the products of up to $3$ real symmetric matrices that let us conclude the previous chain of (in)equalities in terms of traces that completes the proof of this lemma.
\end{proof}

The statement of the result above cover many polynomials. In fact, it covers all polynomials with two variables we are interested in.

\begin{remark}
Indeed, we can observe that the set of RZ polynomials considered in the statement of the previous lemma is actually all the set $\mathbb{R}_{\mbox{RZ}}[\mathbf{x}]$. This is so because the lemma is stated for the case of two variables, where the Helton-Vinnikov theorem is available and therefore we know that we can always write $r(x_{1},x_{2})=\det(I_{d}+A_{1}x_{1}+A_{2}x_{2}),$ where we can suppose that $A_{1}$ is diagonal by a trivial orthogonal diagonalization argument.
\end{remark}

Now, experimentation with the $L$-form plays a key role in finding the identities we look for. For the importance of experiments of this type in order to determine probabilities in trace inequalities, we refer the reader to the work of Greene in \cite{greene2014traces} and the master's theses of his students \cite{schneider2015investigating,huang2009traces}.

\begin{remark}[Experiments on the $L$-form]
We note here that $\tr((I_{d}+K+A_{1}K')^2)\ngeq0$ in general although our random experiments indicate that when $A_{1}$ is diagonal and $K$ and $K'$ are real symmetric $\tr((I_{d}+K+A_{1}K')^2)>0$ happens quite often. Moreover, at this point, our random experiments also indicate that, in general, we should have $L_{r,e}((1+u^{T}x+x_{1}u'^{T}x)^2)\geq0$ and this is what we will prove.
\end{remark}

Equipped with our new knowledge about inequalities using the $L$-form, we try now to expand the indexing column and row of our MMMs just to cover one degree more.

\begin{definicion}[$i$-monomial extensions of the relaxation mold]
Fix the tuple of variables $\mathbf{x}=(x_{1},\dots,x_{n})$. Denote $\gls{mni}:=$\begin{gather}
(1, x_{j_{1}}, \cdots, x_{j_{m}}, x_{i}x_{i_{1}}, \cdots, x_{i}x_{i_{k}})^{T}(1, x_{j_{1}}, \cdots, x_{j_{m}}, x_{i}x_{i_{1}}, \cdots, x_{i}x_{i_{k}})=\\\begin{gathered}
    \label{ext1}
\begin{pmatrix}
1 & x_{j_{1}} & \cdots & x_{j_{m}} & x_{i}x_{i_{1}} & \cdots & x_{i}x_{i_{k}}\\
x_{j_{1}} & x_{j_{1}}^{2} & \cdots & x_{j_{1}}x_{j_{m}} & x_{j_{1}}x_{i}x_{i_{1}} & \cdots & x_{j_{1}}x_{i}x_{i_{k}}\\
\vdots & \vdots & \ddots & \vdots & \vdots & \ddots & \vdots\\
x_{j_{m}} & x_{j_{1}}x_{j_{m}} & \cdots & x_{j_{m}}^2 & x_{j_{m}}x_{i}x_{i_{1}} & \cdots & x_{j_{m}}x_{i}x_{i_{k}}\\
x_{i}x_{i_{1}} & x_{i}x_{i_{1}}x_{j_{1}} & \cdots & x_{i}x_{i_{1}}x_{j_{m}} & x_{i}^{2}x_{i_{1}}^{2} & \cdots & x_{i}^{2}x_{i_{1}}x_{i_{k}}\\
\vdots & \vdots & \ddots & \vdots & \vdots & \ddots & \vdots\\
x_{i}x_{i_{k}} & x_{i}x_{i_{k}}x_{j_{1}} & \cdots & x_{i}x_{i_{k}}x_{j_{m}} & x_{i}^{2}x_{i_{1}}x_{i_{k}} & \cdots & x_{i}^{2}x_{i_{k}}^{2}
\end{pmatrix},
\end{gathered}
\end{gather} where $k,m,i,i_{a},j_{b}\in[n]$ for all $a\in[k], b\in[m]$ and $i_{a}\neq i_{a'}$ whenever $a\neq a'$ and $j_{b}\neq j_{b'}$ whenever $b\neq b'$. We call it the \textit{$i$-monomial $(j_{1},\dots,j_{m};i_{1},\dots,i_{k})$-extension of the relaxation mold $M_{n,\leq1}$}. For simplicity, we will denote it simply $M_{n,i}$ when we take $(j_{1},\dots,j_{m})=(i_{1},\dots,i_{k})=(1,\dots,n)$ and call it the \textit{$i$-monomial extension of the relaxation mold $M_{n,\leq1}$}.
\end{definicion}

It turns out that this MMM produces PSD matrices for the independent term, i.e., the matrices obtained by applying $L_{p,e}$ to all the entries of the matrix in Equation \ref{ext1} above for any RZ polynomial $p\in\mathbb{R}[\mathbf{x}]$ and $e\geq\deg(p)$ are all PSD. We proceed proving this.

\begin{remark}[Looking at the initial matrix]
The initial matrix being PSD is a cut for MMM to be useful for us. Otherwise, we cannot produce a monic representation.
\end{remark}

Remember that we will use our notation for entrywise application. Additionally, in the course of the next results we will adopt some simplifications that will make our life easier wlog.

\setlength{\emergencystretch}{3em}%
\begin{simplificacion}\label{remarkred}
We remember that because of classical results it is enough to show that the biggest possible matrix described in Equation \ref{ext1}, i.e., the one with $k=n=m$, is PSD as we know that if a matrix is PSD then each principal submatrix of it is PSD and moreover simultaneous permutations of corresponding rows and columns do not change the set of eigenvalues. Thus, we fix wlog (as a simple relabeling of the variables shows) $i=1$ instead of $i\in\{1,\dots,n\}$ arbitrary and $k=n=m$ in what follows. Additionally, we introduce temporarily, as in \cite[Lemma 3.22]{main} in order to shorten the formulas, the false variables $x_{0}:=1$ and $x_{n+j}=:x_{i}x_{j}$ for each $j\in\{1,\dots,n\}$ and denote, in accordance with the simplifications just adopted, $M$ the matrix \ref{ext1} with $i=1$ and $k=n=m$ and divide any arbitrary vector $v=(w_{0},w_{1},\dots,w_{n},w'_{1},\dots,w'_{n})^{T}=(w^{T},w'^{T})^{T}=(v_{0},\dots,v_{2n})^{T}\in\mathbb{R}^{2n+1}$ with $w=(w_{0},w_{1},\dots,w_{n})\in\mathbb{R}^{n+1}$ and $w'=(w'_{1},\dots,w'_{n})\in\mathbb{R}^{n}$. We note therefore that we have just extended the vector of variables to $x=(x_{0},x_{1},\dots,x_{n},x_{n+1},\dots,x_{2n})$ althought this extension is artificial because the new elements in the tuple are false variables depending on the original variables.
\end{simplificacion}
\setlength{\emergencystretch}{0em}%

The last point of the previous remark forces us to introduce the next notion in order to keep short notations easier to understand.

\begin{definicion}[Indexed scalar product]
Let $I$ be an index set where we fix two subsets $A,B\subseteq I$ and consider the maps $f\colon A\to F$ and $g\colon B\to F$ to a field $F.$ We can consider the $f$ and $g$ elements of some vector space over $F$ so $f\in F^{A}$ and $g\in F^{B}.$ We denote as usual $\gls{ftg}:=\sum_{i\in A\cap B}f(i)g(i)$ the \textit{indexed scalar product of $f$ and $g$ over $I$} and note that it extends the usual scalar product of vectors when $A=B,$ i.e., when the vectors are (indexed) in the same vector space over the field $F.$
\end{definicion}

The next lemma is analogous to \cite[Lemma 3.22]{main} but applied to our extension of the relaxation. We use the definition above because, as explained in Simplification \ref{remarkred}, $w=(w_{0},w_{1},\dots,w_{n}),w'=(w'_{1},\dots,w'_{n})$ while $x=(x_{0},x_{1},\dots,x_{2n})=(x_{0},x_{1},\dots x_{n},x_{n+1}\dots,x_{n+n})$ with $x_{n+j}=x_{i}x_{j}$ but we set $i=1$ wlog.

\begin{lema}[Computation of the quadratic form of the extension through the $L$-form and decomposition of vectors]\label{lemmaLform}
Let $p\in\mathbb{R}[x]$ be a power series with $p(0)\neq 0$ and consider $e\in\mathbb{N}_{0}$ and $v\in\mathbb{R}^{2n+1}.$ Then $v^{T}(L_{p,e}\circledcirc M)v=$\begin{equation*}L_{p,e}((w^{T}x)^2+2x_{1}(w^{T}x)(w'^{T}x)+x_{1}^2(w'^{T}x)^2).\end{equation*}
\end{lema}

\begin{proof} We proceed simply computing $v^{T}(L_{p,e}\circledcirc M)v=$
\begin{gather*}
\sum_{l=0}^{2n}\sum_{j=0}^{2n}v_{l}v_{j}L_{p,e}(x_{l}x_{j})=L_{p,e}\left(\left(\sum_{l=0}^{2n}v_{l}x_{l}\right)\left(\sum_{j=0}^{2n}v_{j}x_{j}\right)\right)=\\L_{p,e}\left(\left(\sum_{l=0}^{n}v_{l}x_{l}+\sum_{l=n+1}^{2n}v_{l}x_{l}\right)\left(\sum_{j=0}^{n}v_{j}x_{j}+\sum_{j=n+1}^{2n}v_{j}x_{j}\right)\right)=
\\L_{p,e}\left(\left(\sum_{l=0}^{n}w_{l}x_{l}+x_{1}\sum_{l=1}^{n}w'_{l}x_{l}\right)\left(\sum_{j=0}^{n}w_{j}x_{j}+x_{1}\sum_{j=1}^{n}w'_{j}x_{j}\right)\right)=\\L_{p,e}\left(\left(\sum_{l=0}^{n}w_{l}x_{l}\right)^2+2x_{1}\left(\sum_{j=0}^{n}w_{j}x_{j}\right)\left(\sum_{l=1}^{n}w'_{l}x_{l}\right)+x_{1}^2\left(\sum_{l=1}^{n}w'_{l}x_{l}\right)^2\right),\end{gather*} which is the identity we wanted to prove.
\end{proof}

We will need an improvement of the inequality in Lemma \ref{lemmaineq}.

\begin{mejora}[Tightening $L$-form versus trace inequalities]
\label{lemmaineqimpro} Let $$r(x_{1},x_{2})=\det(I_{d}+A_{1}x_{1}+A_{2}x_{2})\in\mathbb{R}_{\mbox{RZ}}[\mathbf{x}]$$ be a bivariate RZ polynomial and fix an arbitrary integer $e\in\mathbb{N}_{0}.$ Then \begin{gather*}L_{r,e}(x_{1}^2(u_{1}x_{1}+u_{2}x_{2})^{2})\geq\\\frac{1}{2}\tr(A_{1}(u_{1}A_{1}+u_{2}A_{2})A_{1}(u_{1}A_{1}+u_{2}A_{2})+A_{1}^{2}(u_{1}A_{1}+u_{2}A_{2})^{2})\end{gather*} for each $(u_{1},u_{2})\in\mathbb{R}^{2}.$
\end{mejora}

\begin{proof}
This proof consists in a different game with the quantities in Equation \ref{playineq} so we have instead $L_{r,e}(x_{1}^2(u_{1}x_{1}+u_{2}x_{2})^2)=L_{r,e}(u_{1}^{2}x_{1}^4+2u_{1}u_{2}x_{1}^3x_{2} + u_{2}^{2}x_{1}^2x_{2}^2)=$
\begin{gather*}\tr(u_{1}^{2}A_{1}^4+2u_{1}u_{2}A_{1}^3A_{2}+\frac{4}{6}u_{2}^{2}A_{1}^2 A_{2}^2+\frac{2}{6}u_{2}^{2}A_{1}A_{2}A_{1}A_{2})\geq\\\tr(u_{1}^{2}A_{1}^4+2u_{1}u_{2}A_{1}^3A_{2}+\frac{1}{2}u_{2}^{2}A_{1}^2 A_{2}^2+\frac{1}{2}u_{2}^{2}A_{1}A_{2}A_{1}A_{2})=\\\tr(\frac{1}{2}u_{1}^{2}A_{1}^4+u_{1}u_{2}A_{1}^3A_{2}+\frac{1}{2}u_{2}^{2}A_{1}^2 A_{2}^2)+\\\tr(\frac{1}{2}u_{1}^{2}A_{1}^4+u_{1}u_{2}A_{1}^3A_{2}+\frac{1}{2}u_{2}^{2}A_{1}A_{2}A_{1}A_{2})=\\\tr(B_{1}^{2}(u_{1}B_{1}+u_{2}B_{2})^{2})+\tr(B_{1}(u_{1}B_{1}+u_{2}B_{2})B_{1}(u_{1}B_{1}+u_{2}B_{2})),
\end{gather*} where we called $B_{i}:=2^{-\frac{1}{4}}A_{i}$ and used Lemma \ref{lemmatrace} to set the inequality.
\end{proof}

All the presented lemmas about the interplay between the $L$-form and the trace in the extension of the relaxation together with the trace inequalities studied above have as a consequence our main theorem. We dedicate the next and final section to the presentation of this result.

\section[Extensions of the multiplicative arrange]{Symmetric matrices constructed via extensions of the multiplicative arrange}

Now we are ready to see that the initial matrix of our extension is indeed PSD. This is the main condition for an extension to be worth studying, as otherwise it is impossible to produce, in general, a monic representation from the extension and therefore to guarantee that the relaxation provides a RZ polynomial itself.

\begin{remark}
For the proof of the next result, we basically copy the strategy followed in the proof of \cite[Theorem 3.35]{main} based in reducing our general problem to the case $n=2$ where the (hermitian) Helton-Vinnikov theorem is available and thus we can use the existence of a \gls{MLMPDR} of RZ polynomials to express the form $L_{p,e}$ in terms of traces of the matrices in that MLMPDR of $p,$ which simplifies drastically the development of the proof.
\end{remark}

We will proceed in this way. The result is the next.

\begin{proposicion}[PSD-ness of initial matrix]
\label{propomatrix}
Applying $L_{p,e}$ for any integer $e\geq d:=\deg(p)$ to all the entries of any MMM of the form given by Equation \ref{ext1} for any RZ polynomial $p\in\mathbb{R}_{\mbox{RZ}}[\mathbf{x}]$ produces a PSD matrix.
\end{proposicion}

\begin{proof}
After the reductions in Remark \ref{remarkred} we have to prove thus only that next matrix is PSD: $L_{p,e}\circledcirc M:=$\begin{gather}\label{biggestmatrix}L_{p,e}\circledcirc((1, x_{1}, \dots, x_{n}, x_{1}^{2}, \dots, x_{1}x_{n})^{T}(1, x_{1},\dots, x_{n}, x_{1}^2, \dots, x_{1}x_{n})).\end{gather} By continuity and homogeneity, it suffices to show the nonnegativity when $w_{0}=1$ so we make a slight modification on the index sets of $w$ and denote now any arbitrary vector with $w_{0}=1$ as $v=(1,w_{1},\dots,w_{n},w'_{1},\dots,w'_{n})^{T}=(1,w^{T},w'^{T})^{T}=(1,v_{1},\dots,v_{2n})^{T}\in\mathbb{R}^{2n+1}.$ Therefore we have to show that\begin{equation*}
    (1,w_{1},\dots,w_{n},w'_{1},\dots,w'_{n})L_{p,e}\circledcirc M(1,w_{1},\dots,w_{n},w'_{1},\dots,w'_{n})^{T}\geq0.
\end{equation*} By Lemma \ref{lemmaLform} and having in mind that $x_{0}=1$ and our new assumption $w_{0}=1$ and the small modification on the index set of $w$ in the notation just introduced accordingly, the previous inequality is equivalent to\begin{equation*}L_{p,e}((1+w^{T}x)^2+2x_{1}(1+w^{T}x)(w'^{T}x)+x_{1}^2(w'^{T}x)^2)\geq0.\end{equation*} In order to prove this inequality, we choose an orthogonal matrix $U\in\mathbb{R}^{n\times n}$ such that $u:=U^{T}w$ and $u':=U^{T}w'$ lie both in $\mathbb{R}^{2}\times\{0\}\subseteq\mathbb{R}^{n}.$ By \cite[Proposition 3.8(b)]{main} it suffices to show \begin{equation*}L_{q,e}((1+u^{T}x)^2+2x_{1}(1+u^{T}x)(u'^{T}x)+x_{1}^2(u'^{T}x)^2)\geq0,\end{equation*} where $q(x):=p(Ux)\in\mathbb{R}_{\mbox{RZ}}[x]$ so we can set $r(x_{1},x_{2}):=q(x_{1},x_{2},0)\in\mathbb{R}_{\mbox{RZ}}[x_{1},x_{2}]$ and reduce to the case of two variables as in Lemma \ref{lemmaineqalmost}. We can then apply Helton-Vinnikov to write $r(x_{1},x_{2})=\det(I_{d}+A_{1}x_{1}+A_{2}x_{2})$ for $(A_{1},A_{2})\in(\sym_{d}(\mathbb{R}))^{2}.$ Now our problem reduces to prove \begin{equation}\label{ineqprin}L_{r,e}((1+u^{T}x)^2+2x_{1}(1+u^{T}x)(u'^{T}x)+x_{1}^2(u'^{T}x)^2)\geq0\end{equation} for the variables $x=(x_{1},x_{2})$ and arbitrary $u=(u_{1},u_{2}), u'=(u'_{1},u'_{2})\in\mathbb{R}^{2}.$ We remember here that Lemma \ref{lemmaineqalmost} does not solve our problem as we cannot guarantee that $\tr((I_{d}+u^{T}A+A_{1}u'^{T}A)^2)\geq0$ always and so we need to implement a different strategy to prove the new inequality in Equation \ref{ineqprin}. For simplicity we denote again $K:=u_{1}A_{1}+u_{2}A_{2}$ and $K':=u'_{1}A_{1}+u'_{2}A_{2}.$ From Lemma \ref{lemmaineqimpro} we know now that \begin{equation*}L_{r,e}(x_{1}^2(u'^{T}x)^{2})\geq\frac{1}{2}\tr(A_{1}K'A_{1}K'+A_{1}^{2}K'^{2})\end{equation*} so we can write the next chain of inequalities instead of those in the proof of Lemma \ref{lemmaineqalmost} $L_{r,e}((1+u^{T}x+x_{1}u'^{T}x)^2)=$ \begin{gather}L_{r,e}((1+u^{T}x)^2+2x_{1}(1+u^{T}x)(u'^{T}x)+x_{1}^2(u'^{T}x)^2)=\\\begin{gathered}\label{intra2}L_{r,e}((1+u_{1}x_{1}+u_{2}x_{2})^2+2x_{1}(1+u_{1}x_{1}+u_{2}x_{2})(u'_{1}x_{1}+u'_{2}x_{2})+\\x_{1}^2(u'_{1}x_{1}+u'_{2}x_{2})^2)=\end{gathered} 
\\e+\tr(2K+K^2+2A_{1}(I_{d}+K)K')+L_{r,e}(x_{1}^2(u'_{1}x_{1}+u'_{2}x_{2})^2)\geq\\d+\tr(2K+K^2+2A_{1}(I_{d}+K)K')+\frac{1}{2}\tr(A_{1}K'A_{1}K'+A_{1}^{2}K'^{2})=\\d+\tr(K+\frac{1}{2}K^2+A_{1}(I_{d}+K)K'+\frac{1}{2}A_{1}K'A_{1}K')\\+\tr(K+\frac{1}{2}K^2+A_{1}(I_{d}+K)K'+\frac{1}{2}A_{1}^{2}K'^{2})=\\\tr(\frac{1}{2}I_{d}+K+\frac{1}{2}K^2+A_{1}(I_{d}+K)K'+\frac{1}{2}A_{1}K'A_{1}K')+\\\tr(\frac{1}{2}I_{d}+K+\frac{1}{2}K^2+A_{1}(I_{d}+K)K'+\frac{1}{2}A_{1}^{2}K'^{2})=\\\tr(\frac{1}{2}(I_{d}+K)^2+A_{1}(I_{d}+K)K'+\frac{1}{2}A_{1}K'A_{1}K')+\\\tr(\frac{1}{2}(I_{d}+K)^2+A_{1}(I_{d}+K)K'+\frac{1}{2}A_{1}^{2}K'^{2})=\\\tr((I_{d}+K)^2+2A_{1}(I_{d}+K)K'+\frac{1}{2}A_{1}K'A_{1}K'+\frac{1}{2}A_{1}^{2}K'^{2})=\\\tr((I_{d}+K+\frac{1}{2}A_{1}K'+\frac{1}{2}K'A_{1})^{2}),\end{gather}
where we used the fact that $e\geq d$ and several times some identities between traces of the products of up to $3$ real symmetric matrices that let us conclude the previous chain of (in)equalities in terms of traces that closes this proof because now we have that $I_{d}+K+\frac{1}{2}A_{1}K'+\frac{1}{2}K'A_{1}\in\sym_{d}(\mathbb{R})$ is a real symmetric matrix and thus all its eigenvalues $\lambda_{1}\leq\cdots\leq\lambda_{d}\in\mathbb{R}$ are real numbers so the trace of its square is $\tr((I_{d}+K+\frac{1}{2}A_{1}K'+\frac{1}{2}K'A_{1})^{2})=\sum_{i=1}^{d}\lambda_{i}^{2}\geq0,$ which finishes our proof.
\end{proof}

The last part of the proof of the proposition above is in fact an improvement of Lemma \ref{lemmaineqalmost}. We see this.

\begin{observacion}[Improvements of previous lemmas in proposition above]
In the last part of the proof of the previous proposition we have indeed an improvement of Lemma \ref{lemmaineqalmost} where instead of $\tr((I_{d}+u^{T}A+A_{1}u'^{T}A)^2)$ we can now write $\tr((I_{d}+u^{T}A+\frac{1}{2}A_{1}(u'{}^{T}A)+\frac{1}{2}(u'{}^{T}A)A_{1})^{2})$. This gives a tighter inequality because $\tr((I_{d}+u^{T}A+\frac{1}{2}A_{1}(u'{}^{T}A)+\frac{1}{2}(u'{}^{T}A)A_{1})^{2})\geq\tr((I_{d}+u^{T}A+A_{1}u'{}^{T}A)^2)$ (or, even better, we can write the trace inequality $\tr((I_{d}+u^{T}A+\frac{1}{2}A_{1}(u'{}^{T}A)+\frac{1}{2}(u'{}^{T}A)A_{1})^{2})>\tr((I_{d}+u^{T}A+A_{1}u'{}^{T}A)^2)$ strictly if $u'_{2}{}^{2}A_{1}^{2}A_{2}^{2}\neq u'_{2}{}^{2}A_{1}A_{2}A_{1}A_{2},$ which can happen only if $u'_{2}\neq 0$) as can be easily seen following the path of the inequalities in both proofs. In fact, we observe that the first inequalities in Equations \ref{playineq} and \ref{intra2} come (apart from the fact that $e\geq d$) from substituting respectively \begin{gather*}
  \tr(\frac{4}{6}A_{1}^2 A_{2}^2+\frac{2}{6}A_{1}A_{2}A_{1}A_{2})\geq\tr(A_{1}A_{2}A_{1}A_{2})
\mbox{\ and\ }\\ \tr(\frac{4}{6}A_{1}^2 A_{2}^2+\frac{2}{6}A_{1}A_{2}A_{1}A_{2})\geq\frac{1}{2}\tr(A_{1}^2 A_{2}^2+A_{1}A_{2}A_{1}A_{2})
\end{gather*} using the inequality $\tr(A_{1}^2 A_{2}^2)\geq\tr(A_{1}A_{2}A_{1}A_{2})$ established in Lemma \ref{lemmatrace} but this also shows that the respective lower bounds verify $\tr(A_{1}A_{2}A_{1}A_{2})\leq\frac{1}{2}\tr(A_{1}^2 A_{2}^2+A_{1}A_{2}A_{1}A_{2})$.\end{observacion}

Thanks to the observation above we see that we actually proved the next corollary of Lemmas \ref{lemmaineqalmost} and \ref{lemmaineqimpro} via the Equation \ref{intra2} while completing the proof of Proposition \ref{propomatrix}.

\begin{corolario}[PSD-ness of initial matrix with degree control]
\label{coroineqalmost}
Fix the vector of variables $x=(x_{1},x_{2})$ and let $r(x)=r(x_{1},x_{2})=\det(I_{d}+A_{1}x_{1}+A_{2}x_{2})=\det(I_{d}+x^{T}A)\in\mathbb{R}_{\mbox{RZ}}[x]_{d}$ be a bivariate RZ polynomial of degree $d:=\deg(r)$ and fix an arbitrary integer $d\leq e\in\mathbb{N}_{0}.$ Then \begin{gather*}L_{r,e}((1+u^{T}x+x_{1}u'^{T}x)^2)\geq\tr((I_{d}+u^{T}A+\frac{1}{2}A_{1}(u'^{T}A)+\frac{1}{2}(u'^{T}A)A_{1})^2)\geq\\ \max\{0,\tr((I_{d}+u^{T}A+A_{1}u'^{T}A)^2)\}\end{gather*} for each $u=(u_{1},u_{2})^{T}, u'=(u'_{1},u'_{2})^{T}\in\mathbb{R}^{2}$ and where $A=(A_{1},A_{2})\in(\Sym_{d}(\mathbb{R}))^{2}$ so we can denote $u^{T}A=u_{1}A_{1}+u_{2}A_{2}$ and $u'^{T}A=u'_{1}A_{1}+u'_{2}A_{2}.$
\end{corolario}

\begin{proof}
Equation \ref{intra2} for the first inequality and \begin{gather*}
    \tr((I_{d}+u^{T}A+\frac{1}{2}A_{1}(u'{}^{T}A)+\frac{1}{2}(u'{}^{T}A)A_{1})^2)=\\d+\tr(2u^{T}A+(u^{T}A)^2+2A_{1}(I_{d}+u^{T}A)(u'{}^{T}A))+\\\frac{1}{2}\tr(A_{1}(u'{}^{T}A)A_{1}(u'{}^{T}A)+A_{1}^{2}(u'{}^{T}A)^2)=
\\
    d+\tr(2u^{T}A+(u^{T}A)^2+2A_{1}(I_{d}+u^{T}A)(u'{}^{T}A))+\\\tr(u'_{1}{}^{2}A_{1}^{4}+2u'_{1}u'_{2}A_{1}^{3}A_{2}+\frac{1}{2}u'_{2}{}^{2}A_{1}^{2} A_{2}^{2}+\frac{1}{2}u'_{2}{}^{2}A_{1}A_{2}A_{1}A_{2})\geq
\\
    d+\tr(2u^{T}A+(u^{T}A)^2+2A_{1}(I_{d}+u^{T}A)(u'{}^{T}A))+\\\tr(u'_{1}{}^{2}A_{1}^4+2u'_{1}u'_{2}A_{1}^{3}A_{2}+u'_{2}{}^{2}A_{1}A_{2}A_{1}A_{2})=
\\
    \tr((I_{d}+u^{T}A+A_{1}u'{}^{T}A)^2),
\end{gather*} where we reused some of the identities used in the previous proof of this section but that are easy to develop in order to make clear the inequality that we have just establish and that finishes this proof.
\end{proof}

Hence we just saw that, actually, the last proposition says a bit more than just being PSD. We can pursue strictness.

\begin{observacion}[Seeking for strictness]
It could well happen that $\tr((I_{d}+K+\frac{1}{2}A_{1}K'+\frac{1}{2}K'A_{1})^{2})>0$ in some cases for all (or enough in the sense of covering all the possible cases of vectors $v$) the possible bivariate polynomials $r$ obtained from $p$ via the mechanism described in the proof. Thus, it is possible that in some cases this path allows us to establish that these matrices are PD. Our experiments show that these matrices are PD almost always and thus it is convenient to study when these matrices have not all zero eigenvalues.
\end{observacion}

Thus, in what follows, we try to establish this strictness. We explain how we obtain such certificate of strictness.

\begin{remark}
For the proof the next result, we just have to analyze the steps followed in the proof of Proposition \ref{propomatrix} above in order to find the necessary strengthening of the conditions that allow us to determine that any of the inequalities that are used in that proof is strict. The complicated statement of the lemma comes thus probably from the somehow artificial procedure that is performed there in order to reduce the multivariate RZ polynomial $p$ to the bivariate RZ polynomial $r$, but we cannot in principle, in general, avoid this procedure as we just have available the Helton-Vinnikov theorem on $2$ variables. This shows us that a simplification of the statement will require a different strategy. Nevertheless, a clearer procedure is likely to exist because, as we said, our experiments show that almost all matrices formed by the procedure described at the beginning of the statement, i.e., $L_{p,e}\circledcirc M$ for random RZ polynomial $p\in\mathbb{R}_{\mbox{RZ}}[x]$ are PD, as expected.
\end{remark}

We will use these insights soon. The result is the following.

\begin{lema}[Strictness]
\label{lemalargoPD}
Applying $L_{p,e}$ for any integer $e\geq d:=\deg(p)$ to all the entries of any matrix of the form given by Equation \ref{ext1} for any RZ polynomial $p\in\mathbb{R}_{\mbox{RZ}}[x]$ produces a PD matrix if, for every pair of vectors $w,w'\in\mathbb{R}^{n}$ not both zero, the RZ polynomial $p$ admits an orthogonal transformation $U\in\mathbb{R}^{n\times n}$ with $u:=U^{T}w, u':=U^{T}w'\in\mathbb{R}^{2}\times\{0\}\subseteq\mathbb{R}^{n}$ such that we can write $r(x_{1},x_{2}):=p(U(x_{1},x_{2},0)^{T})=\det(I_{d}+x_{1}A_{1}+x_{2}A_{2})$ with $A_{1}K'A_{1}K'\neq A_{1}^{2}K'^{2}$ or with $\eigval(K+\frac{1}{2}A_{1}K'+\frac{1}{2}K'A_{1})\neq\{0\}$ and at least one of the following conditions is true: $e>d$ or $\eigval(I_{d}+K+\frac{1}{2}A_{1}K'+\frac{1}{2}K'A_{1})\neq\{0\}.$
\end{lema}

\begin{proof}
We know that $L_{p,e}\circledcirc M$ is PD if and only if for every vector $v\neq0$ of any of the forms $$(1,w^{T},w'^{T})^{T}, (0,w^{T},w'^{T})^{T}\in\mathbb{R}^{2n+1}$$ we have that $v^{T}L_{p,e}\circledcirc M v>0.$ Therefore we have to show that\begin{gather*}
    (1,w^{T},w'^{T})L_{p,e}\circledcirc M(1,w^{T},w'^{T})^{T}>0 \mbox{\ and\ } \\(0,w^{T},w'^{T})L_{p,e}\circledcirc M(0,w^{T},w'^{T})^{T}>0,
\end{gather*} whenever these vectors are not $0.$ But before, first of all, we note that the excluded case $w=0=w'$ is trivial since, in that case, we have that the second equation does not have to be considered at all and the first equation is true because $(1,0)L_{p,e}\circledcirc M(1,0)^{T}=L_{p,e}(1)=e>0.$ Now, by Lemma \ref{lemmaLform} and having in mind that $x_{0}=1$ and our new assumption $w_{0}=1$ or $w_{0}=0$ and the small modification on the index set of $w$ in the notation introduced accordingly as in Proposition \ref{propomatrix}, the previous inequalities are equivalent to\begin{equation*}L_{p,e}((1+w^{T}x+x_{1}w'^{T}x)^2)>0 \mbox{\ and\ } L_{p,e}((w^{T}x+x_{1}w'^{T}x)^2)>0.\end{equation*} In order to prove these inequalities, we can now choose the special orthogonal matrices admitted by $p$ in our assumptions $U\in\mathbb{R}^{n\times n}$ such that (among other properties used in what follows) $u:=U^{T}w$ and $u':=U^{T}w'$ lie both in $\mathbb{R}^{2}\times\{0\}\subseteq\mathbb{R}^{n}.$ By \cite[Proposition 3.8(b)]{main} it suffices then to show \begin{equation*}L_{q,e}((1+u^{T}x+x_{1}u'^{T}x)^2)>0 \mbox{\ and\ } L_{q,e}((u^{T}x+x_{1}u'^{T}x)^2)>0,\end{equation*} where $q(x):=p(Ux)\in\mathbb{R}_{\mbox{RZ}}[x]$ so we can set $r(x_{1},x_{2}):=q(x_{1},x_{2},0)\in\mathbb{R}_{\mbox{RZ}}[x_{1},x_{2}]$ and reduce to the case of two variables as in Lemma \ref{lemmaineqalmost}. We can then apply Helton-Vinnikov to write $r(x_{1},x_{2})=\det(I_{d}+A_{1}x_{1}+A_{2}x_{2})$ for $(A_{1},A_{2})\in(\sym_{d}(\mathbb{R}))^{2}.$ Now our problem reduces to prove \begin{equation*}L_{r,e}((1+u^{T}x+x_{1}u'^{T}x)^2)>0 \mbox{\ and\ } L_{r,e}((u^{T}x+x_{1}u'^{T}x)^2)>0\end{equation*} for the remaining variables $x=(x_{1},x_{2})$ and the vectors $u=(u_{1},u_{2}), u'=(u'_{1},u'_{2})\in\mathbb{R}^{2}$ just obtained via the choice of the special orthogonal transformation $U.$ Thus we have now that \begin{equation*}L_{r,e}((1+u^{T}x+x_{1}u'^{T}x)^2)\geq\tr((I_{d}+K+\frac{1}{2}A_{1}K'+\frac{1}{2}K'A_{1})^{2})\geq0,\end{equation*}
where if $e>d$ the first inequality is strict, and \begin{equation*}L_{r,e}((u^{T}x+x_{1}u'^{T}x)^2)\geq\tr((K+\frac{1}{2}A_{1}K'+\frac{1}{2}K'A_{1})^{2})\geq0.\end{equation*} In general now it can be observed easily following the inequalities in Proposition \ref{propomatrix} that the first inequality in each of the two previous equations is strict if $A_{1}K'A_{1}K'\neq A^{2}_{1}K'^{2}$ and the second inequality is strict if the matrices $I_{d}+K+\frac{1}{2}A_{1}K'+\frac{1}{2}K'A_{1}$ or $K+\frac{1}{2}A_{1}K'+\frac{1}{2}K'A_{1},$ respectively, have not all zero eigenvalues. These observations finish this proof.
\end{proof}

The main question now is how the conditions in the previous lemma translate to the initial polynomial $p$ and how we can find more natural ways of expressing these conditions in terms of the RZ polynomial $p$ itself. In order to do this, we will need to restrict momentarily our point of view.

\begin{ampliacion}[Going beyond reduction to the case $n=2$]
The results before were valid for all RZ polynomials $p$ with $p(0)=1$ because we were always able to reduce our problem to the case $n=2$. However, this reduction implies that we get conditions for the matrices being PD that are complicated and not natural enough. Thus, inasmuch as our problems for clarity in these conditions seem to come from the reduction to the case $n=2$, we will avoid this reduction in order to find other conditions easier to express but, at first, this will lead us naturally towards a restriction to polynomials having monic linear matrix polynomial determinantal representations (MLMPDRs), which are all RZ polynomials for $n\leq 2$ but not anymore all RZ polynomials for $n\geq3$ (due to well-known dimensional counts).
\end{ampliacion}

Doing this ampliation, we hope to find nicer conditions that we will try to extend to all RZ polynomials and not just those admitting a MLMPDR. We notice the similarities between the proof of the next result and the preceding one.

\begin{remark}
The proof of the next corollary corresponds basically with the previous proof with the omission of the restriction to two variables performed through orthogonal transformations.
\end{remark}

Thus, we can proceed in this manner in order to establish the last result of this chapter. This is the topic of the next corollary.

\begin{corolario}[The case of determinantal polynomials]
\label{coroMLMPDRPD}
Applying $L_{p,e}$ for any integer $e\geq d:=\deg(p)$ to all the entries of any matrix of the form given by Equation \ref{ext1} for any RZ polynomial $p\in\mathbb{R}_{\mbox{RZ}}[x]$ with an MLMPDR$(s)$ $p=\det(I_{d}+x_{1}A_{1}+\cdots+x_{n}A_{n})$ produces a PD matrix if, for every pair of vectors $w,w'\in\mathbb{R}^{n}$ not both zero, we have that $A_{1}K'A_{1}K'\neq A_{1}^{2}K'^{2}$ or that $\linspan(\{A_{1},\dots,A_{n},\frac{1}{2}A_{1}A_{1}+\frac{1}{2}A_{1}A_{1},\dots,\frac{1}{2}A_{1}A_{n}+\frac{1}{2}A_{n}A_{1})\}=2n$ and at least one of the following conditions is true: the inequality $e>d$ or the identity $$\linspan(\{I_{d},A_{1},\dots,A_{n},A_{1}A_{1}+A_{1}A_{1},\dots,A_{1}A_{n}+A_{n}A_{1}\})=2n+1.$$
\end{corolario}

\begin{proof}
First we remember that, by \cite[Theorem 2.2]{branden2011obstructions}, we can always suppose that $s=d.$ Using the first part of the proof of Lemma \ref{lemalargoPD}, we know that $L_{p,e}\circledcirc M$ is PD if and only if for every vector $v\neq0$ of any of the forms $(1,w^{T},w'^{T})^{T}, (0,w^{T},w'^{T})^{T}\in\mathbb{R}^{2n+1}$ we have respectively that \begin{equation*}L_{p,e}((1+w^{T}x+x_{1}w'^{T}x)^2)>0 \mbox{\ and\ } L_{p,e}((w^{T}x+x_{1}w'^{T}x)^2)>0.\end{equation*} Since we took now RZ polynomials that we know that admit an MLMPDR($d$) $p=\det(I_{d}+x_{1}A_{1}+\cdots+x_{n}A_{n})$, we do not need to reduce to the case $n=2$ and we can directly perform our computations with $p$ instead of $r$ and, as it can be easily seen that everything works similarly to the case $n=2$, we obtain then that \begin{equation*}L_{p,e}((1+w^{T}x+x_{1}w'^{T}x)^2)\geq\tr((I_{d}+K+\frac{1}{2}A_{1}K'+\frac{1}{2}K'A_{1})^{2})\geq0,\end{equation*}
where, if $e>d$, the first inequality is strict, and \begin{equation*}L_{p,e}((w^{T}x+x_{1}w'^{T}x)^2)\geq\tr((K+\frac{1}{2}A_{1}K'+\frac{1}{2}K'A_{1})^{2})\geq0,\end{equation*} where, in this case, we called $K=w_{1}A_{1}+\cdots+w_{1}A_{1}$ and $K'=w'_{1}A_{1}+\cdots+w'_{1}A_{1}.$ As before but with the difference that our matrices are now not altered by any (orthogonal) transformation on the variables, it can be observed easily that the first inequality in each of the two previous equations is strict if $A_{1}K'A_{1}K'\neq A^{2}_{1}K'^{2}$ and the second inequality is strict if the matrices $I_{d}+K+\frac{1}{2}A_{1}K'+\frac{1}{2}K'A_{1}$ or $K+\frac{1}{2}A_{1}K'+\frac{1}{2}K'A_{1},$ respectively, have not all zero eigenvalues. The rest of the proof is trivial observing the conditions obtained in Lemma \ref{lemalargoPD}.
\end{proof}

The work for non-determinantal polynomial lies beyond our interests here. However, we can lightly mention an interesting possible approach to study that problem using the already cited \cite[Theorem 6]{kummer2017determinantal}.

\begin{ramificacion}[Using Kummer's result]
By \cite[Proposition 3.6]{main} we know that $L_{pq,d+e}=L_{p,d}+L_{q,e}$ for every pair of nonnegative integers $d,e\in\mathbb{N}_{0}$ and every pair of RZ polynomials $p,q\in\mathbb{R}_{\mbox{RZ}}[x].$ Additionally, although not every RZ polynomial $p$ has a MLMPDR, by a result of Kummer we know that some multiple of it by another RZ polynomial $q$ has it so $pq$ can be studied as in Corollary \ref{coroMLMPDRPD}. However, applying directly that corollary to $pq$ gives awful conditions for the matrices studied to be PD. At this point, it is therefore evident that we cannot avoid a deeper study of $L_{p,e}$ in terms of the coefficients of $p$ if we want to go further because the expression of $L_{p,e}$ in terms of traces, although helpful to shorten our previous proofs until this point, seems worn out when we cannot ensure anymore the existence of MLMPDR for our RZ polynomials.
\end{ramificacion}

This ampliation of the relaxation matrix tells us that it is possible to extend the matrix mold keeping the condition that the initial matrix is PSD. However, our experiments show that the polynomial obtained through this method is not in general a relaxation in the sense that it respects the original RCS. For this reason, we turn in the next chapter to other strategies to increase the matrices in the relaxation: adding more variables through interlacers and stability preservers. This is the topic of the next chapter.
\chapter{Invariability under Known Variable Extensions}\label{4}

In this section, we will use three key notions that will allow us to extend a RZ polynomial into a new one having one additional variable. These notions are the Renegar derivative \cite[Section 4]{Renegar2006HyperbolicPA} (see also \cite{Saunderson2017ASR} for understanding better the importance of this construction), interlacing of polynomials \cite{johnson1987interlacing} (see also \cite{fisk2006polynomials} for an extended discussion and many results) and a note of Nuij on keeping real-rootedness under certain transformation \cite[Lemma]{Nuij1968ANO} that was eventually improved and generalized to the multivariate case in the form we can see in \cite[Lemma 2.5.1]{habnetzer} and that we will use here in combination with \cite[Proposition 6.7]{main}. In our case, note that the cited Lemma is applied using one additional \textit{ghost} variable. These kinds of variables will appear quite often in future parts of this work.

\begin{notacion}[Homogenization and Renegar derivative]
For a polynomial $p\in\mathbb{R}[\mathbf{x}]$ and $k\in\mathbb{N}$ we denote $\gls{renegpk}=\left(\frac{\partial^{k}}{\partial x_{0}^{k}}p^{h}\right)\bigg{|}_{x_{0}=1}$ the \textit{Renegar derivative of $p$}, where $\gls{homop}:=x_{0}^{d}p(\frac{1}{x_{1}},\dots,\frac{1}{x_{n}})$ is the \textit{homogenization} of $p$ using the additional variable $x_{0}.$ 
\end{notacion}

\section{Univariate inspiration}\label{section1}

In this first section, we will see that using results from \cite{nuijtype} (only valid for univariate polynomials $p\in\mathbb{R}[x]$) we can extract an idea about how to deal with certain transformations sending the set of RZ polynomials in the variable $x$ to set of RZ polynomials in the variable $y$ involving linear combinations of products of the $k$-th Renegar derivative by the corresponding power $y^{k}$ of the new variable $y$. Here we will see a method that we cannot directly generalize to the multivariate case because we do not have the results on \cite{nuijtype} for multivariate polynomials. However, we will see in future sections that a similar idea to the one that we expose here can be refined to obtain a multivariate version of the result that closes this section. Surprisingly, after some work over the structure of the relaxation defined in \cite[Definition 3.19]{main}, we will find that the mentioned multivariate result admits a straightforward proof not involving any generalization of \cite{nuijtype} but just easy results from \cite{main}. Thus, this section serves uniquely as an introduction to the type of results we are looking for in this document.

\begin{convencion}
We denote $I=I_{d}$ when the size $d$ is understood from the context. For a matrix $A$ we denote $\diag(A)$ the matrix having the same diagonal as $A$ that is zero everywhere else. For ease of writing, with $\mathbb{R}A=\{rA\mid r\in\mathbb{R}\}$ we mean the real linear span of the matrix $A$.
\end{convencion}

\begin{objetivo}[Using transformation respecting real-rootedness]
We will study the transformation $p\to p+y q$ with $q$ the Renegar derivative (\cite[Definition 2.4.2]{habnetzer}) of $p\in\mathbb{R}[\mathbf{x}]$ as first example of a transformation that produces an RZ polynomial having one extra variable. This is the easiest transformation of this kind. However, we will see that it does not add new information about the polynomial and, therefore, it does not improve the relaxation. This analysis will impulse us to search for different transformation (e.g., through stability preservers) or different interlacers able to improve the relaxation.
\end{objetivo}

In particular, when looking at these transformations increasing the number of variables in a polynomial, we will see that finding these is a highly subtle endeavour. This greatly limits what we can do. In particular, adding more than one new variable already becomes a highly technical work.

\begin{limitacion}[For any possible multivariate analogues]
The map $p\to p+y q$ works sending the set of RZ polynomials in the variables $\mathbf{x}$ to the set of RZ polynomials in the variables $(\mathbf{x},y)$. In general, it is only required that $q,p$ are in \textit{proper position} (see \cite[Proposition 2.7]{wagner2011multivariate}). However, adding more than one variable has to pass through recursions, combinatorics and the theory of stability preservers. We will see examples of this in the next parts of this dissertation.
\end{limitacion}

We begin by dealing with the univariate case through known results in order to take inspiration and ideas for our next multivariate steps and results. In particular, the next proposition is \cite[Example 3.6]{nuijtype}.

\begin{proposicion}[Determinantal representation of Renegar extension]
If $p=\det(I+xD)\in\mathbb{R}[x]$ with $D$ diagonal, then $r:=p+yq,$ with $q$ the Renegar derivative of $p,$ has the determinantal respresentation $r=\det(I+xD+yE)$ with $E$ the all ones matrix.
\end{proposicion}

We need a lemma about traces and diagonal matrices.

\begin{lema}[Trace and diagonals]\label{lemmaabove}
Let $A,D$ be square matrices with $A$ arbitrary and $D$ diagonal. Then $\tr(AD)=\tr(\diag(A)D)=\tr(D\diag(A))=\tr(DA)$.
\end{lema}

\begin{proof}
This follows immediately from the definition of trace and the multiplication of matrices, noting that multiplication by diagonal matrices corresponds to multiplying each row or column (depending on left or right multiplication) of that matrix by the corresponding diagonal element.
\end{proof}

Now we relate the relaxation of the Renegar extension with the relaxation of the original polynomial. A restriction shows that they coincide where they can.

\begin{proposicion}[Relaxation of a polynomial equals restriction of the relaxation of the Renegar extension]
\label{analogyprop}
With the notation above, $$S(p)=\{a\in\mathbb{R}\mid (a,0)\in S(r)\}.$$
\end{proposicion}

\setlength{\emergencystretch}{3em}%
\begin{proof}
Using \cite[Proposition 3.33]{main}, we write $S(p)=\{a\in\mathbb{R}\mid\forall M\in U_{p}:\tr(M^{2}(I+aD))\geq0\}$ and $S(r)=\{(a,b)\in\mathbb{R}^{2}\mid\forall M\in U_{r}:\tr(M^{2}(I+aD+bE))\geq0\}$ with $U_{p}:=\{v_{0}I+v_{1}D\mid (v_{0},v_{1})\in\mathbb{R}^2\}\subseteq U_{r}:=\{v_{0}I+v_{1}D+v_{2}E\mid (v_{0},v_{1},v_{2})\in\mathbb{R}^3\}$. Note that $U_{r}=U_{p}+\mathbb{R}E$. The inclusion $\supseteq$ follows from \cite[Lemma 3.26 (c)]{main}. For $\subseteq$, fix $a\in S(p)$, then $\tr(M^{2}(I+aD))\geq0$ for all $M\in U_{p}$. Observe that, for $M\in U_{p}$, $\tr((M+v_{2}E)^{2})=\tr(((M+v_{2}I)^2+v_{2}^{2}(d-1)I)=\tr((M+v_{2}I)^2+(v_{2}\sqrt{d-1}I)^2)$. Note then that $v_{2}\sqrt{d-1}I,M+v_{2}I\in U_{p}$. We see now immediately using Lemma \ref{lemmaabove}, proved above, that $(a,0)\in S(r)$ because for all $M+v_{2}E\in U_{r}$ with $M\in U_{p}$ we have then $\tr((M+v_{2}E)^{2}(I+aD))=\tr(((M+v_{2}I)^2+(v_{2}\sqrt{d-1}I)^2)(I+aD))=\tr((M+v_{2}I)^{2}(I+aD))+\tr((v_{2}\sqrt{d-1}I)^{2}(I+aD))\geq0$ because $a\in S(p)$ and $v_{2}\sqrt{d-1}I,M+v_{2}I\in U_{p}$. This establishes the last inclusion and finishes our proof.
\end{proof}
\setlength{\emergencystretch}{0em}%

We properly introduce now polynomial variable extensions.

\begin{definicion}[Extensions of polynomials]
Consider disjoint sets of variables $\mathbf{x}$ and $\mathbf{y}$. Let $p\in\mathbb{R}[\mathbf{x}]$ be a polynomial and $r\in\mathbb{R}[\mathbf{x},\mathbf{y}]$ be another polynomial such that $p(\mathbf{x})=r(\mathbf{x},0)$. We say that $r$ is a \textit{(polynomial) extension} of the \textit{original polynomial} $p$. Similarly, we say that the set given by $\{a\in\mathbb{R}^{n}\mid (a,0)\in S(r)\}$ is the \textit{intersection with the original subspace of the relaxation of the extension $r$ of the original polynomial $p$}.
\end{definicion}

Intersections of the relaxation will be important for us.

\begin{observacion}[Intersections and restrictions]
The set on the RHS of the equation above appears as the intersection with the line $y=0$ of the relaxation $S(r)$ of the extension $r$ of $p$, i.e., $p=r|_{y=0}$. These sets obtained from intersections of the original subspace with the relaxation of an extension of the original polynomial have sometimes interesting relations to the RCSs of the original polynomials.
\end{observacion}

We will see one of these relations for the bivariate case in the next section.

\section{Another result on intersections in bivariate case}

In this section, we will follow one of the paths drawn by Proposition \ref{analogyprop} at the end of the previous section. That proposition shows a clear relation between the two sets (relaxations) compared. This relation comes in the form of an equality of borders.

\begin{definicion}[Topological objects at the extremes of sets]
Let $(X,\tau)$ be a topological space and $A\subseteq X.$ The \textit{border} $\gls{bda}$ of $A$ equals $\cl(A)\smallsetminus\Int(A).$
\end{definicion}

These are the objects that we will to compare here via intersections.

\begin{objetivo}[Intersection of borders of rigidly convex sets and restrictions of relaxations of extensions]
We will show here a relation given in the form of an intersection of borders but in this case between the borders of the RCS of the original polynomial and the border of the intersection with the original subspace of the relaxation of the given extension of the original polynomial.
\end{objetivo}

We introduce some simplifications in order to facilitate our arguments. In particular, we always normalize our eigenvectors.

\begin{simplificacion}[Normalized vectors]
For ease of writing, we will generally take all our eigenvectors $v$ normalized so they verify $v^{\ast}v=1$. Therefore, a matrix of the form $vv^{\ast}$ has eigenvalues $(1,\mathbf{0})$ and then $\tr(vv^{\ast})=1$.
\end{simplificacion}

As usual, we take by convention our RZ polynomials $p$ with $p(\mathbf{0})=1$.

\begin{convencion}
As we will work with a bivariate RZ original polynomial $p\in\mathbb{R}[\mathbf{x}]$ with $p(0)=1$ wlog, we recall that such a bivariate RZ polynomial is determinantal (i.e., it admits a representation $p:=\det(I+x_{1}A_{1}+x_{2}A_{2})$ with the matrices $A_{i}$ both real symmetric) by Helton-Vinnikov Theorem \cite[Theorem 2.7]{main}.
\end{convencion}

Thus, we can finally establish now the proposition about the intersection of borders of the relaxation and the original RCS.

\begin{definicion}[Signedly extreme eigenvalues, roots around a point]
For a hermitian matrix $A$, we say that $\lambda$ is a \textit{signedly extreme} eigenvalue if $\lambda$ is either positive and the maximal or negative and the minimal. Also, for a univariate polynomial $p\in\mathbb{R}[t]$ (or a univariate LMP $M\in\Her_{d}{\mathbb{C}}[t]$) and $a\in\mathbb{R}$, we define $\innr_{a}(p)$ as the set of roots of $p$ ($\det(M)$) around $a$.
\end{definicion}

\begin{proposicion}[Intersection of borders, univariate view]
Let $p\in\mathbb{R}[\mathbf{x}]$ be a bivariate RZ polynomial and $p:=\det(I+x_{1}A_{1}+x_{2}A_{2})$ one of its determinantal representations with $A_{i}$ a (complex) Hermitian matrix for each $i\in[2]$. For $v$ a (normalized) eigenvector to a signedly extreme eigenvalue $\lambda$ of a nontrivial real linear combination of $A_{1}$ and $A_{2}$, i.e., of a matrix of the form $b_{1}A_{1}+b_{2}A_{2}$ with $\{0\}\neq\{b_{1},b_{2}\}\subseteq\mathbb{R}$, define the new polynomial $q=\det(I+x_{1}A_{1}+x_{2}A_{2}+yvv^{\ast})\in\mathbb{R}[\mathbf{x},y]$. Then $\innr_{0}(p(t(b_{1},b_{2})))\cap\innr_{0}(M_{q}(t(b_{1},b_{2},0)))\neq\emptyset.$
\end{proposicion}

This way of writing the proposition was very geometrical. It says that, in the line spanned by the vector $(b_{1},b_{2})$, some of the roots around $0$ of $p$ and $M_{q}$ (the LMP defining the relaxation of the extended polynomial $q$) coincide. This means that, in some sense, in these special points, we have reached exactness through the consideration of the polynomial variable extension $q$ of $p$. The proposition can also be written in the next more algebraic way. We will use this rewriting to prove it. The equivalence between these two ways of writings the statement in immediate. We will be more precise about the intersection in this rewriting.

\begin{proposicion}[Intersection of borders, bivariate view]
Fix a bivariate RZ polynomial $p\in\mathbb{R}[\mathbf{x}]$ and one of its determinantal representations $p:=\det(I+x_{1}A_{1}+x_{2}A_{2})$ with $A_{i}$ a (complex) Hermitian matrix for each $i\in[2]$. For $v$ a (normalized) eigenvector to the greatest positive or smallest negative eigenvalue of a nontrivial real linear combination of $A_{1}$ and $A_{2}$, i.e., of a matrix of the form $b_{1}A_{1}+b_{2}A_{2}$ with $\{0\}\neq\{b_{1},b_{2}\}\subseteq\mathbb{R}$, define the new polynomial $q=\det(I+x_{1}A_{1}+x_{2}A_{2}+yvv^{\ast})\in\mathbb{R}[\mathbf{x},y]$. Then $$\bd(\{(a_{1},a_{2})\in\mathbb{R}^{2}\mid(a_{1},a_{2},0)\in S(q)\})\cap\bd(C(p))\neq\emptyset.$$ In fact, fixing $B:=b_{1}A_{1}+b_{2}A_{2}$ a nontrivial real linear combination of $A_{1}$ and $A_{2}$.\begin{enumerate}
    \item Suppose that $v$ is an eigenvector to the greatest (then positive by hypothesis) eigenvalue $\lambda_{1}$ of $B$, then we have that $-\frac{1}{\lambda_{1}}(b_{1},b_{2})$ is in the described intersection of borders.
    \item Suppose that $v$ is an eigenvector to the smallest (then negative by hypothesis) eigenvalue $\lambda_{d}$ of $B$, then we have that $-\frac{1}{\lambda_{d}}(b_{1},b_{2})$ is in the described intersection of borders.
\end{enumerate}
In short, suppose that $v$ is an eigenvector to the greatest positive or smallest negative eigenvalue $\lambda$ of $B$, then we have that $-\frac{1}{\lambda}(b_{1},b_{2})$ is in the described intersection of borders $\bd(\{(a_{1},a_{2})\in\mathbb{R}^{2}\mid(a_{1},a_{2},0)\in S(q)\})\cap\bd(C(p))\neq\emptyset$.
\end{proposicion}

\begin{proof}
We will prove the second point because the first is analogous. Take an orthogonal matrix $U$ such that $UBU^{T}=D=\diag(\lambda_{1}\geq\dots\geq\lambda_{d})$ is the ordered diagonal of the eigenvalues of $B$. In particular, this implies $Uvv^{\ast}U^{T}=(0,\dots,0,1)^{T}(0,\dots,0,1)=\diag(0,\dots,0,1)$. We know that $S(q)=\{(a_{1},a_{2},a_{3})\in\mathbb{R}^{3}\mid\forall M\in U_{q}:\tr(M^{2}(I+a_{1}A_{1}+a_{2}A_{2}+a_{3}vv^{\ast}))\geq0\}$ with $U_{q}:=\{v_{0}I+v_{1}A_{1}+v_{2}A_{2}+v_{3}vv^{\ast}\mid(v_{0},\dots,v_{3})\in\mathbb{R}^{4}\}$. In particular, a point is on the border $\bd(S(q))$ of $S(q)$ if it is in $S(q)$ and there exists $0\neq M\in U_{q}$ such that $\tr(M^{2}(I+a_{1}A_{1}+a_{2}A_{2}+a_{3}vv^{\ast}))=0.$ Taking $(a_{1},a_{2},0)=k(b_{1},b_{2},0)$ for some $k\in\mathbb{R}$ and using the transformation $U$ we have \begin{gather*}\tr(M^{2}(I+k(b_{1}A_{1}+b_{2}A_{2})))=\tr(UM^{2}U^{T}U(I+k(b_{1}A_{1}+b_{2}A_{2}))U^{T})=\\ \tr(UM^{2}U^{T}U(I+k(b_{1}A_{1}+b_{2}A_{2}))U^{T})= \tr(UM^{2}U^{T}U(I+kB)U^{T})=\\\tr(UM^{2}U^{T}(I+kD)).\end{gather*} Now it is clear that $UM^{2}U^{T}=(UMU^{T})^{2}$ has its diagonal positive so setting $k=-\frac{1}{\lambda_{d}}$ for $\lambda_{d}$ the smallest nonzero eigenvalue implies $(I+kD)$ has its $i\in\{1,\dots,d\}$ entry equal to $1-\frac{\lambda_{i}}{\lambda_{d}}.$ If $\lambda_{i}=0$, we get $1$; if $\lambda_{i}<0$, then $\lambda_{d}\leq\lambda_{i}<0$ so $1\geq\frac{\lambda_{i}}{\lambda_{d}}>0$ which implies $-1\leq-\frac{\lambda_{i}}{\lambda_{d}}<0$ so summing $1$ we get $0\leq1-\frac{\lambda_{i}}{\lambda_{d}}$; if $\lambda_{i}>0$, as $\lambda_{d}<0$ also, $-\frac{\lambda_{i}}{\lambda_{d}}>0$ so $1-\frac{\lambda_{i}}{\lambda_{d}}>1>0$. Therefore, we see that setting $k=-\frac{1}{\lambda_{d}}$ for $\lambda_{d}$ the smallest negative eigenvalue implies $(I+kD)$ is a diagonal with all its entries nonnegative and in particular the last element of the diagonal is $0$. Taking $v_{3}=1$ we get $M=vv^{\ast}\neq0$ and we have that $UM^{2}U^{T}=(UMU^{T})^{2}=\diag(0,\dots,0,1)$ so $$\tr(UM^{2}U^{T}(I+kD))=\tr(\diag(0,\dots,0,1)(I+kD))=0,$$ which means that $k(b_{1},b_{2},0)$ is in $\bd(S(q))$ (the border of $S(q)$) so $k(b_{1},b_{2})\in\bd(\{(a_{1},a_{2})\in\mathbb{R}^{2}\mid(a_{1},a_{2},0)\in S(q)\}).$ Finally, we see that $k(b_{1},b_{2})\in\bd(C(p))$ because there we have that the LMP describing $p$ is the diagonal positive semidefinite matrix $I+kB$ that has at the last element of the diagonal. Thus $-\frac{1}{\lambda_{d}}(b_{1},b_{2}):=k(b_{1},b_{2})\in\bd(\{(a_{1},a_{2})\in\mathbb{R}^{2}\mid(a_{1},a_{2},0)\in S(q)\})\cap\bd(C(p))$, which finishes the proof.
\end{proof}

We give another more elementary proof that comes even more directly from the definitions. For completeness, we decide to give this proof for the first point in the proposition, but it also works for the second. In the proof that follows, we introduce some shorthands. We denote $a_{ijk}$ the entry in row $j$ and column $k$ of the matrix $A_{i}$ and $a_{ij}$ the $j$-th row (or column because the matrices will be symmetric) vector of the matrix $A_{i}.$ We can now then proceed with the promised alternative proof.

\begin{proof}
If the matrix defining $S(p)$ is denoted $M_{p}$, we have that the matrix $M_{q}$ defining $S(q)$ has an easy expression as a block matrix whose upper left corner is $M_{p}$ and whose additional row is given by $$\begin{pmatrix} \tr(vv^{\ast})+x_{1}\tr(A_{1}vv^{\ast})+x_{2}\tr(A_{2}vv^{\ast})+y\tr(vv^{\ast}vv^{\ast}) \\ \tr(A_{1}vv^{\ast})+x_{1}\Rea\tr(A_{1}^{2}vv^{\ast})+x_{2}\Rea\tr(A_{1}A_{2}vv^{\ast})+y\Rea\tr(A_{1}vv^{\ast}vv^{\ast}) \\ \tr(A_{2}vv^{\ast})+x_{1}\Rea\tr(A_{2}A_{1}vv^{\ast})+x_{2}\Rea\tr(A_{2}^{2}vv^{\ast})+y\Rea\tr(A_{2}vv^{\ast}vv^{\ast}) \\ \tr(vv^{\ast}vv^{\ast})+x_{1}\Rea\tr(A_{1}vv^{\ast}vv^{\ast})+x_{2}\Rea\tr(A_{2}vv^{\ast}vv^{\ast})+y\Rea\tr(vv^{\ast}vv^{\ast}vv^{\ast}) \end{pmatrix}$$ $$=\begin{pmatrix} 1+x_{1}a_{111}+x_{2}a_{211}+y \\ a_{111}+x_{1}||a_{11}||^{2}+x_{2}\Rea\tr(A_{1}A_{2}vv^{\ast})+ya_{111} \\ a_{211}+x_{1}\Rea\tr(A_{2}A_{1}vv^{\ast})+x_{2}||a_{21}||^{2}+y a_{211} \\ 1+x_{1}a_{111}+x_{2}a_{211}+y \end{pmatrix}.$$ In $y=0$ and denoting $w:=\Rea\tr(A_{1}A_{2}vv^{\ast})$, this looks like $$\begin{pmatrix} 1+x_{1}a_{111}+x_{2}a_{211} \\ a_{111}+x_{1}||a_{11}||^{2}+x_{2}w \\ a_{211}+x_{1}w+x_{2}||a_{21}||^{2}\\ 1+x_{1}a_{111}+x_{2}a_{211}.\end{pmatrix}$$ In order to describe a point, it will be more convenient to assume wlog that we take $(b_{1},b_{2})=\lambda(1,0)$ for some $\lambda>0$. Changing the matrices so that they absorb the scalar we can suppose furthermore $\lambda=1$. This is wlog because we can choose a suitable orthogonal transformation $U$ that does this. Moreover, we can also suppose wlog that $A_{1}$ is the ordered diagonal of its eigenvalues $\diag(\lambda_{1}\geq\dots\geq\lambda_{d})$ (with $\lambda_{1}>0$ by hypothesis) using another orthogonal transformation $V$. We will also take the eigenvector unitary wlog. These simplifications result in the fact that we want to find a point in the intersection of the borders of the proposition with $p:=\det(I+x_{1}D+x_{2}A_{2})$ and $q:=\det(I+x_{1}D+x_{2}A_{2}+yL)$ with $D:=A_{1}$ diagonal (with positive and negative entries) and $L$ an all zero matrix except a one in one of the corners of its diagonal (which depends on the choice of the eigenvector to be $(1,0,\dots,0)$ for the greatest positive eigenvalue or $(0,\dots,0,1)$ for the smallest negative eigenvalue). Thus $a_{111}=\lambda_{1}>0$ is the upper left entry of $A_{1}$, $w=\lambda_{1}a_{211}$ with $a_{i11}$ the upper left entry of $A_{i}$ and $||a_{11}||^{2}=\lambda_{1}^{2}$ with $||a_{i1}||$ the norm of the first row (or column) of $A_{i}$ so we have $$\begin{pmatrix} 1+x_{1}\lambda_{1}+x_{2}a_{211} \\ \lambda_{1}+x_{1}\lambda_{1}^{2}+x_{2}\lambda_{1}a_{211} \\ a_{211}+x_{1}\lambda_{1}a_{211}+x_{2}||a_{21}||^{2}\\ 1+x_{1}\lambda_{1}+x_{2}a_{211} \end{pmatrix}=\begin{pmatrix} 1+x_{1}\lambda_{1}+x_{2}a_{211} \\ \lambda_{1}(1+x_{1}\lambda_{1}+x_{2}a_{211}) \\ a_{211}+x_{1}\lambda_{1}a_{211}+x_{2}||a_{21}||^{2}\\ 1+x_{1}\lambda_{1}+x_{2}a_{211} \end{pmatrix}$$ so solving $1+x_{1}\lambda_{1}+x_{2}a_{211}=0$ we have that if $x_{1}=-\frac{1+x_{2}a_{211}}{\lambda_{1}}$ three of the entries are $0$. But, moreover, choosing $x_{2}=0$ we obtain $x_{1}=-\frac{1}{\lambda_{1}}$ that also makes $0$ the third entry and so at the point $(-\frac{1}{\lambda_{1}},0)$ we have that the additional entries of $M_{q}$ with respect to $M_{p}$ are $0$ so $\det(M_{q})=0$ at such point. If we can prove that this point is in $C(p)$, we found here then in particular a point where $C(p)$ and $\bd(\{(a_{1},a_{2})\in\mathbb{R}^{2}\mid(a_{1},a_{2},0)\in S(q)\})$ touch (tangentially) because the intersection of $\bd(\{(a_{1},a_{2})\in\mathbb{R}^{2}\mid(a_{1},a_{2},0)\in S(q)\})$ with $C(p)$ produces a relaxation of $C(p)$, i.e., $C(p)\subseteq\bd(\{(a_{1},a_{2})\in\mathbb{R}^{2}\mid(a_{1},a_{2},0)\in S(q)\})$. In fact, we are going to see that such point is in $\bd(C(p))$ directly. Observe that the LMP defining $p$ has the form $I-\frac{1}{\lambda_{1}}D$ that is diagonal and (strictly) PSD with a $0$ in its upper left corner so $(-\frac{1}{\lambda_{1}},0)$ is in $\bd(C(p))$ and the proof finishes simply undoing the transformation $U$ so we get the intersection point $-\frac{1}{\lambda_{1}}(b_{1},b_{2}).$
\end{proof}

This result shows us the importance of studying restrictions of the relaxation. In the next section, we study restrictions to any plane through the origin. 

\section{Restricting the relaxation to subspaces}

We will see here that the intersection of all restrictions to $2$-dimensional vector subspaces $\pi\subseteq\mathbb{R}^{n}$ through (meaning containing) a vector $0\neq a\in\pi\subseteq\mathbb{R}^{n}$ determines the relaxation defined in \cite{main} on the line $\linspan(\{a\})$ defined by this vector.

\begin{objetivo}[Slicing into planes]
We have to make these points more precise so they make sense and this is precisely our objective in this section. However, we can already say that, roughly speaking, we want to show that we can say that the relaxation $S(p)$ of an RZ polynomial $p\in\mathbb{R}[\mathbf{x}]$ is determined (on any line through the origin) by the corresponding relaxations of the bivariate restrictions of such polynomial (to planes containing such line we are looking at). This fact will be important for the next section about the Renegar derivative.
\end{objetivo}

We begin this section with a (kind of) generalization of the results \cite[Lemma 3.7 and Proposition 3.8]{main} beyond orthogonal transformations.

\begin{lema}[Identity of homogeneous polynomials involving the linear forms]
\label{lemmaforproposition}
Let $L$ be a linear form on the vector subspace $V$ of $\mathbb{R}[x]$ generated by the monomials of degree $k$. Then, for all matrices $A\in\mathbb{R}^{n\times n}$, we have the identity of homogeneous degree $k$ polynomials $$\sum_{|\alpha|=k}\frac{1}{k}\binom{k}{\alpha}L(x^{\alpha})(Ax)^{\alpha}=\sum_{|\alpha|=k}\frac{1}{k}\binom{k}{\alpha}L((A^{T}x)^{\alpha})x^{\alpha}\in V.$$
\end{lema}

\begin{proof}
As in the proof of \cite[Lemma 3.7]{main} we need to clarify some (further) notations before we begin. Denote the $i$-th row of $A:=(a_{i,j})_{i,j=1}^{n}$ by $a_{i}$ and the transpose of the $i$-th column by $a^{i}$ (so both $a_{i}$ and $a^{i}$ are row vectors) for all $i\in\{1,\dots,n\}$. Therefore, the $j$-th entry of $a_{i}$ is $a_{ij}$ and the $j$-entry of $a^{i}$ is $a_{ji}$ for all $i,j\in\{1,\dots,n\}$. Additionally, we denote $e_{i}:=(e_{i1},\dots,e_{in})$ the $i$-th canonical row vector so that $e_{ij}=1$ iff $i=j$ and $e_{ij}=0$ otherwise. In what follows, we often form the product of the row vector $\mathbf{u}$ with the column vector of variables $\mathbf{x}$ which is the element $\mathbf{ux}\in\mathbb{R}[\mathbf{x}]$. Now we can proceed with the proof. For this, we observe, in the same spirit as \cite[Lemma 3.7]{main}, that \begin{gather*}
    \sum_{|\alpha|=k}\binom{k}{\alpha}L(x^{\alpha})(Ax)^{\alpha}=\\\sum_{|\alpha|=k}\binom{k}{\alpha}L(x^{\alpha})(\sum_{j=1}^{n}a_{1j}x_{j},\dots,\sum_{j=1}^{n}a_{nj}x_{j})^{\alpha}=\\\sum_{|\alpha|=k}\binom{k}{\alpha}L(x^{\alpha})(a_{1}x,\dots,a_{n}x)^{\alpha}=\\\sum_{|\alpha|=k}\binom{k}{\alpha}L(x^{\alpha})(\sum_{j=1}^{n}a_{1j}x_{j})^{\alpha_{1}}\cdots(\sum_{j=1}^{n}a_{nj}x_{j})^{\alpha_{n}}=\\\sum_{|\alpha|=k}\binom{k}{\alpha}L(x^{\alpha})(a_{1}x)^{\alpha_{1}}\cdots(a_{n}x)^{\alpha_{n}}=\\\sum_{i_{1},\dots,i_{k}=1}^{n}L(e_{i_{1}}x\cdots e_{i_{k}}x)a_{i_{1}}x\cdots a_{i_{k}}x=\\\sum_{i_{1},\dots,i_{k}=1}^{n}\sum_{j_{1},\dots,j_{k}=1}^{n}e_{i_{1}j_{1}}\cdots e_{i_{1}j_{1}}L(x_{j_{1}}\cdots x_{j_{k}})\sum_{l_{1},\dots,l_{k}=1}^{n}a_{i_{1}l_{1}}\cdots a_{i_{k}l_{k}}x_{l_{1}}\cdots x_{l_{k}}=\\\sum_{i_{1},\dots,i_{k}=1}^{n}L(x_{i_{1}}\cdots x_{i_{k}})\sum_{l_{1},\dots,l_{k}=1}^{n}a_{i_{1}l_{1}}\cdots a_{i_{k}l_{k}}x_{l_{1}}\cdots x_{l_{k}}=\\\sum_{i_{1},\dots,i_{k}=1}^{n}\sum_{l_{1},\dots,l_{k}=1}^{n}L(x_{i_{1}}\cdots x_{i_{k}})a_{i_{1}l_{1}}\cdots a_{i_{k}l_{k}}x_{l_{1}}\cdots x_{l_{k}}=\\\sum_{i_{1},\dots,i_{k}=1}^{n}\sum_{l_{1},\dots,l_{k}=1}^{n}L(a_{i_{1}l_{1}}\cdots a_{i_{k}l_{k}}x_{i_{1}}\cdots x_{i_{k}})x_{l_{1}}\cdots x_{l_{k}}
\end{gather*} while, proceeding similarly, we get that \begin{gather*}
    \sum_{|\alpha|=k}\binom{k}{\alpha}L((A^{T}x)^{\alpha})x^{\alpha}=\\\sum_{|\alpha|=k}\binom{k}{\alpha}L((\sum_{j=1}^{n}a_{j1}x_{j},\dots,\sum_{j=1}^{n}a_{jn}x_{j})^{\alpha})x^{\alpha}=\\\sum_{|\alpha|=k}\binom{k}{\alpha}L((a^{1}x,\dots,a^{n}x)^{\alpha})x^{\alpha}=\\\sum_{|\alpha|=k}\binom{k}{\alpha}L((\sum_{j=1}^{n}a_{j1}x_{j})^{\alpha_{1}}\cdots(\sum_{j=1}^{n}a_{jn}x_{j})^{\alpha_{n}})x^{\alpha}=\\\sum_{|\alpha|=k}\binom{k}{\alpha}L((a^{1}x)^{\alpha_{1}}\cdots(a^{n}x)^{\alpha_{n}})x^{\alpha}=\\\sum_{i_{1},\dots,i_{k}=1}^{n}L(a^{i_{1}}x\cdots a^{i_{k}}x)e_{i_{1}}x\cdots e_{i_{k}}x=\\\sum_{i_{1},\dots,i_{k}=1}^{n}L(\sum_{l_{1},\dots,l_{k}=1}^{n}a_{l_{1}i_{1}}\cdots a_{l_{k}i_{k}}x_{l_{1}}\cdots x_{l_{k}})\sum_{j_{1},\dots,j_{k}=1}^{n}e_{i_{1}j_{1}}\cdots e_{i_{1}j_{1}}x_{j_{1}}\cdots x_{j_{k}}=\\\sum_{i_{1},\dots,i_{k}=1}^{n}L(\sum_{l_{1},\dots,l_{k}=1}^{n}a_{l_{1}i_{1}}\cdots a_{l_{k}i_{k}}x_{l_{1}}\cdots x_{l_{k}})x_{i_{1}}\cdots x_{i_{k}}=\\\sum_{i_{1},\dots,i_{k}=1}^{n}\sum_{l_{1},\dots,l_{k}=1}^{n}L(a_{l_{1}i_{1}}\cdots a_{l_{k}i_{k}}x_{l_{1}}\cdots x_{l_{k}})x_{i_{1}}\cdots x_{i_{k}}
\end{gather*} that changing the auxiliary variables and comparing with what we obtained above tells us that \begin{gather*}
    \sum_{|\alpha|=k}\binom{k}{\alpha}L(x^{\alpha})(Ax)^{\alpha}=\sum_{|\alpha|=k}\binom{k}{\alpha}L((A^{T}x)^{\alpha})x^{\alpha},\end{gather*} which, multiplying by $\frac{1}{k}$, gives the result, thus finishing this proof.
\end{proof}

Lemma \ref{lemmaforproposition} above will help us to prove the next proposition analogously to the procedure followed in \cite[Lemma 3.7 and Proposition 3.8]{main}.

\begin{proposicion}[Transformation of linear forms under the action of matrices]
\label{propositiontransformation}
Let $A\in\mathbb{R}^{n\times n}$ be a square matrix. \begin{enumerate}
    \item If $p\in\mathbb{R}[[x]]$ with $p(0)\neq0$ and $d\in\mathbb{N}_{0}$, then $$L_{p(Ax),d}(q(x))=L_{p(x),d}(q(A^{T}x)).$$
    \item If $p\in\mathbb{R}[x]$ with $p(0)\neq0$, then $$L_{p(Ax)}(q(x))=L_{p(x)}(q(A^{T}x)).$$
\end{enumerate}
\end{proposicion}

\begin{proof}
We only have to prove the first part, where we only need to prove the identity for the monomial basis by linearity. For this reason, we suppose, wlog, that $p(0)=1$ and we will prove that $L_{p(Ax),d}(x^{\beta})=L_{p(x),d}((A^{T}x)^{\beta})$ for all $\beta\neq0$. Following the corresponding definition, it is clear now that we have $\log(p(A(-x)))=(\log(p))(-Ax)$. This means that \begin{gather*}
    \sum_{0\neq\alpha\in\mathbb{N}_{0}^{n}}\frac{1}{|\alpha|}\binom{|\alpha|}{\alpha}L_{p(Ax),d}(x^{\alpha})x^{\alpha}=\sum_{0\neq\alpha\in\mathbb{N}_{0}^{n}}\frac{1}{|\alpha|}\binom{|\alpha|}{\alpha}L_{p(x),d}(x^{\alpha})(Ax)^{\alpha}.
\end{gather*} We use Lemma \ref{lemmaforproposition} above to modify the RHS and write it as \begin{gather*}
    \sum_{0\neq\alpha\in\mathbb{N}_{0}^{n}}\frac{1}{|\alpha|}\binom{|\alpha|}{\alpha}L_{p(x),d}((A^{T}x)^{\alpha})x^{\alpha}.
\end{gather*} Now we only have to compare the coefficient of $x^{\beta}$ in both sides. In the LHS it is $\frac{1}{|\beta|}\binom{|\beta|}{\beta}L_{p(Ax),d}(x^{\beta})$ and in the RHS using the last expression gives us the coefficient $\frac{1}{|\beta|}\binom{|\beta|}{\beta}L_{p(x),d}((A^{T}x)^{\beta})$ so $L_{p(x),d}((A^{T}x)^{\beta})=L_{p(Ax),d}(x^{\beta})$, as we wanted, and thus the proof is finished.
\end{proof}

We make clear that we identify a matrix $A\in\mathbb{R}^{n\times n}$ with a linear transformation $A\colon\mathbb{R}^{n}\to\mathbb{R}^{n}$ and a polynomial $p$ with a map $p\colon\mathbb{R}^{n}\to\mathbb{R}$ in the usual way and therefore we can compose these to form $p\circ A$, which can also be seen as a new polynomial and as a map $p\circ A\colon\mathbb{R}^{n}\to\mathbb{R}$. Thus, now we also want to study how $S(p)$ and $S(p\circ A)$ relate in the particular case when $A=\diag(k,\dots,k)$ for some scalar $k$. We will need this result at the end of this section. The result can also be seen as a corollary of the previous proposition.

\begin{corolario}[Transformation of the relaxation under scaling]
\label{lemmafactor}
Let $p\in\mathbb{R}[x]$ be an RZ polynomial and $0\neq k\in\mathbb{R}$ a scalar, then $\frac{1}{k}S(p)=S(p(kx))$.
\end{corolario}

\begin{proof}
Remember that $S(p):=\{a\in\mathbb{R}^{n}\mid M_{p}(a)\geq0\}$ that is defined via the matrix $M_{p}(a):=(L_{p}(\sum_{k=0}^{n}a_{k}x_{k}x_{i}x_{j}))_{i,j=0}^{n}$, where $x_{0}:=1=:a_{0}$. Note then that, using the Proposition \ref{propositiontransformation} above, we have the identity $M_{p(kx)}(a):=(L_{p(kx)}(\sum_{l=0}^{n}a_{l}x_{l}x_{i}x_{j}))_{i,j=0}^{n}=${\tiny \begin{gather*}
    \begin{pmatrix}
    L_{p(kx)}(1+\sum_{l=1}^{n}a_{l}x_{l}) & L_{p(kx)}(x_{1}+\sum_{l=1}^{n}a_{l}x_{l}x_{1}) & \cdots & L_{p(kx)}(x_{n}+\sum_{l=1}^{n}a_{l}x_{l}x_{n})\\  L_{p(kx)}(x_{1}+\sum_{l=1}^{n}a_{l}x_{l}x_{1}) & L_{p(kx)}(x_{1}^{2}+\sum_{l=1}^{n}a_{l}x_{l}x_{1}^{2}) & \cdots & L_{p(kx)}(x_{1}x_{n}+\sum_{l=1}^{n}a_{l}x_{l}x_{1}x_{n})\\ \vdots & \vdots & \ddots & \vdots\\L_{p(kx)}(x_{n}+\sum_{l=1}^{n}a_{l}x_{l}x_{n}) & L_{p(kx)}(x_{1}x_{n}+\sum_{l=1}^{n}a_{l}x_{l}x_{1}x_{n}) & \cdots & L_{p(kx)}(x_{n}^{2}+\sum_{l=1}^{n}a_{l}x_{l}x_{n}^{2})
    \end{pmatrix}=\\\begin{pmatrix}
    L_{p}(1+\sum_{l=1}^{n}(ka_{l})x_{l}) & L_{p}(kx_{1}+k\sum_{l=1}^{n}(ka_{l})x_{l}x_{1}) & \cdots & L_{p}(kx_{n}+k\sum_{l=1}^{n}(ka_{l})x_{l}x_{n})\\  L_{p}(kx_{1}+k\sum_{l=1}^{n}(ka_{l})x_{l}x_{1}) & L_{p}(k^{2}x_{1}^{2}+k^{2}\sum_{l=1}^{n}(ka_{l})x_{l}x_{1}^{2}) & \cdots & L_{p}(k^{2}x_{1}x_{n}+k^{2}\sum_{l=1}^{n}(ka_{l})x_{l}x_{1}x_{n})\\ \vdots & \vdots & \ddots & \vdots\\L_{p}(kx_{n}+k\sum_{l=1}^{n}(ka_{l})x_{l}x_{n}) & L_{p}(k^{2}x_{1}x_{n}+k^{2}\sum_{l=1}^{n}(ka_{l})x_{l}x_{1}x_{n}) & \cdots & L_{p}(k^{2}x_{n}^{2}+k^{2}\sum_{l=1}^{n}(ka_{l})x_{l}x_{n}^{2})
    \end{pmatrix},
\end{gather*}
}
where in the last expression we used Proposition \ref{propositiontransformation} and grouped the appearing factors adequately.
Therefore we can observe now that \begin{gather*}
(kv_{0},v_{1},\dots,v_{n})M_{p(kx)}(a)(kv_{0},v_{1},\dots,v_{n})^{T}=\\k^{2}(v_{0},v_{1},\dots,v_{n})M_{p}(ka)(v_{0},v_{1},\dots,v_{n})^{T}\end{gather*} for each vector $v\in\mathbb{R}^{n+1}$. This tells us that $M_{p(kx)}(a)$ is PSD iff for all $v\in\mathbb{R}^{n+1}$ \begin{gather*}
    (v_{0},v_{1},\dots,v_{n})M_{p(kx)}(a)(v_{0},v_{1},\dots,v_{n})^{T}\geq0
\end{gather*} iff for all $v\in\mathbb{R}^{n+1}$ \begin{gather*}
    (kv_{0},v_{1},\dots,v_{n})M_{p(kx)}(a)(kv_{0},v_{1},\dots,v_{n})^{T}\geq0
\end{gather*} iff for all $v\in\mathbb{R}^{n+1}$ \begin{gather*}
    k^{2}(v_{0},v_{1},\dots,v_{n})M_{p}(ka)(v_{0},v_{1},\dots,v_{n})^{T}\geq0
\end{gather*} iff for all $v\in\mathbb{R}^{n+1}$ \begin{gather*}
    (v_{0},v_{1},\dots,v_{n})M_{p}(ka)(v_{0},v_{1},\dots,v_{n})^{T}\geq0
\end{gather*} iff $M_{p}(ka)$ is PSD. Now $S(p(kx)):=\{a\in\mathbb{R}^{n}\mid M_{p(kx)}(a)\geq0\}=\{a\in\mathbb{R}^{n}\mid M_{p}(ka)\geq0\}\overset{b:=ka}{=}\{\frac{b}{k}\in\mathbb{R}^{n}\mid M_{p}(b)\geq0\}=\frac{1}{k}\{b\in\mathbb{R}^{n}\mid M_{p}(b)\geq0\}=:\frac{1}{k}S(p)$, which finishes this proof.
\end{proof}

There is a more geometric approach to the proof above that one might try at first but that turns out to produce only a partial proof of the lemma. This is the object of the next observation.

\begin{observacion}[Attention to geometric approaches]
In particular, the geometric approach uses the determinant of the matrix obtained at the beginning of the proof instead of continuing directly with the definition of PSD-ness. This approach has a clear limitation: it needs the mentioned determinant to be nontrivial at the origin and therefore it produces a proof valid only for polynomials $p\in\mathbb{R}[\mathbf{x}]$ verifying that this determinant is not $0$ at the origin, which, in particular, implies that the matrices of their relaxations are PD at the origin. In particular, this means that the determinant of the matrix polynomial defining the relaxation of the corresponding RZ polynomial is, up to a constant, a RZ polynomial.
\end{observacion}

We describe this approach in the next warning as a way to show why the use of the determinant might produce only partial proofs if we are not careful enough. Hence showing that the use of the determinant of the matrix polynomial defining the relaxation might be dangerous and should be avoided whenever possible.

\begin{warning}[Problems taking determinants before it is due]
\label{warningfactor}
We begin this warning just after the expansion of the matrix in the proof above. Now we define the polynomial $D_{p}(a):=\det(M_{p}(a))=\det(M_{p})(a)$. Using properties of the determinant to extract factors $k$ in each row and column except the first one of each in the last expression above and noticing that the bottom right inner submatrix allows for the extraction of a factor $k^{2}$, we can extract several factors $k$ out of the determinant of $M_{p(kx)}(a)$ until we obtain the identity $D_{p(kx)}(a)=\det(M_{p(kx)}(a))=k^{2n}D_{p}(ka)$. We \textbf{have to suppose} now that $D_{p(kx)}(0)=k^{2n}D_{p}(0)\neq0$ in order to continue. Now $S(p(kx))$ is the Euclidean closure of the connected component containing zero of the set $\{a\in\mathbb{R}^{n}\mid D_{p(kx)}(a)\neq0\}$, which coincides (because $D_{p}(a)$ and $D_{p(kx)}(a)$ are now RZ polynomials thanks to the previous additional hypothesis) with the convex hull of the innermost oval defined by the set of zeros $\{a\in\mathbb{R}^{n}\mid D_{p(kx)}(a)=0\}=\{a\in\mathbb{R}^{n}\mid k^{2n}D_{p}(ka)=0\}=\{a\in\mathbb{R}^{n}\mid D_{p}(ka)=0\}\overset{b:=ka}{=}\{\frac{b}{k}\in\mathbb{R}^{n}\mid D_{p}(b)=0\}=\frac{1}{k}\{a\in\mathbb{R}^{n}\mid D_{p}(a)=0\}$. The last expression tells us, in fact, that $\frac{1}{k}S(p)=S(p(kx))$, which finishes this (partial) proof.
\end{warning}

Warning \ref{warningfactor} above shows that we have to take care with determinants when we work with the relaxation because we had to suppose at some point that $D_{p(kx)}(0)=k^{2n}D_{p}(0)\neq0$ in order to produce a (therefore partial) proof. Now we introduce notations that allow us to deal with polynomials derived from $p$ and with the segments produced by intersecting the relaxation with lines through the origin.

\begin{notacion}[Polynomial directions]
\label{problemnotation}
Let $p\in\mathbb{R}[\mathbf{x}]_{d}$ be a polynomial of degree $d$. For each finite tuple of directions $A=(a_{1},\dots,a_{l})$ with $a_{i}\in\mathbb{R}^{n}$ we denote the polynomial corresponding to such directions $\gls{padirection}(x_{1},\dots,x_{l}):=p(\sum_{i=1}^{l}x_{i}a_{i})$, which has generically degree $d$. For a direction $a\in\mathbb{R}^{n}$, denote also $$\gls{sadirection}:=\{\lambda\in\mathbb{R}\mid\lambda a\in S(p)\}.$$
\end{notacion}

There is a problem with the polynomials introduced by this notation and how we are supposed to deal and operate with them. This problem and the solution we implement in order to address it are relevant for the whole current chapter. We explain this in the next warning.

\begin{warning}[Degree might drop]
Let $A=(a_{1},\dots,a_{l})$ with $a_{i}\in\mathbb{R}^{n}$ be a tuple of directions. The restriction $p_{A}(x_{1},\dots,x_{l}):=p(\sum_{i=1}^{l}x_{i}a_{i})$ has only generically the same degree $d$ as $p$. That is, it might happen that for some $A$ and some $p$ we obtain $p_{A}$ such that $\deg(p_{A})\leq\deg(p)$. The easiest nontrivial example with an RZ polynomial is obtained by fixing $p(x_{1},x_{2}):=x_{1}-x_{2}+1$ and $A=(a_{1}):=((1,1))$, which produces the univariate polynomial $p_{A}(y_{1}):=p(y_{1}a_{1})=p(y_{1}(1,1))=p(y_{1},y_{1})=y_{1}-y_{1}+1=1$, so $0=\deg(p_{A})<\deg(p)=1$. We solve this problem by forcing each restriction $p_{A}$ to behave, for every operation that requires taking the degree into account, as if it had the degree of the polynomial $p$ we are restricting at each moment.
\end{warning}

\setlength{\emergencystretch}{3em}%
As mentioned, the solution proposed in the warning will be relevant throughout the whole chapter. This is especially important for the application of two concepts: the relaxation operator $S$ and the Renegar derivative map $\cdot^{(i)}$, as these require a degree as an entry in order to be computed.
\setlength{\emergencystretch}{0em}%

\begin{convencion}[Forcing a degree behaviour]
We force each restriction $p_{A}$ to behave, for every operation that requires taking the degree into account, as if it had the same degree of the polynomial $p$ we are restricting at each moment.
\end{convencion}

Our notation makes easy to take care of this degree convention.

\begin{observacion}[Tracking degree]
This forces us to remember and keep track at each moment of the original polynomial a restriction comes from. Fortunately, this tracking is easy and thus totally possible and doable during all the document because our notation helps with this task. This is so because we always refer to the polynomial $p$ we are restricting to and to the subspace that defines that restriction given by $$\linspan(A):=\linspan(\{a_{1},\dots,a_{l}\})\subseteq\mathbb{R}^{n}$$ in the name that Notation \ref{problemnotation} gives to the corresponding restriction: $p_{A}$.
\end{observacion}

We compress this comment in the next remark.

\begin{remark}[Use of degree shorthands]
\label{remarkdegree}
Whenever we write in the future $S(p_{A})$ or $(p_{A})^{(i)}$ we always mean the subset of $\mathbb{R}^{l}$ given by the relaxation $S_{d}(p_{A})$ (as introduced in \cite[Definition 3.19]{main}) and the polynomial in $\mathbb{R}[\mathbf{x}]$ given by the Renegar derivative $\frac{\partial^{k}}{\partial x_{0}^{k}}(x_{0}^{d}p_{A}(\frac{x}{x_{0}}))|_{x_{0}=1}$ produced by homogenizing using the new variable $x_{0}$ the restricted polynomial $p_{A}$ at degree $d:=\deg(p)$ with $p$ being the polynomial that we are restricting to the subspace $\linspan(A)\subseteq\mathbb{R}^{n}$ regardless of the actual degree $\deg(p_{A})$ of the polynomial $p_{A}$ obtained as a restriction of the polynomial $p$ to the subspace $\linspan(A)$ spanned by the vectors in the tuple of vectors $A=(a_{1},\dots,a_{l})\in(\mathbb{R}^{n})^{l}$ with $a_{i}\in\mathbb{R}^{n}$ for each $i\in\{1,\dots,l\}$, which might be smaller than $d$.
\end{remark}

We introduce now a useful lemma that will help us to understand better how the objects introduced above work and will be useful for a simplification of the next result that we want to prove. Thus, the next lemma serves a double purpose of showing how to handle these objects together and simplifying their treatment in the future.

\begin{lema}[Simplifying work lemma]\label{workinglemma} Write $e_{1}$ for the first unit vector in $\mathbb{R}^{n}$ and fix $b\in\mathbb{R}^{n}$ arbitrary.
\begin{enumerate}
    \item For every tuple $(a,p,k)$ formed by a vector $a\in\mathbb{R}^{n}$, an RZ polynomial $p\in\mathbb{R}[x]$ and a scalar $0\neq k\in\mathbb{R}$ we have $S_{ka}(p)=\frac{1}{k}S_{a}(p)$.
    \item For every tuple $(A,p,k)$ formed by a finite tuple of vectors (or directions) $A=(a_{1},\dots,a_{l})$ with $a_{i}\in\mathbb{R}^{n}$, an RZ polynomial $p\in\mathbb{R}[x]$ and a scalar $0\neq k\in\mathbb{R}$ we have $S_{b}(p_{kA})=S_{b}((p\circ k)_{A})$, where $k$ represents the operator $x\mapsto kx$ in the last expression.
    \item For every pair $(a,p)$ formed by a vector $a\in\mathbb{R}^{n}$ and an RZ polynomial $p\in\mathbb{R}[x]$ there exists an RZ polynomial $q\in\mathbb{R}[x]$ such that the set $S_{a}(p)$ can be written in terms of $S_{||a||e_{1}}$ in the direct way $S_{a}(p)=S_{||a||e_{1}}(q)$. Moreover, we can take the polynomial $q=(p\circ U)$ for some unitary matrix $U$ such that $U^{T}a=||a||e_{1}$.
    \item For every pair $(A,p)$ formed by a finite tuple of directions $A=(a_{1},\dots,a_{l})$ with $a_{i}\in\mathbb{R}^{n}$ and an RZ polynomial $p\in\mathbb{R}[x]$ there exists a finite tuple of directions $\hat{A}=(\hat{a}_{1},\dots,\hat{a}_{l})$ with $\hat{a}_{i}\in\mathbb{R}^{n}$ and $\hat{a}_{1}=||a_{1}||e_{1}$ and an RZ polynomial $q\in\mathbb{R}[x]$ such that the set $S_{b}(p_{A})$ can be written in terms of the relaxation operator $S_{b}$ of the restriction $q_{\hat{A}}$ of the polynomial $q$ to the subspace $\linspan(\hat{A})$ spanned by the tuple $\hat{A}$ in the direct way $S_{b}(p_{A})=S_{b}(q_{\hat{A}})$. Moreover, we can take the polynomial $q=(p\circ U)$ and the vectors $\hat{a}_{i}=U^{T}a_{i}$ for all $i\in\{1,\dots,l\}$ for some unitary matrix $U$.
\end{enumerate}
\end{lema}

\begin{proof}
1. $S_{ka}(p):=\{\lambda\in\mathbb{R}\mid\lambda ka\in S(p)\}\overset{\mu:=\lambda k}{=}\{\frac{\mu}{k}\in\mathbb{R}\mid\mu a\in S(p)\}=\frac{1}{k}\{\mu\in\mathbb{R}\mid\mu a\in S(p)\}:=\frac{1}{k}S_{a}(p)$.

2. $S_{b}(p_{kA}):=\{\lambda\in\mathbb{R}\mid\lambda b\in S(p_{kA})\}\overset{q(x):=p(kx)}{=}\{\lambda\in\mathbb{R}\mid\lambda b\in S(q_{A})\}:=S_{b}(q_{A})$ because we have $p_{kA}:=p(\sum_{i=1}^{l}x_{l}ka_{l})=p(k\sum_{i=1}^{l}x_{l}a_{l})=q(\sum_{i=1}^{l}x_{l}a_{l})=:q_{A}$.

3. First, observe that it is possible to choose an orthogonal matrix $U\in\mathbb{R}^{n\times n}$ such that $||a||e_{1}=U^{T}a\in\mathbb{R}^{2}\times\{0\}^{n-2}\subseteq\mathbb{R}^{n}$. Thence we can write that $S_{a}(p):=\{\lambda\in\mathbb{R}\mid\lambda a\in S(p)\}=\{\lambda\in\mathbb{R}\mid\lambda U^{T}a\in S(p(Ux))\}$ by \cite[Proposition 3.24]{main} because $\lambda a\in S(p)$ iff $\lambda U^{T}a\in S(p(Ux))$. Because we choose $U$ such that $U^{T}a=||a||e_{1}$ we, in fact, have proved the identity $$S_{a}(p)=\{\lambda\in\mathbb{R}\mid\lambda ||a||e_{1}\in S(p(Ux))\}=S_{||a||e_{1}}(p\circ U),$$ which finished the proof of this point.

4. As before, it is clear that it is possible to choose $U$ such that $U^{T}a_{i}=\hat{a}_{i}.$ Using the identity $p_{A}=p(\sum_{i=1}^{l}x_{i}a_{i})=p(U(\sum_{i=1}^{l}x_{i}U^{T}a_{i}))=$ \begin{gather*}
(p\circ U)(\sum_{i=1}^{l}x_{i}U^{T}a_{i})=(p\circ U)(x_{1}||a||e_{1}+\sum_{i=2}^{l}x_{i}\hat{a}_{i})=(p\circ U)_{\hat{A}},\end{gather*} which leads automatically to the identity $S_{b}(p_{A})=S_{b}((p\circ U)_{\hat{A}})$ because this is in fact just $$\{\lambda\in\mathbb{R}\mid\lambda b\in S(p_{A})\}=\{\lambda\in\mathbb{R}\mid \lambda b\in S((p\circ U)_{\hat{A}})\},$$ which gives the result concerning to this point and finishes the proof.
\end{proof}

Equipped with this lemma, we can begin studying restrictions of the relaxation to all the planes containing one vector or direction. This is the endeavour of the next section around the topic of subspaces filling the whole space.

\section[The relaxation via restrictions to planes]{Determination of the relaxation via restrictions to planes}\label{sectioninter}

We begin this section by exposing a way of dividing the space $\mathbb{R}^{n}$ into subspaces. We will also choose a set of transformations over each one of these parts (subspaces) of such divisions that will be useful for the bivariate perspective that the analysis of our relaxation will reach when we use \cite[Lemma 3.22]{main}. We can fill the whole space through a choice of subspaces.

\begin{definicion}[Filling space through subspaces]
\label{definitionspace}
Fix $V\subseteq\mathbb{R}^{n}$ a subspace and $l\in[n-1]$. A \textit{space-filling selection of $l$-dimensional subspaces strictly through $V$} is a set $\Pi$ of $l$-dimensional subspaces $\pi\subseteq\mathbb{R}^{n}$ such that $\bigcup\Pi=\mathbb{R}^{n}$ and $\pi\cap\pi'=V$ for all pairs $\pi,\pi'\in\Pi$ with $\pi\neq\pi'$. The subspace $V\subseteq\mathbb{R}^{n}$ is called the \textit{common subspace} of the space-filling selection of $l$-dimensional subspaces strictly through $V$ given by $\Pi$. An \textit{uniform-$\Pi$ set of transformations} is a set $U_{\Pi}$ of transformations $U_{\pi}\in\mathbb{R}^{n\times n}$ for each subspace $\pi\in\Pi$.
\end{definicion}

The common subspace $V$ can be characterized as the subset of $\mathbb{R}^{n}$ such that any element $b$ not in it determines a unique element of the space-filling selection of $l$-dimensional subspaces strictly through $V$ given by $\Pi$. We see this in the following proposition and use these elements of $\mathbb{R}^{n}$ as indices of the elements in $\Pi$ and in $U_{\Pi}$ due to the mentioned uniqueness that serves the purpose of determining completely these corresponding elements.

\begin{remark}
Notice that it must be $l<n$ so that we have that $\card(\Pi)>1$ and we can therefore find actual pairs of different subspaces defining $V$. We do not want that requirement to be vacuous because we want to have ways to determine $V$.
\end{remark}

\begin{proposicion}[Characterization of common subspace]
Let $\Pi$ be a space-filling selection of $l$-dimensional subspaces strictly through the common subspace $V\subseteq\mathbb{R}^{n}$. Then $$V=\{a\in\mathbb{R}^{n}\mid(\forall b\notin V)\ (\exists! \pi\in\Pi)\ b\in\pi\}.$$
\end{proposicion}

\begin{proof}
With the notation above, any $b\notin V$ is in one and only one of the $l$-dimensional subspaces in $\Pi$ because otherwise, taking $\pi\neq\pi'$ two of these $l$-dimensional subspaces in $\Pi$ containing $b\notin V$, we would get the contradiction $b\in\linspan(\{b\}\cup V)\subseteq\pi\cap\pi'=V$. This finishes the proof.
\end{proof}

This allows us to simplify our notation making reference to the elements instead of the subspaces. Therefore, we denote the corresponding unique subspace using the elements as indices.

\begin{notacion}[Indexing by elements]
Let $\Pi$ be a space-filling selection of $l$-dimensional subspaces strictly through the common subspace $V\subseteq\mathbb{R}^{n}$ and let $b\notin V$. We denote $\pi_{b}\in\Pi$ the unique subspace in $\Pi$ containing $b\notin V$. This indexing is well-defined thanks to the proposition above.
\end{notacion}

Thus we can index using elements $b\notin V$ of the space $\mathbb{R}^{n}$ as indices for these elements of the space-filling selection of $l$-dimensional subspaces strictly through $V$ given by $\Pi$. We similarly index $U_{b}:=U_{\pi_{b}}$ whenever $\Pi\ni\pi_{b}\ni b\notin V$. The uniformity in the choice of the set $U_{\Pi}$ of transformations $U_{\pi}$ for $\pi\in\Pi$ associated to $\Pi$ will be important in the proof of the Fact \ref{factintersection} closing this section. In particular, there is, of course, a canonical way of choosing an uniform-$\Pi$ set of transformations in a way that it sends its corresponding (indexing) $l$-dimensional subspace $\pi\in\Pi$ of the space-filling selection of $l$-dimensional subspaces strictly through $V$ given by $\Pi$ to the canonical $l$-dimensional subspace $P:=\{e_{1},\dots,e_{l}\}$. We put this in symbols.

\begin{definicion}[Canonical uniform set of tranformations]
\label{definitionspace2}
If a uniform-$\Pi$ set of transformations $U_{\Pi}$ verifies that $U_{\pi}^{T}\pi=P$ for each $\pi\in\Pi$, we say that $U_{\Pi}$ is \textit{canonical}. Observe that we can also write this indexing by elements $b\notin V$ as $U_{b}^{T}\pi_{b}=P$ for all $b\notin V$.
\end{definicion}

The instance of the very general Definition \ref{definitionspace} that we will need is the particular case corresponding to setting the dimension of the subspaces in the space-filling selection $l=2$, putting the common subspace $V=\linspan(\{e_{1}\})$ and fixing the elements of $U_{\Pi}$ as orthogonal transformations. We will also establish a convention on which element of $\Pi$ we consider as covering the common subspace $V$. During this particularization we introduce also other conventions along the lines mentioned on this paragraph that will make our steps clearer during our future arguments and proofs. We collect in symbols all these requirements over our particular objects for clarity.

\begin{particularizacion}[$l=2,V=\{e_{1}\}$]
We introduced the notions in Definition \ref{definitionspace} in general, but we noted above that we are mainly interested in the case $l=2$ and $V=\linspan(\{e_{1}\})$. We index then $U_{b}:=U_{\pi_{b}}$ whenever $\Pi\ni\pi_{b}\ni b\notin\linspan(\{e_{1}\})$ and we make sure that $P:=\{e_{1},e_{2}\}\in\Pi$ and that $U_{a_{1}e_{1}+a_{2}e_{2}}:=U_{e_{2}}=U_{P}:=I$ for all $a_{1},a_{2}\in\mathbb{R}$ in order to cover with our point indexation the cases (in the common subspace $V$, when $a_{2}=0$, and in the bidimentional subspace $P\in\Pi$ that we decide that will cover the common subspace, when $a_{2}\neq0$) where we are already in the canonical plane $P\in\Pi$ with the identity transformation. We will also take $U_{\Pi}$ canonical with all its transformations orthogonal.
\end{particularizacion}

We also make several conventions that will facilitate our arguments wlog.

\begin{convencion}
From now on, we will consider moreover that $p\in\mathbb{R}[\mathbf{x}]$ and $p_{A}\in\mathbb{R}[\mathbf{y}]$ with $A\in(\mathbb{R}^{n})^{l}$ have the same $\max(\{l,n\})$ variables, having always in mind that the completions up to that maximum might produce monomials having coefficient zero or vectors in $A$ that are zero.
\end{convencion}

We explain this in more detail and also why it will help us.

\begin{remark}
If $l=n$, we just have to set $\mathbf{x}=\mathbf{y}$ as then $p$ and $p_{A}$ have the same number of variables. We now analyze the cases where we need to adjust the number of variables. If $l<n$, then we consider the extension of the tuple $A$ given by postpending to it as many zero vectors $0\in\mathbb{R}^{n}$ as necessary until the new tuple $A'\in(\mathbb{R}^{n})^{n}$ has $n$ vectors and we consider instead $p_{A'}\in\mathbb{R}[\mathbf{y},y_{l+1},\dots,y_{n}]$ instead of $p_{A}\in\mathbb{R}[\mathbf{y}]$, but in fact $p_{A}=p_{A'}$. Otherwise, if $n<l$ then we consider $p\in\mathbb{R}[\mathbf{x},x_{n+1},\dots,x_{l}]$ with all monomials involving the new extra variables having a zero coefficient. 
\end{remark}

We prove now a first connection between the relaxation of the RCS of an RZ polynomial $p$ and the relaxation of the RCSs of its bivariate restrictions. Such connection is just an initial (and particular) step towards a more general, clear, deep and elegant one that we will establish as the main result of this section in the last point of the ending Corollary \ref{corolariofinal}. In particular, we prove the following fact, which arises as a corollary of Proposition \ref{propositiontransformation} above.

\begin{hecho}[Intersection identity through scaling]
\label{factintersection}
For all scalar $k\in\mathbb{R}$ and RZ polynomial $p\in\mathbb{R}[x]$ we have the identity $$S_{ke_{1}}(p)=\bigcap_{b\in\mathbb{R}^{n}}S_{k(1,0)}(p_{e_{1},b}).$$
\end{hecho}

\begin{proof}
We will prove this identity using \cite[Lemma 3.22]{main} and, for each $b\in\mathbb{R}^{n}$, orthogonal matrices $U_{b}$ fixing $e_{1}$ and sending $b$ to the plane spanned by $\{e_{1},e_{2}\}$, i.e., $U_{b}^{T}b=l_{b}e_{1}+m_{b}e_{2}$ for all $b\in\mathbb{R}^{n}$. In particular, we can take for the trivial cases $U_{\mu e_{1}+\nu e_{2}}=I$ for all $\mu,\nu\in\mathbb{R}$. Moreover, for the nontrivial case, we will choose these matrices uniformly on each element of a space-filling selection of planes strictly through $e_{1}$ (see Definition \ref{definitionspace} and its corresponding discussion). That is, we choose a set $\Pi$ of planes $\pi\subseteq\mathbb{R}^{n}$ such that $\bigcup\Pi=\mathbb{R}^{n}$ and $\pi\cap\pi'=\linspan(\{e_{1}\})$ for all pairs $\pi,\pi'\in\Pi$ with $\pi\neq\pi'$ so any $b\notin\linspan(\{e_{1}\})$ is in one and only one plane $\pi_{b}\in\Pi$. In this case we denote also $U_{b}=U_{\pi_{b}}$ whenever $b\in\pi_{b}\in\Pi$. We make sure that we cover the trivial case $b\in\linspan(\{e_{1}\})$ noting that $\linspan(\{e_{1},e_{2}\})\in\Pi$ and setting $\pi_{k_{1}e_{1}+k_{2}e_{2}}=\linspan(\{e_{1},e_{2}\})$ and $U_{k_{1}e_{1}+k_{2}e_{2}}=U_{e_{2}}=I$ for all $k_{1},k_{2}\in\mathbb{R}$. This uniformity in the choice of the orthogonal transformations will be crucial soon. Observe that with this we have \begin{gather}\label{planetag} U_{b}^{T}\pi_{b}=\linspan(\{e_{1},e_{2}\}) \mbox{\ for all\ } b\in\mathbb{R}^{n}.\end{gather} We will build now a chain of equivalent conditions beginning by the LHS of the statement until we reach an expression that can be adequately connected to a different chain of equivalences reaching the RHS of the statement. In fact, $\lambda\in S_{ke_{1}}(p)$ iff $L_{p}((1+\lambda ke_{1}^{T}x)(1+b^{T}x)^2)\geq0$ for all $b\in\mathbb{R}^{n}$ using \cite[Lemma 3.22]{main}, continuity of the full expression, homogeneity of $(b_{0} + b^{T}x)^{2}$ with respect to $(b_{0},b)$, the fact that squares of real numbers are nonnegative and linearity of $L_{p}$. Now $L_{p}((1+\lambda ke_{1}^{T}x)(1+b^{T}x)^2)\geq0$ for all $b\in\mathbb{R}^{n}$ iff $L_{p(U_{b}x)}((1+\lambda ke_{1}^{T}U_{b}x)(1+b^{T}U_{b}x)^2)\geq0$ for all $b\in\mathbb{R}^{n}$ by \cite[Lemma 3.8(a)]{main}, but by the choice of $U_{b}$ for each $b$ this is equivalent to \begin{gather}\label{beforegrouping}L_{(p\circ U_{b})}((1+\lambda ke_{1}^{T}x)(1+(l_{b}e_{1}+m_{b}e_{2})^{T}x)^2)\geq0 \mbox{\ for all\ } b\in\mathbb{R}^{n}.\end{gather} Moreover, grouping via the uniform choice of $U_{\pi}$ for each plane $\pi\in\Pi$ (see Equation \ref{planetag}), this is equivalent to the plane-by-plane uniform condition \begin{gather}\label{aftergrouping}L_{(p\circ U_{\pi})}((1+\lambda ke_{1}^{T}x)(1+(le_{1}+me_{2})^{T}x)^2)\geq0 \mbox{\ for all\ } l,m\in\mathbb{R} \mbox{\ for all\ } \pi\in\Pi.\end{gather} Note that the trivial case appearing when $b\in\linspan(\{e_{1}\})$) is covered by our choice of $\Pi$ and the indexation within the case when $\pi=\pi_{e_{2}}=\linspan(\{e_{1},e_{2}\})$. Calling from now on $(p\circ U_{b})(\overline{x}):=(p\circ U_{b})(x_{1},x_{2},0)$ and $\overline{x}=(x_{1},x_{2})$ and renaming $e_{1},e_{2}\in\mathbb{R}^{2}$, the condition in Equation \ref{beforegrouping} (before grouping) is equivalent to the bivariate condition \begin{gather}
L_{(p\circ U_{b})(\overline{x})}((1+\lambda ke_{1}^{T}\overline{x})(1+(l_{b}e_{1}+m_{b}e_{2})^{T}\overline{x})^2)\geq0 \mbox{\ for all\ } b\in\mathbb{R}^{n}.\end{gather} After grouping (see Equation \ref{aftergrouping}), this condition is equivalent to the plane-by-plane uniform condition \begin{gather}\label{almostend}L_{(p\circ U_{\pi})(\overline{x})}((1+\lambda ke_{1}^{T}\overline{x})(1+(le_{1}+me_{2})^{T}\overline{x})^2)\geq0 \mbox{\ for all\ } l,m\in\mathbb{R} \mbox{\ and\ } \pi\in\Pi.\end{gather} Note again that the trivial case is covered in the same way as above. Denote now $U_{\Pi}=\{U_{\pi}\mid\pi\in\Pi\}$ and remember that we imposed $I=U_{e_{2}}\in U_{\Pi}$. The condition in Equation \ref{almostend} above is therefore equivalent, by the choice of $\Pi$, to \begin{gather}L_{(p\circ U)(\overline{x})}((1+\lambda ke_{1}^{T}\overline{x})(1+c^{T}\overline{x})^2)\geq0 \mbox{\ for all\ } c\in\mathbb{R}^{2} \mbox{\ for all\ } U\in U_{\Pi}.\end{gather} Finally, this chain of equivalent conditions ends with the equivalent condition \begin{gather}\label{finalLHScondition}
L_{(p\circ U_{b})(\overline{x})}((1+\lambda ke_{1}^{T}\overline{x})(1+c^{T}\overline{x})^2)\geq0 \mbox{\ for all\ } c\in\mathbb{R}^{2}\mbox{\ for all\ }b\in\mathbb{R}^{n}.\end{gather} Now we will pursue a different chain of equivalences beginning on the RHS of the statement that will end up on the condition in Equation \ref{finalLHScondition} above finishing the proof. In order to do this we need some notation and a chain of equalities. Denote $A_{b}=\begin{pmatrix}
1 & l_{b}\\0 & m_{b}
\end{pmatrix}$ for all $b\in\mathbb{R}^{n}$ and observe that $A_{b}$ defines a bijective linear map whenever $m_{b}\neq0$. Observe further that then for all $b\in\mathbb{R}^{n}$ we have \begin{equation}\label{chainofequationsbivariaterestrictions}
\begin{gathered}
(p\circ U_{b})_{e_{1},l_{b}e_{1}+m_{b}e_{2}}(\overline{x})=(p\circ U_{b})(x_{1}e_{1}+x_{2}(l_{b}e_{1}+m_{b}e_{2}))=\\(p\circ U_{b})(x_{1}+l_{b}x_{2},m_{b}x_{2}))=(p\circ U_{b})(A_{b}\overline{x})=\\p(x_{1}U_{b}e_{1}+x_{2}U_{b}(l_{b}e_{1}+m_{b}e_{2}))=p(x_{1}e_{1}+x_{2}b)=p_{e_{1},b}(\overline{x}).\end{gathered}
\end{equation} Now we can continue. Observe, following the condition on the RHS, that $\lambda\in\bigcap_{b\in\mathbb{R}^{n}}S_{k(1,0)}(p_{e_{1},b})$ iff $L_{p_{e_{1},b}}((1+\lambda ke_{1}^{T}\overline{x})(1+c^{T}\overline{x})^2)\geq0$ for all $c\in\mathbb{R}^{2}$ for all $b\in\mathbb{R}^{n}$, which is equivalent, by the chain of equalities in Equation \ref{chainofequationsbivariaterestrictions} above, to $L_{(p\circ U_{b})(A_{b}\overline{x})}((1+\lambda ke_{1}^{T}\overline{x})(1+c^{T}\overline{x})^2)\geq0$ for all $c\in\mathbb{R}^{2}$ for all $b\in\mathbb{R}^{n}$, which is equivalent, by Proposition \ref{propositiontransformation} above, to the condition $L_{(p\circ U_{b})(\overline{x})}((1+\lambda ke_{1}^{T}A_{b}^{T}\overline{x})(1+c^{T}A_{b}^{T}\overline{x})^2)\geq0$ for all $c\in\mathbb{R}^{2}$ for all $b\in\mathbb{R}^{n}$, which can be written as $L_{(p\circ U_{b})}((1+(A_{b}\lambda ke_{1})^{T}\overline{x})(1+(A_{b}c)^{T}\overline{x})^2)\geq0$ for all $c\in\mathbb{R}^{2}$ for all $b\in\mathbb{R}^{n}$, and, developing further, it can be written as $L_{(p\circ U_{b})}((1+\lambda ke_{1}^{T}\overline{x})(1+(A_{b}c)^{T}\overline{x})^2)\geq0$ for all $c\in\mathbb{R}^{2}$ for all $b\in\mathbb{R}^{n}$. Observe now that $m_{b}=0$ iff $A_{b}(\mathbb{R}^{2})=\mathbb{R}\times\{0\}$ iff $U^{T}_{b}b=l_{b}e_{1}$ for some $l_{b}\in\mathbb{R}$ and, as we choose $U^{T}_{b}e_{1}=e_{1}$, by orthogonality (i.e., $U^{T}_{b}U_{b}=U_{b}U^{T}_{b}=I$), this means that this happens iff $b\in\linspan(\{e_{1}\})$. Thus this case is covered by the case $b=e_{2}$ because it produces the same matrix $U_{e_{2}}=U_{k_{1}e_{1}}=I$ and thus the same polynomial $p=p\circ U_{b}$ in the subindex of $L$. Therefore, as the vectors $b\in\mathbb{R}^{n}$ that verify $m_{b}\neq0$ cover all the cases and when $m_{b}\neq0$ we have that, when $c$ varies in all $\mathbb{R}^{2}$, so does $A_{b}c$, we have a final equivalence to the condition \begin{gather}\label{finalRHScondition}L_{(p\circ U_{b})}((1+\lambda ke_{1}^{T}\overline{x})(1+c^{T}\overline{x})^2)\geq0 \mbox{\ for all\ } c\in\mathbb{R}^{2} \mbox{\ for all\ } b\in\mathbb{R}^{n},\end{gather} which coincides with the final condition in Equation \ref{finalLHScondition} (the conditions in Equations \ref{finalLHScondition} and \ref{finalRHScondition} are the same) of the chain above beginning in the LHS and finishes this proof connecting via equivalences the conditions of the LHS and the RHS of the statement.
\end{proof}

\begin{corolario}[Implied intersection identities]
\label{corolariofinal}
For all vector $0\neq a\in\mathbb{R}^{n}$, scalar $k\in\mathbb{R}$ and RZ polynomial $p\in\mathbb{R}[x]$ we have the identities \begin{enumerate}
    \item $S_{k\frac{a}{||a||}}(p)=\bigcap_{b\in\mathbb{R}^{n}}S_{k(1,0)}(p_{\frac{a}{||a||},b})$.
    \item $S_{a}(p)=\frac{1}{||a||}\bigcap_{b\in\mathbb{R}^{n}}S_{(1,0)}(p_{\frac{a}{||a||},b})$.
    \item $S_{a}(p)=\bigcap_{b\in\mathbb{R}^{n}}S_{(1,0)}(p_{a,b})$.
\end{enumerate}
\end{corolario}

\begin{proof}
1. Using the corresponding point (3 or 4) of the Lemma \ref{workinglemma} above in each part and choosing the same matrix $U$ for both parts of the identity in the current statement reduces our proof to show the identity $$S_{ke_{1}}(p)=\bigcap_{b\in\mathbb{R}^{n}}S_{k(1,0)}(p_{e_{1},b}),$$ whose advantage is that now our initial unitary vector $\frac{a}{||a||}$ is the simplest possible and we can establish it using Fact \ref{factintersection} above.

2. Using now the first point of such lemma and setting $k=||a||$ if $a\neq0$ we can particularize the already proved point to the identity $$S_{a}(p)=\bigcap_{b\in\mathbb{R}^{n}}S_{||a||(1,0)}(p_{\frac{a}{||a||},b})=\frac{1}{||a||}\bigcap_{b\in\mathbb{R}^{n}}S_{(1,0)}(p_{\frac{a}{||a||},b}).$$

3. We use the point above and Lemma \ref{lemmafactor}. Fix $b\in\mathbb{R}^{n}$ arbitrary. Observe that $p_{\frac{a}{||a||},b}:=p(x_{1}\frac{a}{||a||}+x_{2}\frac{||a||}{||a||}b)=p(a\frac{x_{1}}{||a||}+||a||b\frac{x_{2}}{||a||})=p_{a,||a||b}(\frac{1}{||a||}\overline{x})$. Then $S_{(1,0)}(p_{\frac{a}{||a||},b}):=\{\lambda\in\mathbb{R}\mid\lambda(1,0)\in S(p_{\frac{a}{||a||},b})\}=$\begin{gather*}
\{\lambda\in\mathbb{R}\mid\lambda(1,0)\in S(p_{a,||a||b}(\frac{1}{||a||}\overline{x}))\}=\\\{\lambda\in\mathbb{R}\mid\lambda(1,0)\in ||a||S(p_{a,||a||b})\}\overset{\mu:=\frac{\lambda}{||a||}}{=}\\\{||a||\mu\in\mathbb{R}\mid\mu(1,0)\in S(p_{a,||a||b})\}=\\||a||\{\mu\in\mathbb{R}\mid\mu(1,0)\in S(p_{a,||a||b})\}=:\\ ||a||S_{(1,0)}(p_{a,||a||b}),\end{gather*} where we used the identity found above and Lemma \ref{lemmafactor}. Thus, since this is true for any $b\in\mathbb{R}^{n}$, we can continue the chain of equalities of the point above with $\cdots=\frac{1}{||a||}\bigcap_{b\in\mathbb{R}^{n}}||a||S_{(1,0)}(p_{a,||a||b})=\bigcap_{b\in\mathbb{R}^{n}}S_{(1,0)}(p_{a,b})$, which finishes the proof of this point and establishes the corollary.
\end{proof}

The last point of Corollary \ref{corolariofinal} above is the fact that we wanted to establish on this section. This gets us ready to begin exploring cubic boundedness in the next section towards a more general result.

\section{Reducing subindex polynomial via cubic boundedness}\label{sectionL}

We will use Helton-Vinnikov again to restrict to planes. These restrictions will be key in our arguments during this section.

\begin{observacion}[Consequences of Helton-Vinnikov]
Using this theorem and orthogonal diagonalization we know that every bivariate RZ polynomial $p$ with $p(0)=1$ and $d:=\deg(p)$ admits an expression of the form $p=\det(I+x_{1}D+x_{2}A)$ with $I,D,A\in\mathbb{R}^{d\times d}$ and $D$ diagonal and $A$ symmetric.
\end{observacion}

Now we can consider the extension of that polynomial given by the RZ polynomial $\det(I+x_{1}D+x_{2}A+yI)$. This polynomial has a nice representation in terms of the Renegar derivatives of $p$ in a way similar to what we saw in \cite{nuijtype} for univariate polynomials. We will see this in the next proposition. However, to see this, we have first to deal with the size of the determinantal representation. This is important for degree considerations.

\begin{remark}[Degree drops]
Above, we took the care of fixing the size of the determinantal representation equal to the degree of the polynomial $p$. However, this does not have to happen for a general determinantal representation.
\end{remark}

The next definition introduces the polynomial related to $p$ and its derivatives that we will end obtaining. The advantage of this definition is that it does not produce degree drops.

\begin{definicion}[Extended polynomial]
\label{hatp}
Let $p\in\mathbb{R}[\mathbf{x}]$ be a polynomial with $d:=\deg(p)$. We define the polynomial $$\hat{p}:=\sum_{k=0}^{\infty}\frac{1}{k!}y^{k}p^{(k)}=\sum_{k=0}^{\infty}\frac{1}{k!}y^{k}p^{(k)}=\sum_{k=0}^{d}\frac{1}{k!}y^{k}p^{(k)}\in\mathbb{R}[\mathbf{x},y].$$
\end{definicion}

The next proposition will establish the connection between the just defined $\hat{p}$ and extensions of determinantal representations of size equal to the degree of the initial polynomial $p\in\mathbb{R}[\mathbf{x}]$. This gives the relation between determinantal representations and Renegar derivatives that we need. Note that the care taken to require the hypothesis of size and degree equality in the next proposition is what makes the bulk of the writing in its statement, although that is not at all an important part of the result.

\begin{remark}
This will happen more times in the future. These kind of problems with degrees will in fact annoy us until we reach our final result even though, as happens here, they appear as a consequence of the consideration of cases that might occur but that do not constitute, however, an important part of our findings in general.
\end{remark}

Thus, we should not be scared by the length of the statement in the proposition as these hypotheses are easy to fulfill. The result is the following.

\begin{proposicion}[Relation between extensions and determinants]
\label{intermedia}
Let $p\in\mathbb{R}[\mathbf{x}]$ be a determinantal (therefore RZ) polynomial with determinantal representation $$\det(I+x_{1}A_{1}+\cdots+x_{n}A_{n})$$ of size $d:=\deg(p)$ (i.e., $A_{i}\in\mathbb{R}^{d\times d}$). Then the polynomial $$\det(I+x_{1}A_{1}+\cdots+x_{n}A_{n}+yI)$$ obtained as the determinant of the extension $$I+x_{1}A_{1}+\cdots+x_{n}A_{n}+yI$$ of the LMP $$I+x_{1}A_{1}+\cdots+x_{n}A_{n}$$ inside the determinant in the given determinantal representation of $p$ can be recovered solely from $p$ and its Renegar derivatives as $\hat{p}$. That is, under the hypotheses of equality of size and degree, $$\det(I+x_{1}A_{1}+\cdots+x_{n}A_{n}+yI)=\sum_{k=0}^{\infty}\frac{1}{k!}y^{k}p^{(k)},$$ where $p^{(k)}$ represents the $k$-th Renegar derivative.
\end{proposicion}

\begin{proof}
We begin homogenizing $p$ in order to be able to \textit{follow} the variable $y$ easily. We know that $p=\det(I+x_{1}A_{1}+\cdots+x_{n}A_{n})$. Consider instead its homogenization $p^{h}=\det(x_{0}I+x_{1}A_{1}+\cdots+x_{n}A_{n})$. Remember that $p^{h}=\sum_{i=0}^{d}x_{0}^{i}p_{d-i}$ is obtained from $p$ by completing each homogeneous part $p_{i}$ of degree $i$ of $p$ with the additional factor $x_{0}^{d-i}$ so the total product $x_{0}^{d-i}p_{i}$ reaches degree $d$ and note that the definition of the Renegar derivative and $\hat{\cdot}$ allows us to write $\det((1+y)I+x_{1}A_{1}+\cdots+x_{n}A_{n})=p^{h}(1+y,x)=\sum_{i=0}^{d}(1+y)^{i}p_{d-i}=\sum_{i=0}^{d}\sum_{k=0}^{i}\binom{i}{k}y^{k}p_{d-i}=(\sum_{i=0}^{d}\sum_{k=0}^{i}\binom{i}{k}x_{0}^{i}y^{k}p_{d-i})|_{x_{0}=1}=$\begin{gather*}
(\sum_{i=0}^{d}\sum_{k=0}^{i}\frac{i(i-1)\cdots(i-k+1)}{k!}x_{0}^{i}y^{k}p_{d-i})|_{x_{0}=1}=(\sum_{i=0}^{d}\sum_{k=0}^{i}\frac{1}{k!}\frac{\partial^{k}}{\partial x_{0}^{k}}x_{0}^{i}y^{k}p_{d-i})|_{x_{0}=1}\\=(\sum_{i=0}^{d}\sum_{k=0}^{d}\frac{1}{k!}\frac{\partial^{k}}{\partial x_{0}^{k}}x_{0}^{i}y^{k}p_{d-i})|_{x_{0}=1}=(\sum_{k=0}^{d}\frac{1}{k!}y^{k}\frac{\partial^{k}}{\partial x_{0}^{k}}\sum_{i=0}^{d}x_{0}^{i}p_{d-i})|_{x_{0}=1}=\end{gather*}$(\sum_{k=0}^{d}\frac{1}{k!}y^{k}\frac{\partial^{k}}{\partial x_{0}^{k}}p^{h})|_{x_{0}=1}=\sum_{k=0}^{d}\frac{1}{k!}y^{k}(\frac{\partial^{k}}{\partial x_{0}^{k}}p^{h})|_{x_{0}=1}=\sum_{k=0}^{d}\frac{1}{k!}y^{k}p^{(k)}=:\hat{p}$, where we completed the the sum indexed by $k$ up to $d$ adding terms that are zero when the index $k>i$ and thus we disentangled one sum from the other allowing the subsequent permutation. This finishes the proof.
\end{proof}

Something in this proof has already appeared in the literature while studying important results. We talk about this.

\begin{observacion}[Shifted hyperbolic polynomials]
The expansions of the polynomials appearing in this proof are directly connected with the \textit{shifted} hyperbolic polynomial appearing in \cite[Lemma 3.10]{main}, \cite[Theorem 2.2]{branden2011obstructions}, \cite[Theorem 3.7]{netzer2012polynomials}. We will greatly play with these shifts.
\end{observacion}

Note that $\hat{p}\in\mathbb{R}[\mathbf{x},y]$ does not depend on the particular determinantal representation of $p\in\mathbb{R}[\mathbf{x}]$ that we begin with and therefore the next corollary is immediate from Proposition \ref{intermedia}.

\begin{corolario}[Expansion independence of the chosen determinantal representation]
Let $p\in\mathbb{R}[\mathbf{x}]$ be a determinantal polynomial with determinantal representation $\det(I+x_{1}A_{1}+\cdots+x_{n}A_{n})$ of size $d:=\deg(p)$. Then the polynomial $\hat{p}\in\mathbb{R}[\mathbf{x},y]$ has a determinantal representation given by $\det(I+x_{1}A_{1}+\cdots+x_{n}A_{n}+yI)$. In particular, in the case of two variables, for any bivariate RZ polynomial $p\in\mathbb{R}[\mathbf{x}]$, the trivariate polynomial $\hat{p}\mathbb{R}[\mathbf{x},y]$ has a determinantal representation. 
\end{corolario}

\begin{proof}
The general part is a rewriting of Proposition \ref{intermedia} and the particular sentence is a direct consequence of applying that rewriting together with the Helton-Vinnikov theorem.
\end{proof}

As happens here, there will be problems with the homogenizing degree in the future. In this particular case, there is a problem with a possible difference between the degree of a determinantal polynomial and the size of the LMP in a given determinantal representation.

\begin{remark}[Fixing problems with homogenizing degree]
Note how we had to circumvent in our exposition of the statement of the result in Proposition \ref{intermedia} the fact that a determinantal representation of $p\in\mathbb{R}[\mathbf{x}]$ might be given by an LMP of size greater than the degree of the polynomial $p\in\mathbb{R}[\mathbf{x}]$. Not contemplating this carefully enough would have created the problem that $\hat{p}\in\mathbb{R}[\mathbf{x},y]$, as defined at the beginning of the section in Definition \ref{hatp}, is not exactly the determinant of an extension of the LMP in the given determinantal representation of $p\in\mathbb{R}[\mathbf{x}]$ but something similar due to the different homogenization that we could obtain in the course of the proof of Proposition \ref{intermedia}. We let as an easy exercise verifying this and showing how does $\hat{p}\in\mathbb{R}[\mathbf{x},y]$ relate then to the determinant of an extension of the LMP in the given determinantal representation of $p\in\mathbb{R}[\mathbf{x}]$.
\end{remark}

Sadly, we could not find more nontrivial multivariate extensions than the one presented in Proposition \ref{intermedia} (and others obtained in a similar way by considering determinantal representations of size higher than the degree $d$ of $p$, which always exist trivially whenever $p$ has a determinantal representation) for the results in \cite{nuijtype}. But, fortunately, the one in Proposition \ref{intermedia} above will be enough for our task here. For clarity, the multivariate extensions of the results in \cite{nuijtype} that we would look for are solutions to the problem described below.

\begin{problem}[Other expansions and extensions]
Suppose $p\in\mathbb{R}[\mathbf{x}]$ is an RZ polynomial and call $d=\deg(p)$. Find instances of coefficients $a_{i}$ for $i\in\{0,\dots,d\}$ such that the polynomial $\sum_{k=0}^{d}a_{k}y^{k}p^{(k)}\in\mathbb{R}[\mathbf{x},y]$ admits a determinantal representation.
\end{problem}

Anyway, continuing with the instance of this problem that we have settled in Proposition \ref{intermedia}, we proceed using \cite[Example 3.5]{main} and without caring about RZ-ness (i.e., in a just formal sense) in order to reach the chain of identities between relaxations given by $S(\hat{p})=S(\sum_{k=0}^{\infty}\frac{1}{k!}y^{k}p^{(k)})=S(\trun_{3}(\sum_{k=0}^{\infty}\frac{1}{k!}y^{k}p^{(k)}))=S(\sum_{k=0}^{3}\frac{1}{k!}y^{k}p^{(k)})$. Not even all the monomials appearing in the last expression are of importance to determine that set (as the second expression implies), but this expression guarantees that we only have to care of up to the third Renegar derivative $p^{(3)}$ in our study during this section. This proves the next proposition.

\begin{proposicion}[RZ-ness and relaxation preservation for extensions of bivariate polynomials]
\label{propoinstead}
Let $p\in\mathbb{R}[\mathbf{x}]$ be a bivariate RZ polynomial. Then we have that the polynomial $\mathbb{R}[\mathbf{x},y]\ni\hat{p}:=\sum_{k=0}^{\infty}\frac{1}{k!}y^{k}p^{(k)}$ is a trivariate RZ polynomial and $S(\hat{p})=S(\sum_{k=0}^{3}\frac{1}{k!}y^{k}p^{(k)})$.
\end{proposicion}

We are especially interested in the set $\{a\in\mathbb{R}^{2}\mid (a,0)\in S(\hat{p})\}$. Observe that a point of the form $(a,0)$ is in $S(\hat{p})$ iff $M_{\hat{p}}(a,0):=A_{0}+a_{1}A_{1}+a_{2}A_{2}\geq0$ in the notation of \cite[Definition 3.19]{main}. Observe the greatest $3\times3$ upper left submatrix of $M_{\hat{p}}(a,0)$ is, in fact, $M_{p}(a)$. Thus, $a\in S(p)$ iff such upper left submatrix verifies $M_{p}(a)\geq 0$. The monomials of $\hat{p}$ involved in the computation of $M_{\hat{p}}(a,0)$ are in fact only the monomials up to degree $3$ of the polynomial $\hat{p}$, i.e., the monomials of $\trun_{3}(\hat{p})$. However, we can go further noting that the only monomial of degree $3$ corresponding to the term $k=3$ of the sum $\sum_{k=0}^{d}\frac{1}{k!}y^{k}p^{(k)}$ defining $\hat{p}$, which is in fact $y^{3}$, does not appear in $M_{\hat{p}}(a,0)$ because it only appears in (a corner of) $A_{3}$, which is annihilated by the $0$ in the third entry (due to the restriction). Therefore, calling $\gls{overlinep}:=\sum_{k=0}^{2}\frac{1}{k!}y^{k}p^{(k)}$, we have that $M_{\hat{p}}(a,0)=M_{\overline{p}}(a,0)$. Hence, we have immediately the next fact.

\begin{hecho}[Expansion and truncation relaxation matrix identity]
\label{immediatefact}
Let $p\in\mathbb{R}[\mathbf{x}]$ be an RZ polynomial. Then $$\{a\in\mathbb{R}^{2}\mid M_{\hat{p}}(a,0)\geq0\}=\{a\in\mathbb{R}^{2}\mid M_{\overline{p}}(a,0)\geq0\}.$$
\end{hecho}

Thus we can prove easily in particular the next corollary.

\begin{corolario}[Expansion and truncation relaxation identity]
Let $p\in\mathbb{R}[\mathbf{x}]$ be an RZ polynomial. We have the identity $S_{(1,0,0)}(\hat{p}):=\{\lambda\in\mathbb{R}\mid \lambda(1,0,0)\in S(\hat{p})\}=\{\lambda\in\mathbb{R}\mid \lambda(1,0,0)\in S(\overline{p})\}=:S_{(1,0,0)}(\overline{p})$.
\end{corolario}

\begin{proof}
From the discussion proving Fact \ref{immediatefact} it is also immediate that $\{a\in\mathbb{R}^{2}\mid (a,0)\in S(\hat{p})\}=\{a\in\mathbb{R}^{2}\mid (a,0)\in S(\overline{p})\}$ because $(a,0)\in S(\hat{p})$ iff $M_{\hat{p}}(a,0)\geq 0$ if and only if $M_{\overline{p}}(a,0)\geq 0$ iff $(a,0)\in S(\overline{p})$.
\end{proof}

At this point, we have to take care of RZ-ness. Note that the polynomial $\overline{p}$ might not be RZ, but it gets very close to the RZ polynomial we are most interested in: $\gls{ptilde}:=p+y p^{(1)}$.

\begin{remark}[Loss of guarantee of RZ-ness]
The main problem for keeping RZ-ness is just our lack of control over the effects of the additional term $\frac{y^2}{2}p^{(2)}$ of $\overline{p}$.
\end{remark}

Our next step will be comparing the matrix polynomials associated to $\Tilde{p}$ and $\overline{p}$ given by $M_{\Tilde{p}}(a_{1},a_{2},0), M_{\overline{p}}(a_{1},a_{2},0)\in\mathbb{R}^{4\times4}[\mathbf{a}]$ corresponding to the relaxations of these two polynomials. For this, we will use \cite[Example 3.5]{main} in order to compute the value of the entries of these matrix polynomials.

\begin{remark}[Reduction of computations in related matrix polynomials]
For a practical reason, we will only need the values of the first and last row and column of these two matrices, which reduces significantly the quantities that we need to compute (remember that these matrices are symmetric).
\end{remark}

We introduce notation to talk independently about the coefficients of a particular monomial within a polynomial. The notation is, naturally, the expected.

\begin{notacion}[Coefficients of polynomials]
For ease of writing during the next computation, we shorten the usual names of our operators and we denote $\gls{coefmp}$ the coefficient of the monomial $m$ in the polynomial $p\in\mathbb{R}[\mathbf{x}]$ and $d:=\gls{degp}=\shortdeg(\overline{p})=\shortdeg(\hat{p})=\shortdeg(\Tilde{p})$ the degree of the polynomial $p$.
\end{notacion}

We can begin to perform our computations. We will do it in a particular order so that we can maintain clarity during the process.

\begin{computacion}[Necessary values of $L$-forms of extensions]
\label{computation1}
We divide our findings in several steps depending on the first appearance of the term that we compute at each moment for clarity. We also consider the value for each polynomial $\overline{p}$ or $\Tilde{p}$ in the subindex. Observe, in particular, that $\shortcoeff(m,\overline{p})=\shortcoeff(m,\Tilde{p})$ whenever $m$ is not divisible by $y^2$ and therefore $L_{\overline{p}}(m)=L_{\Tilde{p}}(m)$ in such cases (when $\shortdeg(m)\leq3$) as, for monomials $m$ with $\shortdeg(m)\leq3$, \cite[Example 3.5]{main} shows that to compute their corresponding $L$-forms only the coefficients of terms dividing $m$ are required. On the other hand, in the following lists a recurring phenomenon will be the fact that $L_{\overline{p}}(y^2m)\neq L_{\Tilde{p}}(y^2m)$ for each monomial $m$ with $\shortdeg(m)\leq1$ and thus these computations are performed in different lines and numbered independently at the end of each list. Thus, for the independent matrix (of the relaxation) corresponding to the polynomials $\overline{p}$ or $\Tilde{p}$ we have the list corresponding to the set $$\{1,x_{1},x_{2},y,x_{1}y,x_{2}y,y^{2}\}$$ of monomials that appear for the first time on such matrix:\begin{enumerate}
    \item \begin{gather*}L_{\Tilde{p}}(1)=L_{\overline{p}}(1)=\shortdeg(\overline{p})=d,\end{gather*}
    \item \begin{gather*}L_{\Tilde{p}}(x_{1})=L_{\overline{p}}(x_{1})=\shortcoeff(x_{1},\overline{p})=\shortcoeff(x_{1},p) \mbox{\ and, analogously,}\end{gather*}
    \item \begin{gather*}L_{\Tilde{p}}(x_{2})=L_{\overline{p}}(x_{2})=\shortcoeff(x_{1},\overline{p})=\shortcoeff(x_{2},p),\end{gather*}
    \item \begin{gather*}L_{\Tilde{p}}(y)=L_{\overline{p}}(y)=\shortcoeff(y,\overline{p})=\shortcoeff(y,yp^{(1)})=\\\shortcoeff(1,p^{(1)})=p^{(1)}(0)=\shortdeg(p)=d,\end{gather*}
    \item \begin{gather*}L_{\Tilde{p}}(x_{1}y)=L_{\overline{p}}(x_{1}y)=\shortcoeff(x_{1},\overline{p})\shortcoeff(y,\overline{p})-\shortcoeff(x_{1}y,\overline{p})=\\\shortcoeff(x_{1},p)\shortcoeff(y,yp^{(1)})-\shortcoeff(x_{1}y,yp^{(1)})\\=\shortcoeff(x_{1},p)\shortcoeff(1,p^{(1)})-\shortcoeff(x_{1},p^{(1)})=\\\shortdeg(p)\shortcoeff(x_{1},p)-(\shortdeg(p)-1)\shortcoeff(x_{1},p)=\\d\shortcoeff(x_{1},p)-(d-1)\shortcoeff(x_{1},p)=\\(d-d+1)\shortcoeff(x_{1},p)=\shortcoeff(x_{1},p) \mbox{\ and, analogously,}\end{gather*}
    \item \begin{gather*}L_{\Tilde{p}}(x_{2}y)=L_{\overline{p}}(x_{2}y)=\shortcoeff(x_{2},p),\end{gather*}
    \item \begin{gather*}L_{\overline{p}}(y^{2})=\shortcoeff(y,\overline{p})^2-2\shortcoeff(y^2,\overline{p})=\\\shortcoeff(y,yp^{(1)})^2-2\shortcoeff(y^2,\frac{y^2}{2}p^{(2)})=\\\shortcoeff(1,p^{(1)})^2-2\shortcoeff(1,\frac{1}{2}p^{(2)})=\\p^{(1)}(0)^2-2\frac{p^{(2)}(0)}{2}=\shortdeg(p)^2-2\frac{1}{2}\shortdeg(p)(\shortdeg(p)-1)=\\d^2-d(d-1)=d^{2}-d^{2}+d=d, \mbox{\ but}\end{gather*}
    \item \begin{gather*}L_{\Tilde{p}}(y^2)=\shortcoeff(y,\Tilde{p})^2-2\shortcoeff(y^2,\Tilde{p})=\shortcoeff(y,yp^{(1)})^2-2\shortcoeff(y^2,0)=\\\shortcoeff(1,p^{(1)})^2-2\shortcoeff(y^2,0)=p^{(1)}(0)^2=\shortdeg(p)^2=d^2,\end{gather*}
\end{enumerate}
for the matrix coefficient of $a_{1}$ (of the relaxation) corresponding to the polynomials $\overline{p}$ or $\Tilde{p}$ we have the list corresponding to the set $$\{x_{1}^{2},x_{1}x_{2},x_{1}^{2}y,x_{1}x_{2}y,x_{1}y^{2}\}$$ of monomials that appear for the first time on such matrix:\begin{enumerate}
\item \begin{gather*}L_{\Tilde{p}}(x_{1}^2)=L_{\overline{p}}(x_{1}^{2})=\shortcoeff(x_{1},\overline{p})^{2}-2\shortcoeff(x_{1}^2,\overline{p})=\\\shortcoeff(x_{1},p)^{2}-2\shortcoeff(x_{1}^2,p),\end{gather*}
\item \begin{gather*}L_{\Tilde{p}}(x_{1}x_{2})=L_{\overline{p}}(x_{1}x_{2})=\shortcoeff(x_{1},\overline{p})\shortcoeff(x_{2},\overline{p})-\shortcoeff(x_{1}x_{2},\overline{p})=\\\shortcoeff(x_{1},p)\shortcoeff(x_{2},p)-\shortcoeff(x_{1}x_{2},p),\end{gather*}
\item \begin{gather*}L_{\Tilde{p}}(x_{1}^{2}y)=L_{\overline{p}}(x_{1}^{2}y)=\shortcoeff(x_{1}^2 y,\overline{p})-\shortcoeff(x_{1},\overline{p})\shortcoeff(x_{1}y,\overline{p})-\\\shortcoeff(y,\overline{p})\shortcoeff(x_{1}^2,\overline{p})+\shortcoeff(x_{1},\overline{p})^2\shortcoeff(y,\overline{p})=\\\shortcoeff(x_{1}^2 y,yp^{(1)})-\shortcoeff(x_{1},p)\shortcoeff(x_{1}y,yp^{(1)})-\\\shortcoeff(y,yp^{(1)})\shortcoeff(x_{1}^2,p)+\shortcoeff(x_{1},p)^2\shortcoeff(y,yp^{(1)})=\\\shortcoeff(x_{1}^2,p^{(1)})-\shortcoeff(x_{1},p)\shortcoeff(x_{1},p^{(1)})-\\\shortcoeff(1,p^{(1)})\shortcoeff(x_{1}^2,p)+\shortcoeff(x_{1},p)^2\shortcoeff(1,p^{(1)})=\\(\shortdeg(p)-2)\shortcoeff(x_{1}^2,p)-(\shortdeg(p)-1)\shortcoeff(x_{1},p)\shortcoeff(x_{1},p)-\\\shortdeg(p)\shortcoeff(x_{1}^2,p)+\shortdeg(p)\shortcoeff(x_{1},p)^2=\\(d-2)\shortcoeff(x_{1}^2 ,p)-(d-1)\shortcoeff(x_{1},p)\shortcoeff(x_{1},p)-d\shortcoeff(x_{1}^2,p)+\\d\shortcoeff(x_{1},p)^{2}=(d-2)\shortcoeff(x_{1}^2 ,p)-(d-1)\shortcoeff(x_{1},p)^{2}\\-d\shortcoeff(x_{1}^2,p)+d\shortcoeff(x_{1},p)^{2}=\\(d-2-d)\shortcoeff(x_{1}^2,p)+(d-(d-1))\shortcoeff(x_{1},p)^{2}=\\\shortcoeff(x_{1},p)^{2}-2\shortcoeff(x_{1}^2,p),\end{gather*}
\item \begin{gather*}L_{\Tilde{p}}(x_{1}x_{2} y)=L_{\overline{p}}(x_{1}x_{2} y)=\frac{1}{2}(\shortcoeff(x_{1}x_{2} y,\overline{p})-\shortcoeff(x_{1},\overline{p})\shortcoeff(x_{2} y,\overline{p})-\\\shortcoeff(x_{2},\overline{p})\shortcoeff(x_{1} y,\overline{p})-\shortcoeff(y,\overline{p})\shortcoeff(x_{1}x_{2},\overline{p})+\\2\shortcoeff(x_{1},\overline{p})\shortcoeff(x_{2},\overline{p})\shortcoeff(y,\overline{p}))=\\\frac{1}{2}(\shortcoeff(x_{1}x_{2} y,yp^{(1)})-\shortcoeff(x_{1},p)\shortcoeff(x_{2} y,yp^{(1)})-\\\shortcoeff(x_{2},p)\shortcoeff(x_{1} y,yp^{(1)})-\shortcoeff(y,yp^{(1)})\shortcoeff(x_{1}x_{2},p)+\\2\shortcoeff(x_{1},p)\shortcoeff(x_{2},p)\shortcoeff(y,yp^{(1)}))=\\\frac{1}{2}(\shortcoeff(x_{1}x_{2},p^{(1)})-\shortcoeff(x_{1},p)\shortcoeff(x_{2} ,p^{(1)})-\\\shortcoeff(x_{2},p)\shortcoeff(x_{1} ,p^{(1)})-\shortcoeff(1,p^{(1)})\shortcoeff(x_{1}x_{2},p)+\\2\shortcoeff(x_{1},p)\shortcoeff(x_{2},p)\shortcoeff(1,p^{(1)}))=\\\frac{1}{2}((\shortdeg(p)-2)\shortcoeff(x_{1}x_{2},p)-(\shortdeg(p)-1)\shortcoeff(x_{1},p)\shortcoeff(x_{2},p)-\\(\shortdeg(p)-1)\shortcoeff(x_{2},p)\shortcoeff(x_{1} ,p)-\shortdeg(p)\shortcoeff(x_{1}x_{2},p)+\\2\shortdeg(p)\shortcoeff(x_{1},p)\shortcoeff(x_{2},p))=\\\frac{1}{2}((d-2)\shortcoeff(x_{1}x_{2},p)-(d-1)\shortcoeff(x_{1},p)\shortcoeff(x_{2},p)-\\(d-1)\shortcoeff(x_{2},p)\shortcoeff(x_{1},p)-d\shortcoeff(x_{1}x_{2},p)+\\2d\shortcoeff(x_{1},p)\shortcoeff(x_{2},p))=\\\frac{1}{2}((d-2-d)\shortcoeff(x_{1}x_{2},p)+(2d-2(d-1))\shortcoeff(x_{1},p)\shortcoeff(x_{2},p))=\\\frac{1}{2}(-2\shortcoeff(x_{1}x_{2},p)+2\shortcoeff(x_{1},p)\shortcoeff(x_{2},p))=\\\shortcoeff(x_{1},p)\shortcoeff(x_{2},p)-\shortcoeff(x_{1}x_{2},p),\end{gather*}
\item \begin{gather*}L_{\overline{p}}(x_{1}y^{2})=\shortcoeff(x_{1}y^{2},\overline{p})-\shortcoeff(y,\overline{p})\shortcoeff(x_{1}y,\overline{p})-\\\shortcoeff(x_{1},\overline{p})\shortcoeff(y^{2},\overline{p})+\shortcoeff(y,\overline{p})^{2}\shortcoeff(x_{1},\overline{p})\\=\shortcoeff(x_{1}y^{2},\frac{y^2}{2}p^{(2)})-\shortcoeff(y,yp^{(1)})\shortcoeff(x_{1}y,yp^{(1)})-\\\shortcoeff(x_{1},p)\shortcoeff(y^{2},\frac{y^2}{2}p^{(2)})+\\\shortcoeff(y,yp^{(1)})^{2}\shortcoeff(x_{1},p)=\shortcoeff(x_{1},\frac{p^{(2)}}{2})-\shortcoeff(1,p^{(1)})\shortcoeff(x_{1},p^{(1)})-\\\shortcoeff(x_{1},p)\shortcoeff(1,\frac{p^{(2)}}{2})+\shortcoeff(1,p^{(1)})^{2}\shortcoeff(x_{1},p)\\=\frac{1}{2}\shortcoeff(x_{1},p^{(2)})-\shortcoeff(1,p^{(1)})\shortcoeff(x_{1},p^{(1)})-\\\frac{1}{2}\shortcoeff(x_{1},p)\shortcoeff(1,p^{(2)})+\shortcoeff(1,p^{(1)})^{2}\shortcoeff(x_{1},p)=\\\frac{1}{2}(\shortdeg(p)-1)(\shortdeg(p)-2)\shortcoeff(x_{1},p)-\shortdeg(p)(\shortdeg(p)-1)\shortcoeff(x_{1},p)-\\\frac{1}{2}\shortcoeff(x_{1},p)\shortdeg(p)(\shortdeg(p)-1)+\shortdeg(p)^{2}\shortcoeff(x_{1},p)=\\\frac{1}{2}(d-1)(d-2)\shortcoeff(x_{1},p)-d(d-1)\shortcoeff(x_{1},p)-\\\frac{1}{2}\shortcoeff(x_{1},p)d(d-1)+d^{2}\shortcoeff(x_{1},p)=\\(\frac{1}{2}(d-1)(d-2)-d(d-1)-\frac{1}{2}d(d-1)+d^{2})\shortcoeff(x_{1},p)=\\\frac{1}{2}(d^{2}-d-2d+2-2d^{2}+2d-d^{2}+d+2d^{2})\shortcoeff(x_{1},p)=\\\frac{1}{2}2\shortcoeff(x_{1},p)=\shortcoeff(x_{1},p), \mbox{\ but}\end{gather*}
\item \begin{gather*}L_{\Tilde{p}}(x_{1}y^{2})=\shortcoeff(x_{1}y^{2},\Tilde{p})-\shortcoeff(y,\Tilde{p})\shortcoeff(x_{1}y,\Tilde{p})-\\\shortcoeff(x_{1},\Tilde{p})\shortcoeff(y^{2},\Tilde{p})+\shortcoeff(y,\Tilde{p})^{2}\shortcoeff(x_{1},\Tilde{p})=\\\shortcoeff(x_{1}y^{2},0)-\shortcoeff(y,yp^{(1)})\shortcoeff(x_{1}y,yp^{(1)})-\\\shortcoeff(x_{1},p)\shortcoeff(y^{2},0)+\shortcoeff(y,yp^{(1)})^{2}\shortcoeff(x_{1},p)=\\-\shortcoeff(1,p^{(1)})\shortcoeff(x_{1},p^{(1)})+\shortcoeff(1,p^{(1)})^{2}\shortcoeff(x_{1},p)=\\-\shortdeg(p)(\shortdeg(p)-1)\shortcoeff(x_{1},p)+\shortdeg(p)^{2}\shortcoeff(x_{1},p)=\\-d(d-1)\shortcoeff(x_{1},p)+d^{2}\shortcoeff(x_{1},p)=(-d^{2}+d+d^{2})\shortcoeff(x_{1},p)=\\d\shortcoeff(x_{1},p),\end{gather*}
\end{enumerate}
and, finally, similarly by analogy to the corresponding case above, for the matrix coefficient of $a_{2}$ (of the relaxation) corresponding to the polynomials $\overline{p}$ or $\Tilde{p}$ we have the list corresponding to the set $$\{x_{2}^{2},x_{2}^{2}y,x_{2}y^{2}\}$$ of monomials that appear for the first time on such matrix:\begin{enumerate}
\item \begin{gather*}L_{\Tilde{p}}(x_{2}^2)=L_{\overline{p}}(x_{2}^2)=\shortcoeff(x_{2},p)^{2}-2\shortcoeff(x_{2}^2,p),\end{gather*}
\item \begin{gather*}L_{\Tilde{p}}(x_{2}^2 y)=L_{\overline{p}}(x_{2}^2 y)=\\\shortcoeff(x_{2},p)^{2}-2\shortcoeff(x_{2}^2,p),\end{gather*}
\item \begin{gather*}L_{\overline{p}}(x_{2}y^{2})=\shortcoeff(x_{2},p), \mbox{\ but}\end{gather*}
\item \begin{gather*}L_{\Tilde{p}}(x_{2}y^{2})=d\shortcoeff(x_{2},p).\end{gather*}
\end{enumerate}
\end{computacion}

We give a compact form of the identities found in our computations. This compact form will evince a nice pattern between the evaluations of the corresponding $L$-forms that we are interested in (that is, evaluations of the $L$-forms $L_{\overline{p}}$ and $L_{\Tilde{p}}$ corresponding to the polynomials $\overline{p}$ and $\Tilde{p}$ over monomials up to degree $3$ except $y^{3}$).

\begin{observacion}[Compact form of the computed values]
\label{observation1}
The internal identities between the elements of the lists above show that we can compress the quantities obtained in the computations above into the shorter list: \begin{enumerate}
    \item $L_{\overline{p}}(1)=L_{\Tilde{p}}(1)=L_{\overline{p}}(y)=L_{\Tilde{p}}(y)=L_{\overline{p}}(y^{2})=d$,
    \item $L_{\overline{p}}(x_{i}y^{2})=L_{\overline{p}}(x_{i}y)=L_{\Tilde{p}}(x_{i}y)=L_{\overline{p}}(x_{i})=L_{\Tilde{p}}(x_{i})=\shortcoeff(x_{i},p)$,
    \item $L_{\overline{p}}(x_{i}^{2})=L_{\Tilde{p}}(x_{i}^{2})=L_{\overline{p}}(x_{i}^{2}y)=L_{\Tilde{p}}(x_{i}^{2}y)=\shortcoeff(x_{i},p)^{2}-2\shortcoeff(x_{i}^{2},p)$,
    \item $L_{\overline{p}}(x_{1}x_{2})=L_{\Tilde{p}}(x_{1}x_{2})=L_{\overline{p}}(x_{1}x_{2}y)=L_{\Tilde{p}}(x_{1}x_{2}y)=\shortcoeff(x_{1},p)\shortcoeff(x_{2},p)-\shortcoeff(x_{1}x_{2},p)$.
\end{enumerate} This list shows in a nutshell that, whenever the monomials being evaluated have degree at most $3$, we have the chain of identities $L_{\overline{p}}(m)=L_{\Tilde{p}}(m)=L_{\overline{p}}(my)=L_{\Tilde{p}}(my)=L_{\overline{p}}(my^{2})$ for all monomials $m$ in the set of variables $\{x_{1},x_{2}\}$ and where we do not consider evaluations of the corresponding $L$-forms over monomials of degree greater than $3$ whenever they appear. In particular, only the evaluations of the form $L_{\Tilde{p}}(my^{2})$ with $m$ a monomial in the set of variables $\{x_{1},x_{2}\}$ and $\deg(m)\leq1$ present a different behaviour characterized by the identity $L_{\Tilde{p}}(my^{2})=dL_{\overline{p}}(my^{2})$. Finally, observe that for all the monomials $m$ involved in our findings and future computations we have that $L_{\overline{p}}(m)=L_{\hat{p}}(m)$ because in the computation of $L_{\hat{p}}(m)$ only are involved the coefficients in $\hat{p}$ of monomials dividing $m$ as showed by \cite[Example 3.5]{main} all these monomials are in $\overline{p}$ with the same coefficients as in $\hat{p}$.
\end{observacion}

Observation \ref{observation1} above has a clear consequence for the structure of the matrix polynomials $M_{\Tilde{p}}(a_{1},a_{2},0),M_{\overline{p}}(a_{1},a_{2},0)\in\mathbb{R}^{4\times4}[a]$ corresponding to the relaxations of the two polynomials $\Tilde{p}$ and $\overline{p}$ when the last variable $a_{3}$ is fixed to be $0$. In fact, this shows that these matrix polynomials share many of their entries and some internal relations between the entries of each one of these matrix polynomials itself.

\begin{remark}[Overlap of matrices of the relaxation]
\label{remarkmatrix}
In particular, the identities observed above tell us that the matrix polynomials associated to $\Tilde{p}$ and $\overline{p}$ given by $M_{\Tilde{p}}(a_{1},a_{2},0),M_{\overline{p}}(a_{1},a_{2},0)\in\mathbb{R}^{4\times4}[\mathbf{a}]$ share all their entries except the one on the bottom right corner. Moreover, $M_{\overline{p}}(a_{1},a_{2},0)$ has equal first and last row (and then first and last column by symmetry).
\end{remark}

We will see better what we said in Remark \ref{remarkmatrix} above expanding our matrices. Using our computations, we are ready to see the matrix polynomials $M_{\Tilde{p}}(a_{1},a_{2},0),M_{\overline{p}}(a_{1},a_{2},0)\in\mathbb{R}^{4\times4}[\mathbf{a}]$ corresponding to the relaxations of the two polynomials $\Tilde{p}$ and $\overline{p}$ when the last variable $a_{3}$ is fixed to be $0$. We already said in Remark \ref{remarkmatrix} above that these matrices have a very nice (and to a great extent shared) structure that will help us to get the result that we need.

\setlength{\emergencystretch}{3em}%
\begin{computacion}[Matrices of the relaxation of the expansions]
\label{estructurematrices}
As we said, we will see that we do not need to write these matrices fully but just the first and last rows and columns, i.e., the entries in the geometric border of such matrices. Note, moreover, that the computations above guarantee that the upper left largest $3\times3$ submatrices of $M_{\Tilde{p}}(a_{1},a_{2},0)$ and $M_{\overline{p}}(a_{1},a_{2},0)$ are equal because monomials of the form $y^{2}m$ only appear in the lower right smallest ($1\times1$) submatrices of $M_{\Tilde{p}}(a_{1},a_{2},0)$ and $M_{\overline{p}}(a_{1},a_{2},0)$ and the corresponding $L$-forms for $\overline{p}$ or $\Tilde{p}$ are different only for such monomials, as we saw in the mentioned computations. Therefore, in particular, the entries not in the corresponding geometric border (the $2\times2$ matrices at the corresponding core) coincide for both matrices $M_{\Tilde{p}}(a_{1},a_{2},0),M_{\overline{p}}(a_{1},a_{2},0)$ and that symmetric submatrix at the ``core" of these matrices will be simply called $U=(u_{ij})_{i,j=2}^{3}\in\mathbb{R}^{2\times2}[a]$. We will concentrate the expression of $U$ directly in the first (independent) term of each expansion of the corresponding matrix polynomials $M_{\Tilde{p}}(a_{1},a_{2},0)$ or $M_{\overline{p}}(a_{1},a_{2},0)$ for simplicity. Now we proceed with such expansions. First, $M_{\overline{p}}(a_{1},a_{2},0)=$\begin{gather*}
    \begin{pmatrix}
    d & C(x_{1},p) & C(x_{2},p) & d\\
    C(x_{1},p) & u_{22} & u_{23} & C(x_{1},p)\\
    C(x_{2},p) & u_{23} & u_{33} & C(x_{2},p)\\
    d & C(x_{1},p) & C(x_{2},p) & d
    \end{pmatrix}+\\\tiny{a_{1}\begin{pmatrix}
    C(x_{1},p) & C(x_{1},p)^{2}-2C(x_{1}^{2},p) & C(x_{1},p)C(x_{2},p)-C(x_{1}x_{2},p) & C(x_{1},p)\\
    C(x_{1},p)^{2}-2C(x_{1}^{2},p) & 0 & 0 & C(x_{1},p)^{2}-2C(x_{1}^{2},p)\\
   C(x_{1},p)C(x_{2},p)-C(x_{1}x_{2},p) & 0 & 0 & C(x_{1},p)C(x_{2},p)-C(x_{1}x_{2},p)\\
    C(x_{1},p) & C(x_{1},p)^{2}-2C(x_{1}^{2},p) & C(x_{1},p)C(x_{2},p)-C(x_{1}x_{2},p) & C(x_{1},p)\end{pmatrix}}+\\\tiny{a_{2}\begin{pmatrix}
   C(x_{2},p) & C(x_{1},p)C(x_{2},p)-C(x_{1}x_{2},p) & C(x_{2},p)^{2}-2C(x_{2}^{2},p) & C(x_{2},p)\\
    C(x_{1},p)C(x_{2},p)-C(x_{1}x_{2},p) & 0 & 0 & C(x_{1},p)C(x_{2},p)-C(x_{1}x_{2},p)\\
    C(x_{2},p)^{2}-2C(x_{2}^{2},p) & 0 & 0 & C(x_{2},p)^{2}-2C(x_{2}^{2},p)\\
    C(x_{2},p) & C(x_{1},p)C(x_{2},p)-C(x_{1}x_{2},p) & C(x_{2},p)^{2}-2C(x_{2}^{2},p) & C(x_{2},p)
    \end{pmatrix}}.
    \end{gather*} Observe that, similarly, we have also $M_{\Tilde{p}}(a_{1},a_{2},0)=$\begin{gather*}
    \begin{pmatrix}
    d & C(x_{1},p) & C(x_{2},p) & d\\
    C(x_{1},p) & u_{22} & u_{23} & C(x_{1},p)\\
    C(x_{2},p) & u_{23} & u_{33} & C(x_{2},p)\\
    d & C(x_{1},p) & C(x_{2},p) & d^{2}
    \end{pmatrix}+\\\tiny{a_{1}\begin{pmatrix}
    C(x_{1},p) & C(x_{1},p)^{2}-2C(x_{1}^{2},p) & C(x_{1},p)C(x_{2},p)-C(x_{1}x_{2},p) & C(x_{1},p)\\
    C(x_{1},p)^{2}-2C(x_{1}^{2},p) & 0 & 0 & C(x_{1},p)^{2}-2C(x_{1}^{2},p)\\
   C(x_{1},p)C(x_{2},p)-C(x_{1}x_{2},p) & 0 & 0 & C(x_{1},p)C(x_{2},p)-C(x_{1}x_{2},p)\\
    C(x_{1},p) & C(x_{1},p)^{2}-2C(x_{1}^{2},p) & C(x_{1},p)C(x_{2},p)-C(x_{1}x_{2},p) & dC(x_{1},p)\end{pmatrix}}+\\\tiny{a_{2}\begin{pmatrix}
   C(x_{2},p) & C(x_{1},p)C(x_{2},p)-C(x_{1}x_{2},p) & C(x_{2},p)^{2}-2C(x_{2}^{2},p) & C(x_{2},p)\\
    C(x_{1},p)C(x_{2},p)-C(x_{1}x_{2},p) & 0 & 0 & C(x_{1},p)C(x_{2},p)-C(x_{1}x_{2},p)\\
    C(x_{2},p)^{2}-2C(x_{2}^{2},p) & 0 & 0 & C(x_{2},p)^{2}-2C(x_{2}^{2},p)\\
    C(x_{2},p) & C(x_{1},p)C(x_{2},p)-C(x_{1}x_{2},p) & C(x_{2},p)^{2}-2C(x_{2}^{2},p) & dC(x_{2},p)
    \end{pmatrix}}.
\end{gather*}
\end{computacion}
\setlength{\emergencystretch}{0em}%

\begin{notacion}
For a matrix $A$, we call $A_{ij}$ the submatrix (not the cofactor) obtained from $A$ after deleting the row $i$ and the column $j$. For ease of writing, we denote also $|N|=\det(N)$ the determinant of the square matrix $N$.
\end{notacion}

Therefore, we denote the cofactor corresponding to the entry $(i,j)$ by $|A_{ij}|$. As happened in the case of Lemma \ref{lemmafactor}, we can try to preview our next step using determinants. This shows us what to expect geometrically. This preview will show that the line defined by the top left entry of our matrices plays a special role, but at the same time it will not help because of the triviality of the determinant of the first matrix polynomial. We see this in more detail.

\begin{warning}[Determinantal preview of the geometry]
\label{bigproblemwithdeterminant}
Fix $a\in\mathbb{R}^{2}$ and, for ease of writing, call $M:=M_{\Tilde{p}}(a_{1},a_{2},0)$. Thus, we call the top left $3\times3$ shared submatrix of both matrix polynomials $M_{44}$. Now we develop the determinant of $M$ by the last column so we obtain $\det(M)=$\begin{gather*}
    (-1)^{1+4}(d+a_{1}\shortcoeff(x_{1},p)+a_{2}\shortcoeff(x_{2},p))|M_{14}|+(-1)^{2+4}l|M_{24}|+\\(-1)^{3+4}m|M_{34}|+(-1)^{4+4}d(d+a_{1}\shortcoeff(x_{1},p)+a_{2}\shortcoeff(x_{2},p))|M_{44}|=\\(-1)^{1+4}(d+a_{1}\shortcoeff(x_{1},p)+a_{2}\shortcoeff(x_{2},p))(-1)^{2}|M_{44}|+0+\\0+(-1)^{4+4}d(d+a_{1}\shortcoeff(x_{1},p)+a_{2}\shortcoeff(x_{2},p))|M_{44}|=\\d(d+a_{1}\shortcoeff(x_{1},p)+a_{2}\shortcoeff(x_{2},p))|M_{44}|-\\(d+a_{1}\shortcoeff(x_{1},p)+a_{2}\shortcoeff(x_{2},p))|M_{44}|=\\(d-1)(d+a_{1}\shortcoeff(x_{1},p)+a_{2}\shortcoeff(x_{2},p))|M_{44}|,
\end{gather*} where for clarity we preferred not to expand the entries $l,m$ corresponding to their obvious cofactors because they have to disappear anyway (because of repetition of rows $|M_{24}|=0=|M_{34}|$) and we performed several permutations of rows to allow for the appearance of the uniform cofactor $|M_{44}|$ showing that the matrix $M_{14}$ is matrix $M_{44}$ after we permute first and last rows first and then second and third rows; that is why we wrote $(-1)^{2}$ as a factor whose exponent accounts for the two permutations performed. However the first matrix polynomial $M_{\overline{p}}(a_{1},a_{2},0)$ has determinant $\det(M_{\overline{p}}(a_{1},a_{2},0))=0$ because the last and first row are equal. Thus the analysis of the corresponding determinants does not give us a clear picture of what is happening (even when the polynomial $p$ is chosed so that $\Tilde{p}$ produces a relaxation whose matrix at the origin is PD) but points out clearly to the line $(d+a_{1}\shortcoeff(x_{1},p)+a_{2}\shortcoeff(x_{2},p))$ (corresponding to the top left entry of our matrix polynomials) that appears as a factor of the determinant $\det(M_{\Tilde{p}}(a_{1},a_{2},0))$ of the matrix polynomial $M_{\Tilde{p}}(a_{1},a_{2},0)$ defining the relaxation of $\Tilde{p}$ when the third coordinate is set equal to $0$.
\end{warning}

Hence, contrary to the partial proof depicted in Warning \ref{warningfactor} before we effectively proved fully Lemma \ref{lemmafactor} without using determinants, the analysis of the determinants here is not helpful at all (it does not even give a partial proof) because one of our matrix polynomials will always and trivially have determinant $0$, i.e., $\det(M_{\overline{p}})=0$. Therefore we need alternative ways of dealing with our problem of relating the sets where these two matrix polynomials are PSD. That alternative way will fundamentally use the structure of these matrices together with Sylvester's criterion through the next lemma. We recall Sylvester's criterion for completeness.

\begin{hecho}[Sylvester's criterion]
\label{Sylves}
Let $A$ be a Hermitian matrix. $A$ is PD iff all its leading principal minors are positive. $A$ is PSD iff all its principal minors are nonnegative.
\end{hecho}

Lamentably, we work in the second case. Here, more minors have to be considered.

\begin{remark}[Many principal minors]
Note the many more principal minors needed to be considered in the last case of the Fact \ref{Sylves} above, not just the leading ones. In particular, we will find ourselves in that last and less compact case because we will work with the PSD-ness of our matrices.
\end{remark}

Now we can use Sylvester's criterion in order to prove a lemma about matrices having the form of those found in the Computation \ref{estructurematrices} above. This lemma will overcome the difficulties found in Warning \ref{bigproblemwithdeterminant}.

\begin{lema}[PDS equivalence under big enough scaling of a corner]
\label{lemmaclave}
Fix $k\geq1$. Suppose that we have two symmetric matrices in $\Sym_{4}(\mathbb{R})$ of the form \begin{gather*}
    A:=\begin{pmatrix}
    a_{11} & a_{12} & a_{13} & a_{11}\\
    a_{12} & a_{22} & a_{23} & a_{12}\\
    a_{13} & a_{23} & a_{33} & a_{13}\\
    a_{11} & a_{12} & a_{13} & a_{11}\\
    \end{pmatrix},  B:=\begin{pmatrix}
    a_{11} & a_{12} & a_{13} & a_{11}\\
    a_{12} & a_{22} & a_{23} & a_{12}\\
    a_{13} & a_{23} & a_{33} & a_{13}\\
    a_{11} & a_{12} & a_{13} & ka_{11}\\
    \end{pmatrix}.
\end{gather*} Then $A$ is PSD iff $B$ is PSD.
\end{lema}

\begin{proof}
$\Rightarrow.$ Suppose that $A$ is PSD. Then, using Sylvester's criterion, we know in particular that the $1\times1$ principal minor on the top left corner of $A$ is nonnegative, i.e., $a_{11}\geq0$. Therefore $B$ is obtained from the PSD matrix $A$ by increasing a nonnegative term of the diagonal of such PSD matrix $A$ by a factor $k\geq1$, which immediately implies that $B$ is also PSD.

$\Leftarrow.$ Suppose that $B$ is PSD. Then, using Sylvester's criterion, we know in particular that all the principal minors involving only the $3\times3$ top left submatrix of $B$ are nonnegative. This submatrix is shared with $A$ and thus we can cover already a big portion of the principal minors of $A$ guaranteeing that these are nonnegative. We do this because we want to use the other direction of the Sylvester's criterion in order to prove that, under the assumption that $B$ is PSD, $A$ is also PSD. For this, we already covered many principal $3\times 3$ minors of $A$ but we still have to analyze (guaranteeing nonnegativity) all the principal minors of $A$ involving the last row and column of $A$. There is $\binom{3}{0}=1$ additional principal $1\times1$ minor of $A$ involving the last row and column, which is the minor given by the bottom right corner $a_{11}$, but this is nonnegative because it coincides with the already covered (shared with $B$) nonnegative top left minor, which is also $a_{11}$. There are additionally $\binom{3}{1}=3$ principal $2\times2$ minors involving the last row and column of $A$ given by \begin{gather*}
\begin{vmatrix}
a_{11} & a_{11}\\
a_{11} & a_{11}
\end{vmatrix}=0,\\ \begin{vmatrix}
a_{22} & a_{12}\\
a_{12} & a_{11}
\end{vmatrix}=-\begin{vmatrix}
a_{12} & a_{11}\\
a_{22} & a_{12}
\end{vmatrix}=\begin{vmatrix}
a_{11} & a_{12}\\
a_{12} & a_{22}
\end{vmatrix},\\ \begin{vmatrix}
a_{33} & a_{13}\\
a_{13} & a_{11}
\end{vmatrix}=-\begin{vmatrix}
a_{13} & a_{11}\\
a_{33} & a_{13}
\end{vmatrix}=\begin{vmatrix}
a_{11} & a_{13}\\
a_{13} & a_{33}
\end{vmatrix}
\end{gather*} that, as we see, can be transformed into already covered and then nonnegative principal $2\times2$ minors via elementary transformations (in particular, we just need $2$-chains of elementary transformations given by a row permutation followed by its analogous column permuation and thus we reach an already known principal $2\times2$ minor of $B$ hence nonnegative by hypothesis). Similarly, there are additionally $\binom{3}{2}=3$ principal $3\times3$ minors involving the last row and column of $A$ given by \begin{gather*}
\begin{vmatrix}
a_{11} & a_{12} & a_{11}\\
a_{12} & a_{22} & a_{12}\\
a_{11} & a_{12} & a_{11}
\end{vmatrix}=0,\\\begin{vmatrix}
a_{11} & a_{13} & a_{11}\\
a_{13} & a_{33} & a_{13}\\
a_{11} & a_{13} & a_{11}
\end{vmatrix}=0,\\
\begin{vmatrix}
a_{22} & a_{23} & a_{12}\\
a_{23} & a_{33} & a_{13}\\
a_{12} & a_{13} & a_{11}
\end{vmatrix}=-\begin{vmatrix}
a_{12} & a_{13} & a_{11}\\
a_{23} & a_{33} & a_{13}\\
a_{22} & a_{23} & a_{12}
\end{vmatrix}=\\\begin{vmatrix}
a_{11} & a_{13} & a_{12}\\
a_{13} & a_{33} & a_{23}\\
a_{12} & a_{23} & a_{22}
\end{vmatrix}=-\begin{vmatrix}
a_{11} & a_{13} & a_{12}\\
a_{12} & a_{23} & a_{22}\\
a_{13} & a_{33} & a_{23}
\end{vmatrix}=\begin{vmatrix}
a_{11} & a_{12} & a_{13}\\
a_{12} & a_{22} & a_{23}\\
a_{13} & a_{23} & a_{33}
\end{vmatrix}
\end{gather*} that, as we see, can be transformed into already covered and then nonnegative principal $3\times3$ minors via elementary transformations (in particular, we just need $4$-chains of elementary transformations given by row permutations intercalated with its analogous column permutation until we reach an already known principal $3\times3$ minor of $B$ hence nonnegative by hypothesis). Finally, there lacks only $\binom{3}{3}=1$ principal $4\times4$ minor of $A$ involving necessarily the last row and column, which is in fact the (principal) minor given by $\det(A)$, but this is nonnegative because trivially $\det(A)=0$. Now we only need to use one direction of Sylvester's criterion to prove that $A$ is PSD, which establishes this implication and thus finishes the proof of this lemma.
\end{proof}

Using Lemma \ref{lemmaclave} above we can finally prove the identity between the subsets of $\mathbb{R}^{2}$ where each of these two matrix polynomials are PSD. In particular, remember that Computation \ref{estructurematrices} showed us that the structure of these matrix polynomials (at each $a\in\mathbb{R}^{2}$) coincides precisely with the structure of the symmetric matrices related in the mentioned Lemma \ref{lemmaclave}.

\begin{corolario}[Identity of restrictions]
\label{goodidentity}
We have $\{a\in\mathbb{R}^{2}\mid M_{\hat{p}}(a,0)\geq0\}=\{a\in\mathbb{R}^{2}\mid M_{\Tilde{p}}(a,0)\geq0\}$.
\end{corolario}

\begin{proof}
Fix $a\in\mathbb{R}^{2}$ arbitrary. Then Lemma \ref{lemmaclave} says that $M_{\hat{p}}=M_{\overline{p}}(a,0)$ is PSD iff $M_{\Tilde{p}}(a,0)$ is PSD. Therefore, the identity follows immediately. 
\end{proof}

Now that we know that the sets in $\mathbb{R}^{2}$ characterized by the PSD-ness of these two matrices are the same, we can proceed analogously as in Section \ref{section1}. But, instead of doing it with $\Tilde{p}$, which in the current bivariate scenario has not a nice known determinantal representation derived from a Helton-Vinnikov determinantal representation of $p$ contrary to what happened in Section \ref{section1} for univariate polynomials, we will work with $\hat{p}$, which has, as we already saw, a nice determinantal representation derived from the determinantal representation of $p$ in the bivariate case.

\begin{remark}[Proceeding under unavailability of determinantal representations]
We can precisely do this because Corollary \ref{goodidentity} guarantees us that the sets defined by the relaxations of $\hat{p}$ and $\Tilde{p}$ are the same. Thus, we proceed with the proof of the analogous to Proposition \ref{analogyprop} of Section \ref{section1}.
\end{remark}

The proof of the new result has indeed some advantages with respect to its previous analogue. In particular, some of the matrices involved are simpler to understand.

\begin{observacion}[Advantages of the next proof]
The proof of this proposition (which is analogous to Proposition \ref{analogyprop}) is very similar to the proof of the analogous one, and even simpler because now we do not have to deal with the all ones matrix $E$ but only with the identity matrix.
\end{observacion}

We can now then proceed with the result and its proof. Recall first that we have the determinantal representations $p=\det(I+x_{1}D+x_{2}A)$ and now $r:=\hat{p}=\sum_{k=0}^{\infty}\frac{1}{k!}y^{k}p^{(k)}=\det(I+x_{1}D+x_{2}A+yI).$

\begin{proposicion}[Relaxation commutes with restriction for extension]
\label{propRenegar}
Let $p\in\mathbb{R}[\mathbf{x}]$ be a bivariate RZ polynomial and denote $r:=\sum_{k=0}^{\infty}\frac{1}{k!}y^{k}p^{(k)}\in\mathbb{R}[x,y]$. Then we have the identity $$S(p)=\{a\in\mathbb{R}^{2}\mid (a,0)\in S(r)\}.$$
\end{proposicion}

\begin{proof}
Using \cite[Proposition 3.33]{main}, we write $S(p)=\{a\in\mathbb{R}^{2}\mid\forall M\in U_{p}:\tr(M^{2}(I+a_{1}D+a_{2}A))\geq0\}$ and $S(r)=\{(a,b)\in\mathbb{R}^{3}\mid\forall M\in U_{r}:\tr(M^{2}(I+a_{1}D+a_{2}A+bI))\geq0\}$ with $U_{p}:=\{v_{0}I+v_{1}D+v_{2}A\mid (v_{0},v_{1},v_{2})\in\mathbb{R}^{3}\}= U_{r}:=\{v_{0}I+v_{1}D+v_{2}A+v_{3}I\mid (v_{0},v_{1},v_{2},v_{3})\in\mathbb{R}^{4}\}$. Note that now $U_{r}=U_{p}$, contrary to the case in the analogous Proposition \ref{analogyprop}. These observations lead to direct proof of the identity now because it means that $S(r)=\{(a,b)\in\mathbb{R}^{3}\mid\forall M\in U_{p}:\tr(M^{2}(I+a_{1}D+a_{2}A+bI))\geq0\}$ and therefore our set of interest $\{a\in\mathbb{R}^{2}\mid (a,0)\in S(r)\}=\{a\in\mathbb{R}^{2}\mid \forall M\in U_{p}:\tr(M^{2}(I+a_{1}D+a_{2}A))\geq0\}=:S(p),$ which finishes the proof.
\end{proof}

The proof of this result was happily much more straightforward than the proof of Proposition \ref{analogyprop}. Note that the argument using the trace also works here but it is less direct, clear and elegant than the proof we just gave.

\begin{remark}[Another path in previous proof]
Clearly, we could have proceeded with the proof analogously because the inclusion $\supseteq$ follows from \cite[Lemma 3.26 (c)]{main} in the same way. Similarly, for $\subseteq$, fix $a\in S(p)$, then $\tr(M^{2}(I+a_{1}D+a_{2}A))\geq0$ for all $M\in U_{p}=U_{r}$ (this equality holds in this new case). We see now immediately that $(a,0)\in S(r)$ because for all $M\in U_{r}=U_{p}$ we have then $\tr(M^{2}(I+a_{1}D+a_{2}A))\geq0$ because $a\in S(p)$ and $U_{p}=U_{r}$. This establishes the last inclusion and finishes our proof.
\end{remark}

Combining our last two results Corollary \ref{goodidentity} and Proposition \ref{propRenegar} we obtain immediately the next corollary about $\Tilde{p}$.

\begin{corolario}[Relaxation commutes with restriction for truncated extension]
\label{tildein}
Let $p\in\mathbb{R}[x]$ be a bivariate RZ polynomial and denote now $s:=p+yp^{(1)}\in\mathbb{R}[x,y]$. Then we have the identity $$S(p)=\{a\in\mathbb{R}^{2}\mid (a,0)\in S(s)\}.$$
\end{corolario}

\begin{proof}
Denote as before $r:=\sum_{k=0}^{\infty}\frac{1}{k!}y^{k}p^{(k)}\in\mathbb{R}[x,y]$. Proposition \ref{propRenegar} implies that $S(p)=\{a\in\mathbb{R}^{2}\mid (a,0)\in S(r)\}$, but $S(r):=\{(a,b)\in\mathbb{R}^{3}\mid M_{r}(a,b)\geq0\}$ so $(a,0)\in S(r)$ iff $M_{r}(a,0)\geq0$ iff $M_{s}(a,0)\geq0$ iff $(a,0)\in S(s)$, where the last two equivalences come from Corollary \ref{goodidentity} and the definition of the relaxation respectively. Therefore, we can change $r$ by $s$ in the initial identity coming from Proposition \ref{propRenegar} to get the one that we want now $S(p)=\{a\in\mathbb{R}^{2}\mid (a,0)\in S(s)\}$, which finishes this proof.
\end{proof}

Finally, we collect a few ideas about the connections between the results we obtained in this section. This will clarify our next steps in this chapter. 

\begin{remark}[Summary of the results in this section]
We were able to prove that for bivariate polynomials $S(p)=\{a\in\mathbb{R}^{2}\mid (a,0)\in S(\Tilde{p})\}$. The next step is combining Corollary \ref{tildein}, which is valid only for bivariate polynomials $p$, with the very last result Corollary \ref{corolariofinal} in Section \ref{sectioninter} above, which allows us to change perspective from a multivariate approach to a bivariate one via an infinite intersection.
\end{remark}

In the sense of the remark above, Corollaries \ref{corolariofinal} and \ref{tildein} fit together perfectly and let us deduce that the operation of adding to an RZ polynomial $p$ its first Renegar derivative multiplied by a new constant (producing the RZ polynomial $\Tilde{p}$) does not help us to improve the relaxation defined in \cite[Definition 3.19]{main}. We dedicate the next section to prove this. We finish leaving an easy exercise that can help us to understand how the relaxation does not change through some easy transformations.

\begin{ejercicio}[Trivial variations of the extension]
A good ending exercise is to see that everything we did on this section can be directly extended to trivial variations of the extension of the results in \cite{nuijtype} that we found at the beginning. That is, we can define $\hat{p}_{s}:=\hat{p}|_{y=sy}$ and similarly for $\overline{p}_{s}$ and $\Tilde{p}_{s}$ for any $0\neq s\in\mathbb{R}$ and see that the relaxations defined by these are related in the same way as we saw here for the particular case $s=1$ and that, in fact, they do not change when $s$ does.
\end{ejercicio}

Now we have finally all the tools we need to prove in the next section that the operation we are studying $p\mapsto p+yp^{(1)}$ lets the restriction of the relaxation unchanged. Thus we will see that using the interlacer $p^{(1)}$ obtained through the Renegar derivative together with a Nuij type transformation preserving RZ-ness while adding an additional variable does not help us. The fundamental importance of this result for this thesis is that it will naturally lead us towards the task of finding out what happens for similar Nuij type transformation using more involved interlacers that actually contain more information about the polynomials or the families these polynomials are embedded in. That is, this result about the invariance of the relaxation for this particular transformation will automatically open the door for us to consider similar but less direct transformations that will actually produce a change in the respective restriction of the relaxation, which is the topic of future parts of this dissertation.

\section{Renegar derivative does not help}

In this section, we put together the two final results from the two previous sections in order to prove that the relaxation does not improve considering $\Tilde{p}$ instead of $p$. Thanks to these results, the proof of this fact will be almost immediate. However, some notation is needed before continuing.

\begin{notacion}[Bivariate restriction determinantal representation]\label{notacionres}
Let $p\in\mathbb{R}[x]$ be an RZ polynomial. For each pair of vectors $a,b\in\mathbb{R}^{n}$ we can form the bivariate restriction of the RZ polynomial $p$ to the subspace $\linspan{a,b}$ that we denoted already $p_{a,b}$. As $p_{a,b}$ is clearly a bivariate RZ polynomial, we can apply Helton-Vinnikov to find a pair of matrices such that $$p_{a,b}(x_{1},x_{2})=\det(I+x_{1}D_{a,b}+x_{2}A_{a,b})$$ with $D_{a,b}$ diagonal and $A_{a,b}$ symmetric.
\end{notacion}

Here the problem with the degree dropping might appear again and therefore we have to circumvent it order to fit all our results together correctly. There is an advantage here with respect to what we had to do before when we talk about the relaxation or the Renegar derivative of a restriction: the solution in this case will be much more natural than what we had to in Remark \ref{remarkdegree} because any polynomial with a determinantal representation of size $s\in\mathbb{N}$ admits determinantal representations of size higher than $s$ in a trivial way constructed just by enlarging the coefficient matrices with zero entries. We devote the next remark to fix this helpful convention.

\begin{remark}[Determinantal degree drop]\label{remarksize}
Let $p\in\mathbb{R}[x]$ be an RZ polynomial and $a,b\in\mathbb{R}^{n}$ a pair of vectors. Independently of the degree of $p_{a,b}$, we will take the matrices $A_{a,b}$ and $D_{a,b}$ introduced in Notation \ref{notacionres} above of size $d:=\deg(p)$ so the representation has always size $d$ (even when we have to artificially increase the size of these matrices by adding zero entries).
\end{remark}

This remark therefore forces us to make the next assumptions about our choices in order to work smoothly through our proofs. This will greatly simplify our work without loss of generality.

\begin{convencion}[Fixing degree]
Let $p\in\mathbb{R}[x]$ be an RZ polynomial. We fix a choice of these matrices verifying the requirement in Remark \ref{remarksize} for each pair $a,b\in\mathbb{R}^{n}$. Moreover, we remember how we defined in the previous section the relevant operators whose bidimensional or bivariate particular cases we will also use here.
\end{convencion}

Now we look in particular to the bivariate case of the main operator studied in sections above.

\begin{particularizacion}[Hat operators in bivariate case]
Let $p\in\mathbb{R}[x]$ be a polynomial. Remember that we defined in Section \ref{sectionL} above the operators between RZ polynomials $$\hat{p_{a,b}}(x_{1},x_{2}):=\sum_{k=0}^{\infty}\frac{1}{k!}y^{k}(p_{a,b})^{(k)}$$ and $$\Tilde{p_{a,b}}:=\sum_{k=0}^{1}\frac{1}{k!}y^{k}(p_{a,b})^{(k)}=p_{a,b}+y^{k}(p_{a,b})^{(1)}$$ for each choice of a pair of vectors $a,b\in\mathbb{R}^{n}$. Remember also the convention fixed in Remark \ref{remarksize} with respect to the application to restrictions of polynomials of operators whose definition requires making use of the degree of the argument polynomial for computing its image.
\end{particularizacion}

Thanks to the care that we took for restrictions $p_{A}$ of polynomials $p$ to the subspace $\linspan(A)$ generated by a tuple of vectors $A\in(\mathbb{R}^{n})^{l}$ at the time of taking Renegar derivatives (or computing, in general, any operator between polynomials involving degree, e.g., $\hat{\cdot},\Tilde{\cdot},\overline{\cdot},\cdot^{(i)}$), the relaxation operator $S_{\cdot}$ (which also involves the degree) and the size of the representations given by Helton-Vinnikov in the bivariate restrictions of our RZ polynomials in this section, we reached a point where all our conventions fit correctly together and let us advance even over the cases where the degrees of the restriction polynomial $p_{A}$ differ from the degree of the original RZ polynomial $p$ we began with.

\begin{remark}[Conventions fit]
As we took the care in Remark \ref{remarksize} above of fixing all the representations of the restrictions $p_{a,b}$ of $p$ for any pair of vectors $a,b\in\mathbb{R}^{n}$ of size $d:=\deg(p)$ and convene on Remark \ref{remarkdegree} on computing $\hat{p_{a,b}}$ as if it had degree $d$, the same proof as the one in Proposition \ref{intermedia} guarantees (taking care of the degree convention on restrictions) now the identity $$\hat{p_{a,b}}(x_{1},x_{2})=\det(I+x_{1}D_{a,b}+x_{2}A_{a,b}+y I).$$
\end{remark}

Along the same lines, we need to prove that the notations above behave well in another respect. In particular, we will now show that the operators between polynomials involving the Renegar derivative represented by $\cdot^{(i)}$, $\hat{\cdot}$ and $\Tilde{\cdot}$ commute with the operator $\cdot_{A}$ representing the restriction of the polynomial to a subspace $\linspan(A)$ spanned by the vectors in a tuple $A\in(\mathbb{R}^{n})^{m}$ for some $m\in\mathbb{N}$. Note that these are all maps that respect RZ polynomials.

\begin{hecho}[Operators are linear]
\label{facteasy}
The maps mentioned above $\Tilde{\cdot}$ and $\hat{\cdot}$ are linear in the vector space of polynomials $\mathbb{R}[x]$. The same happens with $\cdot_{A}$.
\end{hecho}

\begin{proof}
This follows immediately from the definitions and the straightforward fact that $\cdot^{(i)}$ and $\cdot_{A}$ are linear.
\end{proof}

We introduce an easy lemma that will simplify the proof of the result that we want. At this point, we have to remember again Remark \ref{remarkdegree}, where we explained how we deal with operators involving the use of the degree for their definitions (e.g., the Renegar derivative) when these are applied to restrictions $p_{A}$ of a polynomial $p$ to a subspace $\linspan(A)\subseteq\mathbb{R}^{n}$. In particular, we apply the definition of these operators as if $p_{A}$ had always degree $d:=\deg(p)$ (which is not always the case, as we mentioned several times already).

\begin{lema}[Subspace restriction and Renegar derivative commute]
\label{lemaRene}
Let $p\in\mathbb{R}[\mathbf{x}]$ be a polynomial and $A\in(\mathbb{R}^{n})^{l}$ for some $l\in\mathbb{N}$ a finite tuple of vectors. Then $(p^{(k)})_{A}=(p_{A})^{(k)}$.
\end{lema}

\begin{proof}
Write the polynomial $p:=\sum_{\alpha\in\mathbb{N}^{n}}a_{\alpha}x^{\alpha}\in\mathbb{R}[x]$ with $d=\deg(p)$. Remember that the $k$-th Renegar derivative of the polynomial $p$ is $p^{(k)}(x):=$\begin{gather*}
\frac{\partial^{k}}{\partial x_{0}^{k}}(x_{0}^{d}p(\frac{x}{x_{0}}))|_{x_{0}=1}=\frac{\partial^{k}}{\partial x_{0}^{k}}(x_{0}^{d}\sum_{\alpha\in\mathbb{N}^{n}}a_{\alpha}(\frac{x}{x_{0}})^{\alpha})|_{x_{0}=1}=\\\frac{\partial^{k}}{\partial x_{0}^{k}}(\sum_{\alpha\in\mathbb{N}^{n}}a_{\alpha}x_{0}^{d-|\alpha|}x^{\alpha})|_{x_{0}=1}=\sum_{\alpha\in\mathbb{N}^{n}}a_{\alpha}(\frac{\partial^{k}}{\partial x_{0}^{k}}x_{0}^{d-|\alpha|})|_{x_{0}=1}x^{\alpha}=\\\sum_{\alpha\in\mathbb{N}^{n}}a_{\alpha}\frac{(d-|\alpha|)!}{((d-|\alpha|)-k)!}x^{\alpha},\end{gather*} where we set $\frac{(d-|\alpha|)!}{((d-|\alpha|)-k)!}=0$ if $k>d-|\alpha|$. So, taking into account the convention adopted in Remark \ref{remarkdegree}, we obtain the chain of identities $(p_{A})^{(k)}(y)=\frac{\partial^{k}}{\partial x_{0}^{k}}(x_{0}^{d}p_{A}(\frac{y}{x_{0}}))|_{x_{0}=1}=$\begin{gather*}
\frac{\partial^{k}}{\partial x_{0}^{k}}(x_{0}^{d}p(\frac{\sum_{i=1}^{l}y_{i}a_{i}}{x_{0}}))|_{x_{0}=1}=\frac{\partial^{k}}{\partial x_{0}^{k}}(x_{0}^{d}\sum_{\alpha\in\mathbb{N}^{n}}a_{\alpha}(\frac{\sum_{i=1}^{l}y_{i}a_{i}}{x_{0}})^{\alpha})|_{x_{0}=1}=\\\frac{\partial^{k}}{\partial x_{0}^{k}}(\sum_{\alpha\in\mathbb{N}^{n}}a_{\alpha}x_{0}^{d-|\alpha|}(\sum_{i=1}^{l}y_{i}a_{i})^{\alpha})|_{x_{0}=1}=\sum_{\alpha\in\mathbb{N}^{n}}a_{\alpha}(\frac{\partial^{k}}{\partial x_{0}^{k}}x_{0}^{d-|\alpha|})|_{x_{0}=1}(\sum_{i=1}^{l}y_{i}a_{i})^{\alpha}=\\\sum_{\alpha\in\mathbb{N}^{n}}a_{\alpha}\frac{(d-|\alpha|)!}{((d-|\alpha|)-k)!}(\sum_{i=1}^{l}y_{i}a_{i})^{\alpha}=(p^{(k)})_{A},\end{gather*} as the computation above shows. Note that the effect of the convention of Remark \ref{remarkdegree} is that at very beginning of the chain of identities above the factor $x_{0}$ out of the polynomial evaluation $p_{A}(\frac{y}{x_{0}})$ is $d=\deg(p)$ even though it could be that $\deg(p_{A})<d$, this is the effect that we want and that we were looking for when we wrote that remark. This finishes the proof of this lemma. 
\end{proof}

We make clear that the comment at the end of this proof is a landmark of this section because without the convention of Remark \ref{remarkdegree} we could not continue.

\begin{warning}[Failure if degree drops]
Without the convention of Remark \ref{remarkdegree}, Lemma \ref{lemaRene} above is not true, as one can easily see. Thus, what follows is also not true without said convention. The convention on that Remark is then fundamental for the rest of this section.
\end{warning}

We have to notice something about the behaviour of our operators with respect to the new variables the introduce. These new variables can be permanent or auxiliary.

\begin{remark}[Action of operators introducing additional variables]
Contrary to the map $\cdot^{(i)}$ that annihilates the homogenizing variable $x_{0}$, the maps $\Tilde{\cdot}$ and $\hat{\cdot}$ introduce a new variable that does not disappear.
\end{remark}

We therefore need the next notation for ease of writing during the rest of this section.

\begin{notacion}[Augmented subspace restriction]
\label{notation1more}
For a vector $a\in\mathbb{R}^{n}$ we denote $\gls{overa}:=(a,0)=\in\mathbb{R}^{n}\times\{0\}$. For a tuple of vectors $A=(a_{1},\dots,a_{l})\in(\mathbb{R}^{n})^{l}$ with $a_{i}\in\mathbb{R}^{n}$ we denote $$\gls{overaa}:=(\overline{a_{1}},\dots,\overline{a_{l}},(0,\dots,0,1))\in(\mathbb{R}^{n}\times\{0\})^{l}\times(\{0\}^{n}\times\{1\}).$$
\end{notacion}

Now, we can use Lemma \ref{lemaRene} to prove immediately something similar for the maps $\Tilde{\cdot}$ and $\hat{\cdot}$ instead of $\cdot^{(i)}$ because of how these maps are related.

\begin{corolario}[Operators and augmented subspace restriction commute]
\label{corohattilde}
Let $p\in\mathbb{R}[\mathbf{x}]$ be a polynomial and $A\in(\mathbb{R}^{n})^{l}$ for some $l\in\mathbb{N}$ a finite tuple of vectors. Then $(\hat{p})_{\overline{A}}=\hat{(p_{A})}$ and $(\Tilde{p})_{\overline{A}}=\Tilde{(p_{A})}$.
\end{corolario}

\begin{proof}
We prove it for $\hat{\cdot}$ because the other case is analogous. We have the chain of identities $(\hat{p})_{\overline{A}}:=$\begin{gather*}
    (\sum_{k=0}^{\infty}\frac{1}{k!}y^{k}p^{(k)})_{\overline{A}}=\sum_{k=0}^{\infty}\frac{1}{k!}y^{k}(p^{(k)})_{A}=\sum_{k=0}^{\infty}\frac{1}{k!}y^{k}(p_{A})^{(k)}=\hat{p_{A}},
\end{gather*} where we used the definition of $\hat{\cdot}$, the definition of $\overline{A}$, linearity of the restriction to $A$ showed in Fact \ref{facteasy}, and the commutativity proved in Lemma \ref{lemaRene} respectively always under the assumption taken in Remark \ref{remarkdegree}. This concludes the proof.
\end{proof}

We make clear that the action of Remark \ref{remarkdegree} extends beyond the Renegar derivative until it reaches the maps $\hat{\cdot}$ and $\Tilde{\cdot}$, object of the last corollary. This supposition, the convention on how to take Renegar derivatives of restrictions of a polynomial $p$ to a subspace $\linspan(A)$, is necessary in order to establish the Corollary \ref{corohattilde} above.

\begin{remark}[Degree fixing again]
Note that we did not say it explicitly, but $\hat{\cdot}$ and $\Tilde{\cdot}$ are operators depending on the Renegar derivative and hence on the degree of the argument and, therefore, because of Remark \ref{remarkdegree}, these objects apply to restrictions in the form $\hat{(p_{A})}$ and $\Tilde{(p_{A})}$ as if $p_{A}$ had the degree of $p$ even though it could happen that $\deg(p_{A})<\deg(p)$ if the actual degree drops at that particular restriction.
\end{remark}

Now we are finally ready to prove the main result of this section.

\begin{teorema}[Restriction of the relaxation equals relaxation of the restriction]
\label{teoremabuscado}
Let $p\in\mathbb{R}[\mathbf{x}]$ be an RZ polynomial. Then $$S(p)=\{a\in\mathbb{R}^{n}\mid(a,0)\in S(\Tilde{p})\}.$$
\end{teorema}

\begin{proof}
The inclusion $\supseteq$ is clear using \cite[Lemma 3.26 (c)]{main}. For other inclusion $\subseteq$, fix $a\in\mathbb{R}^{n}$ an arbitrary direction. We will show that for such direction $a\in\mathbb{R}^{n}$ we have the identity \begin{gather}\label{equationinmid}
S_{(a,0)}(\Tilde{p}):=\{\lambda\in\mathbb{R}\mid\lambda(a,0)\in S(\Tilde{p})\}=\{\lambda\in\mathbb{R}\mid\lambda a\in S(p)\}=:S_{a}(p).\end{gather} Using the results proved until this point, we have the chain of identities $S_{a}(p)=$ \begin{gather*}
    \bigcap_{b\in\mathbb{R}^{n}}S_{(1,0)}(p_{a,b}):= \bigcap_{b\in\mathbb{R}^{n}}\{\lambda\in\mathbb{R}\mid\lambda(1,0)\in S(p_{a,b})\}=\\\bigcap_{b\in\mathbb{R}^{n}}\{\lambda\in\mathbb{R}\mid\lambda(1,0)\in\{a\in\mathbb{R}^{2}\mid (a,0)\in S(p_{a,b}+y(p_{a,b})^{(1)})\}\}=\\\bigcap_{b\in\mathbb{R}^{n}}\{\lambda\in\mathbb{R}\mid\lambda(1,0,0)\in S(p_{a,b}+y(p_{a,b})^{(1)})\}=\\\bigcap_{b\in\mathbb{R}^{n}}S_{(1,0,0)}(p_{a,b}+y(p^{(1)})_{a,b})=\\\bigcap_{b\in\mathbb{R}^{n}}\bigcap_{b'\in\mathbb{R}^{3}}S_{(1,0)}((p_{a,b}+y(p^{(1)})_{a,b})_{(1,0,0),b'})=\\\bigcap_{b\in\mathbb{R}^{n+1}}S_{(1,0)}((p+y(p^{(1)}))_{\overline{a},b})=:\bigcap_{b\in\mathbb{R}^{n+1}}S_{(1,0)}(\Tilde{p}_{\overline{a},b})=S_{\overline{a}}(\Tilde{p}),
\end{gather*} where we used Corollary \ref{corolariofinal} in the first identity, the definition of $S_{(1,0)}(p_{a,b})$ given via Notation \ref{problemnotation} in the second, the identity $S(p_{a,b})=\{a\in\mathbb{R}^{2}\mid (a,0)\in S(p_{a,b}+y(p_{a,b})^{(1)})\}$ given by Corollary \ref{tildein} in the third, a simple rewriting or simplification of conditions in the fourth, Lemma \ref{lemaRene} in the fifth to put $(p_{a,b})^{(1)}=(p^{(1)})_{a,b}$ together with Notation \ref{problemnotation} used to compact the writing of the sets under the intersection noting that $(p_{a,b}+y(p^{(1)})_{a,b})$ is a trivariate polynomial, and Corollary \ref{corolariofinal} again in the sixth in order to be able to write $$S_{(1,0,0)}(p_{a,b}+y(p^{(1)})_{a,b})=\bigcap_{b'\in\mathbb{R}^{3}}S_{(1,0)}((p_{a,b}+y(p^{(1)})_{a,b})_{(1,0,0),b'}).$$ Importantly, the seventh holds because of the chain of identities \begin{gather*}(p_{a,b}+y(p^{(1)})_{a,b})_{(1,0,0),b'}=((p(x_{1}a+x_{2}b)+y(p^{(1)})(x_{1}a+x_{2}b))(x_{1},x_{2},y))_{(1,0,0),b'}=\\
(p(x_{1}a+x_{2}b)+y(p^{(1)})(x_{1}a+x_{2}b))(z_{1}(1,0,0)+z_{2}b')=\\(p(x_{1}a+x_{2}b)+y(p^{(1)})(x_{1}a+x_{2}b))(z_{1}+z_{2}b'_{1},z_{2}b'_{2},z_{2}b'_{3}))=\\p((z_{1}+z_{2}b'_{1})a+(z_{2}b'_{2})b)+(z_{2}b'_{3})(p^{(1)})((z_{1}+z_{2}b'_{1})a+(z_{2}b'_{2})b)=\\p(z_{1}a+z_{2}(b'_{1}a+b'_{2}b))+(z_{2}b'_{3})(p^{(1)})(z_{1}a+z_{2}(b'_{1}a+b'_{2}b))=\\(p+yp^{(1)})(z_{1}(a,0)+z_{2}(b'_{1}a+b'_{2}b,b'_{3}))=(p+yp^{(1)})_{(a,0),(b'_{1}a+b'_{2}b,b'_{3})}=
\end{gather*} $(p+yp^{(1)})_{\overline{a},\Tilde{b}}=\Tilde{p}_{\overline{a},\Tilde{b}}$ with $\Tilde{b}\in\mathbb{R}^{n+1}$ so that it varies freely in that space while $b$ varies in $\mathbb{R}^{n}$ and $b'$ does so in $\mathbb{R}^{3}$. Finally, the eight is the rewriting $p+yp^{(1)}=:\Tilde{p}$ and the ninth and last is again a final (and third) application of Corollary \ref{corolariofinal}. This proves the identity in Equation \ref{equationinmid} (i.e., $S_{(a,0)}(\Tilde{p})=S_{a}(p)$) for each direction $a\in\mathbb{R}^{n}$ and establishes the remaining inclusion finishing therefore the proof of this theorem.
\end{proof}

During the chain of identities in the proof above, one might want to continue the third line with the equivalent expression $\bigcap_{b\in\mathbb{R}^{n}}\{\lambda\in\mathbb{R}\mid\lambda(1,0,0)\in S((p+yp^{(1)})_{\overline{a},\overline{b},(0,0,1)})\}=:\bigcap_{b\in\mathbb{R}^{n}}S_{(1,0,0)}((\Tilde{p})_{\overline{a},\overline{b},(0,0,1)})$, obtained using Notation \ref{notation1more} and which does not produce a proof of the theorem but instead gives something similar to the last point of Corollary \ref{corolariofinal} at the end of Section \ref{sectioninter} but for the special case of $\Tilde{p}$. The reason why that expression does not produce a proof (contrary to the chosen expression) rests in the fact that the additional variable $y$ of $\Tilde{p}$ requires a special treatment (given in the form of an additional vector $(0,0,1)$) and this implies a tridimensional restriction that we did not study in this document. We collect this fact here anyway.

\begin{porisma}[Relaxation on directions also behave well under restrictions]
\label{porismarnadom}
$S_{a}(p)=\bigcap_{b\in\mathbb{R}^{n}}S_{(1,0,0)}((\Tilde{p})_{\overline{a},\overline{b},(0,0,1)})=S_{\overline{a}}(\Tilde{p}).$
\end{porisma}

We can proceed in the same way to obtain analogous results for other similar operators.

\begin{remark}[Results on other operators]
Using Proposition \ref{propoinstead} instead of Corollary \ref{tildein} in the proof of the last Theorem \ref{teoremabuscado} and Porism \ref{porismarnadom} produces similar results for $\hat{\cdot}$ instead of $\Tilde{\cdot}$.
\end{remark}

Now this result opens automatically the question about the behaviour of the relaxation under similar extensions involving other interlacers. We do this through generalizations of the construction of new RZ polynomial studied in this chapter like the ones presented in \cite{fisk2006polynomials} for general interlacers under very mild conditions. We study a case of these extensions in the next part. In particular, we will see that for other interlacers of a particularly well-understood family of polynomials we can make the relaxation improve asymptotically along the diagonal. Thus we will see that finding other forms of involving interlacers into the polynomial in order to augment the number of variables can result in improvements of the relaxation, which points towards such direction in order to increase our understanding of both the relaxation and, with some luck, possibly the GLC. Before we do that, we need to introduce the kind of polynomials over which we will study this. These are special because we do not have in principle a general construction but just an approach that works in certain cases and that needs to be put in the correct setting in order to allow for a generalization that could allow us to apply it to further cases until we could eventually cover all the RZ polynomials. We begin therefore humbly and simply with Eulerian polynomials and their rigidly convex sets. See \cite{bona2022combin, petersen2015eulerian, kitaev2011patterns} for various degrees of deeper and wider information about Eulerian polynomilas, their generalizations, their properties and their coefficients. This is the topic of the next part of this work, where we will see possibilities of extending the relaxation. This new objective goes in contrast to the limitations of the relaxation explored in this part that finishes here.
\part{On Eulerian Polynomials}\label{II}
\chapter[A Primer on Eulerian Polynomials]{A Primer on Eulerian Polynomials and Statistics and Patterns on Permutations}\label{ChPrimer}

We gave several general references on Eulerian polynomials and numbers at the end of the previous part in order to introduce the reader into the wonderful world they open and prepare the journey we have ahead. However, in this chapter we will use more specific references that address the main topics about these polynomials that are interesting, useful and fundamental for us. As the reader could see, the realm of these polynomials, their meaning, their properties and their generalizations is vast and at least as ample as the realms of the objects they encode: descent statistics on permutations. We will however center here our attention in their roots and where they are located: the roots of the univariate Eulerian polynomials are real and negative and the roots of the multivariate generalizations we will work with here satisfy nice properties in terms of polynomial stability, hyperbolicity or RZ-ness depending on how we decide to look at these generalizations. We already can say that, since our ultimate objective is to apply the relaxation, our preferred point of view is the one given by the framework of RZ polynomials. Here we will see how to proceed from the literature about these polynomials in order to end in the RZ framework we want. Hence, our main references during this chapter will be the papers \cite{visontai2013stable, haglund2012stable} that deal with stability of these multivariate generalizations of the Eulerian polynomials. Starting from this stability, we will be able to derive hyperbolicity and RZ-ness of related polynomials using what we know about the subtle relations between these types of polynomials.

\section{Counting patterns in permutations}

We begin this section defining a classical statistic over the symmetric groups defined as permutations of an ordered set. This paves the road for the introduction of the Eulerian polynomials.

\begin{definicion}[Descents]\cite[Subsection 2.1]{haglund2012stable}
\label{primer}
Call $\gls{S}=\bigcup_{n=0}^{\infty} \mathfrak{S}_{n}$ the union of all the symmetric groups $\mathfrak{S}_{n}$ whose elements are understood as bijections from the set $[n]\subseteq\mathbb{N}$ to itself. We define the \textit{descent counting function} \begin{gather*}
\des\colon\mathfrak{S}\to\mathbb{N},\sigma\mapsto\des(\sigma):=|\{i\mid\sigma_{i}>\sigma_{i+1}\}|.
\end{gather*}
\end{definicion}

Now observe that almost all the information about where descents happen and which elements are affected by them is lost just considering the map $\des$ above. For this reason, we have to think about ways of collecting more information in a meaningful way.

\begin{observacion}
[Tags and variables]
The way we can collect more information is therefore keeping track of the elements at which a descent happens. We can use these elements as tags for new variables, for example. Therefore, instead of having just one variable we can have several variables stemming from the original one when we tag these by the elements at which a descent happens. Thus, we can build more sophisticated objects carrying this extra information through these tags.
\end{observacion}

The use of these tags means that we will indeed count the features of the permutations in a finer way. Thus, we will be able to collect more information about the permutations at play.

\begin{remark}[Counting finer]
The particular tagging elements that we will use and describe here refer to the descent top and the ascent top statistics over permutations. In order to introduce this extension of counting descents, we first need to refine the concepts introduced in Definition \ref{primer} used for counting descents in a very simple way.
\end{remark}

Top and bottoms appear naturally in sequences of elements in an ordered set. As we see permutation as ordered sequences of elements of ordered sets through the one-line notation here, tops and bottoms appear as a very natural feature of these permutations. We see this next.

\begin{definicion}[Tops]\label{defexcat}\cite[Definition 3.1]{visontai2013stable}
For a permutation given as $\sigma=(\sigma_{1}\cdots\sigma_{n+1})\in\mathfrak{S}_{n+1}$ with $\sigma_{i}\in[n+1]$ for all $i\in[n+1]$ so we use the one-line notation for it, we say that $\sigma_{i}\in[n+1]$ is a \textit{descent top} if $\sigma_{i}>\sigma_{i+1}.$ Similarly, we say that $\sigma_{i+1}\in[n+1]$ is an \textit{ascent top} if $\sigma_{i}<\sigma_{i+1}.$ 
\end{definicion}

Observe that $1\in[n]$ can never be a top of anything because $1=\min([n])$. Now that we know already the objects we will deal with, we have to see how we pass from counting objects to polynomials. This will also help us realize how finer-counting these objects yields several types of liftings of the original polynomials. This is the very general and broad topic of the next section.

\section{From counting to polynomials}

Now we begin with an object and we want to transform the information coming from that object into functions. Usually the functions we like are either formal power series or, even better for their straightforward study, polynomials. We concentrate on finite objects and finite features to count with finitely many options. Thus we look here therefore to polynomials.

\setlength{\emergencystretch}{3em}%
\begin{definicion}[Generating polynomials]\cite[Section 2.2]{wilf2005generatingfunctionology}
The \textit{ordinary generating polynomial associated to the tuple} $(a_{1},\dots,a_{n})\in\mathbb{R}^{n}$ is the polynomial $p(x):=\sum_{i=1}^{n}a_{i}x^{i}.$ Building over this, if we consider a finite set $A$ of objects to be enumerated in some way through a \textit{measure function} $m\colon A\to\mathbb{N}$ we call the \textit{ordinary generating polynomial associated to the set $A$ through the measure function $m$} the polynomial $p(x):=\sum_{n>0}a_{i}x^{i}$, where $a_{i}:=|\{a\in A\mid m(a)=i\}|$ is the number of objects of size $i$ as measured by $m$ in the set $A.$  
\end{definicion}
\setlength{\emergencystretch}{0em}%

Note that, in general, the variable in generating polynomials only serves as a placeholder for keeping track of the coefficients accompanying the corresponding monomials and does not fulfill any additional role. We will change this through tags in the future. These tags will allow our variables to carry further information about the objects measured and stored in these polynomial functions.

\begin{remark}[Tagging]
Through the tags we can enrich our generating univariate polynomials into multivariate ones.
\end{remark}

We begin with a set of objects and form its generating polynomial in the usual way described above. Once we have this generating polynomial, we can see that counting finer over these initial objects and adding the new information into the mix through their respective tags allows us to carry and store more information in our polynomial. In particular, the monomials formed by the new tagged variables will allow us to produce a natural lifting of the univariate polynomial we began with. This is how we build the multivariate versions of the original polynomial. This is the general theory. We introduce the particular polynomials we are interested in in the next section. These are the Eulerian polynomials. These polynomials are connected to the descents and descent tops described in the first section of the current chapter. Remember that these polynomials emerge naturally when we count features in permutations happening (or defined) over ordered sets.

\section{Univariate Eulerian polynomials}

In this section we introduce the main object of our interest. How well we can approximate the roots of these polynomials will be the measure that we will use to check how much we can improve the relaxation through the different methods to augment the number of variables (and coefficients) at play that we will see in this dissertation. We introduce then now the univariate Eulerian polynomials.

\begin{definicion}\cite[Section 2.2]{haglund2012stable}[Univariate Eulerian polynomials]
Let $\mathfrak{S}_{n+1}$ be the symmetric group considered as permutations on the ordered set $[n+1]$ and written in its one-line form. Then we define the univariate polynomial $$\gls{an}(x)=\sum_{\sigma\in \mathfrak{S}_{n+1}}x^{\des(\sigma)}\in\mathbb{N}[x]$$ as the generating polynomial for the descent statistic over the symmetric group $\mathfrak{S}_{n+1}$.
\end{definicion}

These polynomials were classically first introduced by Euler while discussing a method of summation of series. Starting from there, these polynomials and their coefficients have been widely studied and applied in many other different mathematical settings. We go now through some of these settings.

\begin{remark}[Applications in (algebraic) combinatorics]
Through their combinatorial interpretation, the Eulerian polynomials, their coefficients and their multiple extensions and generalizations to new frameworks and settings appear among the most studied objects in combinatorics and algebra, making special appearance in questions related to, e.g., unimodality, log-concavity, gamma-positivity or real-rootedness. Recently, many generalizations of these polynomials and their coefficients have been connected to deep questions in order theory and algebraic geometry through the study of the geometry of matroids. To know more about these new and very hot topics in (algebraic) combinatorics we refer the reader to \cite{gu2016interlacing,shin2012symmetric,shareshian2017gamma,alexandersson2024rook,ALCO_2024__7_5_1479_0}.
\end{remark}

But apart from this face, Eulerian polynomials and Eulerian numbers find also their space in a seemingly distant branch of mathematics. These uses are more closely related to their origins in Euler's work \cite{euler} as an analytical tool for implementing summation techniques.

\begin{remark}[Applications in numerical analysis]
Eulerian polynomials and their coefficients are used in the construction of quadrature formulas similar to well-known Euler-MacLaurin quadrature formula. In this setting, they are also related to $B$-spline interpolation methods, as one can consult in \cite{he2012eulerian}.
\end{remark}

\setlength{\emergencystretch}{3em}%
Hence these two disparate appearances of the Eulerian polynomials and their multiple generalizations call in general for a deeper analysis involving an exchange of information between these two realms through the different generalizations, properties and extensions they allow. Thus Eulerian polynomials appear as an unexpected bridge between two seemingly unrelated areas of mathematics. These bridges are usually productive hotspots for new mathematics. It is therefore also interesting to look at the definition of these polynomials that is used in the analytic setting.
\setlength{\emergencystretch}{0em}%

\begin{definicion}[Euler-Frobenius polynomials]\label{eulerfrobenius}\cite[Chapter 25]{sobolev2006selected}
Let $n\in\mathbb{N}$ be a nonnegative integer. We define the \textit{$n$-th Euler-Frobenius polynomial} as $$\gls{en}(x):=\frac{(1-x)^{n+2}}{x}\left(x\frac{d}{dx}\right)^{[n]}\frac{x}{(1-x)^2},$$ where $[n]$ denotes the $n$-th power of the symbolic convolution.
\end{definicion}

Now it is well-known that Eulerian polynomials can be equivalently defined using a recurrence relation. This result is well known and it can be found in \cite[Section 3]{visontai2013stable}.

\begin{proposicion}[Recurrent definition of Eulerian polynomials]\label{recurrenceEuler}
Setting the initial polynomial $A_{1}(x)=1+x$ we have, for $n>1$, that $$A_{n}(x)=x(1-x)A'_{n-1}(x)+(1+nx)A_{n-1}.$$
\end{proposicion}

\begin{proposicion}[Sameness]\label{samenesseulerfrob}
For all nonnegative integers $n\in\mathbb{N}$ we have the equality of polynomials $E_{n}(x)=A_{n}(x).$
\end{proposicion}

\begin{proof}
The proof is by induction noting that $E_{1}=1+x=A_{1}$ and that, with the induction hypothesis for $n-1$, for a positive integer $n>1$, we have the next identity emanating from the well-known recurrence relation of Eulerian polynomials mentioned above in Proposition \ref{recurrenceEuler}
\begin{gather*}
A_{n}(x)=x(1-x)A_{n-1}'(x)+(1+nx)A_{n-1}(x)=\\x(1-x)E_{n-1}'(x)+(1+nx)E_{n-1}(x)=\\x(1-x)\frac{d}{dx}\frac{(1-x)^{n-1+2}}{x}\left(x\frac{d}{dx}\right)^{[n-1]}\frac{x}{(1-x)^2}+\\(1+nx)\frac{(1-x)^{n-1+2}}{x}\left(x\frac{d}{dx}\right)^{[n-1]}\frac{x}{(1-x)^2}=E_{n}(x),\end{gather*} where the first equality where we substitute $A$ by $E$ is consequence of the induction hypothesis.\end{proof}

Therefore, now it is clear that, although referred by different names in the literature of different mathematical settings, these polynomials are the same. This will be helpful in the future. Coming back to the combinatorial setting, the way we lift these polynomials to something having more variables is by keeping track of the elements at which a descent happens, as can be seen in \cite{haglund2012stable, visontai2013stable}. We can use that element to tag a new variable and therefore force our polynomials to carry this extra information about the descents at play in the corresponding permutation we are measuring or looking at. The additional information is therefore encoded in the form of a multivariate polynomial. This is the topic of the next section.

\section{Multivariate Eulerian polynomials}

The multivariate extensions that we will describe here use the descent top and the ascent top statistics over permutations. In order to introduce this extension, we first need to refine the concepts introduced in Definition \ref{primer} used for the univariate case.

\begin{definicion}[Tops sets]\cite[Subsection 2.1]{haglund2012stable}
and \cite[Definition 3.1]{visontai2013stable}\label{dtat}
For $\sigma\in\mathfrak{S}_{n+1}$, we define the \textit{descent top set} $$\gls{dt}=\{\sigma_{i}\in[n+1]\mid i\in[n],\sigma_{i}>\sigma_{i+1}\}\subseteq[2,n+1].$$ Similarly, we define the \textit{ascent top set} $$\gls{at}=\{\sigma_{i+1}\in[n+1]\mid i\in[n],\sigma_{i}<\sigma_{i+1}\}\subseteq[2,n+1].$$ 
\end{definicion}

Notice that there are indices at which we cannot possibly have a descent (or an ascent) and elements that can never be top of anything. In particular, $1$ can never be a top because it is smaller than any other element. Also, we cannot have a descent (or an ascent) at index $n+1$ because nothing follows it in the one-line notation.

\begin{observacion}[Ghost variables]
Thus, if we use tops to tag variables, the variable tagged by the element $1$ will never actually appear. This artefact makes the variable tagged by $1$ a ghost variable that exists just because we are tagging by a set that needs a minimum that can never be a top and therefore will never actually appear as a tag of a variable.
\end{observacion}

It is important to keep this in mind. In the future, not contemplating this could cause misunderstandings. We also look at the many generalizations of being real-rooted that there exist in the literature.

\begin{remark}[Towards generalizations of real-rootedness]
We already saw RZ polynomials as a natural generalization of being real-rooted. Polynomial hyperbolicity \cite[Definition 6.1]{main} and polynomial stability \cite[Subsection 2.5]{haglund2012stable} and \cite[Definition 2.5]{visontai2013stable} are different generalizations of this concept.
\end{remark}

Now we introduce the multivariate generalizations of Eulerian polynomials that we will use.

\begin{definicion}\cite[Theorem 3.2]{haglund2012stable}
and \cite[Theorem 3.3]{visontai2013stable}\label{multieulerian}
Let $$\gls{an}(\mathbf{x},\mathbf{y}):=\sum_{\sigma\in\mathfrak{S}_{n+1}}\prod_{i\in\mathcal{DT}(\sigma)}x_{i}\prod_{j\in\mathcal{AT}(\sigma)}y_{j}$$ be the \textit{descents-ascents $n$-th Eulerian polynomial}.
\end{definicion}

We talked about ghost variables. Here they are.

\begin{remark}\label{numberofvars}
Observe that we always have that $1\notin \mathcal{DT}(\sigma)$ and therefore $A_{n}$ has actually only the $n$ variables $x_{2},\dots x_{n+1}$.
\end{remark}

For univariate Eulerian polynomials, Frobenius noted that all zeros of $A_{n}(x)$ are real. This observation has also been made about different generalizations of Eulerian polynomials. For the generalization of the property of having real roots to multivariate polynomials, a notion that has been fruitful is that of stable polynomials.

\begin{definicion}\cite[Subsection 2.5]{haglund2012stable} and \cite[Definition 2.5]{visontai2013stable}
Let $p\in\mathbb{C}[\mathbf{x}]$ be a polynomial and define the set $$\gls{H}:=\{z\in\mathbb{C}\mid\Ima(z)>0\}.$$ We say that $p$ is \textit{stable} if $p\equiv0$ or $p(x)\neq0$ whenever $x\in\mathcal{H}^{n}$. We call $p$ \textit{real stable} when it is stable and has only real coefficients so $p\in\mathbb{R}[\mathbf{x}]$.
\end{definicion}

Now, finally, we can state the main theorem that allows us to apply the relaxation to this multivariate generalizations of the Eulerian polynomials.

\begin{teorema}\label{realstable}\cite[Theorem 3.3]{visontai2013stable} and \cite[Theorem 3.2]{haglund2012stable}
$A_{n}(\mathbf{x},\mathbf{y})$ is stable.
\end{teorema}

In the previously cited literature about these polynomials, particular cases of these multivariate generalizations appeared as important objects. Two important particularizations are the polynomials $A_{n}(\mathbf{x},\mathbf{1})$ and, even more particular, our univariate $A_{n}(x,\dots,x,\mathbf{1})$. Polynomial stability and its relation to real-zeroness will play a fundamental role in the next chapter. There we will prove that the connection between these two concepts allows us to prove real-zeroness for a restriction of the polynomial above, which we know to be real stable.
\chapter[Root Properties of Eulerian Polynomials]{Root Properties of Extensions of Eulerian Polynomials}\label{ChRoot}

In this chapter, we will basically stick with the same references as in the previous one. We will explore the connections between stability, hyperbolicity and RZ-ness in the context of Eulerian polynomials.

\section{From stability to RZ-ness}

We managed to find a stable multivariate generalization of Eulerian polynomials. This step gets us closer to the point where we can apply the relaxation to find bounds for the extreme roots of some related polynomials. However, in order to do this, we need to translate our stable polynomial setting into a setting where we talk about RZ polynomials, which is where our relaxation actually works. Fortunately, translating from stable to RZ polynomials is easy. First, we need to properly introduce the three kinds of polynomials that will play a role in this translation between settings. Two of them have already been discussed. We introduce now also hyperbolicity.

\begin{definicion}[Hyperbolicity]\cite[Definition 6.1]{main}
Let $p\in\mathbb{R}[\mathbf{x}]$ be a polynomial. We say that $p$ is \textit{hyperbolic with respect to a vector} $e\in\mathbb{R}^{n}$ if $p$ is homogeneous, $p(e)\neq0$ and, for any vector $a\in\mathbb{R}^{n}$, we have that $p(a-te)\in\mathbb{R}[t]$ is real rooted.
\end{definicion}

The relation between hyperbolic and real stable polynomials is as follows.

\begin{proposicion}[Real stability iff hyperbolicity in all positive directions]\cite[Proposition 1.3]{pemantle2012hyperbolicity} and \cite[Section 5]{kummer2015hyperbolic}
A homogeneous polynomial $p\in\mathbb{R}[\mathbf{x}]$ is real stable if and only if, for any vector $e\in\mathbb{R}^{n}_{>0}$ with positive coordinates and $a\in\mathbb{R}^{n}$, $p(a-te)$ is real-rooted. 
\end{proposicion}

Rewriting, this means that we have the following characterization.

\begin{corolario}[Homogeneous setting]
\label{rstoposorth}
A real homogeneous polynomial is real stable if and only if it is hyperbolic in every direction in the positive orthant.
\end{corolario}

Also, remember the following equivalence that we can find, e.g., in \cite[Proposition 6.7]{main}.

\begin{proposicion}[Hyperbolicity and RZ-ness]
\label{dehomo}
Let $p\in\mathbb{R}[x_{0},\mathbf{x}]$ be a homogeneous polynomial. Then $p$ is hyperbolic in the direction of the first unit vector $u=(1,\mathbf{0})\in\mathbb{R}^{n+1}$ if and only if its dehomogenization $q=p(1,x_{1},\dots,x_{n})\in\mathbb{R}[\mathbf{x}]$ is a RZ polynomial.
\end{proposicion}

All these relations allow us to deduce the following immediate result.

\begin{corolario}[Real stability and RZ-ness]
Let $p\in\mathbb{R}[x_{0},\mathbf{x}]$ be a homogeneous real stable polynomial such that $p(u)\neq0$ for $u=(1,\mathbf{0})\in\mathbb{R}^{n+1}$ the first unit vector. Then its dehomogenization $q=p(1,x_{1},\dots,x_{n})\in\mathbb{R}[\mathbf{x}]$ is a RZ polynomial.
\end{corolario}

\begin{proof}
Using Proposition \ref{rstoposorth} we get that $p$ is hyperbolic with respect to any vector $e\in\mathbb{R}_{>0}^{n+1}$. This means that $p((x_{0},x)-te)$ is real rooted for any $e\in\mathbb{R}_{>0}^{n+1}$ and $(x_{0},x)\in\mathbb{R}^{n+1}$. Taking the limit when $e\to u$ this implies that $p((x_{0},x)+tu)$ is real rooted for any $(x_{0},x)\in\mathbb{R}^{n+1}$. Which means that $p$ is hyperbolic in the direction of $u$ because $p(u)\neq0$ by hypothesis. Thus, by the Proposition \ref{dehomo} above, its dehomogenization $q:=p(1,x_{1},\dots,x_{n})\in\mathbb{R}[x]$ is an RZ polynomial.
\end{proof}

Now we are ready to apply this to the multivariate version of the Eulerian polynomial introduced in Definition \ref{multieulerian}.

\section[Homogenization]{Translation into the RZ setting: homogenization}

In this section, we fix our attention in the task of translating these real stable multivariate Eulerian polynomials into the RZ setting. We do this through the study of the degree of the monomials in the expansion defining these polynomials.

Remember that we begin with the polynomials $$A_{n}(\mathbf{x},\mathbf{y})=\sum_{\sigma\in\mathfrak{S}_{n+1}}\prod_{i\in\mathcal{DT}(\sigma)}x_{i}\prod_{j\in\mathcal{AT}(\sigma)}y_{j}$$ introduced in Definition \ref{multieulerian}. We first ask what means to homogenize these polynomials. Is it necessary?

\begin{observacion}[Variables and permutations]
For each term $\sigma\in\mathfrak{S}_{n+1}$ observe that, writing $\sigma=(\sigma_{1},\dots,\sigma_{n+1}$), for each pair $(\sigma_{i},\sigma_{i+1})$ we have that either $\sigma_{i}>\sigma_{i+1}$ or $\sigma_{i}<\sigma_{i+1}$ so, for each such pair indexed by $i\in[n]$, we add a variable to the corresponding monomial. Since there are $n$ of these pairs, all monomials have the same degree $n$ in the variables $(\mathbf{x},\mathbf{y})$.
\end{observacion}

Thus we have the following result.

\begin{proposicion}[Real stability of Eulerian polynomials]
$A_{n}(\mathbf{x},\mathbf{y})$ is homogenenous of degree $n$ and real stable.
\end{proposicion}

We will forget about ascents in order to get closer to a RZ polynomial. This is so because forgetting about ascents we will get polynomials that do not become $0$ at the origin.

\begin{convencion}[Forgetting ascent bottoms]
Now we set all the $\mathbf{y}$ variables equal to $1$, i.e., $\mathbf{y}=\mathbf{1}$.
\end{convencion}

We have to observe then what happens when we dehomogenize by the first unit vector in order to obtain a RZ polynomial. This is the topic of the next section.

\setlength{\emergencystretch}{3em}%
\section[Dehomogenization in ascents]{Translation into the RZ setting: dehomogenization in ascents}
\setlength{\emergencystretch}{0em}%

Here we continue the task that we began in the previous section. Our translation from one setting to another had a middle step.

\begin{remark}[Translation effort]
We remind that we are translating our real stable multivariate Eulerian polynomials to the RZ setting. Before, we concentrated our efforts into noticing that these polynomials are in fact homogeneous. Now we perform a demohomogenization of the variables coming from ascent tops that will finally give us the RZ polynomial we are searching for.
\end{remark}

We begin this exercise developing an example that shows why we have to specialize the most general Eulerian polynomial $A_{n}(\mathbf{x},\mathbf{y})$ when we try to directly apply the relaxation. First we need a convention.

\begin{convencion}[Setting the ascent top variables]
As we need real-zero polynomials, we will consider the polynomial obtained when $\mathbf{y}=\mathbf{1}$. This polynomial is clearly also stable and in this case it provides a direct object over which we can apply the relaxation.
\end{convencion}

Thus, we can work through the next example.

\begin{ejemplo}
Now we fix the vector $\mathbf{u}=(\mathbf{x},\mathbf{y})=(\mathbf{0},\mathbf{1})\in\mathbb{R}^{(n+1)+(n+1)}$ and consider the evaluation $A_{n}(\mathbf{u})$. This means that each summand for which we have $\mathcal{DT}(\sigma)\neq\emptyset$ will be zero. So the only terms surviving are these for which $\mathcal{DT}(\sigma)=\emptyset$. Notice what this set of constraints means now: we have to be able to show that there exists an element of the group of permutations $\mathfrak{S}_{n+1}$ that does not cancel out completely. We need to find then a permutation $\sigma$ such that we have the condition \begin{center}
    for all $i\in[n]$ we have $\sigma_{i}\leq \sigma_{i+1}.$
\end{center} But it is clear that the permutation $\id=(1\cdots(n+1))$ verifies $\mathcal{DT}(\sigma)=\emptyset.$ 
\end{ejemplo}

This example shows us directly a clear relation between $A_{n}(\mathbf{x},\mathbf{y})$ and $A_{n}(\mathbf{x},\mathbf{1})$.

\begin{remark}[Homogenization]
$A_{n}(\mathbf{x},\mathbf{y})$ can be seen as the homogenization of $A_{n}(\mathbf{x},\mathbf{1})$ when we set all the variables in the tuple $\mathbf{y}$ equal to the same variable $y$.
\end{remark}

With the remark above in mind, now it is clear that the most natural dehomogenization that we can consider in our task of translating to the RZ setting is the one obtained setting $\mathbf{y}=\mathbf{1}$.

\begin{remark}[A degree $0$ independent term and choice of $\mathbf{y}$ setting]
We use the polynomial with less variables instead of the one with the extra tuple $\mathbf{y}$ of variables although our relaxation is expected to always work better when we have more variables because we need that our polynomial is not $0$ at the origin.
\end{remark}

We could have made other choices, but we want to keep things simple and nice.

\begin{remark}[Complete dehomogenization]
We considered only the polynomial obtained setting $\mathbf{y}=\mathbf{1}$. We know that for our relaxation it is always more convenient to keep as many variables as possible. However, it seems more natural in a first attempt to completely \textit{dehomogenize out} the variables coming from the ascent top sets.
\end{remark}

Using what we saw in the previous sections, the next corollary is therefore clear.

\begin{corolario}[RZ-ness of dehomogenization]
$A_{n}(\mathbf{x},\mathbf{1})$ is RZ.
\end{corolario}

\begin{proof}
$A_{n}(\mathbf{x},\mathbf{1})$ is the dehomogenization with respect to the unit vector corresponding to the homogenizing variable $y$ of $A_{n}(\mathbf{x},y,\dots,y)$. Using Corollary \ref{rstoposorth} and that $A_{n}(\mathbf{x},\mathbf{y})$ is real stable by Theorem \ref{realstable} we get therefore that $A_{n}(\mathbf{x},y,\dots,y)$ is hyperbolic with respect to the unit vector corresponding to the homogenizing variable $y$. Thus, finally, using Proposition \ref{dehomo} we get that that $A_{n}(\mathbf{x},\mathbf{1})$ is RZ, which finishes this proof.
\end{proof}

Now we have enough knowledge about the roots of these polynomials within our setting to proceed. The next task is purely combinatorial. We will have to develop our strategies to count permutations in terms of how many descents they have and at which elements they top. This is the task on which we work in the next chapter.
\chapter{Counting Permutations to Compute $L$-forms}\label{ChCounting}

Once we know that the polynomials we are studying fit in our setting, we have to proceed with the work of effectively computing the relaxation corresponding to the application to these polynomials. In order to do that, we will need effective formulas for the computation of the corresponding coefficients of these polynomials. These formulas will allow us to, ultimately, compute the corresponding $L$-forms we need in order to build the relaxation. Fortunately, much is known about these coefficients as they are numbers widely studied in combinatorics. We will mainly look at results about these coefficients coming from the papers \cite{davis2018pinnacle,hall2008counting}.

\section[Unrestricted setting]{Counting permutations in terms of descents: unrestricted setting}

Before applying the relaxation, we have to introduce the combinatorial tools necessary to simplify the computations of the entries of its matrices. Only through such simplifications we can expect to establish a manageable bound. We determine what we need to know in order to compute the relaxation.

\begin{remark}[Sets of permutations]
The main object of our study during the computation of the relaxation are the next sets of permutations. These sets are related to the permutations that descend at a set of elements introduced in Definition \ref{defexcat}.
\end{remark}

We will introduce notation for sets of permutations in terms of their sets of descents.

\begin{definicion}[Exact descent]
\label{rns}
Let $S\subseteq[n+1]$. We introduce the set of permutations that \textit{descend exactly} at $S$, i.e., $$\gls{Rns}:=\{\sigma\in\mathfrak{S}_{n+1}\mid S=\mathcal{DT}(\sigma)\}.$$
\end{definicion}

In fact, the cardinals of these sets are the basic objects that will appear quite often in the computations of the entries of the coefficient matrices forming the LMP defining the relaxation.

\begin{remark}[Limited cardinality and shorthands]\label{rshortening}
For our relaxation, we always have the inequality $|S|\leq3.$ When $n$ is fixed we simply denote these numbers $R(S)$ for any subset $S\subseteq[n+1].$
\end{remark}

Therefore, before computing the relaxation, it is better to study first (the cardinalities of) these sets. It will be helpful to find general expressions for them in easily computable terms. Luckily, our polynomials are multiaffine.

\begin{proposicion}[Location of exact descent sets in Eulerian polynomials] we have the identity
$$R(n,S)=A_{n}(x,(1,\dots,1))|_{x_{i}=0 \mbox{ when } i\in[n+1]\smallsetminus S \mbox{ and } x_{i}=1 \mbox{ when } i\in S}.$$
\end{proposicion}

\begin{proof}
Obvious.
\end{proof}

But there are also direct methods for computing these terms. Before introducing these formulas we need to refine the concepts introduced in Definition \ref{dtat}

\begin{definicion}[Pairs]\cite[Section 1]{hall2008counting}
Let $\sigma=(\sigma_{1}\cdots\sigma_{n+1})\in\mathfrak{S}_{n+1}$ be a permutation written in its one-line notation. A \textit{descent pair} of $\sigma$ is a pair $(\sigma_{i},\sigma_{i+1})$ for $i\in[n]$ such that $\sigma_{i}>\sigma_{i+1}.$ An \textit{ascent pair} is defined analogously.
\end{definicion}

For counting purposes, it is interesting to consider pairs whose top lies in some fixed set $X$ and whose bottom lies in another fixed set $Y$. For brevity and because of the form of our polynomials, we will describe this just for descent pairs.

\begin{definicion}[Adequate descents]\cite[Definition 1.1]{hall2008counting}
Fix subsets $X,Y\subseteq\mathbb{N}$ and a permutation $\sigma\in\mathfrak{S}_{n+1}.$ We define the set of \textit{adequate descents} of $\sigma$ for the pair of sets $(X,Y)$ as $$\Des_{X,Y}(\sigma):=\{i\in[n]\mid\sigma_{i}>\sigma_{i+1}\mbox{\ and\ }\sigma_{i}\in X \mbox{\ and\ } \sigma_{i+1}\in Y\}.$$
\end{definicion}

It turns out that there are known formulas for computing the number of these permutations.

\begin{notacion}[Number of adequate descents]\cite[Theorems 2.3 and 2.5]{hall2008counting}
Denote \gls{pnsxy} the number of $\sigma\in\mathfrak{S}_{n}$ with at least $s$ adequate descents for the pair $(X,Y)$.
\end{notacion}

The explanation of the notation above might be confusing. We need to take care of which permutations we are actually counting through these numbers.

\begin{warning}[Exactness in descents]
Beware that $P_{n,s}^{X,Y}$ counts \textit{all} the permutations $\sigma\in\mathfrak{S}_{n}$ with at least $s$ adequate descents for the pair $(X,Y)$ in the sense that it counts permutations having also other additional descents. Later, we will need to refine these numbers because we want to be able to count the number of permutations $\sigma\in\mathfrak{S}_{n}$ with \textit{exactly} $s$ adequate descents \textbf{and no others} in order to find the coefficients of our polynomials.
\end{warning}

Here we will be mainly interested in the case where we let unrestricted the set $Y$. That is, when we fix $Y=\mathbb{N}$. Observe that we do not set a restriction on the bottom of the descent pairs.

\begin{particularizacion}[Unrestricted bottom set]
In this case, we denote $P_{n,s}^{X}:=P_{n,s}^{X,\mathbb{N}}$ the number of $\sigma\in\mathfrak{S}_{n}$ with at least $s$ adequate descents for the pair $(X,\mathbb{N})$.
\end{particularizacion}

We still need a couple of functions counting elements in a subset of $\mathbb{N}$ through order restrictions.

\begin{definicion}[Technical counters]\cite[Theorems 2.3 and 2.5]{hall2008counting}
For $A\subseteq\mathbb{N}$ and $n\in\mathbb{N}$ we write $A_{n}:=A\cap[n]$ and $A_{n}^{c}:=(A^{c})_{n}=[n]\smallsetminus A$. Additionally, for $j\in[n],$ we denote \begin{gather*}
\alpha_{A,n,j}:=|A^{c}\cap\{j+1,\dots,n\}|=|\{x\in A^{c}\mid j<x\leq n\}| \mbox{\ and\ }\\
\beta_{A,n,j}:=|A^{c}\cap\{1,\dots,j-1\}|=|\{x\in A^{c}\mid 1\leq x<j\}|,
\end{gather*} where the complement operator is taken with respect to $\mathbb{N}$ so $A^{c}=\mathbb{N}\smallsetminus A.$
\end{definicion}

Now we can count the number of $n$-permutations with a fixed number $s$ of adequate descent pairs for the pair $(X,\mathbb{N})$.

\begin{remark}[Counting permutations with fixed number of adequate descent pairs]
Since $Y$ is unrestricted, this will give us the number of permutations with a fixed number of descents $s$ whose top lies in $X.$
\end{remark}

We proceed with the main tool that we describe in this section.

\begin{teorema}[Formula for number of $n$-permutations with a fixed number of descents and whose tops are in a fixed set]\cite[Theorems 2.3 and 2.5]{hall2008counting} We have the following two formulas for $P_{n,s}^{X}=$\begin{gather*}
|X^{c}_{n}|!\sum_{r=0}^{s}(-1)^{s-r}\binom{|X^{c}_{n}|+r}{r}\binom{n+1}{s-r}\prod_{x\in X_{n}}(1+r+\alpha_{X,n,x})=\\
|X^{c}_{n}|!\sum_{r=0}^{|X_{n}|-s}(-1)^{|X_{n}|-s-r}\binom{|X^{c}_{n}|+r}{r}\binom{n+1}{|X_{n}|-s-r}\prod_{x\in X_{n}}(r+\beta_{X,n,x}).
\end{gather*}\end{teorema}

But this theorem is not enough for us, as we noted above.

\begin{warning}[Appearance of additional descents]
As we said before, the map $P_{n,s}^{X}$ counts \textit{all} the permutations with at least $s$ descents with tops in the set $X$, but these permutations might have other \textit{additional} descents. We, however, need a map that computes the permutations with no other descents than the $s$ whose descent tops are in $X$.
\end{warning}

This means that, setting $s=|X|$, the number $P_{n+1,s}^{X}=P_{n+1,|X|}^{X}$ will actually not be the cardinal (denoted as $|R(n,X)|$ in Definition \ref{rns}) of the set $$R(n,X):=\{\sigma\in\mathfrak{S}_{n+1}\mid X=\mathcal{DT}(\sigma)\},$$ as $P_{n+1,|X|}^{X}$ counts some extra permutations.

\begin{remark}[Counting carefully]
In fact, $P_{n+1,|X|}^{X}$ is the cardinal already mentioned in Definition \ref{defexcat} of the set $$(\mathfrak{S}_{n+1})(X):=\{\sigma\in\mathfrak{S}_{n+1}\mid X\subseteq\mathcal{DT}(\sigma)\}.$$
\end{remark}

However, considering now $P_{n+1,|X|}^{X}$ we can work out the remaining details with a much shorter and more convenient formula in the following section. Thus, since for our purposes in this paper we want to fill the whole restriction set $X$ with \textit{actual} descents (i.e., we set $s=|X|$ and $X\subseteq[n]$), then we have the following direct corollary that will be our main tool for counting in the next section.

\begin{corolario}[Totally realized fixed descent top set]\label{xordenadas}
For our special case of interest $s=|X|$ and $\{x_{1}<\cdots<x_{s}\}=X\subseteq[n]$, we have $P_{n,|X|}^{X}=$
\begin{gather*}
(n-|X|)!\prod_{i=1}^{s}(x_{i}-i).\end{gather*}
\end{corolario}

\begin{proof}
\begin{gather*}
(n-|X|)!\sum_{r=0}^{|X|-|X|}(-1)^{|X|-|X|-r}\binom{n-|X|+r}{r}\binom{n+1}{|X|-|X|-r}\prod_{x\in X}(r+\beta_{X,n,x})=\\(n-|X|)!\sum_{r=0}^{0}(-1)^{-r}\binom{n-|X|+r}{r}\binom{n+1}{-r}\prod_{x\in X}(r+\beta_{X,n,x})=\\(n-|X|)!\prod_{x\in X}(\beta_{X,n,x})=\\(n-|X|)!\prod_{x\in X}(|\{y\in [n]\smallsetminus X\mid 1\leq y<x\}|)=(n-|X|)!\prod_{i=1}^{s}(x_{i}-i),\end{gather*}
because we indexed the set $X$ according to its order so $x_{1}<\cdots<x_{s}.$ \end{proof}

The warnings in this section have to be addressed before any further computation can be performed. This is the topic of the next section.

\section[Restricted setting]{Counting permutations in terms of descents: restricted setting}

At the end of the previous section, we simplified the general formula already provided in the literature adapting it to our setting. This allows us to work with a much shorter formula in our case. Having this shorter formula will be very helpful in our task of addressing the warnings mentioned before: we need formulas for the number of permutations having exactly $s$ descents with tops in $X$ with $|X|=s$ and no other descents outside of $X.$

\begin{observacion}[Inclusions-exclusion]
The direct way to obtain these formulas is through the combination of the last Corollary \ref{xordenadas} of our previous section and the well-known inclusion-exclusion principle.
\end{observacion}

But first we need to think about the sets to which we need to apply this principle.

\begin{notacion}[Mirror]\label{defe}
Mirroring Definition \ref{defexcat}, we denote $$\gls{XS}:=\{\sigma\in\mathfrak{S}_{n+1}\mid\mathcal{DT}(\sigma)\subseteq X\}.$$
\end{notacion}

Computing the cardinal of these sets require the use of two maps on sets and tuples.

\begin{notacion}[Operators]\cite[Appendix A]{davis2018pinnacle}
For an ordered set $X=\{x_{1}<\dots<x_{k}\}$, we define the operator $$\alpha(X)=(x_{1}-1,x_{2}-x_{1},x_{3}-x_{2},\dots,x_{k}-x_{k-1}).$$ And for a tuple $\beta=(\beta_{1},\dots,\beta_{k})$ we denote by $$\beta\hat{!}=(k+1)^{\beta_{1}}k^{\beta_{2}}\cdots3^{\beta_{k-1}}2^{\beta_{k}}.$$
\end{notacion}

For completeness, we convene $\alpha(\emptyset)=()$ and $()\hat{!}=1$. With this, we are ready to find the cardinal of the sets of permutations considered in Notation \ref{defe}.

\begin{proposicion}\cite[Theorem A.1]{davis2018pinnacle}
Let $X\subseteq[n]$. Then $$|(X)(\mathfrak{S}_{n})|=\alpha(X)\hat{!}$$
\end{proposicion}

We then obtain immediately the following two ways of approaching the numbers $|R(n,X)|.$

\begin{corolario}[Cardinal of sets with exact descent top set]
\label{coroR}
Fix $s=|X|$ and $\{x_{1}<\cdots<x_{s}\}=X\subseteq[n]$ and $\{y_{1}<\cdots<y_{n-s}\}=Y\subseteq[n]$ with $X\cup Y=[n]$. For subsets $S\subseteq Y$, name the ordered chain of elements obtained through the union $X\cup S=\{x_{S,1}<\cdots<x_{S,s+|S|}\}$. Thus, going through the complement, we have that
\begin{gather}\label{coroR1}
    |R(n-1,X)|=\sum_{S\subseteq[n]\smallsetminus{X}}(-1)^{|S|}(n-|X\cup S|)!\prod_{i=1}^{s+|S|}(x_{S,i}-i).
\end{gather} Similarly, we can express this number in terms of deletions in the initial set as
\begin{gather}\label{coroR2}
    |R(n-1,X)|=\sum_{J\subseteq X}(-1)^{|X\smallsetminus J|}\alpha(J)\hat{!}.
\end{gather}
\end{corolario}

\begin{proof}
We only prove the first identity, the second follows similarly and its detailed proof can be consulted in \cite[Appendix A]{davis2018pinnacle}. Using the inclusion-exclusion principle we have that $$
\sum_{S\subseteq[n]\smallsetminus{X}}(-1)^{|S|}P_{n,|X\cup S|}^{X\cup S},$$ which, after the discussion on the section before, clearly equals $$\sum_{S\subseteq[n]\smallsetminus{X}}(-1)^{|S|}(n-|X\cup S|)!\prod_{i=1}^{s+|S|}(x_{S,i}-i).$$
\end{proof}

With this, we can begin the task of computing the numbers relevant towards the objective of effectively building our relaxation. These numbers are the ones obtained for sets $X$ with $|X|\leq3$.

\section{Counting permutations in terms of descents: up to degree $3$}

In order to ease the notation in what follows, we need to introduce an important warning as we will abuse our notation because from this point on we will be only interested in counting cardinalities on sets and not so much in the actual aspect of the sets we are counting. We do this because it is more comfortable to write and to read and because there is no risk of confusion. Anyway, let this warning serve the purpose of avoiding any possible confusion in what follows.

\begin{warning}[Abuse of notation]
\label{abuseofnotationforR}
Since in this section we are not so interested in the sets $R(X)$ themselves but only in their cardinalities, we will abuse the notation in what follows and refer to the cardinal $|R(X)|$ of $R(X)$ simply as $R(X)$ for ease of notation.
\end{warning}

Using the corollary above, now this part is easy. We just have to apply it to the cases $|X|\in\{1,2,3\}.$ We fix $n$ and call $R(X)=R(n,X).$

\begin{computacion}[Cardinal 1]
For $|X|=1$, we have $X=\{x_{1}\}$ so $R(X)=$ \begin{gather*}
(-1)^{0}\alpha(X)\hat{!}+(-1)^{1}\alpha(\emptyset)\hat{!}=(x_{1}-1)\hat{!}-1=2^{x_{1}-1}-1.\end{gather*}
\end{computacion}

\begin{computacion}[Cardinal 2]
For $|X|=2$, we have $X=\{x_{1}<x_{2}\}$ so $R(X)=$ \begin{gather*}(-1)^{0}\alpha(X)\hat{!}+(-1)^{1}\alpha(\{x_{1}\})\hat{!}+(-1)^{1}\alpha(\{x_{2}\})\hat{!}+(-1)^{2}\alpha(\emptyset)\hat{!}=\\3^{x_{1}-1}2^{x_{2}-x_{1}}-(2^{x_{1}-1}+2^{x_{2}-1})+1.\end{gather*}
\end{computacion}

\begin{computacion}[Cardinal 3]
For $|X|=3$, we have $X=\{x_{1}<x_{2}<x_{3}\}$ so $R(X)= (-1)^{0}\alpha(X)\hat{!}+$ \begin{gather*}(-1)^{1}\alpha(\{x_{1},x_{2}\})\hat{!}+(-1)^{1}\alpha(\{x_{2},x_{3}\})\hat{!}+(-1)^{1}\alpha(\{x_{1},x_{3}\})\hat{!}+\\(-1)^{2}\alpha(\{x_{1}\})\hat{!}+(-1)^{2}\alpha(\{x_{2}\})\hat{!}+(-1)^{2}\alpha(\{x_{3}\})\hat{!}+(-1)^{3}\alpha(\emptyset)\hat{!}=\\4^{x_{1}-1}3^{x_{2}-x_{1}}2^{x_{3}-x_{2}}-\\(3^{x_{1}-1}2^{x_{2}-x_{1}}+3^{x_{2}-1}2^{x_{3}-x_{2}}+3^{x_{1}-1}2^{x_{3}-x_{1}})+(2^{x_{1}-1}+2^{x_{2}-1}+2^{x_{3}-1})-1.\end{gather*}
\end{computacion}

Now we are ready to compute the corresponding $L$-forms.

\section{Evaluations of the $L$-forms at monomials up to degree $3$}

We are still counting, so the Warning \ref{abuseofnotationforR} about the notation $R(X)$ of previous section still applies here.

\setlength{\emergencystretch}{3em}%
\begin{convencion}[Fixing as Eulerian polynomials]
In order to further shorten the notation, we fix, from now on, $n$ and call $p:=A_{n}(\mathbf{x},\mathbf{1})$, which is already normalized because $A_{n}(\mathbf{0},\mathbf{1})=1.$
\end{convencion}
\setlength{\emergencystretch}{0em}%

Fix $i,j\in\{0\}\cup[n].$ The entry in position $ij$ of the matrix polynomial produced by the relaxation is $$\sum_{k\in\{0\}\cup[n]}x_{k}L_{p,d}(x_{k}x_{i}x_{j})|_{x_{0}=1}.$$ We have to compute the $L$-form applied to monomials of different degrees. The degree $0$ part is easy because $L_{p}(1)=n$, since $n$ is the degree of $p$. So we begin our computations by the degree $1$ and use the formulas for these values computed in \cite[Example 3.5]{main}.

\begin{computacion}[Degree 1]
\label{degree1form}
The first evaluations of the $L$-form come directly from the polynomial as $L_{p}(x_{i}):=\coeff(x_{i},p)$ $$=|\{\sigma\in\mathfrak{S}_{n+1}\mid\{i\}=\mathcal{DT}(\sigma)\}|=R(\{i\})=2^{i-1}-1.$$
\end{computacion}

The terms of degree $2$ have several forms depending on the degree of the variables involved.

\begin{computacion}[Degree 2]
\label{degree2form}
Looking at the computation above, it is clear that the evaluation on squares the has the form $\frac{1}{2}L_{p,d}(x_{i}^{2}):=-\coeff(x_{i}^{2},p)+\frac{1}{2}\coeff(x_{i},p)^{2}=$ $$\frac{1}{2}|\{\sigma\in\mathfrak{S}_{n+1}\mid \{i\}=\mathcal{DT}(\sigma)\}|^{2}$$ because the polynomial $p$ is multiaffine and then, automatically, we have that $\coeff(x_{i}^{2},p)=0.$ This implies $$L_{p,d}(x_{i}^{2})=R(\{i\})^{2}=(2^{i-1}-1)^{2}.$$ Meanwhile, for $i<j$, we have that the evaluation of the $L$-form in the monomials of the form $x_{i}x_{j}$ equals $L_{p,d}(x_{i}x_{j}):=$ \begin{gather*}-\coeff(x_{i}x_{j},p)+\coeff(x_{i},p)\coeff(x_{j},p)=\\-|\{\sigma\in\mathfrak{S}_{n+1}\mid\{i,j\}=\mathcal{DT}(\sigma)\}|+\\|\{\sigma\in\mathfrak{S}_{n+1}\mid\{i\}=\mathcal{DT}(\sigma)\}||\{\sigma\in\mathfrak{S}_{n+1}\mid\{j\}=\mathcal{DT}(\sigma)\}|=\\R(\{i\})R(\{j\})-R(\{i,j\})=\\(2^{i-1}-1)(2^{j-1}-1)-(3^{i-1}2^{j-i}-(2^{i-1}+2^{j-1})+1)=\\2^{-2 + i + j} - 2^{-i + j} \cdot 3^{-1 + i}.\end{gather*}
\end{computacion}

Finally, for degree $3$, we have to perform the following computations.

\begin{computacion}[Degree 3]
\label{degree3form}
Using again that $p$ is multiaffine, we have that $\frac{1}{3}L_{p,d}(x_{i}^{3}):=\coeff(x_{i}^{3},p)-\coeff(x_{i},p)\coeff(x_{i}^{2},p)+\frac{1}{3}\coeff(x_{i},p)^{3}=\frac{1}{3}\coeff(x_{i},p)^{3}=$ \begin{gather*}\frac{1}{3}|\{\sigma\in\mathfrak{S}_{n+1}\mid\{i\}=\mathcal{DT}(\sigma)\}|^{3}.\end{gather*} This implies $$L_{p,d}(x_{i}^{3})=R(\{i\})^{3}=(2^{i-1}-1)^{3}.$$ Similarly, for $i<j$, we have that the evaluation of the $L$-form of the monomials of the form $x_{i}^{2}x_{j}$ equals $L_{p,d}(x_{i}^{2}x_{j}):=$ \begin{gather*}\coeff(x_{i}^{2}x_{j},p)-\coeff(x_{i},p)\coeff(x_{i}x_{j},p)-\coeff(x_{j},p)\coeff(x_{i}^{2},p)+\\\coeff(x_{i},p)^{2}\coeff(x_{j},p)=\\-\coeff(x_{i},p)\coeff(x_{i}x_{j},p)+\coeff(x_{i},p)^{2}\coeff(x_{j},p)=\\-|\{\sigma\in\mathfrak{S}_{n+1}\mid\{i\}=\mathcal{DT}(\sigma)\}||\{\sigma\in\mathfrak{S}_{n+1}\mid\{i,j\}=\mathcal{DT}(\sigma)\}|+\\\{\sigma\in\mathfrak{S}_{n+1}\mid\{i\}=\mathcal{DT}(\sigma)\}|^2|\{\sigma\in\mathfrak{S}_{n+1}\mid\{j\}=\mathcal{DT}(\sigma)\}|=\\R(\{i\})^{2}R(\{j\})-R(\{i\})R(\{i,j\})=\\(2^{i-1}-1)^{2}(2^{j-1}-1)-(2^{i-1}-1)(3^{i-1}2^{j-i}-(2^{i-1}+2^{j-1})+1)=\\\frac{1}{3}\cdot 2^{-3 - i + j} (-2 + 2^i) (-4 \cdot 3^{i} + 3 \cdot 4^{i}).\end{gather*} Here there is an asymmetry we have to take care of: proceeding as before we see that when $j<i$, we have that the evaluation $L_{p,d}(x_{i}^{2}x_{j})=$ \begin{gather*}(2^{i-1}-1)^{2}(2^{j-1}-1)-(2^{i-1}-1)(3^{j-1}2^{i-j}-(2^{i-1}+2^{j-1})+1)=\\\frac{1}{3}\cdot2^{-3 + i - j} (-2 + 2^i) (-4 \cdot 3^j + 3 \cdot 4^j).\end{gather*} Finally, for distinct $i<j<k$ we have $L_{p,d}(x_{i}x_{j}x_{k})=$ \begin{gather*}
\frac{1}{2}(\coeff(x_{i}x_{j}x_{k},p)-\coeff(x_{i},p)\coeff(x_{j}x_{k},p)-\coeff(x_{j},p)\coeff(x_{i}x_{k},p)-\\\coeff(x_{k},p)\coeff(x_{i}x_{j},p)+2\coeff(x_{i},p)\coeff(x_{j},p)\coeff(x_{k},p))=\\\frac{1}{2}(|\{\sigma\in\mathfrak{S}_{n+1}\mid\{i,j,k\}=\mathcal{DT}(\sigma)\}|-\\|\{\sigma\in\mathfrak{S}_{n+1}\mid\{i\}=\mathcal{DT}(\sigma)\}||\{\sigma\in\mathfrak{S}_{n+1}\mid\{j,k\}=\mathcal{DT}(\sigma)\}|-\\|\{\sigma\in\mathfrak{S}_{n+1}\mid\{j\}=\mathcal{DT}(\sigma)\}||\{\sigma\in\mathfrak{S}_{n+1}\mid\{i,k\}=\mathcal{DT}(\sigma)\}|-\\|\{\sigma\in\mathfrak{S}_{n+1}\mid\{k\}=\mathcal{DT}(\sigma)\}||\{\sigma\in\mathfrak{S}_{n+1}\mid\{i,j\}=\mathcal{DT}(\sigma)\}|+\\2\prod_{h\in\{i,j,k\}}|\{\sigma\in\mathfrak{S}_{n+1}\mid\{h\}=\mathcal{DT}(\sigma)\}|)=\\\frac{1}{2}(R(\{i,j,k\})-(R(\{i\})R(\{j,k\})+R(\{j\})R(\{i,k\})+R(\{k\})R(\{i,j\}))+\\2R(\{i\})R(\{j\})R(\{k\}))=\\\frac{1}{2}(4^{i-1}3^{j-i}2^{k-j}-(3^{i-1}2^{j-i}+3^{j-1}2^{k-j}+3^{i-1}2^{k-i})+\\(2^{i-1}+2^{j-1}+2^{k-1})-1\\-((2^{i-1}-1)(3^{j-1}2^{k-j}-(2^{j-1}+2^{k-1})+1)\\+(2^{j-1}-1)(3^{i-1}2^{k-i}-(2^{i-1}+2^{k-1})+1)\\+(2^{k-1}-1)(3^{i-1}2^{j-i}-(2^{i-1}+\\2^{j-1})+1))+2(2^{i-1}-1)(2^{j-1}-1)(2^{k-1}-1))=\\2^{-3 + i + j + k} - 2^{-1 - i + j + k} \cdot 3^{-1 + i} - 2^{-2 + i - j + k} \cdot 3^{-1 + j} + 2^{-3 + 2 i - j + k} \cdot 3^{-i + j}.\end{gather*}
\end{computacion}

Now we can feed our software directly with these computations so it does not have to perform computations through sets of permutations. This is the topic of the next part of this work.
\part{On the Comparison of the Relaxation with Previously Known and Used Methods Providing Bounds}\label{III}
\chapter[Sobolev, Stanley, Stump and Mez\H{o}]{Sobolev, Stanley, Stump and Mez\H{o}: A Tale of a Lost Estimation and a Refinement of an Idea of Stanley}\label{ChApp}

Many paths meet in this chapter. Knowing that our polynomials are in the correct RZ setting will allow us to apply the relaxation while the counting arguments and formulas studied in the previous part will help us in the process of building it correctly computing the corresponding LMPs efficiently. Additionally, as we will provide bounds for roots of widely studied and researched polynomials, we will find some prominent people along this way. These are the people we mentioned in the title of the chapter. Thus, the main references here will be the books \cite{sobolev2006selected,mezo2019combinatorics} and the MO question \cite{287547}.

\section[Stump meets Stanley]{Stump meets Stanley with better asymptotics of roots of univariate Eulerian polynomials}

Stump and Stanley have an interesting exchange of information in the MO question \cite{287547} about the roots of univariate Eulerian polynomials. In particular, Stanley provides an asymptotic answer to a question of Stump for explicit forms of the roots of univariate Eulerian polynomials. We claim that Stanley proves in that post more than what he claims. He claims the following about the asymptotic behaviour of these roots. Beware that we are translating to our notation where the $n$-th multivariate Eulerian polynomial has degree $n$ in the variables $x_{2},\dots,x_{n+1}$. We decompose the $n$-th univariate Eulerian polynomials as $$A_{n}(x)=\prod_{i=1}^{n}(x-\gls{qin})=\prod_{i=1}^{n}(x+|q_{i}^{n}|)$$ so that $q_{i}^{(n)}$ is the $i$-th root of the $n$-th univariate Eulerian polynomial counting from the left to the right, i.e., for each $n$, the roots are sorted as follows $$q_{1}^{(n)}\leq\cdots\leq q_{n}^{(n)}<0.$$

\begin{notacion}[Eulerian numbers]\cite[Definition 6.6.3]{mezo2019combinatorics}
We write the Eulerian number $\gls{enk}$ for the number of permutations in $\mathfrak{S}_{n}$ having exactly $k$ descents.
\end{notacion}

Almost repeating names for different things is very unfortunate. Here, we better take care of this before continuing in order to avoid very unnecessary confusions.

\begin{warning}[Euler versus Eulerian numbers in the literature]
There are many different mathematical terms called after Euler. This is good for Euler and for those who, like me, admire him. But it is also a problem when two disparate sequences of numbers end up having a very similar name when they are not (directly) related. Thus, we make clear that here we are talking about (the triangle of) Eulerian numbers coming from counting permutations with a particular number of descents and not about the other sequence known as \textit{Euler numbers}.
\end{warning}

Stump initially asks a more complicated question that we do not fully address here. We only care about the asymptotic behaviour.

\begin{remark}[Interest in the form of the roots]
Stump asks in general what is known about the roots of univariate Eulerian polynomials and their values. He also asks if there is an explicit description for them. We only address the first part, dealing with the asymptotic behaviour of these roots. We note that, although it lies out of our scope here, there is a short discussion (by Stanley) in the comments of the cited MO post about the (im-)possibility of giving a nice  explicit description of these roots.
\end{remark}

However, afterwards Stanley proceeds to give an answer to Stump's question that only gives a very weak asymptotic estimation for these roots. In particular, Stanley claims something very weak. However, analyzing his strategy, it is clear that he could have indeed claimed much more. This is what we do here studying his approach in that MO post in more detail.

\begin{observacion}[Squeezing a strategy]
In particular, we will be able to use the same idea of Stanley and extend it until we obtain the asymptotic growth of the roots. Moreover, by solving some asymptotic equations, we will be able to go even beyond and analyze further terms of the asymptotic growth of these roots.
\end{observacion}

We look now at the estimate provided by Stanley in the MO post. As we will see, this estimate is far from what he could have claimed.

\begin{proposicion}[Original Stanley claim]\label{orstcl}\cite{287547}
Fix numbers $n,k\in\mathbb{N}$ with $k\leq n$ and decompose the $n$-th univariate Eulerian polynomial $$A_{n}(x)=\prod_{i=1}^{n}(x-q_{i}^{(n)})=\prod_{i=1}^{n}(x+|q_{i}^{n}|)$$ so that $$0<|q_{n}^{(n)}|\leq\cdots\leq |q_{1}^{(n)}|.$$ Then we can estimate the growth of the absolute value of the $k$-th root of the $n$-th univariate Eulerian polynomial $A_{n}$ as $$\lim_{n\to\infty}(|q_{k}^{(n)}|)^{\frac{1}{n+1}}=\frac{k+1}{k}.$$
\end{proposicion}

\begin{proof}
The $n$-th univariate Eulerian polynomial is $A_{n}(x)=\sum_{k=0}^{n}E(n+1,k)x^{k}$. Observe that, for $k\leq\frac{n}{2}$ fixed, $E(n+1,k)=E(n+1,n-k)$ grows roughly like $(k+1)^{n+1}$ in the particular sense that an explicit formula for that number has the form $$E(n+1,k)=\sum_{i=0}^{k}(-1)^{i}\binom{n+2}{i}(k+1-i)^{n+1}.$$ We will take $n$ arbitrarily big while fixing $k$ and therefore we can always suppose that we are in the case where $k\leq\frac{n}{2}$ so we have that, for fixed $k$, $\lim_{n\to\infty}\frac{E(n+1,k)}{(k+1)^{n+1}}=$ \begin{gather*}\lim_{n\to\infty}\frac{1}{(k+1)^{n+1}}\sum_{i=0}^{k}(-1)^{i}\binom{n+2}{i}(k+1-i)^{n+1}=\\\lim_{n\to\infty}\frac{1}{(k+1)^{n+1}}(-1)^{0}\binom{n+2}{0}(k+1-0)^{n+1}=1.\end{gather*} Now, $E(n+1,k)$ is the $k$-th elementary symmetric function of the absolute values of the roots, i.e., $$E(n+1,k)=\sum_{S\in\binom{[n]}{n-k}}\prod_{i\in S}|q^{(n)}_{i}|.$$ This in particular means that \begin{gather}\label{vagueproblem}E(n+1,1)=E(n+1,n-1)=\sum_{i=1}^{n}|q_{i}^{(n)}|\end{gather} and, therefore, as $E(n+1,n-1)=2^{n+1}-n-2$, we have that the biggest absolute value of roots $|q^{(n)}_{1}|$ has to grow like $2^{n+1}$ in the (very weak) sense that $$\lim_{n\to\infty}(|q^{(n)}_{1}|)^{\frac{1}{n+1}}=\lim_{n\to\infty}E(n+1,n-1)^{\frac{1}{n+1}}=2.$$ From there, it is easy to continue with a similar inductive argument in order to establish the general limit $\lim_{n\to\infty}(|q_{k}^{(n)}|)^{\frac{1}{n+1}}=\frac{k+1}{k}$ in the statement.
\end{proof}

This limit tells us how these roots grow in a very rough sense, as recognized by Stanley himself in the cited MO post. We explain why this estimate is so rough.

\begin{observacion}[Weak claim]
We can see, in particular, that from such a weak statement we cannot possibly identify if $q_{1}$ grows like $n^{a}2^{n+b}$ for some $a,b\in\mathbb{R}$. This is so because the $n$-th root kills the factor $n^{a}2^{b}$ for all $a,b\in\mathbb{R}$ as, in that case, $$\lim_{n\to\infty}(n^{a}2^{b})^{\frac{1}{n}}=1.$$
\end{observacion}

Clearing this up is the sense of a question of Mez\H{o} that we will treat in the next sections. However, a small modification of the argument given by Stanley in the MO post allows for a much stronger asymptotic estimation that already answers Mez\H{o}'s question. We see this next.

\begin{remark}[Seeking higher concreteness]
We modify here the argument of Stanley used in the proof above in order to extract from it a stronger statement. Our modification begins just after Equation \ref{vagueproblem} meaning that the proof is the same up to this point.
\end{remark}

The precise path we follow in order to modify the proof so we can be more precise is based on solving an asymptotic inequality. We explain how we do this.

\begin{estrategia}[Devising a way to improve the estimate]
We wrote Equation \ref{vagueproblem} describing Eulerian numbers in terms of the quantities $q_{i}^{(n)}$. Instead of estimating as roughly as Stanley did in the MO post and we saw above, we can estimate a few asymptotic inequalities and proceed to use the next Equation involving the absolute values of the roots $q_{i}^{(n)}$ and the Eulerian numbers.
\end{estrategia}

This procedure will allow us to extract more information from Equation \ref{vagueproblem}. In particular, as we want, we will be able to extract the asymptotic growth. More than that, we will build a procedure providing us with a general strategy that can be further applied in order to obtain additional information about these roots and roots of polynomials behaving similarly. We need to introduce some natural notation about asymptotic inequalities before proceeding with our result. We will only introduce the segment of these definitions that is useful for us. In particular, we only need to work here with positive sequences.

\begin{notacion}[Asymptotic (in-)equalities] \cite{khoshnevisan2007probability}\ \label{insteadO}
Let $f,g\colon\mathbb{N}\to\mathbb{R}_{>0}$ be sequences taking positive values. We write \begin{gather*}
    f(n)=O(g(n)) \  (\mbox{as\ } n\to\infty)
\end{gather*} if $\limsup_{n\to\infty}\frac{f(n)}{g(n)}<\infty$ and \gls{l} if $\lim_{n\to\infty}\frac{f(n)}{g(n)}=1.$ More than that, we write \gls{m} if there exist another sequence of positive numbers $h$ with $h\leq g$ and  $\lim_{n\to\infty}\frac{f(n)}{h(n)}=1,$ i.e., if $f\sim h\leq g.$ Finally, we write \begin{gather*}
    f(n)=o(g(n)) \  (\mbox{as\ } n\to\infty)
\end{gather*} if $\lim_{n\to\infty}\frac{f(n)}{g(n)}=0.$ If $f=\sum_{i=1}^{k}a_{i}\phi_{i}+o(\phi_{k+1})$ for some $k$ and some asymptotic scale (see \cite[Subsection 1.1.2]{paris2001asymptotics}) given by the sequence of sequences $\{\phi_{i}\}_{i=1}^{\infty}$, we say that $f$ \textit{grows like} $\sum_{i=1}^{k}a_{i}\phi_{i}+o(\phi_{k+1})$.
\end{notacion}

Now we have the tools we need to state and establish the next refinement of the argument of Stanley. In particular, now we have the asymptotic notions necessary to develop the proof of the next theorem. This proof will be the same as that of Stanley up until Equation \ref{vagueproblem}. After that equation, we will follow a different path that will allow us to say more than him using the asymptotic notions introduced above. 

\begin{teorema}[Improved estimate and exact asymptotics]
\label{refinementStanley}
Fix $n,k\in\mathbb{N}$ with $k\leq n$ and decompose the $n$-th univariate Eulerian polynomial $$A_{n}(x)=\prod_{i=1}^{n}(x-q_{i}^{(n)})=\prod_{i=1}^{n}(x+|q_{i}^{(n)}|)$$ so that $$0<|q_{n}^{(n)}|\leq\cdots\leq |q_{1}^{(n)}|.$$ Then we can estimate the growth of the absolute value of the $k$-th root of the $n$-th univariate Eulerian polynomial $A_{n}$ as $$\lim_{n\to\infty}\frac{|q^{(n)}_{k}|}{\left(\frac{k+1}{k}\right)^{n+1}}=1.$$
\end{teorema}

\begin{proof}
As we said, all stays the same until Equation \ref{vagueproblem}. In principle, from such equation alone we can extract a bit more. Observe first that $|q^{(n)}_{1}|$ can grow at most like $2^{n+1}$ and at least like $\frac{2^{n+1}}{n+1}$ or, otherwise, the left-hand side would grow faster or slower than the right-hand side. To see this, suppose that there exists $N$ such that for all $n>N$ we have $|q^{(n)}_{1}|>2^{n+1}$. Then $\sum_{i=1}^{n}|q_{i}^{(n)}|>2^{n+1}$ because we are summing positive quantities. But then such sum cannot equal $E(n+1,n-1)=2^{n+1}-n-2<2^{n+1}.$ For the other inequality, suppose that there exists $N$ such that for all $n>N$ we have $|q_{i}^{(n)}|\leq |q^{(n)}_{1}|<\frac{2^{n+1}}{n+1}$ for all $i\in[1,n].$ Thus, in short, we obtain \begin{gather*}
    \frac{\sum_{i=1}^{n}|q_{i}^{(n)}|}{2^{n+1}}\leq \frac{\sum_{i=1}^{n}\frac{2^{n+1}}{n+1}}{2^{n+1}}=\frac{n}{n+1}
\end{gather*} while \begin{gather*}
    \frac{E(n+1,n-1)}{2^{n+1}}=\frac{2^{n+1}-n-2}{2^{n+1}}=1-\frac{n+2}{2^{n+1}}.
\end{gather*} Now we see how these quantities approach $1$ from the left. In particular observe that $$1-\frac{\sum_{i=1}^{n}|q_{i}^{(n)}|}{2^{n+1}}\geq1-\frac{n}{n+1}=\frac{1}{n+1}$$ while $$1- \frac{E(n+1,n-1)}{2^{n+1}}=\frac{n+2}{2^{n+1}}$$ The disparity within these distances while approaching $1$ from the left guarantee therefore that it cannot be that there exists $N$ such that for all $n>N$ we have $|q^{(n)}_{1}|<\frac{2^{n+1}}{n+1}$. In symbols, that means that we have established that there exists $N$ such that for all $n>N$ $$\frac{2^{n+1}}{n+1}\leq |q^{(n)}_{1}|\leq2^{n+1}.$$ Suppose that $|q_{1}^{(n)}|$ grows actually slower than $2^{n+1}$ in the sense that there exist a subsequence $A=\{\frac{|q_{1}^{(n_{k})}|}{2^{n_{k}+1}}\}_{k=1}^{\infty}$ of the sequence $\{\frac{|q_{1}^{(n)}|}{2^{n+1}}\}_{n=1}^{\infty}$ with $A\to l<1$ when $k\to\infty$. We have more equations at this point of the proof. Namely, we also know, e.g., that \begin{equation}
    \begin{gathered}\label{products2}3^{n+1}-(n+2)2^{n+1}+\frac{1}{2}(n+1)(n+2)=E(n+1,2)=\\E(n+1,n-2)=\sum_{S\in\binom{[n]}{2}}\prod_{i\in S}|q_{i}^{(n)}|.\end{gathered}\end{equation} This equation puts a limit on how close can $q_{1}^{(n)}$ and $q_{2}^{(n)}$ be. In particular, $q_{1}^{(n_{k})}q_{2}^{(n_{k})}\leq 3^{n_{k}+1}$ for $k>K$ for some $K\in\mathbb{N}$ as, otherwise, the right-hand side would grow faster than the left-hand side. Therefore we have that $|q_{2}^{(n_{k})}|\leq\frac{3^{n_{k}+1}}{|q_{1}^{(n_{k})}|}\leq (n_{k}+1)\frac{3^{n_{k}+1}}{2^{n_{k}+1}}$ for $k>K$ for some $K\in\mathbb{N}.$ But now, under our hypothesis that $\frac{|q_{1}^{(n_{k})}|}{2^{n_{k}+1}}\to l<1$ this implies that the sum on the right-hand side of Equation \ref{vagueproblem} can never catch up with  the growth of the term $2^{n+1}$ in the left-hand side as such sum is asymptotically bounded as follows \begin{gather*}\sum_{i=1}^{n_{k}}|q^{(n_{k})}_{i}|\leq |q^{(n_{k})}_{1}|+(n_{k}-1)|q^{(n_{k})}_{2}|\leq\\ |q^{(n_{k})}_{1}|+(n_{k}-1)(n_{k}+1)\frac{3^{n_{k}+1}}{2^{n_{k}+1}}=:I(n_{k})\end{gather*} for $k>K$ for some $K\in\mathbb{N}.$ So it is clear that this implies $\lim_{k\to\infty}\frac{I(n_{k})}{2^{n_{k}+1}}=\lim_{k\to\infty}\frac{|q^{(n_{k})}_{1}|}{2^{n_{k}+1}}=l<1,$ which is a contradiction. Thus we obtain that any subsequence with a limit of the sequence $\{\frac{|q_{1}^{(n)}|}{2^{n+1}}\}_{n=1}^{\infty}$ has limit $1.$ As such sequence is bounded, there exists an accumulation point and therefore the sequence has a limit. This proves $\lim_{n\to\infty}\frac{|q_{1}^{(n)}|}{2^{n+1}}=1$ and finishes the proof for the extreme root. From this, we can now inductively establish the stronger result in the statement.
\end{proof}

There are several features of this proof worth mentioning. Some of these will be easier to see after we sketch the proof in a less formal way.

\begin{observacion}[Separating roots]
The fast growth of the coefficients allows us jump into the next equation involving a smaller degree coefficient in order to establish separation results about the roots. This is what happens in Equation \ref{products2}. Establishing this separation is fundamental in order to be able to compute the limits we end up studying. The fact that the roots of the Eulerian polynomials split naturally so fast is also of great help when dealing with these limits.
\end{observacion}

There is another trick that we used in the proof and that helped us to avoid dealing with limits that could have possibly not existed. This trick consists in supposing that we have a subsequence of the sequence of quotients tending to a limit different from the one we desire. After supposing these subsequences exist, we can use the known form of the coefficients of the polynomial in order to rule out the possibility of these subsequences reaching a contradiction. A well-known result about the fact that a bounded sequence converges if and only if it has exactly one accumulation point does the rest of the job.

\setlength{\emergencystretch}{3em}%
\begin{observacion}[Guaranteeing the existence of limits in bounded sequences]
If we want to check whether a bounded sequence $A$ has a limit $l$, we just need to explore an arbitrary subsequence $B$ with a limit $m\neq l$ and determine that the existence of $B$ implies a contradiction.
\end{observacion}
\setlength{\emergencystretch}{0em}%

The proof above is written in a very formal way that does not let us see very clearly what is happening. For this reason, it is beneficial to add a sketch of the proof written in a more informal way. This is how we actually reasoned in order to obtain this result. This is a teaser about how to reason asymptotically.

\begin{proof}[Sketch of the proof]
All stays the same until Equation \ref{vagueproblem}. In principle, from that equation alone we can extract a bit more. Observe first that $|q^{(n)}_{1}|$ can grow at most like $2^{n+1}$ and at least like $\frac{2^{n+1}}{n}$ or, otherwise, the left-hand side would grow faster or slower than the right-hand side. In symbols, this means that we have $\frac{2^{n+1}}{n}\lesssim |q^{(n)}_{1}|\lesssim2^{n+1}.$ Suppose that $|q_{1}^{(n)}|$ grows actually slower than $2^{n+1}$ in the sense that $\lim_{n\to\infty}\frac{|q_{1}^{(n)}|}{2^{n+1}}<1$. We have more equations at this point of the proof. Namely, we also know, e.g., that \begin{equation}
    \begin{gathered}\label{products2-repe}3^{n+1}-(n+2)2^{n+1}+\frac{1}{2}(n+1)(n+2)=E(n+1,2)=\\E(n+1,n-2)=\sum_{S\in\binom{[n]}{2}}\prod_{i\in S}|q_{i}^{(n)}|.\end{gathered}\end{equation} This equation places a limit on how close $q_{1}^{(n)}$ and $q_{2}^{(n)}$ can be. In particular, $q_{1}^{(n)}q_{2}^{(n)}\lesssim 3^{n+1}$ as, otherwise, the left-hand side would grow faster than the right-hand side. Therefore we have the asymptotic inequality $|q_{2}^{(n)}|\lesssim\frac{3^{n+1}}{|q_{1}^{(n)}|}\lesssim n\frac{3^{n+1}}{2^{n+1}}.$ But now, under our hypothesis that $|q_{1}^{(n)}|\lesssim2^{n+1}$ this implies that the sum on the right-hand side of Equation \ref{vagueproblem} can never catch up with 
the growth of the term $2^{n+1}$ in the left-hand side as such sum is asymptotically bounded as follows $$\sum_{i=1}^{n}|q^{(n)}_{i}|\lesssim |q^{(n)}_{1}|+(n-1)|q^{(n)}_{2}|\lesssim |q^{(n)}_{1}|+(n-1)n\frac{3^{n+1}}{2^{n+1}}=I(n)$$ and it is clear that $\lim_{n\to\infty}\frac{I(n)}{2^{n+1}}=\lim_{n\to\infty}\frac{|q^{(n)}_{1}|}{2^{n+1}}<1.$ Thus we obtain that $|q^{(n)}_{1}|\sim2^{n+1}.$ From this, we can inductively establish the stronger result in the statement.
\end{proof}

The refinement of the proof of Stanley that we made above and the development of the ideas it raised impulses us to go beyond. In particular, the proof invokes new ways of thinking about the asymptotic behaviour of these roots and also open questions about dealing with asymptotic (in-)equations in a manner that will be important for us in future sections.

\begin{remark}[Answer opens new doors]
The answer to the question of Stump given by Stanley and refined here opens up several questions. First, this answer gives the asymptotic growth of the extreme roots, thus opening the question of what more is possible to establish about such growth. Moreover, the techniques refined here deal with a kind of (in-)equations that provoke much interest to us: asymptotic (in-)equations have made appearance in our work this way. 
\end{remark}

Thus, we see how this answer leads us spontaneously into several different paths. We will have to choose the most promising ones and those that fit better with the object we are studying here. Recall that we want to use and study the power of the relaxation.

\begin{observacion}[Devising forms to look beyond]
In due time, we will generalize and expand our understanding of asymptotic equations in order to extract further information about the growth of these roots. Following this path, we will find new tools and ways to look at these asymptotic equations when dealing with further terms of the growth. Indeed, we will have to make clear what it means to extract further terms of these growths.
\end{observacion}

Now we explore another setting in which the question of Stump that inspired the developments in this section also appears. In particular, Mez\H{o} asks in \cite[Subsection 7.6.5]{mezo2019combinatorics} (also in this document as a question just before the beginning of \cite[Section 7.6.6]{mezo2019combinatorics}) a similar question about the growth of the extreme roots of univariate Eulerian polynomials. The answer provided by our refinement of the argument of Stanley is also an answer to his question. Remember that Stanley chose to answer this: Stump asked a more broad and general question about what is known about the behaviour or expressions of these roots. Exploring Mez\H{o} approach is the object of the next section.

\section[Mez\H{o}'s approach]{Mez\H{o}'s approach to the asymptotic growth of extreme roots of univariate Eulerian polynomials}

We have already answered the question of Mez\H{o} that we will treat in this section. In particular, we showed that the answer to that question was already hidden in the literature, as the proof is just a slight modification of Stanley's argument. In the next sections, we ask if we can actually improve our understanding of the extreme roots beyond these asymptotics establishing actual bounds and not just asymptotic estimations. We start by noticing that there are already well-known bounds for these roots. In particular, we know from \cite[Subsection 7.6.5]{mezo2019combinatorics} that, using 
Laguerre-Samuelson theorem and Colucci's estimation, we obtain the following result.

\begin{proposicion}[Known bounds]
\label{clearerasin}
We have the chain of inequalities \begin{gather*}I(n):=\frac{1}{n}\E(n+1,n-1)+\frac{n-1}{n}\sqrt{\E(n+1,n-1)^2-\frac{2n}{n-1}\E(n+1,n-2)}=\\\hspace*{-1cm}\frac{2^{n+1}-n-2}{n}+\frac{n-1}{n}\sqrt{(2^{n+1}-n-2)^2-\frac{2n}{n-1}(3^{n+1}-(n+2)2^{n+1}+\frac{1}{2}(n+1)(n+2))}\\\geq |q_{1}^{(n)}|\geq\frac{2^{n+1}}{n}-1-\frac{2}{n}\end{gather*} with $\lim_{n\to\infty}\frac{I(n)}{2^{n+1}}=1.$
\end{proposicion}

\begin{proof}
We obtain the asymptotic equivalence from the fact that $\E(n+1,n-1)\sim 2^{(n+1)}$ and $\E(n+1,n-2)\sim 3^{(n+1)}$ so it is clear that, when $n\to\infty$, \begin{gather*}
   \frac{I(n)}{2^{n+1}}\to 0+1\sqrt{1-2(0-0+0)}=1.
\end{gather*}
\end{proof}

In the proof above, we used the growth of some extreme Eulerian numbers. We can go more general.

\begin{proposicion}[Asymptotic growth of Eulerian numbers]
Fix $i\in\mathbb{N}$. Then $E(n+1,n-i)\sim(i+1)^{n+2}$.
\end{proposicion}

\begin{proof}
It is clear looking at the expansions of these numbers as sums, e.g., at the beginning of Proposition \ref{orstcl}.
\end{proof}

Contrary to what happened in the previous section where we dealt directly with the asymptotics, here we want to establish actual bounds at each side. With close enough bounds, we can extract again the asymptotics if the sandwich formed by the inequalities in these bounds is tight enough.

\begin{observacion}[Bounds versus asymptotics bounds]
In principle, asymptotic bounds do not need to be bounds in the strict sense as long as the growth is respected. Thus, for $n$ both $n+1$ and $n-1$ are lower asymptotic bounds, but only $n-1$ is really a lower bound.
\end{observacion}

This means that when we bound asymptotic growth we do not have to take so much care about the bound breaking sometimes as long as the growth is respected.

\begin{proposicion}[Bounding asymptotic growth]
Let $f,g,h,a,b\colon\mathbb{N}\to\mathbb{R}_{\geq0}$ be sequences of positive real numbers. Suppose that $a\lesssim f\leq g\leq h\lesssim b$ and that $a\sim b$. Then $a\sim f\sim g\sim h\sim b.$
\end{proposicion}

\begin{proof}
Evident.
\end{proof}

This proposition allows us to simplify our sequences when we deal with a sandwich of inequalities in order to establish asymptotic growth. This natural proposition will be used frequently in what follows. Moreover, we have to mention how far this asymptotic study actually reaches.

\begin{remark}[Further on asymptotics]
Bounding asymptotic growth will acquire even more importance when we use it to extract further terms of the growth of the roots. In that case, we will have to deal with differences and the growth that survives them.
\end{remark}

Before continuing we turn briefly our attention to the upper bound we discuss in this section. This will serve us in the next chapters. 

\section{Deeper analysis of the growth of the upper bound}

For the future, it will be important to know more about this upper bound. In particular, we want to extract more terms of its asymptotic growth in the following sense that will appear several times in our next analyses of these bounds.

\begin{definicion}[Convention on further terms of the asymptotic expansion]\label{defasy}
Let $f\colon\mathbb{N}\to\mathbb{R}$ be a sequence for which we have determined that $f\sim g$ for some (generally easier to understand) $g\colon\mathbb{N}\to\mathbb{R}$ and suppose that we call $g$ the \textit{up to the $n$-th term asymptotic growth of $f$}. We say that we \textit{extract a further term of the asymptotic growth of} $f$ if we compute a function $h\sim f-g$. In this case, we say that $g+h$ is the \textit{up to the $(n+1)$-th term asymptotic growth of $f$} and $h$ is the \textit{$(n+1)$-th term of the asymptotic growth of $f$}.
\end{definicion}

\begin{remark}
Observe that in the previous definition we speak about \textit{the} up to the $n$-th term asymptotic growth of $f$ and about \textit{the} $(n+1)$-th term of the asymptotic growth of $f$. We use ``the'' instead of ``a'' to refer to these objects because, for us, once a growth term is found for $f$ we will stick to it for the rest of this document. This means that we compute our expansions always using the same terms that we computed previously and in the fixed scale given by the almost exponential functions of the form $n^{r}a^{n}$ for $a>0$ and $r\in\mathbb{R}$. We will further clarify this next.
\end{remark}

As we see, we define the asymptotic growth as a recursive convention with respect to the functions we have extracted before. We could imagine that we are defining a basis of asymptotics and we generally will choose the easier available functions in order to define such \textit{basis}. This concept is actually widely used in asymptotic analysis, we are speaking here about \textit{scales}. See \cite[Subsection 1.1.2]{paris2001asymptotics}. However, we will not go so technical here because we do not need so much in this work. More on these bases for asymptotic expansions can be found in the cited reference and this is in accordance with our convention as long as we have in mind the kind of easy functions we are choosing in order to deal with our growth, i.e., as long as our basis is clear from the context. This is why we do not need more than the recursive definition-convention above.

\begin{convencion}[Our general asymptotic scale]
\label{conasy}
Since we will be dealing with sequences that grow following some exponential times a polynomial all the time and we will just need one fixed basis in order to describe the details of the asymptotic growths we are interested in, we will keep our basis completely multiplicative and therefore the kind of sequences that we will use to describe each term of our growth are the sequences of the form $\{n^{r}a^{ln+m}\}_{n=1}^{\infty}$ with $a,l,m\in\mathbb{R}_{\geq0}$ and $r\in\mathbb{R}.$ These sequences in use can clearly be simplified even more to these of the form $n^{r}a^{n}$ with $a\in\mathbb{R}_{\geq0}$ and $r\in\mathbb{R}$ following easy modifications of the coefficients, the base $a$ and the exponent $r$, but it will be more comfortable for us to keep some more freedom in the form that these sequences might have. The most important feature is that we do not allow sums or differences to appear in our basis sequences, thus letting us distinguish our growth terms by the different terms appearing on a sum of these sequences.
\end{convencion}

Apart from that and also as a consequence of our interest for extracting further terms of the asymptotic growth, now, for the first time in this dissertation, a problem will arise when dealing with the square roots appearing in our bounds. We will first deal with it here in order to see, in a somewhat easy example, how we overcome the complications caused by it.

\begin{warning}
We have to deal with the square roots nicely. For that, we use, of course, conjugates.
\end{warning}

Thus, the rest of the growth of the upper bound above is, having in mind that \begin{gather*}
\frac{2^{n+1}-n-2}{n}-2^{n+1}=\frac{2^{n+1}-n-2-n2^{n+1}}{n}=\frac{-n-2+(1-n)2^{n+1}}{n},\end{gather*} given by $I(n)-2^{n+1}=\frac{N}{D}$ with numerator $N=$ \begin{gather*}
   \bigg{(}\frac{-n-2+(1-n)2^{n+1}}{n}+\\\frac{n-1}{n}\sqrt{(2^{n+1}-n-2)^2-\frac{2n}{n-1}(3^{n+1}-(n+2)2^{n+1}+\frac{1}{2}(n+1)(n+2))}\bigg{)}\\\cdot\bigg{(}\frac{-n-2+(1-n)2^{n+1}}{n}-\\\frac{n-1}{n}\sqrt{(2^{n+1}-n-2)^2-\frac{2n}{n-1}(3^{n+1}-(n+2)2^{n+1}+\frac{1}{2}(n+1)(n+2))}\bigg{)}\\=\frac{3\left(2+3^{n}2(n-1)+n\right)}{n}\sim\frac{3^{n+1}2n}{n}=3^{n+1}2
\end{gather*} and denominator $D=$\begin{gather*}\bigg{(}\frac{-n-2+(1-n)2^{n+1}}{n}-\\\frac{n-1}{n}\sqrt{(2^{n+1}-n-2)^2-\frac{2n}{n-1}(3^{n+1}-(n+2)2^{n+1}+\frac{1}{2}(n+1)(n+2))}\bigg{)}\\\sim-2^{n+1}-2^{n+1}=-2^{n+1}2\end{gather*} so the second growth term of this bound is $\frac{N}{D}=\frac{3^{n+1}2}{-2^{n+1}2}=-(\frac{3}{2})^{n+1}$. This tells us that this bound grows like $2^{n+1}-(\frac{3}{2})^{n+1}.$ We can now continue and repeat this process in order to obtain a third growth term for this bound. Thus, having in mind that \begin{gather*}
\frac{-n-2+(1-n)2^{n+1}}{n}+\left(\frac{3}{2}\right)^{n+1}=\\\frac{-2^{n+1}n-2^{n+1}2+4^{n+1}(1-n)+3^{n+1}n}{2^{n+1}n},\end{gather*} we have that the third growth term is given by $I(n)-(2^{n+1}-(\frac{3}{2})^{n+1})=\frac{N}{D}$ with numerator $N=$ \begin{gather*}
\bigg{(}\frac{-2^{n+1}n-2^{n+1}2+4^{n+1}(1-n)+3^{n+1}n}{2^{n+1}n}+\\\frac{n-1}{n}\sqrt{(2^{n+1}-n-2)^2-\frac{2n}{n-1}(3^{n+1}-(n+2)2^{n+1}+\frac{1}{2}(n+1)(n+2))}\bigg{)}\\\cdot\bigg{(}\frac{-2^{n+1}n-2^{n+1}2+4^{n+1}(1-n)+3^{n+1}n}{2^{n+1}n}-\\\frac{n-1}{n}\sqrt{(2^{n+1}-n-2)^2-\frac{2n}{n-1}(3^{n+1}-(n+2)2^{n+1}+\frac{1}{2}(n+1)(n+2))}\bigg{)}\\=\left(\frac{9}{4}\right)^{n+1}-2^{-n} 3^{n+1}-\frac{6 \left(\left(\frac{3}{2}\right)^n-1\right)}{n}+3\sim\left(\frac{9}{4}\right)^{n+1}
\end{gather*} and denominator $D=$ \begin{gather*}
\bigg{(}\frac{-2^{n+1}n-2^{n+1}2+4^{n+1}(1-n)+3^{n+1}n}{2^{n+1}n}-\\\frac{n-1}{n}\sqrt{(2^{n+1}-n-2)^2-\frac{2n}{n-1}(3^{n+1}-(n+2)2^{n+1}+\frac{1}{2}(n+1)(n+2))}\bigg{)}\\\sim-2^{n+1}-2^{n+1}=-2^{n+1}2
\end{gather*} so we obtain that $\frac{N}{D}\sim\left(\frac{9}{4}\right)^{n+1}\frac{1}{-2^{n+1}2}=-\frac{1}{2}\left(\frac{9}{8}\right)^{n+1}.$ This tells us information about the first three growth terms of the upper bound. We collect what we know and have proved in the next result.

\begin{proposicion}
The upper bound given by $I(n)$ grows like $$I(n)\sim 2^{n+1}-\left(\frac{3}{2}\right)^{n+1}-\frac{1}{2}\left(\frac{9}{8}\right)^{n+1}+o\left(\left(\frac{9}{8}\right)^{n}\right).$$
\end{proposicion}

Observe how we followed our Definition \ref{defasy} of extracting further terms of the asymptotic growth and our Convention \ref{conasy} about the scale of sequences we use when dealing with the extraction of these terms. This serves as an example on why these notes were important before continuing our work on the asymptotic behaviour of the roots of our polynomials.

\begin{remark}[Keeping the same scale to compare]
Another advantage of keeping the same scale along the whole dissertation is the fact that, thanks to this uniformity, we will be able to establish comparisons between different bounds. The simplicity of our scale will further facilitate these comparisons that will become abundant and prevalent in what follows.
\end{remark}

We stop extracting information about the growth in the third term for reasons that will be disclosed and better understood when we talk about the bounds given by Sobolev. In the next section, we come back to our extension of the work of Mez\H{o} in \cite{mezo2019combinatorics}, which is our immediate interest besides the short detour we took here in order to understand better his upper bound and our treatment of asymptotic expansions.

\section[An answer]{An answer to a question of Mez\H{o} on asymptotics of extreme roots of univariate Eulerian polynomials}

We deeply explored the upper bound above. By the other side, the lower bound for the absolute value of this extreme root is obtained through the use of a Colucci estimation that comes from the next classic theorem in the theory of polynomial inequalities. This bound is not as good as the upper one, as we will see.

\begin{teorema}[Colucci estimation]\cite[Subsection 7.6.5]{mezo2019combinatorics}
Consider $p(x)=\sum_{i=0}^{n}a_{i}x^{i}\in\mathbb{C}[x]$ a polynomial whose zeros are bounded in absolute value by some positive real number $M\in\mathbb{R}_{\>0}$ and fix $k\in\{0,\dots,n\}$. Then the absolute value of the $k$-th derivative of $p$ admits the bound $$|p^{(k)}(x)|\leq k!\binom{n}{k}|a_{n}|(|x|+M)^{n-k}.$$
\end{teorema}

Now a simple analysis shows a way to perform a transformation producing the bound above. Thus is how we obtain that bound using the theorem.

\begin{remark}[Using Colucci estimation]
Rearranging the terms in the inequality provided by the theorem, we see that this result actually allows us to estimate the majorant $M$ whenever we have information about the derivatives of $p$ at some points. This is how this theorem is used to obtain the lower bound exposed above.
\end{remark}

Now we will use our relaxation in order to find these bounds instead. We will see that the relaxation has the ability to give even better bounds if we make the correct choices. We do not even need the whole power of the relaxation to do this.

\begin{convencion}[Simplified relaxation]
For ease of our analysis of the vectors involved, we will consider in this section not the full relaxation, but only the one obtained through the deletion of the first row and column of the original relaxation (see (c) of \cite[Definition 3.19]{main}). We do this because the generators of the kernel in this case seem to follow an easier pattern.
\end{convencion}

We remind that we stay all the time now in the RZ setting. In general, since we start using the relaxation we have to be in the RZ setting, as the relaxation needs to be fed with RZ polynomials.

\begin{notacion}[Simplified name for Eulerian polynomials]
For ease of notation, from now on, we will denote $\gls{an}(\mathbf{x},\mathbf{1}):=A_{n}(\mathbf{x})$, dropping directly the $y$ variables out as we will only work in the RZ setting in what follows.
\end{notacion}

We have to put attention into the size of our matrices. In this sense, we have to pay special care to ghost variables in multivariate Eulerian polynomials.

\begin{remark}[Size of the matrix of simplified relaxation of multivariate Eulerian polynomials]
As noted before, for each $n$, the multivariate Eulerian polynomial $A_{n}$ has actually $n$ variables $x_{2},\dots x_{n+1}$ (because $1$ can never be a descent top) and thus our method generates a matrix polynomial of size $(n+1)+1-1=n+1$ for the full relaxation. This is so because the row, column and matrix coefficient corresponding to the never appearing variable $x_{1}$ are in fact always $0$ and therefore can be omitted giving thus $-1$ row and columns. Notice then that the matrix of the simplified relaxation has additionally one row and column less. This makes therefore it into an $n\times n$ matrix for the matrix associated to the simplified relaxation of the $n$-th multivariate Eulerian polynomial $A_{n}.$
\end{remark}

We need some notation so we can speak about the matrix polynomials of our interest. This is what we do next.

\begin{notacion}[LMP obtained]
Remember that for multivariate Eulerian polynomials $\mathbf{x}=(x_{1},\dots,x_{n+1})$ but we can eliminate the row and column corresponding to $x_{1}$ because this variable does not actually play any role in the polynomial, i.e., all monomials containing it have coefficient $0$ because $1$ can never be a top. The relaxation applied to the $n$-th multivariate Eulerian polynomial $A_{n}(\mathbf{x})\in\mathbb{R}[\mathbf{x}]$ provides us therefore with a matrix polynomial $$\gls{mn}(\mathbf{x}):=M_{n,0}+\sum_{j=2}^{n+1}x_{j}M_{n,j}\in\Sym_{n+1}[\mathbf{x}]$$ positive semidefinite at the origin, i.e., $M_{n,0}$ is PSD. Setting $x_{i}=x$ for all $i\in[n+1]$ we look at the diagonal, where the univariate Eulerian polynomial $A_{n}(x)\in\mathbb{R}[x]$ lies. Looking at the diagonal like this produces the LMP $$M_{n}(x,\dots,x)=M_{n,0}+x\sum_{i=2}^{n+1}M_{n,i}=M_{n,0}+xM_{n,\suma}\in\Sym_{n+1}[\mathbf{x}]$$ whose determinant is a univariate polynomial whose biggest root bounds from below the biggest root $q_{n}^{(n)}$ (which is always negative) of the corresponding univariate Eulerian polynomial $A_{n}(x,\dots,x)$. Now, if we want a linear bound, we have to analyze the behaviour of the kernel of the matrix obtained when $x$ is the largest root $x_{n,r}$ of $\det(M_{n,0}+xM_{n,\suma})$ (which has to be negative as it must in fact be $x_{n,r}\leq q_{n}^{(n)}$ because of the property of being, in fact, a relaxation).
\end{notacion}

For analyzing the behaviour of such kernel we will have to act cleverly and guess through numerics and experimention a good approximate eigenvector in each case. Thus, the study of the behaviour of this kernel will be done through a procedure giving many different linear inequalities. We describe precisely such procedure in order to shed light over the reason why it works.

\begin{procedimiento}[Obtaining inequalities]
\label{procesobound}
In particular, our procedure will consist in applying the following criterion. A matrix polynomial $$p(x)=A+xB\in\Sym_{n}(\mathbb{R})[x],$$ where $A$ is a PSD matrix, will verify $$v^{\top}(A+xB)v=v^{\top}Av+xv^{\top}Bv\geq0$$ for all $v\in\mathbb{R}^{n}$ and $x\in\rcs(\det(p))$. Hence, fixing $v\in\mathbb{R}^{n}$, this gives us many inequalities of the form $a+xb\geq0$ (one for each $v$) that lower-bound the set $\rcs(\det(p))$ through the rearranging (solving for $x$) $x\geq\frac{-a}{b}$ whenever $b>0$. In short, $x\in\rcs(\det(p))$ must verify $x\geq\frac{-a}{b}.$
\end{procedimiento}

We have to be mindful of some subtleties that appear when we are dealing with inequalities. In particular, we have to ensure that these do not change direction.

\begin{remark}[Avoid reverse inequalities]
It is important to extreme the care in the last part of the procedure: we need $b>0$ in order to lower-bound through this method. Otherwise the inequality gets reversed.
\end{remark}

Thus we can recover exactly previous estimations. In particular, we can easily notice that applying this method to the easier case of our relaxation, which corresponds to considering only the top left entry of the relaxation, gives already Colucci estimation. We see in detail how this recovering of the same estimation is done.

\begin{observacion}[Relaxation allows Colucci estimation]
Applying the relaxation to any RZ polynomial $p\in\mathbb{R}[\mathbf{x}]$, considering just its top left entry and imposing that this entry has to be positive (or, equivalently using the vector $(1,\mathbf{0})$) gives directly the inequality $L_{p}(1)+x\sum_{i=2}^{n+1}L_{p}(x_{i})=n+x\sum_{i=2}^{n+1}(2^{i-1}-1)=n+x(-2 + 2^{1 + n} - n)\geq0$. It is easy to see that this inequality translates (remembering that univariate Eulerian polynomials are palindromic) easily into our first bound. This translation happens in this way: $q_{1}^{(n)}\leq\frac{-2 + 2^{1 + n} - n}{-n}$ or, in absolute value, $|q_{1}^{(n)}|\geq\frac{2^{1 + n}}{n}-\frac{2}{n}-1$. Now this last inequality is \textit{exactly} Colucci's estimation, as we saw some sections above.
\end{observacion}

Now we claim that an even more trivial application of our method already radically improves the previous lower bound. Additionally, this application also provides a tight first order asymptotic estimation for such root. In this way, we answer, through a method different from the one used in the previous section, the next question of Mez\H{o} (which was in fact already answered much before he posed it by Sobolev, but we will come to this in future parts of this work).

\begin{cuestion}[Mez\H{o}]\cite[Subsection 7.6.5]{mezo2019combinatorics}\label{mezo}
Provide $\beta\in[0,1]$ and $d>0$ such that $$\lim_{n\to\infty}\frac{n^\beta|q^{(n)}_{1}|}{2^{n+1}}=d.$$
\end{cuestion}

This question originally started our research and discussion of the application of the relaxation in order to find bounds for the extreme roots of univariate Eulerian polynomials and for this reason it was productive for us. However, as we noted, this question as it is posed was already answered decades ago by Sobolev. We will examine this work in the future. But we have to remark that, although this question was answered, thinking about it from the perspective of applying the relaxation has produced interesting results for us that go beyond an answer to the mere question posed and let us dig into the properties of the relaxation as measured by the effects of increasing the number of variables when dealing with the task of bounding the roots of univariate Eulerian polynomials injected in the diagonal of multivariate Eulerian polynomials. We make this clear now.

\begin{comentario}[Question already answered]
This question was already answered by Sobolev in \cite{sobolev2006selected}. It might have been difficult to find because he used a different name for our Eulerian polynomials and because the original publication was in Russian. We will comment on that work in the future. Here we will try to go beyond because, as we said, dealing with this question actually served us the purpose of understanding better the relaxation. In this sense, the path towards our answer is in fact more important than the answer itsef, as this answer was anyway already in the literature about univariate Eulerian polynomials. As Gauss said: ``It is not knowledge, but the act of learning, not the possession of but the act of getting there, which grants the greatest enjoyment.'' And, in this sense, it is not the knowledge of the asymptotic growth of the extreme roots of the univariate Eulerian polynomials, not an answer to the original question of Mez\H{o}, but the processes and techniques by which we reach there which gives us the big satisfaction of seeing how our techniques develop and serve us, not just to find an answer by a different path, but to also show properties about the behaviour and weaknesses and strengths of the tools used and applied in the techniques that drew our particular way to give another path towards the answer of such question. In particular, applying the relaxation during our struggle to answer this question showed us many properties about the relaxation and, \textit{in addition}, also eventually solved the question in a different manner, following a different way, drawing a different path.
\end{comentario}

Our answer to Question \ref{mezo} has to be, as we already know by the previous section, that $\beta=0$ and $d=1$. In order to see this here, we need to compute the $L$-forms of the univariate polynomial. Doing this is much easier than what we already did for the multivariate ones in the previous chapter.

\begin{computacion}[$L$-forms of the univariate polynomial]
In fact, as usual, for degree $0$, we have $L_{p}(1)=n$; for degree $1$, $L_{p}(x)=\coeff(x,p)=\E(n+1,1)=2^{n+1}-(n+2);$ for degree $2$, $L_{p}(x^2)=-2\coeff(x^2,p)+\coeff(x,p)^2=-2\E(n+1,2)+\E(n+1,1)^2=$ \begin{gather*}
-2(3^{n+1}-(n+2)2^{n+1}+\frac{1}{2}(n+1)(n+2))+(2^{n+1}-(n+2))^2=\\2 - 3^{1 + n}2 + 4^{1 + n} + n
;\end{gather*} and, finally, for degree $3$, $L_{p}(x^3)=$ \begin{gather*}3\coeff(x^3,p)-3\coeff(x,p)\coeff(x^2,p)+\coeff(x,p)^3=\\3\E(n+1,3)-3\E(n+1,1)\E(n+1,2)+\E(n+1,1)^3=\\3(4^{1 + n} - 3^{1 + n}(2 + n) + 2^n(1 + n)(2 + n) - \frac{1}{6}n(1 + n)(2 + n))-\\3(2^{n+1}-(n+2))(3^{n+1}-(n+2)2^{n+1}+\frac{1}{2}(n+1)(n+2))+\\(2^{n+1}-(n+2))^3=-2 - 2^{1 + n} 3^{2 + n} + 4^{1 + n}3 + 8^{1 + n} - n.
\end{gather*}
\end{computacion}

Now we get the answer fast just considering the relaxation of the univariate polynomial.

\begin{proposicion}[Improving Corollary \ref{clearerasin} using the relaxation]\label{improvementcearer}
We have the asymptotic inequality $|q_{1}^{(n)}|\gtrsim 2^{n+1}$.
\end{proposicion}

\begin{proof}
We can compute now the matrix of the relaxation as \begin{gather*}
    \begin{pmatrix}
    n & 2^{n+1}-(n+2)\\2^{n+1}-(n+2) & 2 - 3^{1 + n}2 + 4^{1 + n} + n
    \end{pmatrix} + \\ x\begin{pmatrix}
    2^{n+1}-(n+2) &  2 - 3^{1 + n}2 + 4^{1 + n} + n\\ 2 - 3^{1 + n}2 + 4^{1 + n} + n & -2 - 2^{1 + n} 3^{2 + n} + 4^{1 + n}3 + 8^{1 + n} - n.
    \end{pmatrix}.
\end{gather*} Next, we compute its determinant \begin{gather*}
    -4(-1 + 2^n)^2 + (-2 + 2^{2 + n} - 3^{1 + n}2 + 4^{1 + n})n
 +\\ x(-4(-1 + 2^n)(1 + 2^{1 + 2n} - 3^{1 + n}) \\+ 2(1 - 2^n + 2^{3 + 2n} + 2^{2 + 3n} - 3^{1 + n} - 2^n 3^{2 + n})n
)\\ + x^{2}(-2^{2 + n} - 2^{3 + 2n}5 + 3^{1 + n}8 + 6^{2 + n} + 8^{1 + n} - 9^{1 + n}4 +\\12^{1 + n} - 2(2^n + 2^{1 + 2n}5 + 2^{2 + 3n} - 3^{1 + n}2 - 2^n 3^{2 + n})n).
\end{gather*} Finally, writing it as $c+bx+ax^{2}$ as usual, we can easily establish its roots as it is a polynomial of degree $2$ and taking its root (both roots are negative because we are using a relaxation that respects the rigidly convex set of the original polynomial) of smallest absolute value we get our 
outer approximation to the largest root of the univariate Eulerian polynomial. This root equals $z_{n,r}:=$ \begin{gather*}
   \frac{-b+\sqrt{b^2-4ac}}{2a}
\end{gather*} because in $b$ the terms winning asymptotically by exponentials are $$-2^{1+3n}4+8^{n+1}n=8^{1+n}n-8^{1+n}=(n-1)8^{1+n}>0$$ so, for big enough $n>0$, $-b<0$ and, as all the roots of these polynomials must be real, the term inside the square root $b^2-4ac\geq0$ so the root with smallest absolute value is the root in which the $+$ is taken before the square root. With this, and using the palindromicity of Eulerian polynomials, we can see that $$q_{1}^{(n)}\leq\frac{2a}{-b+\sqrt{b^2-4ac}}$$ or, in absolute value, $$|q_{1}^{(n)}|\geq\frac{2a}{b-\sqrt{b^2-4ac}}\sim 2^{n+1}$$ because it is clear that\footnote{Be careful, some computer software is unable to compute this limit correctly because of the square root in the denominator that goes to $0$ as it seems like the software gets lost before the step of multiplying by the conjugate of the expression is done and thus proceeds wrongly!}, when $n\to\infty$, \begin{gather*}
\frac{\frac{2a}{b-\sqrt{b^2-4ac}}}{2^{n+1}}=\frac{\frac{2a}{2^{n+1}}}{b-\sqrt{b^2-4ac}}=\frac{\frac{2a}{2^{n+1}}(b+\sqrt{b^2-4ac})}{b^2-(b^2-4ac)}=\\\frac{\frac{2a}{2^{n+1}}(b+\sqrt{b^2-4ac})}{4ac}\to 1\end{gather*} as the term winning asymptotically in the last denominator is $2^{4 + 3 n} n$ and the term winning asymptotically in the numerator is $2^{3 + 3 n} n+\sqrt{2^{6 + 6 n} n^2}=2^{4 + 3 n} n,$ which confirms that the quotient tends to $1$ as $n\to\infty$.
\end{proof}

As an immediate corollary, we therefore see that we can answer Question \ref{mezo} already. The proof is clear and immediate using the sandwich of asymptotic inequalities provided by Proposition \ref{improvementcearer} and Corollary \ref{clearerasin}.

\begin{corolario}[Answer using the relaxation]
We have the asymptotic identity $|q_{1}^{(n)}|\sim 2^{n+1}$.
\end{corolario}

We did this here to obtain the best possible bound that the univariate relaxation gives. As we will soon see in the next Proposition \ref{anteserajemplo}, for univariate Eulerian polynomials and in the first asymptotic term, the best possible bound given by the relaxation and the one obtained through the strategy of checking through a vector are equivalent.

\begin{remark}[Identifying the improvement]
The bound obtained above is the best one that the relaxation can give when applied to the univariate polynomial. In this case, it was possible to explicitly compute this optimal bound because the matrices were small enough. In the future, we will have matrices too big to allow us to follow this exact approach. This is why we look now at approximate approaches through exercises of \textit{eigenvector guessing}. This technique will be of fundamental importance in the future when we go fully multivariate.
\end{remark}

Hence, we found that the best possible univariate bound given by the relaxation can be written explicitly. We make sure to collect this here.

\begin{observacion}[Best possible univariate bound through the relaxation]
The optimal bound obtainable applying the relaxation to the univariate Eulerian polynomial is $$\un:=\frac{2a}{b-\sqrt{b^2-4ac}}$$ with $c+bx+ax^{2}=$ \begin{gather*}
    -4(-1 + 2^n)^2 + (-2 + 2^{2 + n} - 3^{1 + n}2 + 4^{1 + n})n
 +\\ x(-4(-1 + 2^n)(1 + 2^{1 + 2n} - 3^{1 + n}) \\+ 2(1 - 2^n + 2^{3 + 2n} + 2^{2 + 3n} - 3^{1 + n} - 2^n 3^{2 + n})n
)\\ + x^{2}(-2^{2 + n} - 2^{3 + 2n}5 + 3^{1 + n}8 + 6^{2 + n} + 8^{1 + n} - 9^{1 + n}4 +\\12^{1 + n} - 2(2^n + 2^{1 + 2n}5 + 2^{2 + 3n} - 3^{1 + n}2 - 2^n 3^{2 + n})n).
\end{gather*} These nice expressions will become more difficult to obtain for polynomials involving more variables and, therefore, LMP of bigger size because the degree of the corresponding determinant will then grow accordingly until this growth makes computing its roots a practically (and also theoretically, symbolically and even numerically) undoable task. 
\end{observacion}

For the reason above, we need to think about strategies that can jump over the difficulties of degree growth of the determinant because our LMP grows when the number of variables increases. Thus, instead of proceeding via the determinant, we can choose the vector $v=(1,1)$ and use it to extract a different bound with the same asymptotic growth.

\begin{proposicion}[Same asymptotics checking through a vector]\label{anteserajemplo}
The inner bound sequence $b_{v}\colon\mathbb{N}\to\mathbb{R}$ for the smallest root of the univariate Eulerian polynomial obtained through the guess of the eigenvector of its relaxation given by $v=(1,1)$ grows like $b_{v}\sim2^{n+1}.$
\end{proposicion} 

\begin{proof}
Using the vector $v=(1,1)$  we obtain the inequality \begin{gather*}L_{p}(1)+2L_{p}(x)+L_{p}(x^2)+x(L_{p}(x)+2L_{p}(x^2)+L_{p}(x^3))=\\-2 + 2^{2 + n} -3^{1 + n}2 + 4^{1 + n}+\\x(2^{1 + n} + 2^{3 + 2n} -  3^{1 + n}4 - 2^{1 + n} 3^{2 + n} + 4^{1 + n}3 + 8^{1 + n}
)\geq0.\end{gather*} Thus, proceeding as we did previously, we obtain the bound \begin{gather*}
q_{1}^{(n)}\leq\frac{2^{1 + n} + 2^{3 + 2n} -  3^{1 + n}4 - 2^{1 + n} 3^{2 + n} + 4^{1 + n}3 + 8^{1 + n}}{2 - 2^{2 + n} + 3^{1 + n}2 - 4^{1 + n}}:=-b_{v}(n)\end{gather*} or, in absolute value and asymptotically, we obtain what we want as we can see that $|q_{1}^{(n)}|\geq b_{v}(n)\sim\big{(}\frac{8}{4}\big{)}^{n+1}=2^{n+1}$ because it is clear that, when $n\to\infty$, we have \begin{gather*}\frac{2^{1 + n} + 2^{3 + 2n} -  3^{1 + n}4 - 2^{1 + n} 3^{2 + n} + 4^{1 + n}3 + 8^{1 + n}}{2^{n+1}(-2 + 2^{2 + n} - 3^{1 + n}2 + 4^{1 + n})}=\frac{b_{v}(n)}{2^{n+1}}\to1.\end{gather*}
\end{proof}

The procedure followed in the proof above clearly shows us why it is beneficial to look through guesses of eigenvectors instead of spending so much energy trying to establish the optimal bound given by the relaxation. Now we also notice that this, together with the asymptotic bound by the other side given already by Mez\H{o}, provides a clear first order asymptotic estimate for $|q_{1}^{(n)}|$ which immediately answers Mez\H{o}'s question: $\beta=0$ and $d=1.$ So we get actually the following more precise result. 

\setlength{\emergencystretch}{3em}%
\begin{proposicion}[Answering the question of Mez\H{o} with actual better bounds and tight asymptotics]
$I(n)\geq|q_{1}^{(n)}|\geq\un\geq b_{(1,1)}.$
\end{proposicion}
\setlength{\emergencystretch}{0em}%

The univariate relaxation is already powerful as it allows us to establish the first order growth. However, this relaxation is not new. It has appeared before in the literature.

\begin{observacion}[Previous appearances of univariate relaxation]
The univariate relaxation has appeared already in \cite[Equation 2.2.9]{szeg1939orthogonal}. For more details on this early appearance, we refer also to \cite{blekherman2020generalized}.
\end{observacion}

For this reason, we are particularly interested in going beyond the univariate case. We want to apply the multivariate relaxation as this is new and it already depends in results coming from real algebraic geometry. We expect that going multivariate will greatly help and we will see how this happens soon. There are several ways to go multivariate. An intermediate step is going bivariate, which is already not trivial and improves our relaxation.

\begin{objetivo}[Pushing beyond the univariate relaxation]
There are two ways to leave the univariate setting. The most immediate way is considering a Nuij-type bivariate extension that allows us to introduce the bivariate RZ polynomial $A_{n+1}+yA_{n}$ using the well-known fact that $A_{n+1}$ interlaces $A_{n}$ and results from \cite{fisk2006polynomials}. This approach is similar to the one we studied in the first part when we increased the number of variables through the use of the interlacer provided by the Renegar derivative of the original polynomial. The other way consists in using the multivariate Eulerian polynomials introduced above and constructed through the collection of the information of the tops of the descents in the permutations. These approaches are similar, but just the last one allows us to increase the number of variables in each step. The first one only allows us to go to the bivariate setting.
\end{objetivo}

In order study our bounds in detail, we remind that we will use a particular form of representing and extracting asymptotic expansions for our setting. In particular, we will expand the asymptomatic growth using very simple functions. Our reference for asymptotic expansions is \cite{malham2005introduction}, although we will use a very tiny and particular segment of this theory.

\begin{remark}[Necessary tools in asymptotics to push beyond]
We already mentioned that our asymptotic scale will be given by functions of the form $a^{n}$ for $a\in\mathbb{R}_{\geq0}$. Additionally, we fix this scale all the time in the sense that, whenever we compute some term of the growth of a sequence in this scale, we understand that we do not ever change the scale when we extract more terms of the asymptotic expansion of that same sequence. Thus our choice of the scale used is homogeneous and invariable along the whole text.
\end{remark}
The relaxation has proven itself valuable at determining an inner bound having a tight first growth term for these roots. However, as we saw, we could establish this first asymptotic term using Stanley's approach already (although not a bound). In future parts we will even see that the method of Sobolev will in fact provide better bounds in a first iteration. This makes us wonder what is the real power of the relaxation. We can mention several things at this point.

\begin{remark}[The real power of the relaxation]
The main power of the relaxation comes from the fact that it is relatively easy to compute using only the degree three part of the polynomial at play. Additionally, when we deal with the multivariate extension, the relaxation not only bounds the original univariate polynomial, but the whole ovaloid of the multivariate extension, as we saw in \cite{main}. This uniformity means that we can extract knowledge not just about the original univariate polynomials but also about the very close polynomials that lie in restrictions close to this one. Moreover, the relaxation can be successfully combined with other methods to approximate roots in order to sharpen the bounds obtained even more before performing more demanding iterations that sometimes can run into the problem of numerical instability. Remarkably, we will combine it with the DLG method used by Sobolev in \cite{sobolev2006selected}. We refer the reader to \cite{macnamee2007numerical,macnameeii} to learn more about instability problems of some numerical methods used to estimate roots of polynomials. Also, around \cite[Equation 27]{macnameeii} one can find interesting information about methods to estimate roots of polynomials computing eigenvalues of companion matrices. These methods lie therefore certainly close to the method using the relaxation that we are exploring and exploiting here. However, there, no care is taken about RZ-ness and therefore a big chunk of the theory that helps us to build the story of our method is missing and non-reconstructible from the approaches taken there. But not everything is bad: at least, we found a broad family where our method could fit. Finally, applying the relaxation to these well-known polynomials and using the just discussed extension methods in order to enhance the power of the relaxation allow us to reveal new venues in the applications, limits and extensions of the relaxation itself and also in new features and generalizations of the polynomials we are studying. All in all, applying the relaxation results in an irresistible impulse to look in several different and varied directions. All these new directions emanate from our early explorations about the many venues where this relaxation could be used or exploited in order to achieve further understanding of real-rooted (or RZ) polynomials and their features.
\end{remark}

Before continuing towards the next sections of this work we have to make sure that the reader completely understands the nuances of our approach. Making this clear will ensure a smoother trip through our future explorations and, therefore, a greater enjoyment along the paths we will follow next.

\begin{warning}[Ground necessary before going beyond]
What we did in this section can be considered an easy example of what is to come afterwards. Guessing the correct structure for the approximated eigenvectors, dealing with the huge expressions that will appear, extracting several terms of the asymptotics and computing optima will become more complicated due to the cumbersome length of some expressions. However, we saw already the main problems that will appear and how we will approach them in order to overcome them.
\end{warning}

\setlength{\emergencystretch}{3em}%
\begin{remark}
Notice that, although we talk frequently about ``approximated eigenvectors'', we do not have, in principle, any guarantee that these guesses approximate the actual eigenvectors. We just know that they get \textit{close enough} for our interests to provide good approximations to a generalized eigenvalue problem through linearization. This is how we obtain our bounds. For this reason, we will sometimes just refer to these as ``guesses'' instead of ``approximations''.
\end{remark}
\setlength{\emergencystretch}{0em}%

In the next section, we show the problems that might appear when we have more variables. In particular, we start by the bivariate extension. For this, we will need to look again closely at interlacing and the tools it offers in \cite{fisk2006polynomials}.

\section{Interlacing goes beyond Renegar derivative}\label{sectioninterlacing}

In Part \ref{I}, we used the Renegar derivative to extend our polynomials into something admitting one additional variable. Here we will abandon such rigidity and instead of the Renegar derivative we will work with any interlacer. This will allow us to explore how the relaxation reacts when we use interlacers arising from different processes. In the future we will briefly explore interlacers with higher structure and extremal interlacers as these improving the relaxation the most. In order to start this study here, we begin recalling our notion of interlacing and its amplifications. We stay in the next particular case extracted from \cite[Definition 1.1]{kummer2015hyperbolic} but we will expand it accordingly here in order for it to fit with the notions we need to use from \cite[Section 8]{blekherman2023linear} and, by extension, \cite[Section 5]{branden2007polynomials}.

\begin{definicion}
Let $p,q\in\mathbb{R}[x]$ be univariate real-rooted polynomials with $\deg(p)-\deg(q)\leq1$ and decompose $p=k_{p}\prod_{i=1}^{\deg(p)}(x-a_{i})$ and $q=k_{q}\prod_{i=1}^{\deg(q)}(x-b_{i})$ for some $k_{p},k_{q},a_{i},b_{j}\in\mathbb{R}$ for all $i\in[\deg(p)]$ and $j\in[\deg(q)]$ so that $a_{1}\leq\cdots\leq a_{\deg(p)}$ and $b_{1}\leq\cdots\leq b_{\deg(q)}.$ We say that $q$ \textit{interlaces} $p$ and write $q\ll p$ if $a_{i}\leq b_{i}\leq a_{i+1}$ for all $i\in[\deg(p)-1]$. For the multivariate generalization, if $p\in\mathbb{R}[\mathbf{x}]$ is hyperbolic with respect to $e\in\mathbb{R}^{n}$ and $q\in\mathbb{R}[\mathbf{x}]$ is homogeneous of degree $\deg(p)-1$ we say that $q$ \textit{interlaces $p$ with respect to $e$} if $q(te+a)$ interlaces $p(te+a)$ for all $a\in\mathbb{R}^{b}$ and we write $q\ll_{e} p$.
\end{definicion}

We can compare with \cite[Section 5]{branden2007polynomials} in order to ensure that we can safely include self-interlacing.

\begin{convencion}
We consider that a polynomial interlaces itself. This is helpful and is in agreement with the definition of \textit{proper position} in \cite[Section 5]{branden2007polynomials} and with the notion of interlacing used in \cite{blekherman2023linear}, as it can be seen at the beginning of their proof of \cite[Lemma 7.6]{blekherman2023linear}.
\end{convencion}

Two important consequences can be extracted from the definition. We explore them.

\begin{remark}
For why $e$ is not actually that special, see also \cite[Theorem 2.1]{kummer2015hyperbolic} and remember Corollary \ref{rstoposorth} also commented in the discussion after \cite[Definition 5.2]{kummer2015hyperbolic}. This will allow us expanding the definition to stable polynomials with respect to vectors in the connected component of the unit vector corresponding to the homogenizing variable that we will use to transform RZ polynomials into real stable ones (hyperbolic with respect to every direction in the positive orthant). It is clear that if $q$ interlaces $p$, then $q$ is also hyperbolic with respect to $e$. Additionally, this definition directly extends also to RZ polynomials using the characterization in Proposition \ref{dehomo}. In this case, since we have to dehomogenize by the homogenizing variable to come back to the RZ setting, we write $q\ll p$ with any reference to $e.$
\end{remark}

Intermediate between going fully multivariate and the univariate relaxation is an application of a kind of Nuij type extension as we tried when we wanted to extend the amount of variables of our polynomials in the previous part using the Renegar derivarite as an interlacer easy to build. For this application, we will need to keep real stability under certain sums of interlacing polynomials, something we can get under mild conditions using the next result.

\begin{lema}\cite[Lemma 8.1]{blekherman2023linear}
Assume that the homogeneous stable polynomials $p,q\in\mathbb{R}[\mathbf{x}]$ have nonnegative coefficients and a common interlacer. Then $p+q$ is stable.
\end{lema}

This has an immediate consequence important for us. We basically rewrite this in our setting. In order to proceed, we also need the following.

\begin{proposicion}
The homogenization of any real stable polynomial $p\in\mathbb{R}[\mathbf{x}]$ with non-negative coefficients is also real stable.
\end{proposicion}

\begin{proof}
The proof is an easy exercise. For a hint, see \cite{leakeexer}. For a proof, see \cite[Proposition 5.3]{pemantle2012hyperbolicity} or, originally, \cite[Theorem 4.5]{borcea2009negative}.
\end{proof}

Now the next corollary is clear. We can translate to our setting the lemma above. Recalling that for univariate polynomials being real-rooted is equivalent to being real stable, we can now easily prove what we need in order to proceed.

\begin{corolario}\label{interlacerextension}
Assume that the real-rooted polynomials $p,q\in\mathbb{R}[x]$ have nonnegative coefficients, $0$ is not a root of any of them and $q\ll p$. Then $p^{h}+(yq)^{h}$ is stable. As a consequence, $p+yq\in\mathbb{R}[x,y]$ is RZ.
\end{corolario}

\begin{proof}
As $p,q\in\mathbb{R}[x]$ are real-rooted and have nonnegative coefficients, we have that $p^{h},(yq)^{h}=yq^{h}\in\mathbb{R}[x_{0},x,y]$ are homogeneous, real stable and have nonnegative coefficients. We need a common interlacer. We claim that $q^{h}$ is so. First we see that $q^{h}\ll_{u} p^{h}$ with $u=(1,0,0)\in\mathbb{R}^{3}.$ We have to see that $q^{h}(tu+a)\ll p^{h}(tu+a)$. Rewriting, we see that we want to prove that $x_{0}^{\deg(q)}q(\frac{x}{x_{0}})|_{tu+a}\ll x_{0}^{\deg(p)}p(\frac{x}{x_{0}})|_{tu+a}$. So we want to see that $(a_{0}+t)^{\deg(q)}q(\frac{a_{1}}{a_{0}+t})\ll (a_{0}+t)^{\deg(p)}p(\frac{a_{1}}{a_{0}+t})$ and calling $s=a_{0}+t$ we get that we want to see $s^{\deg(q)}q(\frac{a_{1}}{s})\ll s^{\deg(p)}p(\frac{a_{1}}{s}),$ which is clear using the common properties of inequalities and that $q\ll p$. Proceeding similarly we also get $q^{h}\ll_{u} yq^{h}.$ Then we can apply the lemma to ensure that $p^{h}+(yq)^{h}$ is stable. Now we just have to use what we just proved and Proposition \ref{dehomo} to dehomogenize into a RZ polynomial for the last part.
\end{proof}

The last part admits an alternative direct proof that we also introduce here because it is more direct, although it does not show the relations between the different notions of interlacing in the different settings, as the previous discussion did. In particular, we have to introduce first the notion of interlacing for RZ polynomials.

\begin{definicion}[RZ interlacing]
Let $p,q$ be RZ polynomials. We say that $q$ \textit{interlaces} $p$ if, for each direction $a\in\mathbb{R}^{n}$ we have that the univariate polynomials $q(\frac{a}{t})$ and $p(\frac{a}{t})$ verify that $\deg(q(\frac{a}{t}))+1=\deg(p(\frac{a}{t}))$ and $t^{d-1}q(\frac{a}{t})$ interlaces $t^{d}p(\frac{a}{t})$. We also denote $q\ll p$ when this happens.
\end{definicion}

Notice how we had to pass through infinity, using the reciprocals, in order to use the univariate definition of interlacing to extend it to the multivariate one naturally. The direct proof of the fact is therefore the following.

\begin{teorema}
Let $p,q\in\mathbb{R}[\mathbf{x}]$ be RZ polynomials with $q\ll p$. Then $p+yq\in\mathbb{R}[\mathbf{x},y]$ is RZ.
\end{teorema}

\begin{proof}
Fix $(a,a_{n+1})\in\mathbb{R}^{n}\times\mathbb{R}$. Then we have that $(p+yq)(ta,ta_{n+1})=p(ta)+ta_{n+1}q(ta)$. We need to see that this last univariate polynomial in the variable $t$ has only real roots. A univariate polynomial has only real roots if and only if its reciprocal has only real roots. Moreover, two univariate polynomials interlace if and only if their reciprocals interlace. The reciprocal of $p(ta)+ta_{n+1}q(ta)$ is $t^{d}(p(\frac{a}{t})+\frac{a_{n+1}}{t}q(\frac{a}{t}))=t^{d}p(\frac{a}{t})+a_{n+1}t^{d-1}q(\frac{a}{t})$, where we used that $d$ is the degree of $p(ta)$ and therefore $d-1$ is the degree of $q(ta)$. Now, we can use the fact that we obtained a real linear combination of the reciprocals of the original polynomials and an application of the well-known Obreschkoff-Dedieu Theorem (see \cite[Theorem 8]{branden2004operators}) to immediately finish this proof.
\end{proof}

What we will end up seeing as a consequence of this is that, although we saw that some simple stability preservers do not improve the performance of the relaxation, some others do in fact accomplish this \textit{in some sense along the diagonal}. This contrast between different stability preservers will become our main focus in this thesis.

\begin{observacion}[Generality of the statement]\label{pyp1observastab}
Observe that the fact that $p+yp{(1)}$ is RZ if and only if $p$ is RZ could have been proved used that such transformation is a stability preserver. This is what it is done, for example, in \cite[Section 2]{nuijtype} to prove that transformations of this type preserve stability. Here we proved something more general because we did it for \textit{any} interlacer and not just for these obtained through differential operators. However, at the end, in the next section we will apply it to sequences coming from stability preservers because we will form the polynomials $A_{n}+yA_{n-1}$ and we know that $A_{n}=(n+1)xA_{n-1}+(1-x)(xA_{n-1})'.$ Looking at homogenization we see that the recurrence for them is $A_{n}^{h}=(n+1)xA^{h}_{n-1}+(x_{0}-x)\frac{\partial}{\partial x}(xA^{h}_{n-1})$ so in the homogenized sense our transformation with the Renegar derivative is therefore $A_{n}^{h}+yA_{n-1}^{h}=(nx+x+y)A^{h}_{n-1}+(x_{0}-x)\frac{\partial}{\partial x}(xA^{h}_{n-1})$ and thus we can easily see why the stability is preserved and even build a recurrence for this new family. 
\end{observacion}

Corollary \ref{interlacerextension} is very important for this section and overall for this thesis as it allows us to now expand beyond what we did for the Renegar derivative into other interlacers. These interlacers will appear in Part \ref{IV} when we talk about how to beat the bounds provided by Sobolev in \cite{sobolev2006selected}. First, we explore how the relaxation reacts to these new expansions using new interlacers as opposed to what happened with the Renegar derivative, which did not produce any improvement.

\section[On the similarities and differences of what follows]{On the similarities and differences of the structure of what follows}

We use the next short section to explain briefly why the structure of four of the following last five sections rhyme. In the first two sections that follow we explain the jump of the univariate case to the bivariate using the construction with interlacers presented in the section above. This jump requires us to talk about possible problems arising in our analysis when we go bivariate. Thus, we need two sections: one for cautions and the other for the approach itself and the results we can obtain. After this, it follows another short section about the benefits of increasing the number of variables in each step. This short section announces and promotes our next jump in the number of variables admitted. Thus, this middle section prepares us for the last two sections of this chapter. These last two sections follow the same structure as the two sections that introduced the bivariate jump. The only difference is that the last two sections deal now with multivariateness instead of bivariateness. Thus, these last two sections follow the same structure: one for cautions and the following for the approach itself and the obtained results. In this sense, the reader should not be confused about the similarities between these sections and their structures, as this similarity is totally intentional because the approaches follow an analogous pattern. However, at the end, we will see that going multivariate gives a true edge on the asymptotic behaviour to the relaxation over just going bivariate. This is not surprising because having more variables at play allows us to increase the size of the LMPs defining the relaxation. Thus we can collect finer information about the combinatorial object encoded by the Eulerian polynomials. This information seems to travel well through the relaxation when there are enough variables. Hence, as we have thus more room for fitting this further information into the relaxation because of an increased size of the LMPs defining it, we can expect to obtain better (asymptotic) results for our bounds obtained in this way. Once we have made this comment and the reader is warned about what follows, we can proceed with our first jump. This is the jump that goes from univariateness to bivariateness. As we mentioned, this is the topic of the next two sections.

\section[Bivariate cautions]{Cautions before using the relaxation in the bivariate case}

We saw before that using the Renegar derivative to increase the number of variables does not change the relaxation. However, we just proved that this kind of extension works for other interlacers obtained in more exotic ways and we know that Eulerian polynomials interlace. So our intermediate case here consists in considering the bivariate polynomials $B_{n}(x,y):=A_{n} + y A_{n-1}\in\mathbb{N}[x,y]$, which we just saw that are RZ in Corollary \ref{interlacerextension}. For these, we can also try to compute the best bound that can be obtained through the use of the relaxation. However, this can become too complicated as the matrices will now have size $3\times3$. We will explore other ideas connected to guessing approximations for eigenvectors solving generalized eigenvalue problems in terms of their corresponding kernels by linearization. We have to begin, as usual, computing the relaxation. Now our matrices have size $3\times 3$ and we have to compute the corresponding values of the associated $L$-forms. Fix $p=B_{n}$ for short in what follows.

\begin{computacion}
Proceeding as usual by degree of the evaluation and using that some values coincide with previously computed ones, we obtain that, for degree $0$, we have $L_{p}(1)=n$; for degree $1$, we have \begin{gather*}L_{p}(x)=\coeff(x,p)=E(n+1,1)=2^{n+1}-(n+2), \mbox{\ and}\end{gather*} \begin{gather*}L_{p}(y)=\coeff(y,p)=1;\end{gather*} for degree $2$, we have \begin{gather*}L_{p}(x^2)=-2\coeff(x^2,p)+\coeff(x,p)^2=\\-2E(n+1,2)+E(n+1,1)^2=2 - 3^{1 + n}2 + 4^{1 + n} +n,\end{gather*} \begin{gather*}L_{p}(y^2)=-2\coeff(y^2,p)+\coeff(y,p)^2=\coeff(y,p)^2=1, \mbox{\ and}\end{gather*} \begin{gather*}L_{p}(xy)=-\coeff(xy,p)+\coeff(x,p)\coeff(y,p)=\\-\coeff(x,A_{n-1})+\coeff(x,p)=-E(n,1)+E(n+1,1)=\\-(2^{n}-(n+1))+(2^{n+1}-(n+2))=-1 + 2^n;\end{gather*} and finally, for degree $3$, we have \begin{gather*}L_{p}(x^3)=3\coeff(x^3,p)-3\coeff(x,p)\coeff(x^2,p)+\\\coeff(x,p)^3=-2 - 2^{1 + n} 3^{2 + n} + 4^{1 + n}3 + 8^{1 + n} - n,\end{gather*} \begin{gather*}L_{p}(y^3)=\coeff(y,p)^3=1,\end{gather*} \begin{gather*}L_{p}(x^{2}y)=\coeff(x^{2}y,p)-\coeff(x,p)\coeff(xy,p)-\\\coeff(y,p)\coeff(x^2,p)+\coeff(x,p)^2\coeff(y,p)=\\\coeff(x^{2},A_{n-1})-\coeff(x,A_{n})\coeff(x,A_{n-1})-\\\coeff(x^2,A_{n})+\coeff(x,A_{n})^{2}=\\E(n,2)-E(n+1,1)E(n,1)-E(n+1,2)+E(n+1,1)^2=\\(3^{n}-(n+1)2^{n}+\frac{1}{2}(n)(n+1))-(2^{n+1}-(n+2))(2^{n}-(n+1))-\\(3^{n+1}-(n+2)2^{n+1}+\frac{1}{2}(n+1)(n+2))+(2^{n+1}-(n+2))^2=\\1 - 2^{n} + 2^{1 + 2 n} - 3^{n}2, \mbox{\ and}\end{gather*} \begin{gather*}L_{p}(xy^2)=\coeff(y^{2}x,p)-\coeff(y,p)\coeff(xy,p)-\\\coeff(x,p)\coeff(y^2,p)+\coeff(y,p)^2\coeff(x,p)=\\-\coeff(x,A_{n-1})+\coeff(x,A_{n})=-E(n,1)+E(n+1,1)=\\-(2^{n}-(n+1))+(2^{n+1}-(n+2))=-1 + 2^n.\end{gather*} 
\end{computacion}

These $L$-forms are very similar to the ones we computed for the univariate case. Putting everything together, we obtain the matrix of the relaxation of the just introduced bivariate polynomials.

\begin{computacion}
The matrix of the relaxation of the new bivariate polynomials is $M(x,y)=$ \begin{gather*}
    \begin{pmatrix}
        n & 2^{n+1}-(n+2) & 1\\
        2^{n+1}-(n+2) & 2 - 3^{1 + n}2 + 4^{1 + n} +n & -1+2^{n}\\
        1& -1+2^{n} & 1
    \end{pmatrix} +\\\hspace*{-1cm}
     x\begin{pmatrix}
        2^{n+1}-(n+2) & 2 - 3^{1 + n}2 + 4^{1 + n} +n & -1+2^{n}\\
        2 - 3^{1 + n}2 + 4^{1 + n} +n & -2 - 2^{1 + n} 3^{2 + n} + 4^{1 + n}3 + 8^{1 + n} - n & 1 - 2^{n} + 2^{1 + 2 n} - 3^{n}2 \\
        -1+2^{n} & 1 - 2^{n} + 2^{1 + 2 n} - 3^{n}2 & -1 + 2^n
    \end{pmatrix}\\
    + y\begin{pmatrix}
        1 & -1 + 2^n & 1\\
        -1 + 2^n & 1 - 2^{n} + 2^{1 + 2 n} - 3^{n}2 & -1 + 2^n\\
        1 & -1 + 2^n & 1
    \end{pmatrix}.
\end{gather*}
\end{computacion}

As we are mainly interested now only in root bounds for the univariate polynomial, we can set $y=0$ above and study the resulting matrix polynomial. We could be brave enough and find the optimal bound given by that relaxation, but that would require solving a very long cubic polynomial hoping that we have made all the correct choices of roots along the way while dealing with awfully long expressions. It is indeed not impossible to do. However, here, for brevity and sanity, we prefer to use a technique that will also help us in the future, when solving equations exactly will not even be a far-fetched possibility. Hence, this way of approaching will save us some energy at the same time that it allows us to contemplate, through a kind of toy example, how we will proceed in the future. We proceed.

\begin{remark}
First, we make some numeric computations that allow us to guess that a vector of the form $v=(\alpha,3,-8)$ for some still unknown map $\alpha\colon\mathbb{N}\to\mathbb{R}$ will produce a linearization $vM(x,0)v^{T}=vAv^{T}+xvBv^{T}$ that produces a bound for the absolute value of our extreme. We obtain such bound solving the inequation $vAv^{T}+xvBv^{T}\geq0$ in the form $x\geq\frac{-vAv^{T}}{vBv^{T}}$. Thus, the other extreme root of the univariate Eulerian polynomial, by palindromicity, is bounded from above by $\frac{1}{x}\geq\frac{vBv^{T}}{vAv^{T}}$ root in the form $\frac{-vBv^{T}}{vAv^{T}}$. Here, we will look at an inner bound for the absolute value of that leftmost root. Therefore we are interested in maximizing $\frac{vBv^{T}}{vAv^{T}}$ in terms of the parameter $\alpha$ so that we can make it get closer to the actual root than the bound obtained through the application of the relaxation to the univariate Eulerian polynomial that we analyzed above.
\end{remark}

Hence, now we just have to consider $\frac{vBv^{T}}{vAv^{T}}$ and find, for each $n$, the $\alpha(n)$ that maximizes that fraction. This is easy.

\begin{procedimiento}
For short, we write the fraction we are interested to maximize as $\frac{N}{D}$ so its derivative is $\frac{N'D-ND'}{D^2}$. We are looking for our optimal $a$ among the roots of the polynomial in $a$ given by the numerator $N'D-ND'$, which depends on the parameter $n$ and is quadratic in $a$. We solve this quadratic equation as usual and obtain two roots $\frac{-b\pm\sqrt{b^2-4ac}}{2a}$. We have to find a method to determine which one of these roots corresponds to the maximum of our fraction $\frac{N}{D}$.
\end{procedimiento}

Before seeing this, we have to notice that $D$ must always be nonnegative for any $\alpha$ due to the fact that the relaxation has always its initial matrix PSD. For $D$, we have to check that, for the chosen $\alpha$, we have $D\neq0.$ We do this using that the discriminant of $D$ is $\Delta=-9\ 4^{n+2} n+8\ 3^{n+3} n-9\ 2^{n+4} n+3\ 2^{n+6} n-184 n-21\ 2^{n+5}+9\ 2^{2 n+4}+784$, which is clearly negative when $n$ is big enough. Moreover, the quadratic term in $D$ viewed as a polynomial in $\alpha$ is $n\alpha^{2}\neq0.$ Thus, all this tells us that $D$ is a quadratic polynomial with no real zeros and therefore strictly positive for whichever $\alpha$ we choose when $n$ is big enough. This solves completely one of our problems. We also need that $N>0$ for $n$ big enough so that we can isolate in the inequality with the correct orientation of the signs. Since $N=\alpha^2 2^{n+1}-\alpha^2 n-2 \alpha^2-4 \alpha 3^{n+2}-\alpha 2^{n+4}+3 \alpha 2^{2 n+3}+6 \alpha n+28 \alpha+32\ 3^{n+1}+27\ 4^{n+1}+9\ 8^{n+1}+3\ 2^{n+4}-2^{n+1} 3^{n+4}+2^{n+6}-3\ 2^{2 n+5}-9 n-130$, the requirement that $N>0$ for $n$ big enough\footnote{This not necessary to check because we can check later directly with the choice of $\alpha$ we will eventually make.} means that it is enough if we choose the function $\alpha$ so that \begin{gather}\label{conditionsa}
\alpha>1 \mbox{\ and\ } \lim_{n\to\infty}\frac{N}{2^{n+1}\alpha^{2}}=1.\end{gather} We can therefore continue with the quest to find that optimal $\alpha$.

\begin{procedimiento}
Considering the derivative $\frac{N'D-ND'}{D^2}$ of the fraction $\frac{N}{D}$, we have to look for our optimal $a$ among the roots of the polynomial in $a$ given by the numerator $N'D-ND',$ which is quadratic in $a$ and depends on $n$. We obtain two roots using the well known quadratic formula $\alpha=\frac{-b\pm\sqrt{\Delta}}{2a}$ with $\Delta:=b^{2}-4ac$ now the discriminant of $N'D-ND':=a\alpha^{2}+b\alpha+c$. We have to determine which one gives a maximum for $\frac{N}{D}$ as a function of $\alpha$.
\end{procedimiento}

To see this, observe that we have already seen that the discriminant of the denominator $D$ verifies $\Delta<0$ once $n$ is big enough (which was expectable beforehand also because the fraction $\frac{N}{D}$ cannot go to infinity as a function of $\alpha$ because it is bounded by the absolute value of the extreme root of the $n$-th Eulerian polynomial for any given $n$) so it has no real zeros and therefore we can easily deduce now that the roots of $N'D-ND'$ correspond to a maximum and a minimum of $\frac{N}{D}$ and that such function is continuous as a function of $\alpha$.

\begin{observacion}
It might seem worth studying the horizontal asymptotes of such function but observe that $\lim_{\alpha\to\infty}\frac{N}{D}=\frac{-2 + 2^{1 + n} - n}{n}=\lim_{\alpha\to-\infty}\frac{N}{D},$ which does not directly help in our decision. However, now it is evident that, as a function of $\alpha$, this bivariate bound $\bi(n):=\frac{N}{D}$ cuts the horizontal asymptote $f(\alpha)=\frac{-2 + 2^{1 + n} - n}{n}$ in exactly one point and therefore the difference with this horizontal asymptote in any of the limits determines the relative position of the maximum and the minimum, as there are just two options on how this function looks.
\end{observacion}

This observation finally puts us in the correct way to distinguish between the maximum and the minimum. The names in the next compation are locally set.

\begin{computacion}
Observe that the difference between the function $\bi(n)$ and the asymptote equals \begin{gather*}
    \bi(n) - \frac{-2 + 2^{1 + n} - n}{n}= \frac{v}{w}
\end{gather*} with $v=5\alpha2^{n+4}-3\alpha2^{2 n+3}-4 \alpha3^{n+2} n+3\alpha2^{n+3} n-\alpha2^{n+4} n+3\alpha2^{2 n+3} n-12\alpha n-56\alpha-9\ 8^{n+1}-65\ 2^{n+2}+2^{n+2} 3^{n+3}-4\ 3^{n+3}-3\ 2^{n+5}+9\ 2^{2 n+3}+3\ 2^{2 n+5}-9\ 2^{n+1} n+32\ 3^{n+1} n+9\ 8^{n+1} n+9\ 4^{n+2} n-2\ 3^{n+3} n-2^{n+1} 3^{n+4} n+2^{n+6} n-3\ 2^{2 n+5} n+18 n+260$ (which has at maximum one zero because it is linear) and $w=\alpha^2 n^2-6\alpha n^2+3\alpha 2^{n+2} n-28 \alpha n+9 n^2+9\ 4^{n+1} n-2\ 3^{n+3} n-3\ 2^{n+4} n+130 n$ so the denominator $w$ is always positive at infinity while the numerator $v$, when $|\alpha|$ goes to infinity, is dominated by the term multiplying $\alpha$ given by $-4\ 3^{n+2} n+2^{n+3} n+3\ 2^{2 n+3} n-12 n+5\ 2^{n+4}-3\ 2^{2 n+3}-56,$ which is dominated by the term $3\ 2^{2 n+3} n>0$ so the $\lim_{\alpha\to-\infty}\frac{v}{w}=0^{-}$ (because $v\approx2^{2 n+3}3n\alpha<0$ as $\alpha\to-\infty$) which implies that, for $\alpha$ close enough to $-\infty$, $\bi(n)-\frac{-2 + 2^{1 + n} - n}{n}<0$ or, equivalently, $\bi(n)<\frac{-2 + 2^{1 + n} - n}{n}$ and $\lim_{a\to\infty}\frac{v}{w}=0^{+}.$
\end{computacion}

Looking at the computations, we can see therefore that the function $\bi(n)$, when viewed from left to right, decreases from below the asymptote until it reaches a minimum under it and then increases towards a maximum reached over the asymptote after crossing it before decreasing back towards the asymptote approaching it from above.

\begin{remark}
All in all, we have established that the maximum is reached at the rightmost of the two zeros of $N'D-ND'.$ Thus, we have determined which root of $N'D-ND'$ corresponds to a maximum of $\frac{N}{D}$ and now we know exactly the optimal choice of $\alpha$ in our way towards the best bound attainable through the family $(\alpha,3,-8)$ we are studying in this section.
\end{remark}

Do not be confused with the location of the rightmost root. We explain this in the next warning.

\begin{warning}
Observe that, if we write (as we did above) $N'D-ND':=a\alpha^{2}+b\alpha+c$, then $a=4\ 3^{n+2} n+3\ 2^{n+3} n+2^{n+5} n-2^{n+6} n+3\ 2^{2 n+3} n-3\ 2^{2 n+4} n+12 n+5\ 2^{n+4}-5\ 2^{n+5}-3\ 2^{2 n+3}+3\ 2^{2 n+4}+56\sim-3\ 2^{2 n+3} n$ and, therefore, for $n$ big enough, this coefficient $a$ is negative. Thus, the rightmost of the two zeros of $N'D-ND'$ corresponds to $\frac{-b-\sqrt{b^2-4ac}}{2a}$ because of the change of sign that happens because of the denominator $2a$ being negative. 
\end{warning}

Thus, in fact, this means that our choice of $\alpha$ to reach this best bound must be the one corresponding to the minus sign before the square root, i.e., we fix now $\alpha:=\frac{-b-\sqrt{b^2-4ac}}{2a}$. Now we have to extreme our care again.

\begin{remark}[Necessity of freeing radicals]
We remind that there was a condition that we should check for $\alpha$, as exposed in Equation \ref{conditionsa}. In the spirit of our warning, observe that the terms winning in the numerator (inside and outside of the root) annihilate as the expression obtained if we only look at dominant terms is, as $-b\sim 9 n 2^{3 n+4}$ and $b^2-4ac\sim 81\ 4^{3 n+4} n^2$, \begin{gather*}9 n 2^{3 n+4}-\sqrt{81\ 4^{3 n+4} n^2}=0.\end{gather*}
\end{remark}

Therefore, in order to really determine the growth of the numerator, we have to proceed differently (as we will have to do many times in the future).

\begin{computacion}
In particular, here we see that the conjugate does not annihilate in dominant terms as it gives $9 n 2^{3 n+5}$ and therefore we can multiply by it to obtain that the corresponding difference of squares is dominantly $3^{n+5} 16^{n+2}n$ so our numerator is dominantly $\frac{3^{n+5} 16^{n+2}n}{9 n 2^{3 n+5}}=6^{n+3}.$ For the denominator we proceed in the usual way and obtain that it grows like $ -2^{2 n+4}n\ 3$ and therefore we get that $\alpha\sim$ \begin{gather*}\frac{6^{n+3}}{-2^{2 n+4}n\ 3}=-\frac{3^{n+2}}{n2^{n+1}}\end{gather*} and so now an easy computation makes us sure of the fact that $\alpha$ verifies the condition that $N(\alpha)>0$ for big enough $n$ because $N(\alpha)\to\infty$ when $n\to\infty$. In order to see this, observe that, when we substitute for the recently computed growth of $\alpha$, we get that $N(\alpha)\sim N(-\frac{3^{n+2}}{n2^{n+1}})\to\infty$ when $n\to\infty.$
\end{computacion}

This computation proves that $\alpha$ thus chosen verifies our growth constraints and therefore it works as a choice for providing a \textit{correct} bound. Now, finally, we only have to check that this choice of $\alpha$ provides a bivariate bound \textit{effectively} improving the optimal univariate bound obtained above.

\begin{remark}
Substituting $\alpha$ for the just computed root $\alpha:=\frac{-b-\sqrt{b^2-4ac}}{2a}$ of $N'D-ND'=a\alpha^{2}+b\alpha+c$ in $\frac{N}{D}$ gives us the improved bound we were pursuing here. Hence, the only thing that we still need to do is comparing this new bound $\bi(n)$ with the previous optimal bound $\un(n)$ obtained through the application of the relaxation to the univariate Eulerian polynomial. Remember that this last bound does not come from a linearization and it is, therefore, the optimal bound obtained through the application of the relaxation to the univariate Eulerian polynomials. We are going to compute $\bi-\un$ and we expect it to be positive in order to confirm the improvement we are seeking for here.
\end{remark}

We proceed with the computations step by step for maximum clarity. We have to make a slight change.

\begin{notacion}[Renaming for standardization within text]
We rename $\alpha$ as $y$ in what follows because $\alpha$ will have a different use in the rest of this section. We named the root as $\alpha$ that way because $y$ was chosed as the extra variable in the construction of the bivariate Eulerian polynomials through adding interlacers. However, in the future $y$ will take the role of this parameter in the vector we have to fix. As the definition of bivariate Eulerian polynomials is already too far in the text, we can therefore come back now to this standard and call this root $y$ in order to fit our text into the same scheme as it will have in future section where similar techniques will be applied. So, from now on, the names $\alpha,v,w$ are free again (we will use them soon). What $\alpha$ represented above is now called $y$ (in order to stay in accordance with the future of this text).
\end{notacion}

\begin{observacion}
First, note that we can expand the root as $y=\frac{f-\sqrt{g}}{h}$ so $y^{2}=\frac{f^{2}+g-2f\sqrt{g}}{h^{2}}$ and we can substitute these values in the expression of $\bi=\frac{N}{D}$ and kill denominators inside $N$ and $D$ writing $\bi=\frac{N(y)}{D(y)}=\frac{h^{2}N(y)}{h^{2}D(y)}$ so at the end we obtain the manageable expression $\bi=\frac{\alpha+\beta\sqrt{g}}{\gamma+\delta\sqrt{g}}.$ Finally, remember that $\un=\frac{p+r\sqrt{q}}{s}$ and, therefore, in order to study the difference between these two bounds, the most convenient expression is $\bi-\un=$\begin{gather*}
\frac{s(\alpha+\beta\sqrt{g})-(p+r\sqrt{q})(\gamma+\delta\sqrt{g})}{s(\gamma+\delta\sqrt{g})}=\\\frac{(s\alpha-p\gamma)+(s\beta-p\delta)\sqrt{g}-r\gamma\sqrt{q}-r\delta\sqrt{gq}}{s(\gamma+\delta\sqrt{g})}=\\\frac{k+v+u+w}{s(\gamma+\delta\sqrt{g})}\end{gather*} with $k:=s\alpha-p\gamma,v:=(s\beta-p\delta)\sqrt{g},u:=-r\gamma\sqrt{q}$ and $w:=-r\delta\sqrt{gq}$.
\end{observacion}

We want to be able to estimate correctly the dominant behaviour of both the numerator and the denominator. The main problem with these expressions is that they contain square roots and therefore we need to deal with these carefully because we cannot directly try to proceed with dominant terms due to the many cancellations happening both in the denominator and in the numerator. We see in detail the problems produced by these cancellations.

\begin{remark}
The dominant terms of the sumands in our decomposition of the numerator are \begin{gather*}
k\sim -2^{11 n+15} 3^{n+5} n^{4},\\
v\sim(2^{8 n+11} 3^{n+3} n^4)\sqrt{81\ 4^{3 n+4} n^2}=2^{11 n+15} 3^{n+5} n^4,\\
u\sim-(2^{2 n+3} 3^{n+1})(2^{6n+9} n^{4}3^{4})\sqrt{2^{6n+6} n^{2}}=-2^{11 n+15} 3^{n+5} n^4,\\
w\sim-(2^{2 n+3} 3^{n+1})(-3^{2} 2^{3 n+5} n^3)\sqrt{(81\ 4^{3 n+4} n^2)(2^{6n+6} n^{2})}=2^{11 n+15} 3^{n+5} n^4.\end{gather*}
It is easy to see that these dominant terms annihilate, which is a big problem because we need to keep track of the terms surviving the big cancellation in order to be able to really establish the asymptotics of our difference.
\end{remark}

In order to continue then we have to find the way to keep the biggest surviving terms. We do this freeing some radicals through the use of the correct conjugates. The situation in the denominator is similar.

\begin{computacion}
In particular, we have the luck that the conjugate expressions in the numerator \begin{gather*}(k+u)-(v+w)\sim-2^{11 n+15} 3^{n+5} n^{4}-2^{11 n+15} 3^{n+5} n^4\\-(2^{11 n+15} 3^{n+5} n^4+2^{11 n+15} 3^{n+5} n^4)=\\-4\cdot2^{11 n+15} 3^{n+5} n^4\end{gather*} and in the denominator \begin{gather*}\gamma-\delta\sqrt{g}\sim 81\ 8^{2 n+3} n^{3}-(-9) 2^{3 n+5} n^2\sqrt{81\ 4^{3 n+4} n^2}=\\ 81\ 8^{2 n+3} n^{3}-((-9) 2^{3 n+5} n^{2})(9\ 2^{3 n+4} n)=81\ 4^{3 n+5} n^3\end{gather*} do not annihilate.
\end{computacion}

Now it is easy to see the following problematic cancellation.

\begin{observacion}
Notice that \begin{gather*}\gamma+\delta\sqrt{g}\sim 81\ 8^{2 n+3} n^{3}+(-9) 2^{3 n+5} n^2\sqrt{81\ 4^{3 n+4} n^2}=\\ 81\ 8^{2 n+3} n^{3}+((-9) 2^{3 n+5} n^{2})(9\ 2^{3 n+4} n)=0.\end{gather*}
\end{observacion}

We use the good behaviour of these conjugates to free some radicals via multiplication. We proceed similarly with the denominator. This is what we do next.

\begin{procedimiento}
Since managing the denominator is easy, we proceed first with the numerator. We do this multiplying by the nice conjugates saw above. This will free some radicals.
\end{procedimiento}

We warn the reader that in the next explanation of the structure of the computations involved all the names are locally set again.

\begin{remark}
Since we have seen that the expression $(k+u)-(v+w)$ is dominantly $-4\cdot2^{11 n+15} 3^{n+5} n^4$, we can multiply by it in order to get a nice difference of squares. Hence, multiplying the numerator $k+u+v+w$ by $(k+u)-(v+w)$ we obtain precisely $(k+u)^2-(v+w)^2=k^2+u^2+2ku-v^2-w^2-2vw=(k^2+u^2-v^2-w^2)+(2ku)+(-2vw)=r+s+t.$
\end{remark}

We will need a further use of conjugates to determine the real growth of this because the expressions of $s$ and $t$ contain $\sqrt{q}.$ We have to see first what is the growth of the conjugate $r-(s+t)$.

\begin{remark}
In order to see this, we first need to establish the real growth of $s+t$, which is not immediate because there is a cancellation again as it is easy to see that $ku\sim (2^{11 n+15} 3^{n+5} n^{4})^2\sim vw$.
\end{remark}

Thus the use of conjugates is necessary again. We see this in the next section. Now that we have made clear where and why we have to be cautious about the management of radicals, we can actually proceed to determine that the bivariate bound is better than the univariate.

\section{Bivariate approach beats univariate}

\begin{computacion}
We can immediately compute the asymptotic growth of the part with no roots involved $r\sim-2^{19 n+33} 27^{n+3} n^7$. It is clear that $s-t=(2ku)-(-2vw)=2(ku+vw)\sim 4(2^{11 n+15} 3^{n+5} n^{4})^2=2^{22 n+32} 3^{2n+10} n^{8}$ while computing again we can see that $$(s+t)(s-t)=s^2-t^2\sim -2^{41 n+65} 3^{5 n+19} n^{15}$$ so we obtain therefore that $$s+t=\frac{s^2-t^2}{s-t}\sim\frac{-2^{41 n+65} 3^{5 n+19} n^{15}}{2^{22 n+32} 3^{2n+10} n^{8}}=-2^{19 n+33} 27^{n+3} n^{7}\sim r.$$ This implies that $r+s+t\sim 2r\sim-2^{19 n+34} 27^{n+3} n^{7}$ and therefore we obtain that $k+v+u+w=\frac{(k+u)^2-(v+w)^2}{(k+u)-(v+w)}=\frac{r+s+t}{(k+u)-(v+w)}\sim\frac{-2^{19 n+34} 27^{n+3} n^{7}}{-2^{11 n+17} 3^{n+5} n^4}=2^{8 n+17} 9^{n+2} n^3.$
\end{computacion}

We can repeat the same argument with the denominator. In that way, we obtain the following.

\begin{computacion}
Proceeding similarly with the denominator we get that the conjugate of the problematic factor is dominantly $\gamma-\delta\sqrt{g}\sim$ \begin{gather*}81\ 8^{2 n+3} n^{3}-(-9) 2^{3 n+5} n^2\sqrt{81\ 4^{3 n+4} n^2}=\\ 81\ 8^{2 n+3} n^{3}-((-9) 2^{3 n+5} n^{2})(9\ 2^{3 n+4} n)=81\ 4^{3 n+5} n^3\end{gather*} while the resulting difference of squares is dominantly $$\gamma^2-\delta^{2}g\sim 6561\ 16^{3 n+5} n^{5}$$ so, finally, we can see that the problematic factor is dominantly $\gamma+\delta\sqrt{g}=\frac{\gamma^2-\delta^{2}g}{\gamma-\delta\sqrt{g}}\sim\frac{6561\ 16^{3 n+5} n^{5}}{81\ 4^{3 n+5} n^3}=81\ 4^{3 n+5} n^2$ so, as $s=2ku\sim2(2^{11 n+15} 3^{n+5} n^{4})^2=2^{22 n+31} 3^{2n+10} n^{8},$ the denominator is dominantly $$s(\gamma+\delta\sqrt{g})\sim(2^{22 n+31} 3^{2n+10} n^{8})(81\ 4^{3 n+5} n^2)=2^{28 n+41} 9^{n+7} n^{10}.$$
\end{computacion}

All in all, we see that the difference is dominantly $$0<\frac{k+v+u+w}{s(\gamma+\delta\sqrt{g})}\sim\frac{2^{8 n+17} 9^{n+2} n^3}{2^{28 n+41} 9^{n+7} n^{10}}=\frac{16^{-5 n-6}}{59049 n^7}\to0^{+} \mbox{\ when\ } n\to\infty.$$ This proves immediately the next result.

\begin{teorema}
The application of the relaxation to the bivariate Eulerian polynomials produces asymptotically better bounds than the application of the relaxation to the univariate Eulerian polynomials. This improvement grows asymptotically, at least, as $\frac{16^{-5 n-6}}{59049 n^7}.$
\end{teorema}

\begin{proof}
See the computations above.
\end{proof}

Although positive news because we get an actual improvement in the asymptotics of the bounds obtained, this also shows us that our gain with respect to the univariate optimal is actually quite small and, in fact, vanishes as $n$ grows towards infinity. We will have to work with polynomials in more variables in order to obtain better improvements. As we saw already, the fully multivariate generalization of the Eulerian polynomials is not needed to provide an answer to the question about the first order asymptotic approximation to these roots, as this was completely determined already through the relaxation of the univariate polynomials. This univariate relaxation has appeared previously in the literature of statistics but not (to our knowledge) in the literature of approximations of roots of polynomials and it has shown here to be surprisingly good to estimate extreme roots of this family of polynomials. As the novelty of our method resides then in the use of the \textit{multivariate} relaxation we want to push things beyond this first order asymptotics result. Thus, in order to see the power of the relaxation of the multivariate polynomials, we will now turn to the study of the higher orders of the asymptotic approximation. This will allow us to eventually find much better improvements to the asymptotics of the bounds than the ones obtained in this section, which, although disappointingly small, have shown to us that, in fact, these improvements are possible when we go to more variables. But first we have to be cautious about what we are actually doing in the multivariate scenario. We do this study following a similar scheme as we did in the sections above for the bivariate case. As we explained already in this direction, for this very reason, many steps and arguments that follow will look quite similar to the ones we already saw above. We get started into the art of increasing variables on each step of the computation. This means that our LMPs will actually grow in each step, making thus any approximation through the use of determinants unfeasible in practice and therefore pushing us towards the art of guessing structural approximate solutions for generalized eigenvalue problems through appropriate guesses of approximate eigenvectors used to linearize our problems. We explain this in more detail next, although the main ideas were already covered in the preceding sections.

\section{Increasing the variables in each step}

We recall what we have done. We began applying the relaxation to univariate polynomials. We could see that doing this already provides a very good bound in asymptotic terms for the extreme roots of the univariate Eulerian polynomials. This bound is actually so good that it already matches the first asymptotic growth term in the natural exponential scale we are using for asymptotic expansions.

\begin{remark}[Univariate bound already good]
The univariate bound computed sections before grows like $2^{n+1}+o(2^{n})$ just as the actual extreme root of the univariate Eulerian polynomial $A_{n}.$ Thus we have to search for improvements down further terms in the asymptotic. We cannot know a priori how far we have to go, this is why we use differences to compare here. 
\end{remark}

Using certain properties of interlacing polynomials, we found a way to extend the unviariate Eulerian polynomials into multivariate ones. In particular, for this, we used the fact that they interlace sequentially. This allowed us to perform a construction similar to the one we made in Part \ref{I} using the Renegar derivative. The generalization of that construction allowed us to get rid of the rigidity of the Renegar derivative. Through this more general construction we could explore different bivariate construction that have let us show that the use of certain interlacers can indeed improve the relaxation (at least, when we look through the diagonal). In contrast, what happened when we used the Renegar derivative in Part \ref{I} is that the relaxation did not change at all. We find thus desirable to explore the interlacers that manage to change the relaxation through this process. A natural first step in this direction is looking to the interlacers that get us the closest possible to the original RCS of the RZ polynomial we are studying. These are the objects introduced in the next definition.

\begin{definicion}[Relaxation-extremal interlacers]\label{reinter}
Let $p\in\mathbb{R}[\mathbf{x}]$ be a RZ polynomial. Then it is clear by continuity that there exists an interlacer $r\in\mathbb{R}[\mathbf{x}]$ of $p$ (which does not need to be unique) such that for any other interlacer $s\in\mathbb{R}[\mathbf{x}]$ of $p$ we have that the respective relaxations verify \begin{gather*}\vol(\{a\in\mathbb{R}^{n}\mid (a,0)\in S(p+ys)\}\smallsetminus\rcs(p))\geq\\\vol(\{a\in\mathbb{R}^{n}\mid (a,0)\in S(p+yr)\}\smallsetminus\rcs(p)),\end{gather*} where $\gls{volA}$ is the volume of the Lebesgue-measurable set $A\subseteq\mathbb{R}^{n}.$ We call such an interlacer $r$ a \textit{relaxation-extremal interlacer for $p$}.
\end{definicion}

Thus, a relaxation-extremal interlacer is an interlacer $r\in\mathbb{R}[\mathbf{x}]$ that minimizes the volume of the difference between the original RCS of $p$ and the intersection with the original subspace of the RCS of the extended RZ polynomial constructed through the well-known transformation $p\mapsto p+yr\in\mathbb{R}[\mathbf{x},y]$. As David Sawall noted in a discussion, one interesting particular case comes with exactness. In such case, we are computing the volume of the empty set, which is $0$, and we have therefore equality.

\begin{ejemplo}[Analyzing relaxation-extremal interlacers under exactness]\label{extremalexactness}
Let $p\in\mathbb{R}[\mathbf{x}]$ be a RZ polynomial such that $\rcs(p)=S(p)$, i.e., such that the relaxation is exact. This happens, e.g., for quadratic RZ polynomials and for certain determinantal polynomials, as it is noted in \cite[Theorem 5.1 and Proposition 3.33]{main}, respectively. Hence, in the case of exactness, as we will easily see as a direct consequence of Proposition \ref{improinter}, for any interlacer $r$ of $p$, we will have $$\{a\in\mathbb{R}^{n}\mid (a,0)\in S(p+yr)\}=S(p)=\rcs(p).$$ Thus, in the case of exactness, all interlacers are relaxation-extremal.
\end{ejemplo}

In view of the example above, we need to make a point clear. Our Definition \ref{reinter} above is not vacuous. This is so because we know that it is not the case that always all interlacers are relaxation extremal. We remark this.

\begin{remark}[There are non-relaxation-extremal interlacers]
In order to see that the definition above is indeed not vacuous, we notice that, through our study of the sequence of univariate Eulerian polynomials, we were indeed able to find an example actually showing that the use of other different interlacers constructed in more exotic ways can, in fact, improve the relaxation (at least, through the diagonal). In order to see this, we compared the univariate Eulerian sequence with the use of the Renegar derivative and saw that the previous Eulerian polynomial in the sequence does in fact change the RCS in a way that the univariate bound improves over the diagonal. This shows that the provided (intersection with the original space of the) relaxation changes, contrary to what happened when we used the Renegar derivative as our interlacer in the transformation $p+yq$.
\end{remark}

We saw this fact in the previous section and studied the change of the provided relaxation through the diagonal. Thus, that information lets us make the following conjecture about these objects.

\begin{conjetura}[Extremal interlacers]\label{conjeextrem}
Let $p\in\mathbb{R}[\mathbf{x}]$ be a RZ polynomial with compact RCS. Then there exist a relaxation-extremal interlacer for $p$ that is also an extremal interlacer.
\end{conjetura}

We note that we cannot claim that all relaxation-extremal interlacers are extremal interlacers because of what we saw above in Example \ref{extremalexactness} for the case of exactness of the relaxation. A visual example of this was suggested by David Sawall. This example will also help us to understand better the Conjecture \ref{conjeextrem} above.

\begin{ejemplo}[Quadratic exactness and extremal interlacers]
We know, by \cite[Theorem 5.1]{main}, that the relaxation is exact for quadratic RZ polynomials. Take such a polynomial $p$ whose RCS is compact. Any hyperplane not cutting its RCS is an interlacer of $p$. Thus, all these interlacers are relaxation-extremal interlacers for $p$. The hyperplanes tangent to its RCS are, moreover, extremal. 
\end{ejemplo}

This conjecture means that the best performance of the relaxation under the strategy of increasing the number of variables of the original polynomial through the transformation $p\mapsto p+yq$ occurs, in particular, for some $q$ in the set of extremal interlacers (but it can also happen when $q$ is a non-extremal interlacer at the same time). It is a long way ahead before we are able to attack this conjecture.

Additionally, more knowledge about how to combine several interlacers into one polynomial increasing rapidly (and maybe even explosively \cite{raghavendra2019exponential}) the number of variables is necessary. This exercise is highly subtle and advancing in that direction is promising but also problematic. In any case, the example of the sequence of Eulerian polynomials that we worked out here tells us that there is some hope researching in that direction and therefore opens the door to follow that path.

\begin{remark}[Interlacing combinations]
More broadly, we could conjecture that there are other combinations of interlacers and additional variables that produce even better improvements and that these improvements happen at the extremal allowed combinations of interlacers. But this would be going too far already. Other combinations using the Renegar derivative could be extracted exploring results as these in \cite{nuijtype}.
\end{remark}

Now, however, we will explore a different path. The main disadvantage of the bivariate construction is that the number of variables is static there: it is always $2$ regardless of the degree. We know that the size of the LMP defining our relaxation depends on the number of variables and therefore any hope to greatly improve the output of the relaxation has to pass through mechanism allowing us to increase the number of variables sequentially with the degree. This is why we look in the direction of multivariability through the multivariate Eulerian polynomials, whose existence is ultimately connected with the celebrated existence and results about stability preservers. This is how we attack the next step into our understanding of the relaxation applied to these polynomials.

\begin{remark}[The importance of multivariability]
Multivariate Eulerian polynomials have relaxations that grow sequentially with the degree because so does the number of variables. We know that in order to faithfully represent these polynomials we most surely would even require a faster growth in the number of variables as compared to the degree. Thus, they are a step in the correct direction, but clearly not the end of the story. The relaxation is always hungry for additional variables and so are we if we really want to establish better bounds or maybe even find actual determinantal representations through this path we are drawing here.
\end{remark}

Thus, we proceed to study the application to the multivariate Eulerian polynomials with a great hope to find better results than in the bivariate case. This is so because now, as we desired and required, the amount of variables will increase, which is intuitively something positive for the accuracy of the relaxation. We begin by studying this application in a similar way as how we proceeded with the bivariate one. The initial part of the research into the multivariate polynomials is actually quite similar to how we started studying the bivariate ones. Therefore, the next two sections are extremely similar to the two before, as we already mentioned. However, it will be in the next chapter where we will actually see the whole power of multivariability. Thus, we start, as in the bivariate case, with some cautions. This time we do not need to study how to do the constructions, as we did in Section \ref{sectioninterlacing}, because these multivariate Eulerian polynomials were already deeply studied and researched before in \cite{visontai2013stable,haglund2012stable}, which is an advantage now for us in order to get faster and more directly into them. We proceed therefore with these cautions in the next section.

\section[Multivariate cautions]{Cautions before using the relaxation in the multivariate case}

Now we really go multivariate. Our number of variables will increase in each iteration and, with it, also the size of our LMPs. This, of course, complicates things, but the benefits will be evident soon when we start guessing correctly our approximations to the corresponding generalized eigenvector problem. We have to do this because now it is not doable (theoretically, practically, symbolically nor numerically) to compute the actual optimal bound.

\begin{remark}[Guessing the eigenvector]
The only approach that is practical and doable for us in this case is guessing the eigenvector. The growing size of the LMPs in the sequence only let us this option available. Even in this case, finding a correct guess will not be easy, although we will manage to do it in the next section after some numerical experiments and symbolic optimization work.
\end{remark}

We want to use the whole power of the multivariate relaxation and, as we said, for this, we have to be cautious about our choice of vectors. In particular, we have to take care that we do not end up going back to the univariate setting through our choice. In order to show how this happens we introduce now a weak sense of compression that will describe our problem here.

\begin{definicion}[Matrix compression]
Let $R$ be a ring and $M\in R^{n\times m}$ a matrix over that ring. We say that the matrix $N\in R^{n'\times m'}$ with $n'\leq n$ and $m'\leq m$ is a \textit{compression} of $M$ if $N$ is obtained from $M$ after summing some of its rows and columns suppressing the appearance of the rows and columns being added to another one.
\end{definicion}

Thus, we see that the all-ones vector can never achieve anything better than the univariate relaxation. The reason for this is the fact that using this vector has the same effect as compressing the multivariate LMP to the univariate one.

\begin{observacion}[Action of the all-ones vector]
Choosing the vector $\mathbf{1}\in\mathbb{R}^{n+1}$ of all-ones will give the same result that we obtained in the univariate case because this choice compresses the matrix given by the multivariate relaxation to the one obtained in the univariate case above.
\end{observacion}
Thus, we see that we have to take care with our choice of guesses in order to avoid ending up getting exactly the same bounds as before. We have to avoid, in particular, coming back to the univariate setting without realizing. Notice that there is another way of ending up in a univariate setting \textit{in some sense}: guessing a vector that kills the rows and columns corresponding to all the variables except one. We see this.

\begin{remark}[Other forms of univariateness]
Another possibility is using \textit{unit binary} (having all entries but one zero) vectors $(\mathbf{0},1,\mathbf{0})\in\mathbb{R}^{n+1}$ to linearize the inequality given by the relaxation of the multivariate polynomial. We will see that the use of these vectors does not improve the bound provided by the relaxation of the univariate polynomial. As the choice of these polynomials is natural, we will see why this is not providing a better bound. In particular, we will study the bounds given when the vector $(1,0,\dots,0)\in\mathbb{R}^{n+1}$ is considered and when the rest of the vectors $(0,1,0\dots,0),\dots,(0,0\dots,0,1)\in\mathbb{R}^{n+1}$ are considered.
\end{remark}

This is not really totally univariate because, although it seems that we kill information about the other variables through the zeros in the vector, in fact, some information about the splitting of the descents among the tops remains in the variable that survives. Still, knowing this, we have to take care that we do not discard too much information doing this. In fact we do discard too much information, as we will soon see and, for this reason, in the next section we have to deal with the task of finding a guess for an approximate eigenvector that collects enough information to allow a real improvement in the multivariate relaxation with respect to the univariate case. In this section, we will see that we cannot improve the bounds provided by the 
application of the relaxation to the univariate polynomials if we end up discarding too much information through a bad choice of our guess for an approximate eigenvector. But, first, we have to note again that there are variables that do not really play any role and we are therefore not actually using although they seem to float around us all the time.

\begin{warning}[Ghost variables action in the relaxation]
For the study of this second set of vectors we introduce, for clarity, the reduced relaxation, which is the relaxation of size $n+1$ obtained after deleting additionally the first row and column (these corresponding to the application of the $L$-form to monomials up to degree $1$ in the initial matrix) and \textbf{reintroducing} the (always $0$) row and column corresponding to the index $1$ temporarily. We do this to be able to count positions in the vector indexed by the same set as the variables in the multivariate Eulerian polynomial $A_{n}$. This will not affect the final result because the study at position $1$ will not produce anything relevant.
\end{warning}

Since we already saw that what we obtain is actually Colucci estimate when we bound through the vector $(1,
\mathbf{0})\in\mathbb{R}^{n+2}$ (counting by the size used in previous sections), we pass directly to the case of the other unit binary vectors (having all entries but one zero). These considerations will produce bounds under the generalized univariateness we discussed above.

\setlength{\emergencystretch}{3em}%
\begin{remark}[Bounds under generalized univariateness]
Thus, we will study the bounds obtained through the vectors $(\mathbf{0},1,\mathbf{0})\in\mathbb{R}^{n+1}$ in the newly introduced reduced relaxation so the possible positions run from $1$ to $n+1$ and correspond one by one with the indices of the variables used in the degree $n$ multivariate Eulerian polynomial $A_{n}$ including the never appearing variable $x_{1}$. In particular, remember that $L_{p}(x_{1}m)=0$ for any monomial $m$ and therefore, in our sums, we can \textit{jump} (as in \textit{ignoring}) over that index.
\end{remark}
\setlength{\emergencystretch}{0em}%

Therefore, using these vectors accounts to choose to look only at the action of one of the variables. And it is in this sense that we say that this is a form of univariateness: not because we are really in a univariate context, but because we make a choice to look at a univariate action as a part of the a multivariate relaxation. It is the choice (of vector) what is multivariate in this case, not the polynomial. These choices produce the bounds that we study next.

\begin{proposicion}[Generalized univariate bounds]\label{genunibou}
For $n,j\in\mathbb{N}$ with $n\geq1$ and $1<j\leq n+1$ we have the inequality $$|q_{1}^{(n)}|\geq\frac{1 - 2^{j-1} + 2^{j + n} - 2^{2 - j + n} 3^{j-1}}{1-2^{j-1}}.$$ 
\end{proposicion}

\begin{proof}
The linear bound obtained through the choice having a one in position $j$ is formed considering only the $j$-th diagonal element of the (reduced) relaxation:\comm{$L_{p}(x_{j}^{2})+x\sum_{i=2}^{n}L_{p}(x_{i}x_{j}^{2})=(2^{j-1}-1)^{2}+x(\sum_{i=2}^{j-1}L_{p}(x_{i}x_{j}^{2})+L_{p}(x_{j}^{3})+\sum_{i=j+1}^{n}L_{p}(x_{i}x_{j}^{2}))=(2^{j-1}-1)^{2}+x(\sum_{i=2}^{j-1}((2^{j-1}-1)^{2}(2^{i-1}-1)-(2^{j-1}-1)(3^{i-1}2^{j-i}-(2^{j-1}+2^{i-1})+1))+(2^{j-1}-1)^{3}+\sum_{i=j+1}^{n}((2^{j-1}-1)^{2}(2^{i-1}-1)-(2^{j-1}-1)(3^{j-1}2^{i-j}-(2^{j-1}+2^{i-1})+1)))=(2^{j-1}-1)^{2}+x(\frac{1}{24}(-2 + 2^{j}) (2^{1 + j}3 - 3^{j}8 + 4^{j}3)+(2^{j-1}-1)^{3} -\frac{1}{3} 2^{-j - 2} (2^{j} - 2) (4^{j}3 - 3^{j}4) (2^{j} - 2^{n}))=(2^{j-1}-1)^{2}+x(\frac{1}{12} (-2 + 2^j) (6 - 2^{j}3 + 2^{j + n}3 - 2^{2 - j + n} 3^{j}))} $L_{p}(x_{j}^{2})+x\sum_{i=2}^{n+1}L_{p}(x_{i}x_{j}^{2}).$ There the term without $x$ is $L_{p}(x_{j}^{2})=(2^{j-1}-1)^{2}$ while the term with $x$ equals $\sum_{i=2}^{n+1}L_{p}(x_{i}x_{j}^{2})=$ \begin{gather*}L_{p}(x_{j}^{3})+\sum_{i=2}^{j-1}L_{p}(x_{i}x_{j}^{2})+\sum_{i=j+1}^{n+1}L_{p}(x_{i}x_{j}^{2})=(2^{j-1}-1)^{3}\\+\sum_{i=2}^{j-1}((2^{j-1}-1)^{2}(2^{i-1}-1)-(2^{j-1}-1)(3^{i-1}2^{j-i}-(2^{j-1}+2^{i-1})+1))\\+\sum_{i=j+1}^{n+1}((2^{j-1}-1)^{2}(2^{i-1}-1)-(2^{j-1}-1)(3^{j-1}2^{i-j}-(2^{j-1}+2^{i-1})+1))\\=(2^{j-1}-1)^{3}+\frac{1}{24}(-2 + 2^{j}) (2^{1 + j}3 - 3^{j}8 + 4^{j}3)\\-\frac{1}{3} 2^{-j - 2} (2^{j} - 2) (4^{j}3 - 3^{j}4) (2^{j} - 2^{n+1})=\\\frac{1}{12} (-2 + 2^j) (6 - 2^{j}3 + 2^{j + n+1}3 - 2^{2 - j + n+1} 3^{j}).\end{gather*} So we obtain the condition $$(2^{j-1}-1)^{2}+x(\frac{1}{12} (-2 + 2^j) (6 - 2^{j}3 + 2^{j + n+1}3 - 2^{2 - j + n+1} 3^{j}))\geq0$$ which gives, for $x_{n,r}$ the biggest root (which has to be negative) of the determinant of the relaxation of the multivariate Eulerian polynomial $A_{n}$, the bound \comm{\frac{-(2^{j-1}-1)^{2}}{\frac{1}{12} (-2 + 2^j) (6 - 2^{j}3 + 2^{j + n}3 - 2^{2 - j + n} 3^{j})}=} $$x_{n,r}\geq\frac{-6(2^{j-1}-1)}{6 - 2^{j}3 + 2^{j + n+1}3 - 2^{2 - j + n+1} 3^{j}}=\frac{1-2^{j-1}}{1 - 2^{j-1} + 2^{j + n} - 2^{2 - j + n} 3^{j-1}}$$ whenever we fulfill the condition \begin{gather*}6 - 2^{j}3 + 2^{j + n+1}3 - 2^{2 - j + n+1} 3^{j}=\\6 - 2^{j}3 + 2^{j + n+1}3 - 2^{2 - j + n+1} 2^{\log_{2}(3)j}=\\6+2^{j}(-3+2^{n+1}3-2^{2+n+1+(\log_{2}(3)-2)j})=\\6+2^{j}(-3+2^{n+1}(3-2^{2+(\log_{2}(3)-2)j}))>0,\end{gather*} a positivity condition for the denominator that is fulfilled if the exponent $2+(\log_{2}(3)-2)j<\log_{2}(3)$ (i.e., $j>\frac{\log_{2}(3)-2}{\log_{2}(3)-2}=1$) and, choosing $j=2$ as basis and using that the function $-3+2^{n+1}(3-2^{2+(\log_{2}(3)-2)j})$ grows with $j$, if $-3+2^{n+1}(3-2^{2+2\log_{2}(3)-4})\geq0,$ i.e., $n+1\geq\log_{2}\frac{3}{(3-2^{2+2\log_{2}(3)-4})}=\log_{2}\frac{3}{(3-\frac{9}{4})}=\log_{2}(4)=2$ (i.e., $n\geq 1$). This finished the proof for these generalized univariate bounds.
\end{proof}

Notice that the restrictions in the indices in Proposition \ref{genunibou} make sense as the subindex $j=1$ does never actually appear as a subindex of a variable in the multivariate polynomial $A_{n}$, as we mentioned already, and, for $n=0$, we obtain a constant polynomial $A_{0}=1$ that cannot benefit from the multivariate point of view of our relaxation because it has no roots. Now we analyze what happens when $j$ increases.

\begin{remark}[Adding more variables versus looking into the next variable]
When we we change the index $j$ we look deeper in the matrix and therefore we have in mind more information about the split of descents by their tops. Another option that we do not pursue here would be to add variables slowly changing some of the $0$ to $1$ and looking at the sequence of bound we obtain that way.
\end{remark}

Hence, we can see that, for fixed $n$, this bound changes with $j$. This makes us wonder which is the optimal bound attainable through these vectors. The answer to this question is our next result.

\begin{corolario}[Best generalized univariate bound reached at last]
The best bound attainable using unit binary vectors (with all entries $0$ except one) is obtained linearizing the relaxation through the vector $(\mathbf{0},1),$ where the last entry is the only one different from $0$. This means that for the polynomial $A_{n}$ we choose $j=n+1$ and linearize its simplified relaxation through the vector $(\mathbf{0},1)\in\mathbb{R}^{n+1}.$
\end{corolario}

\begin{proof}
Clearly, looking at Proposition \ref{genunibou} we can see that the best bound is given by that choice of the index $j$ that minimizes $$\left|\frac{-6(2^{j-1}-1)}{6 - 2^{j}3 + 2^{j + n+1}3 - 2^{2 - j + n+1} 3^{j}}\right|.$$ It turns out that the (possible) choice of $j\in[2,n+1]$ minimizing that quantity is in fact $j=n+1$, which corresponds to taking the vector $(\mathbf{0},1)$.
\end{proof}

We see here clearly now how adding more information about the splitting helps. This corollary makes formal what we commented above about going deeper into the matrix in order to extract further information about the splitting of the descents among the tops in the multivariate polynomial.

\begin{remark}[At last, reaching the last variable wins]
When we look at the last vector $(\mathbf{0},1)$, we are looking at the row and column of our LMP that carries the more dense information about the last variable added to the polynomial. Thus it is not surprising that looking at the densest part of the matrix in the information about the variable that completes the splitting of the descents provides us with the best possible vector attainable in this way.
\end{remark}

In this way, we obtain immediately a more clear vision about how the information travels from the multivariate polynomial to the relaxation. This is so because, as we can see, the linearizations of the relaxation, in some sense, \textit{feel} when the polynomial has been completed. This happens when we look at the part being the densest in the information about the last variable and we obtain the best possible bound attainable in this way. Now we want to compare these bounds not just via inequalities but more concretely and quantitatively accurately through the comparison of the corresponding asymptotic growths. This is the content of the next result.

\begin{proposicion}[Asymptotic comparison of generalized univariate bounds]\label{ratiogenunivariate}
Fix $i\in[2,\infty)$. Then the ratio of the bound obtained through the linearization of the relaxation by the all-zeros-except-at $i$ vector and the best possible bound obtained using these vectors (corresponding, as we have seen, to the choice $(\mathbf{0},1)$) when $n\to\infty$ goes to $$\frac{2^{-i} (-4 \cdot 3^i + 3 \cdot 4^i)}{3 (-2 + 2^i)}\in[0,1].$$
\end{proposicion}

\begin{proof}
Fix $i\in[2,\infty)$. Now, in Proposition \ref{genunibou}, we take, on the one hand, $j$ to be the best possible choice $j=n+1$ in terms of $n$, as we said before and, on the other hand, $j=i$ for the $i$ we fixed at the beginning. Therefore we obtain now the ratio of absolute values \begin{gather*}\frac{\frac{6(2^{n}-1)}{6 - 2^{n+1}3 + 2^{2(n+1)}3 - 2^{2} 3^{n+1}}}{\frac{6(2^{i-1}-1)}{6 - 2^{i}3 + 2^{i + n+1}3 - 2^{2 - i + n+1} 3^{i}}}\to\frac{2^{-i} (-4 \cdot 3^i + 3 \cdot 4^i)}{3 (-2 + 2^i)} \mbox{ when } n\to\infty,\end{gather*} as we wanted to see.
\end{proof}

We can extract immediately the next result looking at the limit obtained above. In particular, here we see clearly the improvement in the bound and how it happens in a quantitative way. Qualitatively, it is also obvious that choosing an index $i$ that grows while the number of variables grows is beneficial in order to get tighter bounds, as it is expectable in terms of the information being collected that way, as we already noted intuitively above.

\begin{corolario}[Properties of the ratio]
The limit ratio in Proposition \ref{ratiogenunivariate} is a growing sequence for $i$ and is always positive and smaller than one but tends to $1$ as $i$ goes to infinity.
\end{corolario}

All in all, this tells us that the bound improves when $i$ grows. Thus, the best bound attainable through this family is therefore the one provided by the case $i=n+1$, the highest $i$ possible for each $n$ and then the improvement, in geometric terms, with respect to the bound obtained in the case where we choose all the time a fixed index $i$ is given by the function $$\frac{2^{-i} (-4 \cdot 3^i + 3 \cdot 4^i)}{3 (-2 + 2^i)},$$ which takes values in $[0,1]$ when $i\in[2,\infty).$

\begin{remark}[Confirmation that adding more variables helps]
Summing up, this shows us that adding more variables does in fact improve the bound provided by the relaxation through this family of simple vectors. This is so because we have seen that it is in the last variable where the optimal bound is attained. In order to see this, we compared the bound fixing $i$ with a different bound in which $i$ is consistently taken to be the best possible choice in terms of $n$.
\end{remark}

However, as we will see, this does not beat the application of the relaxation to the univariate polynomial. In order to beat that, we will have to use a more cleverly chosen vector for the linearization of the relaxation in the process of establishing the corresponding bound. Here, we could only show, so far, the benefit of adding more variables into the mix \textit{once we have already gone multivariate}.

\begin{remark}[Original univariate still wins against the all-zeros-except-one family]
In order to win against the univariate bound, we will need to choose a different family of vectors. In particular, we will build a family that uses more information about the number of variables (and the coefficients) in each polynomial of the sequence. Doing this will greatly benefit and strengthen what the relaxation can achieve. For this, our vector will have to change more while the number of variables $n$ grows.
\end{remark}

We can formally see why this improvement is not happening. The comparison is clear. 

\begin{proposicion}[Original univariate bound beats generalized univariate bounds]
The optimal upper bound for the leftmost root obtained through the unit binary vectors is worse than the corresponding bound obtained through applying the relaxation to the univariate polynomial.
\end{proposicion}

\begin{proof}
The absolute value of the bound for the leftmost root obtained in this section is $\frac{6 - 2^{n+1}3 + 2^{2(n+1)}3 - 2^{2} 3^{n+1}}{6(2^{n}-1)}$, which, compared to the bound obtained via the univariate polynomial, gives \begin{gather*}\frac{6 - 2^{n+1}3 + 2^{2(n+1)}3 - 2^{2} 3^{n+1}}{6(2^{n}-1)}<\\\frac{2^{1 + n} + 2^{3 + 2n} -  3^{1 + n}4 - 2^{1 + n} 3^{2 + n} + 4^{1 + n}3 + 8^{1 + n}}{-2 + 2^{2 + n} - 3^{1 + n}2 + 4^{1 + n}}.\end{gather*} This comparison finishes the proof.
\end{proof}

As we made the comparison to the optimal, now we can say that no element of this family can compete. This is precisely what impulses us to search for a different sequence of vectors.

\begin{remark}[Importance for a bound to be closer]
The bound obtained through the univariate polynomial is actually closer to the actual root than the one obtained through this family of vectors. This means that, in this way, through this family, we are not using the whole power of going multivariate.
\end{remark}

However, there is something positive about the family used here. The positive part lies in faster computability. Something that we just note but that we are not going to develop here further because it lies beyond our interests here.

\begin{observacion}[Generalized univariate is not better but computes faster]
Killing so many elements of the LMP of the relaxation through these almost always zero vectors produces worse bounds but, at the same time, it kills also many terms making our computations way easier and faster to performs. We will see that the difficulty of these computations will greatly increase when we use vectors with a more complicatedly elaborated structure in the future.
\end{observacion}

Although this family is not improving the application to the univariate relaxation, we can say that it was worth studying what it does. It makes sense to look at these vectors for two reasons. First, they are the simplest non-trivial ones. Second and more interesting, they give a clear picture of what is the effect of just adding more variables. 

\begin{remark}[The real effects of adding more variables to the polynomials but not to the relaxation]
When we just choose a variable to look at (that is, we fix a position in the family of vectors that will be the only one different from zero) we are just looking at how the whole splitting of the descents among the different tops adds more information just at this variable. In that way, we could thus see that it does indeed add more information. However, we also saw that this information is not enough to beat the act of collecting the whole information of the splitting. Moreover, we saw that moving the position towards the last variables instead of fixing it also improves our retrieval of information with respect to fixing. In particular, the best bound was obtained when we moved always the position to fix our look at the splitting suffered by the last variable. Thus, the study of these families provides combinatorial information theoretic insights that we just decide to not pursue here but that justify the interest and future research about these other \textit{suboptimal} families of vectors from different perspectives. We simply decide not to purse these combinatorial information theoretic questions here because of a lack of overall connection with our objectives in this dissertation.
\end{remark}

This leads us to our objective in the next section. We need to find the way of choosing a sequence of vectors that allows us to extract more information about the multivariate splitting. We do this in order to show that the use of the multivariate Eulerian polynomials do indeed increase the accuracy of the relaxation in the diagonal where the univariate Eulerian polynomials lie.

\begin{objetivo}[Going beyond in the addition of the new variable information]
Through some experiments, we know numerically that the application of the relaxation to the multivariate Eulerian polynomials augments the accuracy of our bounds in the for the univariate Eulerian polynomials lying in the diagonal. Finding a family of vectors that allows us to show this is our next task.
\end{objetivo}

In order to accomplish the task above, we will have to be more precise in the numerics. We have to this so we can find the correct structure for the sequence of approximated generalized eigenvectors we need. For effectively and accurately measuring and analyzing the improvements that these processes and new choices of sequences of approximated generalized eigenvectors will produce, we also have to improve our management of asymptotics. We remind here how we go deeper and clearer into the information about the asymptotic growth of the sequences that will appear soon.

\begin{remark}[Asymptotic tools necessary to measure improvements]
Now deeper terms of the growth of the sequences at play will star to play a role. In particular, we will use the exponential scale we talked about before and we will write, analyze and solve (in-)equations in the corresponding coefficients of the expansions of the bounds in these scales. This study opens therefore further the realm of asymptotic (in-)equations in this work.
\end{remark}

We analyzed in this section a sequence that does not improve the action of the relaxation after its application to the multivariate Eulerian polynomials. In the next section, we venture into the constructions we need in order to prove that the relaxation can indeed benefit from going multivariate. Soon we will experience the real power of the relaxation and discover at the same time the complications related to the use, computation and intuitions behind the correct uses and nice constructions behind a successful application of the relaxation. This is how we open the door showing us the path towards our next section.

\section{Multivariate approach beats univariate}

After what we have seen above, we can question the necessity for studying the multivariate relaxation. The first step to justify the study of the multivariate relaxation is establishing that the multivariate case allows us to establish strictly better inner bounds. We see this in this section.

\begin{objetivo}[Justification for going (fully) multivariate (not just in the polynomial, also in the relaxation)]
We saw above that we have to find ways to use the multivariateness that we are introducing into the polynomials in order to improve our bounds. We have this objective because we can see in numerical experiments that the relaxation applied to the multivariate Eulerian polynomials should give us better results and we want to prove that we do indeed obtain something better in this manner.
\end{objetivo}

There were several difficulties in the notation that we need to use that we prefer to remind again here in order to avoid confusion. These difficulties stem ultimately from the fact that $1$ can never be a top.

\setlength{\emergencystretch}{3em}%
\begin{reminder}[Ghost variables, indexing and size]
As a consequence of the special nature of the definition of these polynomials, we decided to write the $n$-th Eulerian polynomial as a polynomial in $n$ variables $x_{2},\dots,x_{n+1}$ ($x_{1}$ is never actually appearing) and degree $n$ and, therefore, the full relaxation corresponding to the $n$-th Eulerian polynomial has size $n+2$ but the second (the first in the simplified relaxation) row, column and matrix coefficient are always just $0$ and therefore irrelevant. We keep these here just for indexing purposes.
\end{reminder}
\setlength{\emergencystretch}{0em}%

We decide to follow the current path guided by what the experiments show. These are the tools that produced the insights making us look into this direction.

\begin{observacion}[Numerical experiments clear the path]
Although the previous section could look depressing in terms of using the multivariate approach to find better bounds, numerical experiments show that the multivariate relaxation gives, in fact, slightly better bounds. This is why the numerical experiments push us in the current direction consisting in exploring further that relaxation through sequences of vectors with a more elaborated structure that could extract more information about the polynomial from the relaxation through linearization.
\end{observacion}

We already know that we cannot improve the first order asymptotics because that order is already tight in the exponential scale we are using. However, a good enough choice of vectors in the multivariate case could in fact still improve the higher order asymptotic estimates for these roots.

Choosing a good sequence of vectors is a necessity. Otherwise our task is impossible. We can see this in the comparison we make in the next two remarks about how are bounds extracted using the determinant or vectors. We start with the practically undoable but more accurate path through the determinant of the relaxation.

\begin{remark}[Through the determinant]
Trying to use the determinant of the relaxation in order to obtain a bound for its innermost root is not a rewarding task. Managing such polynomial is in no way easier than managing the original univariate Eulerian polynomial. In fact, the coefficients are way more complicated, some nice symmetries are lost and there are many extremely cumbersome computations in the middle that, overall, lead to an unreasonably long task in order to obtain a polynomial of around the same degree as the polynomial we begin with.
\end{remark}

Using the determinant is not the path. Some parts of the process above have to be simplified, skipped, jumped, outsourced or divested to other objects or venues in order to produce a path worth taking. Taking the determinant may be more accurate, but we need to balance accuracy with manageability. We accomplish this linearizing through a vector, as we saw before. There is however still one complication when we try to do this: at some point, we will have to optimize over an entry in the vector in order to be able to beat the bound obtained through the univariate application. In any case, performing this optimization is way more doable from a computational and practical point of view than proceeding through the determinant. We explore in more detail this path and the first sequence of vectors that we could find beating the univariate application. This sequence is very simple except for the optimization step we already talked about. In the future, we will look at more complicated sequences. Dealing with this easy sequences serves as an exercise before passing, in the next chapters, towards more complicated structures in the sequences of vectors we have to use. In particular, the optimization step we make in this section in order to establish one of the entries of the sequence we treat here will be reused in the future and it has been, in fact, a very important tool in our path towards the exploration of even better sequences. This is why we show how we deal with this family here as an intermediate step before going to our definitive sequence: because it was indeed an intermediate step for us while we tried to understand how to deal with the task of finding good approximations of the corresponding generalized eigenvalues that we need to compute in order to simplify and make doable the analysis about the relaxation that we need to asses skipping determinants, after we saw above that the use of determinants is not the way to go towards our objective of understanding the behaviour of the relaxation in relation to our Eulerian polynomials. Thus, the sequence we introduce next is simple and we will use it in the rest of this last section of the current chapter.

\begin{remark}[Through a new vector]
Linearizing through a vector is the way to go. Our experiments show that one good choice has the form $(y,0,1,-\mathbf{1})\in\mathbb{R}^{n+2},$ where $y$ has to be determined and instead of $0$ we could put any other number without this affecting the result. The optimization of the $y$ will be a fundamental part of our job.
\end{remark}

The vector $(y,0,1,-\mathbf{1})\in\mathbb{R}^{n+2}$ finally breaks the univariateness. Now we are really using more than one variable. This means that in this case and for the first time in this bounding process of the extreme roots of Eulerian polynomials through the relaxation, we are using a deeper theorem from the theory of real algebraic geometry.

\begin{observacion}[Using Helton-Vinnikov for the first time]
Schweighofer's proof in \cite{main} that the relaxation is indeed a relaxation for polynomials of strictly more than one variable uses Helton-Vinnikov theorem. We did not take advantage of this theorem so far because the winner up to this point was the all-ones vector, whose use is equivalent to considering only the relaxation applied to the univariate Eulerian polynomial. The fact that the relaxation is so in the univariate setting is much easier to see and, therefore, this was already noticed in \cite{szeg1939orthogonal}, as we said.
\end{observacion}

Now we are in the correct path. We just have to figure out the $y$.

\begin{observacion}[In search of a vector beating the univariate bound]
Optimizing the bound provided by $v:=(y,0,1,-\mathbf{1})\in\mathbb{R}^{n+2}$ over the parameter $y$ will give us a function of $n$ that we will use also in future experiments. Thus, looking at this case, puts us in the path towards better guesses.
\end{observacion}

We compute then now the bound obtained through such method and maximize its absolute value for $y\in\mathbb{R}$ in order to get an explicit expression that we can eventually compare with the best univariate bound obtained some sections above. This process will end up giving the following result whose proof is comprised in the next subsections.

\begin{proposicion}[Beating univariate]\label{multundiff}
Call $\gls{uni},\gls{multv}\colon\mathbb{N}\to\mathbb{R}_{>0}$ the sequence of bounds for the absolute value of the smallest (leftmost) root of the univariate Eulerian polynomial $A_{n}$ obtained after applying the relaxation to the univariate $A_{n}$ and the multivariate $A_{n}$ Eulerian polynomials respectively, where $\mult_{v}$ is obtained after linearizing using the vector $v$ and $\un$ is the optimal bound given by the actual root of the determinant of the relaxation. Remember also that $v$ depends on $y$. Then there exists a sequence $y\colon\mathbb{N}\to\mathbb{R}$ such that the difference $\mult_{v}-\un\sim\frac{1}{2}\left(\frac{3}{4}\right)^{n}$ is positive. 
\end{proposicion}

This proposition tells us that we have indeed an improvement with respect to the univariate vision. The positivity of the difference $\mult_{v}-\un$ tells us that the new linearization using the new vector $(y,0,1,-\mathbf{1})\in\mathbb{R}^{n+2}$ ends up being closer to the actual extreme root. This gives immediately the next corollary.

\begin{corolario}
The application of the relaxation to the multivariate Eulerian polynomials produce better bounds for the extreme roots of the univariate Eulerian polynomial than its application of the relaxation to the univariate Eulerian polynomials.
\end{corolario}

The rest of the this section is the proof of Proposition \ref{multundiff}. We divide this task in several subsections for clarity. First, we need to compute the form of the bound from the corresponding $L$-forms.

\subsection{Computation}

We begin computing the inequality that gives the bound. The bound obtained through the choice of the vector $(y,0,1,-\mathbf{1})\in\mathbb{R}^{n+2}$ comes from an inequation of the form $D+xN\geq0$ with with $D=$ \begin{gather*}y(yL_{p}(1)+L_{p}(x_{2})-\sum_{i=3}^{n+1}L_{p}(x_{i}))+yL_{p}(x_{2})+L_{p}(x_{2}^{2})-\sum_{i=3}^{n+1}L_{p}(x_{2}x_{i})\\-\sum_{j=3}^{n+1}yL_{p}(x_{j})+L_{p}(x_{j}x_{2})-L_{p}(x_{j}^{2})-\sum_{j\neq i=3}L_{p}(x_{i}x_{j})
\end{gather*} and $N=y(y\sum_{k=2}^{n+1}L_{p}(x_{k})+\sum_{k=2}^{n+1}L_{p}(x_{2}x_{k})-\sum_{i=3}^{n+1}\sum_{k=2}^{n+1}L_{p}(x_{k}x_{i}))+$ \begin{gather*}
y\sum_{k=2}^{n+1}L_{p}(x_{k}x_{2})+\sum_{k=2}^{n+1}L_{p}(x_{k}x_{2}^{2})-\sum_{i=3}^{n+1}\sum_{k=2}^{n+1}L_{p}(x_{k}x_{2}x_{i})\\-\sum_{j=3}^{n+1}y\sum_{k=2}^{n+1}L_{p}(x_{k}x_{j})+\sum_{k=2}^{n+1}L_{p}(x_{k}x_{2}x_{j})-\sum_{i=3}^{n+1}\sum_{k=2}^{n+1}L_{p}(x_{k}x_{j}x_{i})=\\y^2\sum_{k=2}^{n+1}L_{p}(x_{k})+2y\sum_{k=2}^{n+1}L_{p}(x_{2}x_{k})-2y\sum_{i=3}^{n+1}\sum_{k=2}^{n+1}L_{p}(x_{k}x_{i})+\\
\left(\sum_{k=2}^{n+1}L_{p}(x_{k}x_{2}^{2})-\sum_{i=3}^{n+1}\sum_{k=2}^{n+1}L_{p}(x_{k}x_{2}x_{i})\right)\\+\left(-\sum_{i=3}^{n+1}\sum_{k=2}^{n+1}L_{p}(x_{k}x_{2}x_{i})+\sum_{j=3}^{n+1}\sum_{i=3}^{n+1}\sum_{k=2}^{n+1}L_{p}(x_{k}x_{j}x_{i})\right)=\\y^2\sum_{k=2}^{n+1}L_{p}(x_{k})+2yL_{p}(x_{2}^2)\\-2y\sum_{i=3}^{n+1}\sum_{k=3}^{n+1}L_{p}(x_{k}x_{i})+L_{p}(x_{2}^{3})-\sum_{i=3}^{n+1}\sum_{k=3}^{n+1}L_{p}(x_{k}x_{2}x_{i})\\-\sum_{i=3}^{n+1}L_{p}(x_{2}^{2}x_{i})+\sum_{j=3}^{n+1}\sum_{i=3}^{n+1}\sum_{k=3}^{n+1}L_{p}(x_{k}x_{j}x_{i}).\end{gather*} Thus, computing these last expressions, we obtain the following lemma telling us the form of the bound obtained in this way.

\begin{lema}[Form of the bound]\label{formofboundfirstwinner}
The bound for the extreme root of the $n$-th univariate Eulerian polynomial obtained through linearization of the relaxation applied to the corresponding multivariate Eulerian polynomial by the vector $(y,0,1,-\mathbf{1})\in\mathbb{R}^{n+2}$ is of the form $q_{n}^{(n)}\geq-\frac{D}{N}$ with $D=$ \begin{gather*}
   10 - 2^{2 + n} + 2^{2 + 2n} - 2 \cdot 3^{1 + n} + n + 4y - 2^{1 + n}y + 
 n y + y(4 - 2^{1 + n} + n + ny)
\end{gather*} and $N=$ \begin{gather*}
    -10 + 2^{3 + n} - \frac{1}{3}2^{3 + 2n} - \frac{1}{3}2^{4 + 2n} + \frac{1}{7}2^{4 + 3n} + \frac{1}{7}2^{5 + 3n} +\\ 2 \cdot 3^n - 4 \cdot 3^{1 + n} + 2 \cdot 3^{2 + n} - \frac{1}{5}2^{1 + n}3^{3 + n} - 4^{1 + n} + 4^{2 + n} - \frac{6^{2 + n}}{5} +\\ \frac{8^{1 + n}}{7} - n - 8y - 2^{2 + n}y + 2^{3 + n}y - \frac{1}{3}2^{3 + 2n}y - \frac{1}{3}2^{4 + 2n}y +\\ 4 \cdot 3^{1 + n}y - 2ny - 2y^2 + 2^{1 + n}y^2 - ny^2.
\end{gather*}
\end{lema}

\begin{proof}
The proof consists in performing the last sums of the computations opening this subsection.
\end{proof}

Now we have to optimize in order to find the best $y$. This optimization of the bound will give us $y$ as function of $n$ and this $y$ will be obtained as a solution of a quadratic equation. This means that there will be square roots in our expressions that we will have to learn how to deal with. We first see the optimization step in the next subsection.

\subsection{Optimization}

We now have to choose $y$ wisely having in mind that the bound that we want to obtain has the form $q_{n}^{n}\geq-\frac{D}{N}$. Thus, a first condition on this $y$ is that such $y$ (which will depend only on the parameter $n$) should make $N>0$ for $n$ big enough (in order to be able to isolate the $x$ with the inequality pointing in the correct direction). This means, in particular, that it is enough if \begin{equation}\label{conditionsy} y>1 \mbox{\ and\ } \lim_{n\to\infty} \frac{\max\{2^{n+1}y^2, 2y4^{n+1}\}}{8^{n+1}}=0.\end{equation} We see this.

\begin{proposicion}
If Condition \ref{conditionsy} is met, then $N>0$ for all $n\in\mathbb{N}$ big enough.
\end{proposicion}

\begin{proof}
As $y>1$, the sum of terms that determine the first asymptotic growth of $N$ is $2^{n+1}y^2-\frac{1}{3}2^{3+2n}y-\frac{2}{3}2^{3+2n}y+\frac{2}{7}2^{3+3n}+\frac{4}{7}2^{3+3n}+\frac{1}{7}2^{3+3n}=2^{n+1}y^2-2^{3+2n}y+2^{3+3n}=2^{n+1}y^2-2y4^{1+n}+8^{1+n}.$ Now taking limits does the rest and finished the proof of this claim.
\end{proof}

Notice that $D$ must always be nonnegative for any $y$ as the relaxation has its initial matrix PSD and therefore we do not have to be so careful with $D$ as we have been with $N$ at the choice of $y$; for $D$ we just have to check that, for the chosen $y$, we have $D\neq0$. We can fortunately check this condition about $D$ right away using discriminants. We see it.

\begin{proposicion}
We have $D>0$ for all $n\in\mathbb{N}$ big enough and $y>0.$
\end{proposicion}

\begin{proof}
We see this fast using the fact that the discriminant of $D$ verifies $\Delta=64 - 2^{6 + n} + 2^{4 + 2n} - 8n - 2^{4 + 2n}n + 8 \cdot 3^{1 + n}n=64(1 - 2^{n}) - 8n + 16(1-n)4^{n} + 24 n 3^{n},$ which is clearly negative when $n$ is big enough and moreover the quadratic term in $D$ as a function of $y>1$ is $ny^{2}\neq0$. All this ensures that $D$ is a quadratic polynomials with no real zeros and therefore strictly positive for whichever $y>1$ we might end up choosing, which proves the claim.
\end{proof}

This solves completely one of our problems. Continuing with the quest for the optimal choice of $y$, we rearrange the expression to remind that, as a bound in absolute values, our bound translates to $|q^{(n)}_{n}|\leq\frac{D}{N}$ (with $N>0$ for $n$ big enough), which gives the bound $|q^{(n)}_{1}|\geq\frac{N}{D}$ for the absolute value of the leftmost root $q^{(n)}_{1}$ of the $n$-th univariate Eulerian polynomial. Here we used the palindromicity of these polynomials.

\begin{reminder}
As the univariate Eulerian polynomials are palindromic, if $r$ is a root one of them, so is its reciprocral $\frac{1}{r}.$
\end{reminder}

Additionally, as we want a better bound, we require $\un(n)<\mult_{v}(n),$ where $\un(n)$ is the sequence in $n$ found above giving the optimal bound (actual root of its corresponding determinant) obtained through the application of the relaxation to the univariate polynomial and $\mult_{v}(n)$ is the sequence in $n$ giving the (better multivariate) bound we desire to find during this section, which comes after linearizing through the vector $v$ (depending on the sequence $y\colon\mathbb{N}\to\mathbb{R}$) the relaxation obtained after applying the relaxation to the multivariate Eulerian polynomials.

\begin{reminder}
Thus, we need to have in mind that $\mult_{v}$ is not even the best bound possible. It is just the bound we can work with here. We are linearizing through $v$ and not computing the actual root of the determinant of the LMP giving the spectrahedron defining the relaxation of the multivariate Eulerian polynomial. Contrary to that, $\un$ corresponds to the actual root of the corresponding determinant of the LMP giving the spectrahedron defining the relaxation of the univariate Eulerian polynomial. Thus $\mult_{v}$ is itself just a bound for how good $\mult$ (the one obtained through the root of the determinant of the relaxation applied to the multivariate polynomial) can be.
\end{reminder}

This last condition forces us to study the difference $\mult_{v}(n)-\un(n)$ in order to find a sequence $y$ (verifying the condition expressed in Equation \ref{conditionsy} above) such that $\mult_{v}(n)-\un(n)>0.$ In order to find our optimal $y$, as traditionally, we compute the derivative of the expression $\mult_{v}(n)$ viewed as a function of the parameter $y$ and search for its zeros. We do this now.

\begin{computacion}
In this case, writing $\mult_{v}(n)=\frac{N}{D}$ so the derivative is $\frac{N'D-ND'}{D^2}$, we are then looking for our optimal $y$ among the roots of the polynomial in $y$ given by the numerator $N'D-ND'$, which is quadratic in $y$ and depends on the parameter $n$. From the two roots obtained as usual via the well-known quadratic formula $y=\frac{-b\pm\sqrt{b^{2}-4ac}}{2a}$ we now have to find the way to determine which one gives the maximum of $\mult_{v}(n).$ To see this, observe that we have already seen that the discriminant of the denominator $D$ verifies $\Delta=64 - 2^{6 + n} + 2^{4 + 2n} - 8n - 2^{4 + 2n}n + 8 \cdot 3^{1 + n}n<0$ once $n$ is big enough (which was expectable beforehand also because the fraction $\frac{N}{D}$ cannot go to infinity as a function of $y$ because it is bounded by the absolute value of the extreme root of the $n$-th Eulerian polynomial for any given $n$) so it has no real zeros and therefore we can easily deduce now that the roots of $N'D-ND'$ correspond to a maximum and a minimum of $\mult(n)=\frac{N}{D}$ and that such function is continuous as a function of $y$. It might seem worth studying the horizontal asymptotes of such function but observe that $\lim_{y\to\infty}\frac{N}{D}=\frac{-2 + 2^{1 + n} - n}{n}=\lim_{y\to-\infty}\frac{N}{D},$ which does not directly help in our decision. However, now it is evident that, as a function of $y$, $\mult_{v}(n)$ cut the horizontal asymptote $f(y)=\frac{-2 + 2^{1 + n} - n}{n}$ in exactly one point and therefore the difference with this horizontal asymptote in any of the limits determines the relative position of the maximum and the minimum, as there are just two options on how this function looks. This observation finally puts us in the correct way to distinguish between the maximum and the minimum. Observe that the difference between the function $\mult_{v}(n)$ and the asymptote equals \begin{gather*}
    \mult_{v}(n) - \frac{-2 + 2^{1 + n} - n}{n}= \frac{z}{w}
\end{gather*} with $z=20 - 5 \cdot 2^{2 + n} - 2^{3 + n} + 2^{4 + 2n} - 2^{3 + 3n} - 4 \cdot 3^{1 + n} + 2^{2 + n} \cdot 3^{1 + n} + 2n - 2^{1 + n} n - 2^{2 + n} n + 2^{3 + n} n + 2^{2 + 2n} n - \frac{1}{3} \cdot 2^{3 + 2n} n - \frac{1}{3} \cdot 2^{4 + 2n} n + \frac{1}{7} \cdot 2^{4 + 3n} n + \frac{1}{7} \cdot 2^{5 + 3n} n + 2 \cdot 3^n n - \frac{1}{5} \cdot 2^{1 + n} \cdot 3^{3 + n} n - 4^{1 + n} n + 4^{2 + n} n - \frac{1}{5} \cdot 6^{2 + n} n + \frac{1}{7} \cdot 8^{1 + n} n + 16y - 2^{3 + n} y - 2^{4 + n} y + 2^{3 + 2n} y + 4ny - 2^{2 + n} ny - \frac{1}{3} \cdot 2^{3 + 2n} ny - \frac{1}{3} \cdot 2^{4 + 2n} ny + 4 \cdot 3^{1 + n} ny$ and $w=10n - 2^{2 + n} n + 2^{2 + 2n} n - 2 \cdot 3^{1 + n} n + n^2 + 8ny - 2^{2 + n} ny + 2n^2 y + n^2 y^2$ so the denominator $w$ is always positive at infinity while the numerator $z$, when $y$ goes to infinity, is dominated by the term multiplying $y$ given by $16 - 2^{3 + n} - 2^{4 + n} + 2^{3 + 2n} + 4n - 2^{2 + n}n - \frac{1}{3} \cdot 2^{3 + 2n}n - \frac{1}{3} \cdot 2^{4 + 2n}n + 4 \cdot 3^{1 + n}n
,$ which is dominated by the terms $- \frac{1}{3} \cdot 2^{3 + 2n}n - \frac{1}{3} \cdot 2^{4 + 2n}n=-2^{3 + 2n}n<0$ so the $\lim_{y\to-\infty}\frac{z}{w}=0^{+}$ (because $z\approx-2^{3 + 2n}ny>0$ as $y\to-\infty$) which implies that, for $y$ close enough to $-\infty$, $\mult_{v}(n)-\frac{-2 + 2^{1 + n} - n}{n}>0$ or, equivalently, $\mult_{v}(n)>\frac{-2 + 2^{1 + n} - n}{n}$ and, therefore, the function $\mult_{v}(n)$, when viewed from left to right, increases from above the asymptote until it reaches a maximum over it and then decreases towards a minimum reached under the asymptote after crossing it before increasing back towards the asymptote approaching it from below. All in all, we have established that the maximum is reached at the leftmost of the two zeros of $N'D-ND'.$
\end{computacion}

Thus, we have determined which root of $N'D-ND'$ corresponds to a maximum of $\frac{N}{D}$ and now we know exactly the optimal choice of $y$ in our way towards the best bound attainable through the family $(y,0,1,-\mathbf{1})$ we are studying in this section. This means that our choice of $y$ to reach this best bound must be the one corresponding to the minus sign before the square root, i.e., $\frac{-b-\sqrt{b^2-4ac}}{2a}.$ Thanks to these computations we can therefore write the next result.

\begin{lema}[Optimal choice]\label{optimaly}
Suppose the conditions in Equation \ref{conditionsy} are met. The optimal sequence $y\colon\mathbb{N}\to\mathbb{R}$ providing the best possible bound linearizing through the vector $v=(y,0,1,-\mathbf{1})\in\mathbb{R}^{n+2}$ is $y=\frac{-b-\sqrt{b^2-4ac}}{2a}$ with $N'D-ND':=ay^2+by+c.$
\end{lema}
\begin{proof}
The arguments in the computation above shows that this is the root corresponding to the value that provides the optimal bound of this form. The proof stems directly from that argument.
\end{proof}

Now we have to extreme our care. The reason for this is the appearance of square roots while we try to estimate asymptotic growths. We explain this with more detail in the next section.

\subsection{Correct management of radicals}

This section has to begin with a warning. Square roots can be tricky if we do not properly deal with them.

\begin{warning}[Freeing radicals]
Our expressions will begin to show radicals from now on because of the expression we just obtained for the optimal $y$. Doing asymptotics with radicals is problematic and cannot in general be done directly when the dominant term inside the radical and the dominant term outside cancel. For this reason we are forced to deal with conjugates when this happens. Obtaining a $0$ for dominance in asymptotic growth is usually not acceptable as this kills all the growth hidden behind the annihilating dominant terms. The only time when obtaining a $0$ for a difference in asymptotic algebra is justified is when the two expression are \textit{exactly the same}. Thus, from now on, dealing with expressions of the form $a-\sqrt{b}$ will be a task that will require an additional amount of work.
\end{warning}

We remind that there was a condition that we should check for $y$, as exposed in Equation \ref{conditionsy}. In the spirit of our warning, we show the next relevant example of the phenomena we talk about. In particular, the expression in the next example comes from the computations above and it is therefore stemming from our study of the asymptotics of the bounds we are analyzing.

\begin{ejemplo}[Dominant root growths annihilate]
Observe $N$ of Lemma \ref{formofboundfirstwinner} after we substitute for the optimal $y$ we computed in Lemma \ref{optimaly}. We can see that the terms winning in the numerator (inside and outside of the root) annihilate as the expression obtained if we only look at dominant terms is \begin{gather*}\frac{1}{7}2^{4 + 3n}n + \frac{1}{7}2^{5 + 3n}n + \frac{1}{7}2^{6 + 3n}n
-\sqrt{\frac{1}{49}2^{8 + 6n}n^2 + \frac{1}{49}2^{12 + 6n}n^2 + \frac{1}{49}2^{13 + 6n}n^2}\\=2^{4+3n}n-\sqrt{2^{8+6n}n^2}=0.\end{gather*} This is precisely the problem we are warning about.
\end{ejemplo}

Therefore, in order to really determine the growth of the numerator we have to proceed differently (as we will have to do many times in the future). We have to follow the next path.

\begin{procedimiento}[Multiplication by conjugate, difference of squares and division]\label{processy}
In particular, here we see that the conjugate does not annihilate in dominant terms as it gives $2^{4+3n}n+\sqrt{2^{8+6n}n^2}=2^{5+3n}n$ and therefore we can multiply by it to obtain that the difference of squares is dominantly $2^{4 n+8} 3^{n+1} n$ so our numerator is dominantly $\frac{2^{4 n+8} 3^{n+1} n}{2^{5+3n}n}=3^{n+1}2^{n+3}.$ For the denominator we proceed in the usual way and therefore we get that $y\sim$ \begin{gather*}\frac{3^{n+1}2^{n+3}}{-\frac{1}{3}2^{4 + 2n}n + \frac{1}{3}2^{6 + 2n}n
}=\frac{3^{n+1}2^{n+3}}{2^{4+2n}n}=\frac{3^{n+1}}{2^{n+1}n}.\end{gather*}
\end{procedimiento}

With this, in particular, we can now ensure that the sequence $y\colon\mathbb{N}\to\mathbb{R}$ we just decided to choose verifies our previous restrictions. Indeed, an easy computation makes us sure of the fact that this sequence $y\colon\mathbb{N}\to\mathbb{R}$ verifies the conditions in Equation \ref{conditionsy}.

\begin{lema}[The $y$ we wanted]
The sequence $y\colon\mathbb{N}\to\mathbb{R}$ we determined in Lemma \ref{optimaly} verifies the conditions in Equation \ref{conditionsy}.
\end{lema}

\begin{proof}
As we compute in Procedure \ref{processy} above that $y\sim\frac{3^{n+1}}{2^{n+1}n}$, it is clear that $y>1$ and that $\lim_{n\to\infty}\frac{\max\{2^{n+1}y^2,2y4^{n+1}\}}{8^{n+1}}=0$ because $2^{n+1}y^2\sim\frac{9^{n+1}}{2^{n+1}n^2}$ and $\frac{9}{2}<8$ and $2y4^{n+1}\sim2\frac{6^{n+1}}{n}$ and $6<8$. This completes the proof of the adequacy of the sequence $y\colon\mathbb{N}\to\mathbb{R}$ we chose.
\end{proof}

This proves that $y$ thus chosen verifies our growth constraints and therefore it works as a choice for providing a \textit{correct} bound. Now, finally, we only have to check that this choice of $y$ provides a multivariate bound \textit{effectively} improving the optimal univariate bound obtained above. Substituting $y$ for this root in $\frac{N}{D}$ gives the improved bound we were pursuing here so the only thing that we still need to do is comparing this new bound $\mult_{v}(n)$ with the previous bound $\un(n).$

\begin{lema}[Form of the comparison]
Writing $y=\frac{f-\sqrt{g}}{h}$ and $\un=\frac{p+r\sqrt{q}}{s}$ we can express the difference $\mult_{v}(n)-\un(n)=$ \begin{gather*}
\frac{k+v+u+w}{s(\gamma+\delta\sqrt{g})}\end{gather*} with $\mult_{v}:=\frac{\alpha+\beta\sqrt{g}}{\gamma+\delta\sqrt{g}}$ and $k:=s\alpha-p\gamma,v:=(s\beta-p\delta)\sqrt{g},u:=-r\gamma\sqrt{q}$ and $w:=-r\delta\sqrt{gq}$.
\end{lema}

\begin{proof}
For this we compute $\mult_{v}(n)-\un(n)$. We do it step by step for maximum clarity. First, note that $y^{2}=\frac{f^{2}+g-2f\sqrt{g}}{h^{2}}$ and we can substitute these values in the expression of $\mult_{v}:=\frac{N}{D}$ and kill denominators inside $N$ and $D$ writing $\mult_{v}=\frac{N(y)}{D(y)}=\frac{h^{2}N(y)}{h^{2}D(y)}$ so at the end we obtain the manageable expression $\mult_{v}=\frac{\alpha+\beta\sqrt{g}}{\gamma+\delta\sqrt{g}}.$ Finally, remember that we wrote $\un=\frac{p+r\sqrt{q}}{s}$ and, therefore, in order to study the difference between these two bounds, the most convenient expression is $\mult_{v}-\un=$\begin{gather*}
\frac{s(\alpha+\beta\sqrt{g})-(p+r\sqrt{q})(\gamma+\delta\sqrt{g})}{s(\gamma+\delta\sqrt{g})}=\\\frac{(s\alpha-p\gamma)+(s\beta-p\delta)\sqrt{g}-r\gamma\sqrt{q}-r\delta\sqrt{gq}}{s(\gamma+\delta\sqrt{g})}=\\\frac{k+v+u+w}{s(\gamma+\delta\sqrt{g})}\end{gather*} with $k=s\alpha-p\gamma,v=(s\beta-p\delta)\sqrt{g},u=-r\gamma\sqrt{q}$ and $w=-r\delta\sqrt{gq}$, as we wanted to show.
\end{proof}

We want to be able to estimate correctly the dominant behaviour of both the numerator and the denominator. The main problem with these expressions is that they contain square roots and therefore we need to deal with these carefully because we cannot directly try to proceed with dominant terms due to the many cancellations happening both in the denominator and in the numerator. We see in detail the problems produced by these cancellations. Simple computations show the following.

\setlength{\emergencystretch}{3em}%
\begin{hecho}[Annihilation of dominant terms]
The dominant terms of the sumands in our decomposition of the numerator are \begin{gather*}
k\sim -2^{11 n+15} 3^{n+1} n^{4},\\
v\sim(2^{8n+11} 3^{n+1} n^{3})\sqrt{2^{6 n+8} n^{2}}=2^{11n+15} 3^{n+1} n^{4},\\
u\sim-(2^{2 n+3} 3^{n+1})(2^{6n+9} n^{3})\sqrt{2^{6n+6} n^{2}}=-2^{11n+15}3^{n+1}n^{4},\\
w\sim-(2^{2 n+3} 3^{n+1})(-2^{3 n+5}n^{2})\sqrt{(2^{6 n+8} n^{2})(2^{6n+6} n^{2})}=2^{11n+15}3^{n+1}n^{4}.\end{gather*}
\end{hecho}
\setlength{\emergencystretch}{0em}%

It is easy to see that these dominant terms annihilate, which is a big problem because we need to keep track of the terms surviving the big cancellation in order to be able to really establish the asymptotics of our difference. In order to continue then we have to find the way to keep the biggest surviving terms. We use the good behaviour of these conjugates to free some radicals via multiplication. We proceed similarly with the denominator. This is what we do next.

\begin{remark}[Procedure \ref{processy} saves us]
We continue freeing some radicals through the use of the correct conjugates. The situation in the denominator is similar. In particular, we have the luck that the conjugate expressions in the numerator \begin{gather*}(k+u)-(v+w)\sim-2^{11 n+15} 3^{n+1} n^{4}-2^{11 n+15} 3^{n+1} n^{4}\\-(2^{11 n+15} 3^{n+1} n^{4}+2^{11 n+15} 3^{n+1} n^{4})=\\-4\cdot2^{11 n+15} 3^{n+1} n^{4}=-2^{11n+17} 3^{n+1} n^{4}\end{gather*} and in the denominator \begin{gather*}\gamma-\delta\sqrt{g}\sim 2^{6 n} n^{3}512-(-32) 2^{3 n} n^{2}\sqrt{n^{2} 2^{6 n}256}=\\2^{6 n} n^{3}512-((-32) 2^{3 n} n^{2})(n 2^{3 n}16)=2^{6 n+10} n^{3}\end{gather*} do not annihilate.
\end{remark}

Observe however that \begin{gather*}\gamma+\delta\sqrt{g}\sim 2^{6 n} n^{3}512+(-32) 2^{3 n} n^{2}\sqrt{n^{2} 2^{6 n}256}=0.\end{gather*} Now we can finally proceed with the comparison.

\subsection{Final comparison}

A few computations land us where we want. The result is the following lemma.

\begin{lema}[Difference growth]\label{lemafinal}
We have the asymptotic growth of the difference of bounds $$\mult_{v}-\un\sim\frac{1}{2}\left(\frac{3}{4}\right)^n.$$
\end{lema}

\begin{proof}
Since managing the denominator is easy, we proceed first with the numerator. We do this multiplying by the nice conjugates saw above. This will free some radicals. We warn the reader that in the next explanation of the structure of the computations involved all the names are locally set. Since we have seen that $(k+u)-(v+w)$ is dominantly $-2^{11 n+17} 3^{n+1} n^4$, we can multiply by it in order to get a nice difference of squares. Hence, multiplying the numerator $k+u+v+w$ by $(k+u)-(v+w)$ we obtain precisely $(k+u)^2-(v+w)^2=k^2+u^2+2ku-v^2-w^2-2vw=(k^2+u^2-v^2-w^2)+(2ku)+(-2vw)=r+s+t.$ We will need a further use of conjugates to determine the real growth of this because the expressions of $s$ and $t$ contain $\sqrt{q}.$ We have to see first what is the growth of the conjugate $r-(s+t)$. In order to see this, we first need to establish the real growth of $s+t$, which is not immediate because there is a cancellation again as it is easy to see that $ku\sim (2^{11n+15} 3^{n+1} n^{4})^2\sim vw$. Thus the use of conjugates is necessary again. We can immediately compute the asymptotic growth of the part with no roots involved $r\sim-2^{19 n+31} 3^{3 n+2} n^7$. It is clear that $s-t=(2ku)-(-2vw)=2(ku+vw)\sim 4(2^{11n+15} 3^{n+1} n^{4})^2=2^{22n+32} 3^{2n+2} n^{8}$ while computing again we can see that $(s+t)(s-t)=s^2-t^2\sim -2^{63 + 41 n} 3^{4 + 5 n} n^{15}$ so we obtain therefore that $s+t=\frac{s^2-t^2}{s-t}\sim\frac{-2^{63 + 41 n} 3^{4 + 5 n} n^{15}}{2^{22n+32} 3^{2n+2} n^{8}}=-2^{31 + 19n} 3^{2 + 3n} n^{7}\sim r$. This implies that $r+s+t\sim-2^{19 n+32} 3^{3 n+2} n^{7}$ and therefore we obtain that $k+v+u+w$ is dominantly $\frac{-2^{19 n+32} 3^{3 n+2} n^{7}}{-2^{11 n+17} 3^{n+1} n^4}=2^{8 n+15} 3^{2 n+1} n^3.$ Proceeding similarly with the denominator we get that the conjugate of the problematic factor is dominantly $\gamma-\delta\sqrt{g}\sim$ \begin{gather*} 2^{6 n} n^{3}512-(-32) 2^{3 n} n^{2}\sqrt{n^{2} 2^{6 n}256}=\\2^{6 n} n^{3}512-((-32) 2^{3 n} n^{2})(n 2^{3 n}16)=2^{6 n+10} n^{3}\end{gather*} while the resulting difference of squares is dominantly $2^{12 n+20} n^{5}$ so, finally, we can see that the problematic factor is dominantly $\gamma+\delta\sqrt{g}\sim\frac{2^{12 n+20} n^{5}}{2^{6 n+10} n^{3}}=2^{6 n+10} n^2$ so, as $s\sim3^{n} 2^{4 n} n 192,$ the denominator is dominantly $$s(\gamma+\delta\sqrt{g})\sim(3^{n} 2^{4 n} n 192)(2^{6 n+10} n^2)=2^{10 n+16} 3^{n+1} n^3.$$ All in all, we see that the difference is dominantly $$0<\frac{2^{8 n+15} 3^{2 n+1} n^3}{2^{10 n+16} 3^{n+1} n^3}=\frac{3^n}{2^{1+2n}}=\frac{1}{2}\left(\frac{3}{4}\right)^n\to0^{+} \mbox{\ when\ } n\to\infty,$$ as we wanted. This finishes our proof.
\end{proof}

Observe that this implies $\mult_{v}-\un\to 0^{+} \mbox{\ when\ } n\to\infty$. Therefore we have reached our destination in this section.

\begin{reminder}[Packing up things proves Proposition \ref{multundiff}]
Lemma \ref{lemafinal} finishes our path here. The growth mentioned in the cited proposition has been finally established after we completed several steps of freeing radicals and conjugating.
\end{reminder}

We can immediately observe something not very nice about this difference. Although we get an improvement, this improvement vanishes when $n$ grows towards infinity.

\begin{remark}[Difference is small]
The fact that the difference of bounds verifies $\mult_{v}-\un\to 0^{+} \mbox{\ when\ } n\to\infty$ tells us that we are in the right path. However, at the same time, we see that we obtain here an improvement that diminishes when $n$ get bigger towards infinity. This tells us that our job is not yet done because we would like to see better improvements. In fact, our numerical experiments show that this improvement could be greatly improved. We are searching for an explosive improvement of the bound. However, for this we will need to look at more sophisticated ways of finding approximations to the generalized eigenvectors of the relaxation restricted to the diagonal.
\end{remark}

This points towards a direction promising us a better bound but requiring us to find a more elaborated structure for the vectors we use in our linearizations. We need to think about ideas allowing us to build these vectors. In order to access and create intuitions about the form and the structure of the new vectors we need, we perform many experiments. These experiments show us many things we need to have in mind in our future analysis. These insights are both about the growth of the difference between the relaxations and about the structure of the vectors we look for here.

\begin{observacion}[Possibilities and conditions for better bounds]
What we have seen up until here tells us that in order to improve the bound we have to be able to split the information about the different variables in the multivariate polynomial when we look at the relaxation. The vector we used here $v=(y,0,1,-\mathbf{1})$ does not split the tail of variables $(x_{3},\dots,x_{n+1})$ since all of them get assigned the value $-1.$ We have to find a solution to this.
\end{observacion}

There is still a lot of room for that splitting because there are many equal entries in the vector $v$ above. We will do this charging the entries with variable (in the sense that the values will change with $n$) and different (in the sense that not so many entries repeat for the same value of $n$) sequences that can distinguish between the entries of the vector they appear in. Thus, we will be able to split the information in the LMP defining the spectrahedron of the relaxation through the linearization using that vector.

\begin{procedimiento}[Searching for better bounds through higher multivariability]\label{procedurebeating}
Observe that there is only one different entry (having value $1$) among the significative (the $0$ is not relevant as it corresponds to the ghost variables $x_{1}$) tail of the vector $v$. Setting a greater proportion of entries different will allow us to distinguish and therefore pour more information about the relaxation into our linearization. In order to accomplish this, the number of different entries will have to grow with $n$. At the same time, we will not use anymore static sequences. We need something that varies when $n$ varies. In the vector $v$, just $y$ varies, and that was already enough to beat the application of the relaxation to the univariate polynomial. We need to follow suit and make also a greater proportion of the entries variable depending on $n$.
\end{procedimiento}

We will put this procedure in practice in the next section. This will set in motion the machinery that will allow us to show that the application of the relaxation to the multivariate Eulerian polynomials produces indeed a much greater improvement in the (asymptotic growth) quality of our bounds. This growth is indeed explosive because it is exponential.

\begin{objetivo}[Seeking for explosive exponential differences in deeper order asymptotics]
After we find a structure for a vector that seems to numerically accomplish the improvement we are looking for, we will compare it to $\uni$. This comparison, as before, will be done through a computation of a difference. This difference will therefore give us the point of the exponential asymptotic scale we are using at which such improvement happen. We will be happy if this point is a point in the exponential scale $a^{n}$ with $a>1,$ which will tell us that the new bound explosively (exponentially) diverges from $\uni$ when $n$ grows. That is our objective now.
\end{objetivo}

We are able to accomplish the objective above after some experimentation. The journey towards the discovery of the structure of the approximated generalized eigenvector we use and the proof that it does indeed produce the improvement we seek is the content of the next section. In fact, in the next section, we manage to explosively win against $\uni$ with a vector having the characteristics described in Procedure \ref{procedurebeating}. In particular, this vector produces a linearization giving a bound whose difference with $\uni$ grows exponentially with $n$.
\chapter[Explosively Beating the Univariate Setting]{Explosively Beating the Univariate Setting in Higher Order Asymptotics}\label{ChExplosion}

Finally, here the relaxation will show that going multivariate is actually a dealbreaker in order to improve the bounds we obtain. This chapter will be heavily computational both symbolically and numerically, as many experiments had to be performed in order to find a good approximated solution to the generalized eigenvector problem we had in our hands. Fortunately, we succeeded.

\section{The search for a good family}

Parity will play a role due to the nature of the sequence of vectors we found. Anyway, the odd case follows similarly changing the structure of the sequence of vectors.

\begin{convencion}[Even case]
We will prove that the bounds get apart for $n=2m$ even. We do this for practical reasons related to the shape of the sequence of vectors that we will use.
\end{convencion}

In our proof, we will use some of the computations made above. The optimal $y$ found there is still useful for us here.

\begin{observacion}[Previous optimal helps]
In particular, the optimal value obtained for the first entry $y$ (over which we had to maximize the absolute value of the multivariate bound) will be used also here. This will save us a lot of work with respect to what we had to do in the previous section.
\end{observacion}

Now we claim that the multivariate bound obtained through the vector $$(y,0,(-2^{m-i})_{i=3}^{m},(0,\frac{1}{2}),(1)_{i=1}^{m})\in\mathbb{R}^{n+2}$$ diverge apart from the optimal univariate bound. We have to compute terms again, as before. 

\begin{remark}[Room for improvement]
Notice that the improvement cannot occur at the first term of the asymptotic growth in our scale because the bounds obtained already share the same terms. We take differences because it will happen at some other term.
\end{remark}

\begin{notacion}[Good family of vectors]
As we center our attention into a different vector, we denote now here $v=$ $$(y,0,(-2^{m-i})_{i=3}^{m},(0,\frac{1}{2}),(1)_{i=1}^{m})\in\mathbb{R}^{n+2}$$ and therefore $\mult_{v}$ is now the bound obtained through this new linearization.
\end{notacion}

We obtained this vector through numerical experimentation. We explain the main points in the process of these experiments here.

\begin{experimento}[Guessing the good family of vectors]
We compute the relaxation for small instances of the multivariate Eulerian polynomials, that is, setting $n=1,2,3,\dots$. The first thing we do is looking at the relaxation along the diagonal setting all the variables equal. Then we solve the corresponding generalized eigenvalue problem that this generates and take the rightmost (biggest) generalized eigenvalue, which is negative because we began with a relaxation. Then we compute the generalized eigenvector associated to such generalized eigenvalue. Then we plot the entries of these eigenvectors equally spaced in the interval $[0,1]$. When we do this for several values of $n$, we obtain the Figure \ref{figura}. \begin{figure}[h]\label{figura}
\centering
\scalebox{.75}{\includegraphics{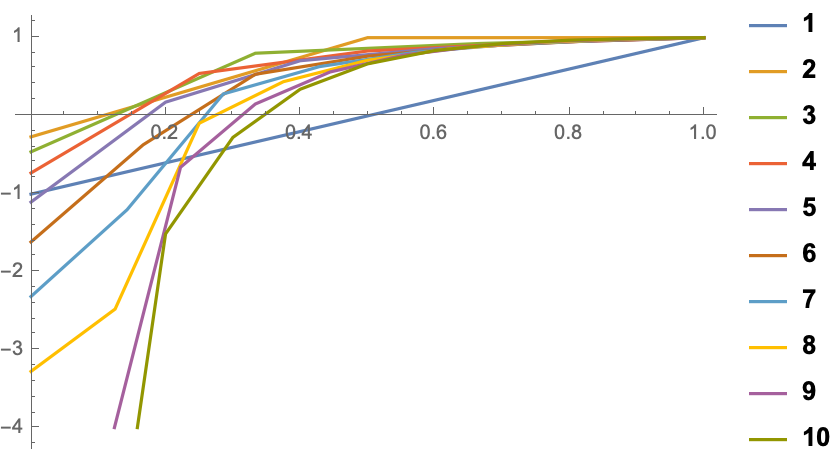}}
\caption{Representation of the entries of the eigenvectors in the interval $[0,1]$ for the relaxation corresponding to the multivariate Eulerian polynomials of degrees 1 to 10.}
\end{figure} The observation of this figure (and the generalized eigenvectors that form it) shows us that we lack real control over what the first entry does. This is why we put there a $y$ that we will have to compute optimizing later. We skip representing the second entry because it does not play any significant role since the corresponding entries are all $0$ in the matrices of the relaxation. We just set it equal to $0$. Now, for the rest of the entries we see that they grow from negative to positive. Passing from negative to positive, there is a jump that we identify as from $0$ to $\frac{1}{2},$ approximately. Finally, the last entries, these lying in the tails of the eigenvectors, seem to stabilize towards $1$. Thus, this only let open how the growth in the head happens. Fortunately, it is easy to see that it goes through reducing the absolute value by powers of $2$. And this is how these experiments allowed us to construct a vector good enough to give us a precise enough approximation to the generalized eigenvector we need in order to linearize our problem and obtain the better asymptotic bounds that we will compute in this chapter. Ideally, we would have to compute the optimal value of $y$ depending on $n$. However, it turns out that the optimal $y$ computed for the examples developed in the previous chapter is already good enough for our objectives here, which will save us a lot of work to do in this chapter.
\end{experimento}

After seeing how we perform these experiments to search for nice enough guesses for our linearizing approximation of generalized eigenvectors, things are finally set in motion. The computations can begin.

\section{The computations for the good family}

As before, we have to compute the form of the linearization in order to study the bound. This computation is the content of this section. We recall, once again, that $n=2m$.

\begin{proposicion}[Form of bound obtained for arbitrary $y$]
The bound for the extreme root of the univariate Eulerian polynomial obtained through linearization of the relaxation applied to the corresponding multivariate Eulerian polynomial by the vector $(y,0,(-2^{m-i})_{i=3}^{m},(0,\frac{1}{2}),(1)_{i=1}^{m})\in\mathbb{R}^{n+2}$ is of the form $x_{n,r}\geq-\frac{D}{N}$ with $D=$ \begin{gather*}
    -\frac{1}{12} + 2^{3m} + 2^{2 + m} + 5 \cdot 2^{-3 + 2m} - 7 \cdot 2^{-1 + 2m} + 3 \cdot 2^{1 + 3m} - 2^{3 + 3m}\\ + \frac{1}{3} \cdot 2^{2 + 4m} + \frac{1}{3} \cdot 2^{3 + 4m}
 - 2 \cdot 3^{-1 + m} - 2^{4 + m} \cdot 3^{-1 + m} + 3^m - 2^{1 + m} \cdot 3^m\\ - 3^{1 + m} + 2^{2 + m} \cdot 3^{1 + m}
 - 2 \cdot 3^{1 + 2m} - \frac{11 \cdot 4^{-2 + m}}{3} + m - 2^{3m} \cdot m\\ - 5 \cdot 2^{-4 + 2m} \cdot m - 2^{-3 + 2m} \cdot m
 + 2^{-1 + 2m} \cdot m + 4^{-2 + m} \cdot m + 2^{-4 + 2m} \cdot m^2\\ (- 3 + 2^{-1 + m} + 2^{1 + m}
 - 2^{2 + m} + 2^{2 + 2m} - 2m - 2^{-1 + m} \cdot m)y + 2my^2
\end{gather*} and $N=$ \begin{gather*}
    \frac{1}{12} + 2^{2m} - 2^{3m} + 5 \cdot 2^{4m} - 2^{2 + m} + \frac{11}{3} \cdot 2^{-4 + 2m} - 5 \cdot 2^{-3 + 2m} - 7 \cdot 2^{-1 + 2m} \\
 + 3 \cdot 2^{1 + 2m} + 9 \cdot 2^{-3 + 3m} - 47 \cdot 2^{-2 + 3m} + 3 \cdot 2^{-1 + 3m} - 2^{2 + 3m} - \frac{1}{7} \cdot 2^{3 + 3m} \\
 + \frac{1}{7} \cdot 2^{4 + 3m} + \frac{5}{7} \cdot 2^{5 + 3m} - \frac{27}{5} \cdot 2^{-3 + 4m} + 5 \cdot 2^{-1 + 4m} - 2^{1 + 4m} + 2^{1 + 5m} \\
 + 3 \cdot 2^{2 + 5m} - 2^{4 + 5m} + \frac{1}{7} \cdot 2^{3 + 6m} + \frac{1}{3} \cdot 2^{4 + 6m} + \frac{1}{21} \cdot 2^{5 + 6m} + 2 \cdot 3^{-1 + m} \\
 - 11 \cdot 2^{2m} \cdot 3^{-1 + m} - \frac{1}{5} \cdot 2^{3 + m} \cdot 3^{-1 + m} + 2^{4 + m} \cdot 3^{-1 + m} + 13 \cdot 2^{2 + 2m} \cdot 3^{-1 + m} \\
 - 2^{5 + 3m} \cdot 3^{-1 + m} - 3^m - 2^{-1 + m} \cdot 3^m + 2^{1 + m} \cdot 3^m + 7 \cdot 2^{-1 + 2m} \cdot 3^m \\
 - 2^{2 + 3m} \cdot 3^m + 3^{1 + m} - 2^{2 + m} \cdot 3^{1 + m} - 2^{3 + 2m} \cdot 3^{1 + m} + 2^{3 + 3m} \cdot 3^{1 + m} \\
 - 2^{1 + 2m} \cdot 3^{2 + m} + 2^{-2 + 2m} \cdot 3^{3 + m} + 4 \cdot 3^{1 + 2m} - 2^m \cdot 3^{1 + 2m} - 2^{1 + m} \cdot 3^{1 + 2m} \\
 + 2^{2 + m} \cdot 3^{1 + 2m} - \frac{1}{5} \cdot 2^{2 + 2m} \cdot 3^{1 + 2m} + 6^m - 6^{1 + m} - \frac{13}{5} \cdot 6^{1 + 2m} \\
 - m - 2^{2m} \cdot m + 2^{3m} \cdot m + 5 \cdot 2^{-3 + 2m} \cdot m + 2^{-2 + 2m} \cdot m - 5 \cdot 2^{-2 + 4m} \cdot m \\
 - 2^{-1 + 4m} \cdot m + 2^{1 + 4m} \cdot m - 2^{1 + 5m} \cdot m + 2^{1 + 2m} \cdot 3^{-1 + m} \cdot m + 2^{-2 + 2m} \cdot 3^m \cdot m \\
 - 2^{-1 + 2m} \cdot 3^{1 + m} \cdot m + 2^{-1 + m} \cdot 3^{1 + 2m} \cdot m - 2^{-4 + 2m} \cdot m^2 + 2^{-3 + 4m} \cdot m^2 \\
 + y(3 - 3 \cdot 2^{-1 + m} + 2^m + 5 \cdot 2^{3m} + 2^{1 + m} - 2^{2 + 2m} + 2^{1 + 3m} \\
 - 2^{3 + 3m} + 2^{3 + 4m} - 2^{4 + m} \cdot 3^{-1 + m} - 2^{1 + m} \cdot 3^m + 2^{2 + m} \cdot 3^{1 + m} \\
 - 4 \cdot 3^{1 + 2m} + 2m + 2^{-1 + m} \cdot m - 2^{3m} \cdot m) +y^{2}(- 2 + 2^{1 + 2m} - 2m).
\end{gather*}
\end{proposicion}

\begin{proof}
The bound obtained through this new vector comes from an equation of the form $D+xN\geq 0$ with $D=$ \begin{gather*}y(yL_{p}(1)-\sum_{i=3}^{m}2^{m-i}L_{p}(x_{i-1})+\frac{1}{2}L_{p}(x_{m+1})+\sum_{i=m+2}^{n+1}L_{p}(x_{i}))-\\\sum_{j=3}^{m}2^{m-j}(yL_{p}(x_{j-1})-\sum_{i=3}^{m}2^{m-i}L_{p}(x_{i-1}x_{j-1})+\\\frac{1}{2}L_{p}(x_{m+1}x_{j-1})+\sum_{i=m+2}^{n+1}L_{p}(x_{i}x_{j-1}))+\\\frac{1}{2}(yL_{p}(x_{m+1})-\sum_{i=3}^{m}2^{m-i}L_{p}(x_{i-1}x_{m+1})+\\\frac{1}{2}L_{p}(x_{m+1}x_{m+1})+\sum_{i=m+2}^{n+1}L_{p}(x_{i}x_{m+1}))+\\\sum_{j=m+2}^{n+1}(yL_{p}(x_{j})-\sum_{i=3}^{m}2^{m-i}L_{p}(x_{i-1}x_{j})+\frac{1}{2}L_{p}(x_{m+1}x_{j})+\sum_{i=m+2}^{n+1}L_{p}(x_{i}x_{j}))
\end{gather*} and $N=$ \begin{gather*}y(y\sum_{k=2}^{n+1}L_{p}(x_{k})-\sum_{i=3}^{m}2^{m-i}\sum_{k=2}^{n+1}L_{p}(x_{i-1}x_{k})+\\\sum_{k=2}^{n+1}\frac{1}{2}L_{p}(x_{k}x_{m+1})+\sum_{i=m+2}^{n+1}\sum_{k=2}^{n+1}L_{p}(x_{k}x_{i}))-\\\sum_{j=3}^{m}2^{m-j}(y\sum_{k=2}^{n+1}L_{p}(x_{k}x_{j-1})-\sum_{i=3}^{m}2^{m-i}\sum_{k=2}^{n+1}L_{p}(x_{i-1}x_{j-1}x_{k})+\\\sum_{k=2}^{n+1}\frac{1}{2}L_{p}(x_{k}x_{j-1}x_{m+1})+\sum_{i=m+2}^{n+1}\sum_{k=2}^{n+1}L_{p}(x_{j-1}x_{k}x_{i}))+\\\frac{1}{2}(y\sum_{k=2}^{n+1}L_{p}(x_{k}x_{m+1})-\sum_{i=3}^{m}2^{m-i}\sum_{k=2}^{n+1}L_{p}(x_{i-1}x_{m+1}x_{k})+\\\sum_{k=2}^{n+1}\frac{1}{2}L_{p}(x_{k}x_{m+1}x_{m+1})+\sum_{i=m+2}^{n+1}\sum_{k=2}^{n+1}L_{p}(x_{m+1}x_{k}x_{i}))+\\\sum_{j=m+2}^{n+1}(y\sum_{k=2}^{n+1}L_{p}(x_{k}x_{j})-\sum_{i=3}^{m}2^{m-i}\sum_{k=2}^{n+1}L_{p}(x_{i-1}x_{j}x_{k})+\\\sum_{k=2}^{n+1}\frac{1}{2}L_{p}(x_{k}x_{j}x_{m+1})+\sum_{i=m+2}^{n+1}\sum_{k=2}^{n+1}L_{p}(x_{j}x_{k}x_{i})).\end{gather*} Computing these last expressions finishes the proof.
\end{proof}

As before, now we have to choose $y$ carefully so that this expression gives us an improvement. It turns out that the $y$ we computed above already does the job.

\section{Managing radicals again and a comment on optimality}

In this section, we will consider the $y$ we choose now and, as before, we will deal with the problems it introduces through the radicals in its expression. The $y$ we choose is easy to understand after the previous chapter. We take actually now a $y$ that equals the opposite of the $y$ computed in previous chapter and we check that this choice is appropriate for us.

\begin{proposicion}[Adequacy of our choice of $y$]
Taking $y=\frac{b+\sqrt{b^2-4ac}}{2a}$ with $a,b,c$ as in the previous section but substituting $n=2m$ makes $N>0$ and $D>0$.
\end{proposicion}

\begin{proof}
We remind the reader to observe the change of sign in the numerator of our new $y$. We study now the conditions on $y$ so $N>0$ for $n$ big enough as in the previous section. Now we already know $y\sim-\frac{3^{n+1}}{2^{n+1}n}$. Now the winning terms (in each power of $y$) are $2^{6 m + 3}=8^{n+1}, 2^{4 m + 3}y=4^{2m+1}2y=4^{n+1}2y$ and $2^{1+2m}y^2=2^{n+1}y^2$. Thus the only term that would represent a problem is $4^{n+1}2y$ because it is negative but $4^{n+1}2y\sim -4^{n+1}2\frac{3^{n+1}}{2^{n+1}n}=-6^{n+1}2$, which fades against $8^{n+1}$. Thus, our $y$ is adequate. For checking that $D\neq0$ for this $y$, we have to do exactly the same as in the previous section. This settles the adequacy of this choice of the sequence $y$.
\end{proof}

Once we have our adequate $y$, we have a bound. We want to see that this bound is not just better, but \textit{much} better than the previous bounds. This is the content of the next result, which is the analogue of Lema \ref{lemafinal} for this new choice of a sequence of vectors. In absolute values, our improved root bound is now $|q_{1}^{(n)}|\geq\frac{N}{D}:=\mult_{v(n)}(n)$ and we require again $\un(n)<\mult_{v(n)}(n)$ with the twist that now we want this difference to grow with $n=2m$.

\begin{lema}[Comparison with univariate bound]
We have the asymptotic growth of the difference of bounds $$\mult_{v(n)}(n)-\un(n)\sim\frac{3}{8}\left(\frac{9}{8}\right)^m.$$
\end{lema}

\begin{proof}
We proceed as in the previous section taking special care when freeing radicals and using the strategy of considering conjugates. In particular, we compute $\mult_{v}-\un$. We proceed again step by step writing $y=\frac{f+\sqrt{g}}{h}$ so $y^2=\frac{f^2+g+2f\sqrt{g}}{h^2}$ and substitute these values in the expression of $\mult_{v}=\frac{N}{D}$ and kill denominators inside $N$ and $D$ writing $\mult_{v}=\frac{N(y)}{D(y)}=\frac{h^{2}N(y)}{h^{2}D(y)}$ in order to obtain the manageable expression $\mult_{v}=\frac{\alpha+\beta\sqrt{g}}{\gamma+\delta\sqrt{g}}.$ Remembering again $\un=\frac{p+r\sqrt{q}}{s}$, we develop conveniently the expression $\mult_{v}-\un=$ $$\frac{k+v+u+w}{s(\gamma+\delta\sqrt{g})}$$ with $k:=s\alpha-p\gamma,v:=(s\beta-p\delta)\sqrt{g},u:=-r\gamma\sqrt{q}$ and $w:=-r\delta\sqrt{gq}$. The dominant terms of the sums in our decomposition of the numerator are \begin{enumerate}
\item $k\sim -2^{22 m+19} 3^{2 m+1} m^4$,
\item $v\sim(3^{2 m+1} 4^{8 m+7} m^3)\sqrt{2^{6 n+8} n^{2}}=2^{22 m+19} 3^{2 m+1} m^4$,
\item $u\sim-(-2^{16 m+15} 3^{2 m+1} m^3)\sqrt{2^{6n+6} n^{2}}=-2^{22 m+19} 3^{2 m+1} m^4$,
\item $w\sim(3^{2 m+1} 4^{5 (m+1)} m^2)\sqrt{(2^{6 n+8} n^{2})(2^{6n+6} n^{2})}=2^{22 m+19} 3^{2 m+1} m^4$.\end{enumerate} As happened in the previous section, these dominant terms annihilate so we have to keep track again of the surviving terms freeing radicals through the use of conjugates. We want to be able to estimate correctly the dominant behaviour of both the numerator and the denominator and, as in the previous section, we have the luck that the conjugate expressions in the numerator $(k+u)-(v+w)\sim$ \begin{gather*}-2^{22 m+20} 3^{2 m+1} m^4-2^{22 m+20} 3^{2 m+1} m^4=-2^{22 m+21} 3^{2 m+1} m^4\end{gather*} and in the denominator \begin{gather*}\gamma-\delta\sqrt{g}\sim 2^{12 m+12} m^3-(-2^{6 m+7} m^2)\sqrt{2^{12 m+10} m^2}=\\2^{12 m+12} m^3-(-2^{6 m+7} m^2)(2^{6 m+5} m)=2^{12 m+13} m^3\end{gather*} do not annihilate. We use the good behaviour of these conjugates to free some radicals via multiplication. We proceed similarly with the denominator. This is what we do next.

As managing the denominator is easy, we proceed first with the numerator. We do this multiplying by the nice conjugates we saw above. This will free some radicals. We warn the reader that in the next explanation of the structure of the computations involved all the names are locally set. As we have seen that $(k+u)-(v+w)$ is dominantly $-2^{22 m+21} 3^{2 m+1} m^4$, we can multiply by it in order to get a nice difference of squares. Hence, multiplying the numerator $k+u+v+w$ by $(k+u)-(v+w)$ we obtain precisely $(k+u)^2-(v+w)^2=k^2+u^2+2ku-v^2-w^2-2vw=(k^2+u^2-v^2-w^2)+(2ku)+(-2vw)=r+s+t.$ We will again need a further use of conjugates to determine the real growth of this because the expressions of $s$ and $t$ contain $\sqrt{q}.$ We have to see first what is the growth of the conjugate $r-(s+t)$. In order to see this, we first need to establish the real growth of $s+t$, which is not immediate because there is a cancellation again as it is easy to see that $ku\sim (2^{22 m+19} 3^{2 m+1} m^4)^2\sim vw$. Thus the use of conjugates is necessary again. It is clear that $s-t=(2ku)-(-2vw)=2(ku+vw)\sim 4(2^{22 m+19} 3^{2 m+1} m^4)^2=2^{44 m+40} 3^{4 m+2} m^8$ while computing we can see that $(s+t)(s-t)=s^2-t^2=(2ku)^2-(2vw)^2=4qr^{2}(((s\alpha-p\gamma)\gamma)^{2}-((s\beta-p\delta)g\delta)^{2})\sim-2^{83 m+76} 3^{10 m+5} m^{15}$ so we obtain that $$s+t=\frac{s^{2}-t^{2}}{s-t}\sim\frac{-2^{83 m+76} 3^{10 m+5} m^{15}}{2^{44 m+40} 3^{4 m+2} m^8}=-2^{39 m+36} 3^{6 m+3} m^7.$$ As $r\sim-2^{39 m+36} 3^{6 m+3} m^7$, we get that the sum $r+s+t\sim-2^{39 m+37} 3^{6 m+3} m^7$ and therefore we obtain $k+v+u+w=\frac{(k+u)^{2}-(v+w)^{2}}{(k+u)-(v+w)}\sim\frac{-2^{39 m+37} 3^{6 m+3} m^7}{-2^{22 m+21} 3^{2 m+1} m^4}=2^{17 m+16} 3^{4 m+2} m^3.$

Proceeding similarly with the denominator we get that the conjugate of the problematic factor is dominantly $\gamma-\delta\sqrt{g}\sim2^{12 m+13} m^3$ while the resulting difference of squares is dominantly $\gamma^2-\delta^{2}g\sim2^{24 m+25} m^5$ so, finally, we can see that the problematic factor is dominantly $\gamma+\delta\sqrt{g}\sim\frac{2^{24 m+25} m^5}{2^{12 m+13} m^3}=2^{12 m+12} m^2$ so, as $s\sim2^{8 m+7} 3^{2 m+1} m,$ the denominator is dominantly $$s(\gamma+\delta\sqrt{g})\sim2^{20 m+19} 3^{2 m+1} m^3.$$ All in all, we see that the difference is dominantly $$\frac{2^{17 m+16} 3^{4 m+2} m^3}{2^{20 m+19} 3^{2 m+1} m^3}=2^{-3 m-3} 3^{2 m+1}=\frac{3}{2^3}\frac{3^{2m}}{2^{3m}}=\frac{3}{8}\left(\frac{9}{8}\right)^m,$$ as we wanted to see.
\end{proof}

The consequences of this lemma are clear. In particular, this result shows us that we obtained the best bound so far using the relaxation. The analysis of this is the content of the next section.

\section{Best bound in this work}

We immediately obtain the next corollary. We will see that we could even go beyond now.

\begin{corolario}[Exponential explosion in the difference]
The difference between $\mult_{v(n)}(n)$ and $\un(n)$ grows exponentially when $n=2m\to\infty$ by steps of size $2$, that is, through even numbers.
\end{corolario}

\begin{proof}
It is clear because $\frac{3}{8}\left(\frac{9}{8}\right)^m\to\infty \mbox{\ when\ } m\to\infty$ and this happens exponentially because $\frac{9}{8}>1.$
\end{proof}

As we see, computing the $y$ in the previous section gave us a hint for a $y$ to use in this case. This choice of $y$ eventually allowed us to carry out the numerical experiment that gave us this last nice bound. However, we could have done it even better directly proceeding similarly as we did in the previous chapter.

\begin{remark}[Better $y$]
A straightforward exercise that should give now an even better bound is to compute the optimal $y$ for the current expression at hand as we did in the previous section. We leave it as an exercise because we obtained already what we wanted using the current suboptimal $y$: we showed that the bound obtained through the multivariate method drastically improves by an \textit{exponential} term with respect to the univariate bound.
\end{remark}

Thus, we have the next nice and easy exercise. Its solution would consist in repeating the whole process that we developed in the previous chapter but for the new sequence of vectors of the form $v$ we use to linearize the relaxation in this current chapter.

\begin{ejercicio}[Finding such $y$]
Find the optimal $y$ for the new family of vectors $v$ used in this chapter and compare it with our current best bound obtained using the suboptimal $y$ obtained through the computations in the previous chapter.
\end{ejercicio}

This finishes our analysis of the bounds obtained through the relaxation. Now we want to compare these bounds with other methods traditionally used to compute bounds for the extreme roots of the univariate Eulerian polynomials. In particular, we will closely look at the work of Sobolev on the bounds and estimations of these roots in the next chapter.

\chapter[Comparisons with Sobolev]{Comparisons with the Estimations of Roots of Univariate Eulerian Polynomials in Sobolev's Work}\label{ChComparisons}


We analyze here error, bounds and asymptotics in Sobolev approach to the roots of Eulerian polynomial. The work of Sobolev in these polynomials can be found in four papers of his selected works translated from Russian in \cite{sobolev2006selected}. We will refer here to these papers translated to English in the collection we just cited instead of the original ones.


\section{Sobolev introduction of Eulerian polynomials}

Sobolev provides bounds and asymptotics for Eulerian polynomials using different techniques and coming from a different approach. In particular, Sobolev uses a different definition for our Eulerian polynomials. In fact, he works with the so called Euler-Frobenius polynomials \footnote{We believe that this difference in the nomenclature, paired with the fact that the original publication was in Russian, provoked that the effect that this result got lost in the literature. There will be more differences in nomenclature, especially when we refer to the method that Sobolev uses to obtain his bound competing with us.} which we introduced in Definition \ref{eulerfrobenius} and which we recall here.

\begin{reminder}[Euler-Frobenius polynomials]
Let $n\in\mathbb{N}$ be a nonnegative integer. We define the $n$-th Euler-Frobenius polynomial as $$E_{n}(x):=\frac{(1-x)^{n+2}}{x}\left(x\frac{d}{dx}\right)^{[n]}\frac{x}{(1-x)^2},$$ where $[n]$ denotes the $n$-th power of the symbolic convolution.
\end{reminder}

We need to be careful when looking at this recurrent definition in other sources. In particular, there is a big problem with indices in the literature.

\begin{warning}[Bad indexing attacks again]
We have mentioned several times along the lines of this work that many references in the literature use a confusing indexing for Eulerian polynomials. This indexing usually makes the $n$-th polynomial have degree $n-1$ and $n-1$ variables. We strongly dislike that way of naming because it is confusing. As a consequence, the previous recurrence may present a different form when viewed in other works, but this is the correct translation to our (we believe) more correct choice of indexing these polynomials.
\end{warning}

This inconsistent indexing is also the root of other confusions. Luckily, Sobolev in \cite{sobolev2006selected} indexed things the same way we do here.

\setlength{\emergencystretch}{3em}%
\begin{remark}[Good indexing in Sobolev]
Note that the indexing in Sobolev's approach in \cite{sobolev2006selected} is actually in line with our indexing. We are therefore not surprised by the fact that this more accurate and tidy indexing is in the work that gives the most accurate results for the roots of these polynomials. As if the good way to index these polynomials is the degree coherent one we use here!
\end{remark}
\setlength{\emergencystretch}{0em}%

Now it is easy to prove by induction, and so we did above in Proposition \ref{samenesseulerfrob}, that these are actually the same as our Eulerian polynomials that just happen to have a different name, use and purpose in the (numerical) analysis literature. Thus, we ensured there that Sobolev and we speak about the same polynomials being used in different mathematical fields or settings. Hence, it is reasonable to translate Sobolev's results to analyze how they fit within our context so we can relate both works.

\section{Sobolev treatment of the roots of Eulerian polynomials}

Now we analyze how Sobolev deals with the roots of these polynomials. Contrary to us, he is initially interested in the distribution of all the roots.

\begin{observacion}[Growth]
Sobolev sees fast that the absolute value of the extreme roots grows and then he sets the tools to develop a further study of the behaviour of these roots at infinity.
\end{observacion}

The first result that he obtains is about the order relations between these roots lying in the interval $[-1,0]$. For this, he goes on to study similar behaviour of roots of related polynomials in the interval $[0,1]$. This strategy of looking at transformations of the original polynomial will prove itself fundamental in the future. Finally, he closes this first paper on Eulerian polynomials with the next result that we collect here using our notations.

\begin{teorema}[Asymptotic behaviour in a interval in $\mathbb{R}_{-}$]
Consider $M>1$ and let the roots $q_{j}^{(k)}$ of the polynomial $E_{k}(x)$ be in the interval $[-M,-\frac{1}{M}]\subseteq\mathbb{R}$. Write $k=2m$ or $k=2m+1$ depending on parity. Then, for a sufficiently large $m$, the roots  $q_{j}^{(k)}$ can be written as $$q_{j}^{(2m)}=-\exp\pi\left[\tan\left(\frac{j\pi}{2m+2}+\frac{\pi}{4m+4}\right)+\epsilon_{j}^{(2m)}\right]$$ and $$q_{j}^{(2m-1)}=-\exp\pi\left[\tan\frac{j\pi}{2m+2}+\epsilon_{j}^{(2m-1)}\right],$$ where $|\epsilon_{j}^{(k)}|\leq C\eta^{k}$ with $C>0$ and $\eta<1$ depend only on $M$.
\end{teorema}

In a subsequent paper, building inspiration in results of Sirazhdinov, Sobolev turns his attention mainly towards what he calls \textit{the end roots} of these polynomials.  These are the roots that end up getting sufficiently close to $0$ or $-\infty$ in the sense that they lie out of the interval $[-M,-\frac{1}{M}]$ for $M>1$ used in the previous Theorem.

\begin{remark}[Unknown behaviour]
Notice that Sobolev mentions before stating the main theorem of this second paper that ``the asymptotic behaviour of a certain portion of the end roots remains unknown for now''. That seems to mean that, when he wrote the proof of such theorem, he did not know how to extend the techniques used in his proof to these other end roots whose asymptotic behaviour remained therefore out of the scope of the result in the previous paper and of the new result that he goes on to present just after that sentence.
\end{remark}

We collect here that main result in order to properly analyze its power and scope in comparison to the bounds we obtained in this thesis. As before, we just translate it to our notations and, for that and clarity, we need first to introduce some terms that Sobolev introduces in the middle of his proof during the derivation of his Equation (7) in his fourth paper dedicated to Eulerian polynomials in \cite{sobolev2006selected}. We first remember some notation.

\begin{recordatorio}
Remember that we can expand the Eulerian numbers as $$E(n+1,k):=\sum_{i=0}^{k}(-1)^{i}\binom{n+2}{i}(k+1-i)^{n+1}.$$ Sobolev calls $a_{s}^{(k)}:=E(k+1,s)$ for short.
\end{recordatorio}

Now we can deal with the expansions of these numbers as Sobolev does. This will help us to understand better his approach.

\begin{notacion}
We can expand the Eulerian number $$E(k+1,s)=(s+1)^{k+1}(1-u_{1}^{(k,s)}+u_{2}^{(k,s)}-\cdots)$$ with $\gls{ujks}:=\binom{k+2}{j}\left(\frac{s+1-j}{s+1}\right)^{k+1}.$ Fix some $s_{0}$ and take $s<s_{0}.$ Thus the number of terms of the form $u_{j}^{(k,s)}$ with $j\in[1,s]$ is therefore finite. Moreover, each of these numbers decreases exponentially when $k$ increases linearly. Therefore, there exists some $\nu>0$, depending only on our choice of $s_{0}$, such that $$E(k+1,s)=(s+1)^{k+1}(1+O(\exp(-\nu k))).$$ We call $\nu(s_{0})$ any such $\nu$.
\end{notacion}

This notation allows us to introduce the theorem in which Sobolev competes with our bounds. The possible refinements of the proof of this theorem constitute the main object of our analysis of Sobolev's work in this chapter.

\begin{teorema}[Sobolev's result on asymptotic behaviour of a portion of the end roots]\cite[Chapter 26, Theorem 1]{sobolev2006selected}
For sufficiently large $k$ we have the asymptotic expansion $$q^{(k)}_{s}=-\left(\frac{s+1}{s}\right)^{k+1}(1+\epsilon_{s}^{(k)}),$$ where $$|\epsilon_{s}^{(k)}|\leq\begin{cases}
    C\exp(-\nu(s_{0})k) &\mbox{\ if\ } s<s_{0} \mbox{\ and\ } \nu(s_{0})>0,\\ \epsilon &\mbox{\ if\ } s+1<\sqrt{\frac{k+1}{\log 2}}(1-\epsilon). $$
\end{cases}$$
\end{teorema}

In particular, Sobolev provides an asymtpotic expansion for the roots. We will analyze what does this asymptotic expansion and its error estimation say about intervals containing the roots. That is, we will not work with the punctual estimation provided by the theorem but with intervals given in terms of minorants and majorants that we will extract through several refinements from of the techniques involved in the proof of this theorem. Thus, the main objective of the next section is building a first approach to this task.

\section{Sobolev's approach in depth}

The arguments of Sobolev can be dealt with much more precision and exactness than he decides to do. If we take care of these, we can greatly improve his stated results.

\begin{computacion}[Refining Sobolev's refinement] The last paper about the roots of Eulerian polynomials in \cite{sobolev2006selected} gives another estimate for the error in Equation (16), which actually rests in Equation (11), which eventually comes from Equation (6) after applying logarithms. Therefore we look back instead at Equation (11) and disentangle it into a form helpful for us to apply exponentials again. Thus we obtain the inequalities $$\frac{E(k+1,j)}{E(k+1,j-1)}\frac{1+\sqrt{1-4\chi_{j}^{(k)}}}{2\sqrt{1-2\chi_{j-1}^{(k)}}}\leq-q_{j}^{(k)}\leq\frac{E(k+1,j)}{E(k+1,j-1)}\frac{2\sqrt{1-2\chi_{j}^{(k)}}}{1+\sqrt{1-4\chi_{j-1}^{(k)}}},$$ where we define \begin{gather*}\gls{chijk}:=\frac{E(k+1,j-1)E(k+1,j+1)}{E(k+1,j)^2}-\frac{E(k+1,j-2)E(k+1,j+2)}{E(k+1,j)^2}+\\\frac{E(k+1,j-3)E(k+1,j+3)}{E(k+1,j)^2}-\cdots.\end{gather*} Notice at this points that values of $\chi_{j}^{(k)}>\frac{1}{4}$ are problematic because there negatives appear under the square roots. Since we are centering currently our attention into the case $j=1$, we explicitly write down the greatly simplified form of that case having in mind that $\chi_{0}^{(k)}=0$ and $\chi_{1}^{(k)}=\frac{E(k+1,0)E(k+1,2)}{E(k+1,1)^2}=\frac{E(k+1,2)}{E(k+1,1)^2}=\frac{3^{k+1}+o(2^{k+1})}{(2^{k+1}+o(1))^2}\to0$ when $k\to\infty$. Hence, for $k$ big enough, we are in the position where these bounds work giving the inequality  $$E(k+1,1)\frac{1+\sqrt{1-4\chi_{1}^{(k)}}}{2}\leq-q_{1}^{(k)}\leq E(k+1,1)\sqrt{1-2\chi_{1}^{(k)}}$$ so we obtain, for $k$ big enough, $$-\sqrt{1-2\chi_{1}^{(k)}}\leq \frac{q_{1}^{(k)}}{E(k+1,1)}\leq-\frac{1+\sqrt{1-4\chi_{1}^{(k)}}}{2}.$$ Now, from this we can try to extract lower order terms of $q_{1}^{(k)}$ proceeding similarly as how we proceeded with our refinements before. If we look first at the bound that mostly interests us, that is, the upper bound (the one in the same direction as the one provided by the relaxation), then we have to look at the inequality $$q_{1}^{(k)}+2^{k+1}\leq -E(k+1,1)\frac{1+\sqrt{1-4\chi_{1}^{(k)}}}{2}+2^{k+1}.$$ The majorant in the RHS expands as \begin{gather*}-\left(\sum_{i=0}^{1}(-1)^{i}\binom{k+2}{i}(1+1-i)^{k+1}\right)\frac{1+\sqrt{1-4\chi_{1}^{(k)}}}{2}+2^{k+1}=\\-\left(2^{k+1}-\binom{k+2}{1}\right)\frac{1+\sqrt{1-4\chi_{1}^{(k)}}}{2}+2^{k+1}=\\-\left(2^{k+1}-(k+2)\right)\left(\frac{1}{2}+\frac{\sqrt{1-4\chi_{1}^{(k)}}}{2}\right)+2^{k+1} 
\end{gather*} and, at this point, we have to deal again with the problem of cancellations because, although it might look like it, this last expression does not grow like $k$, as doing that would mean killing hidden growth inside the radical. Let study this in detail. As the factor $k+2$ is too small to provide any important growth at this stage, we look at the growth of the rest $-2^{k+1}\sqrt{1-4\chi_{1}^{(k)}}+2^{k+1}=-\sqrt{4^{k+1}-4^{k+1}4\chi_{1}^{(k)}}+2^{k+1}=$ \begin{gather*}\frac{(2^{k+1}+\sqrt{4^{k+1}-4^{k+1}4\chi_{1}^{(k)}})(2^{k+1}-\sqrt{4^{k+1}-4^{k+1}4\chi_{1}^{(k)}})}{2^{k+1}+\sqrt{4^{k+1}-4^{k+1}4\chi_{1}^{(k)}}}=\\\frac{4^{k+1}-(4^{k+1}-4^{k+1}4\chi_{1}^{(k)})}{2^{k+1}+\sqrt{4^{k+1}-4^{k+1}4\chi_{1}^{(k)}}}\sim \frac{3^{k+1}4}{2^{k+1}2}=2\left(\frac{3}{2}\right)^{k+1}\end{gather*} that after dividing by the two in the denominator shows that the second growth term of the RHS majorant is therefore $\left(\frac{3}{2}\right)^{k+1},$ which, as we already know, coincides with the second term of the growth of the root. Thus the upper bound provided by Sobolev is, in its second term, as accurate as our bound. In order to see if the bounds obtained by Sobolev actually establish the second growth term, we have to look also at the minorant, which gives \begin{gather*}
    -E(k+1,1)\sqrt{1-2\chi_{1}^{(k)}}+2^{k+1}\leq q_{1}^{(k)}+2^{k+1}.
\end{gather*} Now the LHS of this minorization expands through difference of squares as \begin{gather*}
    \frac{(2^{k+1}-(2^{k+1}-(k+2))\sqrt{1-2\chi_{1}^{(k)}})(2^{k+1}+(2^{k+1}-(k+2))\sqrt{1-2\chi_{1}^{(k)}})}{2^{k+1}+(2^{k+1}-(k+2))\sqrt{1-2\chi_{1}^{(k)}}}\\\sim\frac{2^{2k+2}-(2^{k+1}-(k+2))^{2}(1-2\chi_{1}^{(k)})}{2^{k+2}}=\\\frac{2^{2k+2}-(2^{2k+2}+(k+2)^{2}-2^{k+2}(k+2))(1-2\chi_{1}^{(k)})}{2^{k+2}}\sim\\(k+2)(1-2\chi_{1}^{k})=(k+2)(1-2\frac{E(k+1,2)}{E(k+1,1)^2})=\\(k+2)(1-2\frac{\sum_{i=0}^{2}(-1)^{i}\binom{k+2}{i}(2+1-i)^{k+1}}{(2^{k+1}-(k+2))^{2}})=\\(k+2)\frac{(2^{k+1}-(k+2))^{2}-2(3^{k+1}-(k+2)2^{k+1}+\frac{(k+2)(k+1)}{2})}{(2^{k+1}-(k+2))^{2}}\sim k,
\end{gather*} which tells us that this other bound is actually really bad and does not get close enough to the actual root to allow us establish the second growth term $(\frac{3}{2})^{k+1}$ through a sandwich of inequalities by this side. Therefore, coming back to the majorant, we want to compare it to our bound obtained through the relaxation, but before doing this main exercise we can look again closely at it in case we can elegantly extract a third growth term looking as we conjectured. Remember and notice anyway that Sobolev bounds, as explicitly stated by him in the proof of the theorem, do not allow even establishing the second growth term. Then we have to look now at the inequality $$q_{1}^{(k)}+2^{k+1}-\left(\frac{3}{2}\right)^{k+1}\leq -E(k+1,1)\frac{1+\sqrt{1-4\chi_{1}^{(k)}}}{2}+2^{k+1}-\left(\frac{3}{2}\right)^{k+1}.$$ The majorant in the RHS expands therefore now as \begin{gather*}-\left(\sum_{i=0}^{1}(-1)^{i}\binom{k+2}{i}(1+1-i)^{k+1}\right)\frac{1+\sqrt{1-4\chi_{1}^{(k)}}}{2}+2^{k+1}-\left(\frac{3}{2}\right)^{k+1}=\\-\left(2^{k+1}-\binom{k+2}{1}\right)\frac{1+\sqrt{1-4\chi_{1}^{(k)}}}{2}+2^{k+1}-\left(\frac{3}{2}\right)^{k+1}=\\-\left(2^{k+1}-(k+2)\right)\left(\frac{1}{2}+\frac{\sqrt{1-4\chi_{1}^{(k)}}}{2}\right)+2^{k+1}-\left(\frac{3}{2}\right)^{k+1}.
\end{gather*} The growth in the majorant comes therefore from $-2^{k}(1+\sqrt{1-4\chi_{1}^{(k)}})+2^{k+1}-\left(\frac{3}{2}\right)^{k+1}=2^{k}-\left(\frac{3}{2}\right)^{k+1}-2^{k}\sqrt{1-4\chi_{1}^{(k)}}=$\begin{gather*}
\frac{(2^{k}-\left(\frac{3}{2}\right)^{k+1}-2^{k}\sqrt{1-4\chi_{1}^{(k)}})(2^{k}-\left(\frac{3}{2}\right)^{k+1}+2^{k}\sqrt{1-4\chi_{1}^{(k)}})}{2^{k}-\left(\frac{3}{2}\right)^{k+1}+2^{k}\sqrt{1-4\chi_{1}^{(k)}}}=\\\frac{(2^{k}-\left(\frac{3}{2}\right)^{k+1})^{2}-4^{k}(1-4\chi_{1}^{(k)})}{2^{k+1}}=\frac{4^{k}+\left(\frac{3}{2}\right)^{2k+2}-3^{k+1}-4^{k}+4^{k+1}\chi_{1}^{(k)}}{2^{k+1}}=\\\frac{\left(\frac{3}{2}\right)^{2k+2}-3^{k+1}+4^{k+1}\chi_{1}^{(k)}}{2^{k+1}}\end{gather*} remembering that $4^{k+1}\chi_{1}^{(k)}=4^{k+1}\frac{E(k+1,2)}{E(k+1,1)^2}$ that we can expand as $$4^{k+1}\frac{\sum_{i=0}^{2}(-1)^{i}\binom{k+2}{i}(2+1-i)^{k+1}}{(2^{k+1}-(k+2))^{2}}=3^{k+1}-(k+2)2^{k+1}+\frac{(k+2)(k+1)}{2}$$ so we see now easily that the growth is $\frac{1}{2^{k+1}}\left(\frac{3}{2}\right)^{2k+2}=\frac{3^{2k+2}}{2^{3k+3}}=\left(\frac{9}{8}\right)^{k+1}.$
\end{computacion}

These computations allow us to write down several terms of Sobolev's estimation that will be important for our future analysis on improvements of these estimations.

\begin{teorema}[Sobolev's bound until the third term]
Let $m$ be minorant and $M$ the majorant explicitly stated by Sobolev for the leftmost root of the $k$th Eulerian polynomial $q_{1}^{(k)}$ so that $m\leq q_{1}^{(k)}\leq M.$ Then $$m\sim-2^{k+1}+k+o(k) \mbox{\ and\ } M\sim-2^{k+1}+\left(\frac{3}{2}\right)^{k+1}+\left(\frac{9}{8}\right)^{k+1}+o(\left(\frac{9}{8}\right)^{k}).$$  
\end{teorema}

We see therefore that his majorant is better than the one we obtain through the plain application of the relaxation because his third growth is better. Far from being bad new, this points out towards a new direction in order to improve the relaxation, as we will soon see in the next part of this thesis.

\begin{remark}[Comparison to our bound]The third growth term of the majorant provided by the relaxation and that we computed in the previous chapter is $2\left(\frac{9}{8}\right)^{k+1}$, that is, the double of the one obtained via Sobolev's methods. This is a call to combine our methods with these used by Sobolev.\end{remark}

The first step in order to perform this combination is understanding better how Sobolev establishes these estimations. We do this in the next part through a closer look to the tools he uses. These tools are the well-known Dandelin-Lobachevski-Gr{\"a}ffe (DLG) method to approximate roots (see more on \cite{macnamee2007numerical, macnameeii}) and the \cite[Chapter 26, Lemma 1]{sobolev2006selected} that Sobolev uses very lightly but that has much more power when analyzed and used until its last consequences. These considerations constitute the seminal point that grows until forming the next part of this thesis.
\part{Consequences of Sobolev: a Look through the Glasses of Dandelin-Lobachevski-Gräffe and a Glimpse of a Multivariate Generalization}\label{IV}
\chapter{DLG method in Sobolev argument}\label{ChDLGMeth}

In the work on the roots of Eulerian polynomials that we analyzed above and that can be found in four articles selected in \cite{sobolev2006selected}, Sobolev made use of what he calls \textit{Lobachevsky equations}. These equations have a long history as a method used for estimating the roots of polynomials. We will refer to this method as the \textit{DLG method} for short. We do this because the trick involving this method was used and independently discovered by Dandelin, Lobachevsky and Gr{\"a}ffe, in that order. We describe now the method and identify its importance in Sobolev's papers about the end roots of Eulerian polynomials.

\section{General construction in Sobolev}

In order to provide his bound, Sobolev constructs the coefficients of the polynomial having as roots the squares of the roots of the original Eulerian polynomial. This construction is general. We keep the construction simple supposing that all the roots are different.

\begin{construccion}
Let $p\in\mathbb{C}[x]$ be a univariate monic polynomial with roots $r_{1}< \cdots < r_{n}$. Then we can write the polynomial $p(x)=(x-r_{1})\cdots(x-r_{n})$ so $p(-x)=(-1)^{n}(x+r_{1})\cdots(x+r_{n})$ so $(-1)^{n}p(-x)p(x)=(x+r_{1})\cdots(x+r_{n})(x-r_{1})\cdots(x-r_{n})=(x^{2}-r_{1}^{2})\cdots(x^{2}-r_{n}^{2})=q(x^{2})$ with $q(x)=(x-r_{1}^{2})\cdots(x-r_{n}^{2})$ the polynomial having as roots the squares of the roots of $p$. Thus, we see that $q(x)$ can easily be computed. Put $p=
\sum_{i=0}^{n}a_{i}x^{n-i}$ and $q=\sum_{i=0}^{n}b_{i}x^{n-i}$, where $a_{0}=b_{0}=1$ because we chose our polynomials monic. Then $b_{k}=(-1)^{k}a_{k}^{2}+2\sum_{j=0}^{k-1}(-1)^{j}a_{j}a_{2k-j}.$
\end{construccion}

The construction above gives therefore a computationally easy way of obtaining the polynomial $q$ having as roots the squares of the roots of $p$. The power of this trick lies in the splitting that this process of squaring promotes among the roots of the initial polynomial $p$. The next result is well-known.

\begin{proposicion}
If the absolute values of the roots of $p$ split at least by a multiplicative factor $\rho>1$, then the roots of $q$ split at least by a multiplicative factor $\rho^2>1.$
\end{proposicion}

\begin{proof}
Fix $k$. If there exists $\rho>1$ with $|r_{k+1}|\geq\rho|r_{k}|$, then $|r_{k+1}|^2\geq\rho^2|r_{k}|^2$. As a consequence, if such $\rho$ exists, the $q$ polynomial obtained applying the DLG method admits $\rho^2$ as a splitting factor.
\end{proof}

Splitting the roots enough allows us to apply Vi{\`e}te formulas to approximate the roots so \begin{gather*}
b_{1}=-(r_{n}^{2}+\cdots+r_{1}^{2})\\
b_{2}=r_{n}^{2}r_{n-1}^{2}+\cdots+r_{1}^{2}r_{2}^{2}\\\vdots
\\b_{n}=(-1)^{n}(r_{n}^{2}\cdots r_{1}^{2})
\end{gather*} and therefore we can approximate $b_{1}\approx-r_{n}^{2}$ and, proceeding sequentially, the rest of the squares of the roots. However, since we are only interested in this biggest root, we center our attention on the approximation $\sqrt{-b_{1}}\approx r_{n}$. We analyze the estimate that this direct application provides for the extreme roots of the univariate Eulerian polynomials. In particular, this observation will convince us in several ways why the refinement in the estimation performed by Sobolev is so important.

\section{Direct raw estimate}

Vi{\`e}te formulas directly give us that the extreme root $-q_{1}^{(n)}\approx E(n+1,1)=2^{n+1}-(n+2)$, which approximates well the first asymptotic growth but it is fairly bad beyond that. However, the application of the DLG method gives instead the approximation $(q_{1}^{n})^{2}\approx b_{1}^{k}=(a_{1}^{k})^{2}(1-2\frac{a_{0}^{k}a_{2}^{k}}{(a_{1}^{k})^{2}})=(a_{1}^{k})^{2}-2a_{0}^{k}a_{2}^{k}=(2^{k+1}-(k+2))^{2}-2(3^{k+1}-2^{k+1}(k+2)+\frac{(k+1)(k+2)}{2})=4^{k+1}-2^{k+2}(k+2)+(k+2)^{2}-3^{k+1}2+2^{k+1}(k+2)+\frac{(k+1)(k+2)}{2}=4^{k+1}-2^{k+1}(k+2)+(k+2)^{2}-3^{k+1}2+\frac{(k+1)(k+2)}{2}$. Now, the estimation of the roots comes from a square root $\sqrt{4^{k+1}-2^{k+1}(k+2)+(k+2)^{2}-3^{k+1}2+\frac{(k+1)(k+2)}{2}}$ and we have to proceed as usual in order to establish further asymptotic growth terms. We want to know the asymptotic growth of \begin{gather*}\sqrt{4^{k+1}-2^{k+1}(k+2)+(k+2)^{2}-3^{k+1}2+\frac{(k+1)(k+2)}{2}}-\\\left(2^{k+1}-\left(\frac{3}{2}\right)^{k+1}-\frac{1}{2}\left(\frac{9}{8}\right)^{k+1}\right).\end{gather*} The conjugate grows as $2^{k+2}$ and the difference of squares equals \begin{gather*}
    \frac{3 k^2}{2}-\left(\frac{2}{3}\right)^{-2 k-2}+4^{k+1}+2^{-2 k-2} 9^{k+1}-2^{k+2}-2^{2 k+2}-2^{-6 k-8} 9^{2 k+2}\\-2^{-4 k-4} 3^{3 k+3}-2^{k+1} k+\frac{11 k}{2}+5\sim -2^{k+1} k
\end{gather*} so the fourth growth term of this estimate equals $\frac{-2^{k+1} k}{2^{k+2}}=-\frac{k}{2}.$ Knowing this fourth growth term will be fundamental to see how the estimate breaks because it stops being more accurate than the method used by Sobolev exactly at this growth term. We collect what we know until here.

\begin{proposicion}\label{rawestimate}
The estimate $$\sqrt{4^{k+1}-2^{k+1}(k+2)+(k+2)^{2}-3^{k+1}2+\frac{(k+1)(k+2)}{2}}$$ for the absolute value of the leftmost root of the $n$-th Eulerian polynomial obtained from approximating using Vi{\`e}te's formula for the biggest root after the first iteration of the DLG method expands asymptotically, up to its fourth growth term, as $2^{k+1}-\left(\frac{3}{2}\right)^{k+1}-\frac{1}{2}\left(\frac{9}{8}\right)^{k+1}-\frac{k}{2}+o(k).$
\end{proposicion}

\section{An underestimated lemma}

There are two big problems with the estimate obtained above. We collect them here.

\begin{remark}
Just using Vi{\`e}te's formula we do not immediately know at which side of the actual root the estimate falls.
\end{remark}

This problem could indeed be analyzed and estimated looking carefully at signs and taking care of the growth of other roots, but we will not do this. Instead of that, dealing with the other problem will do this for us at the same time that we build a closer and deeper look at Sobolev's approach.

\begin{remark}
The raw estimate above do not give us an interval where our roots lie and instead provides just a point. With just this it is difficult to estimate confidently how close are we to the actual root.
\end{remark}

This second problem is intimately related to the first and the solution of the second will also solve the first. The solution comes from the \cite[Chapter 26, Lemma 1]{sobolev2006selected} used by Sobolev in the middle of its paper. We will analyze carefully this result here because \cite[Chapter 26, Lemma 1]{sobolev2006selected} is the key of his ability to provide bounds. This study will also help us to hint here that there is a possible productive connection to Cauchy-Schwartz's Lemma and its generalizations. We reproduce here the lemma and comment its proof in order to be able to detect how Sobolev's method differs from our approach.

\begin{lema}[Sobolev's Lemma] \cite[Chapter 26, Lemma 1]{sobolev2006selected}\label{sobolema}
Consider $\nu_{1}>\nu_{2}\geq\dots\geq\nu_{N}\geq0$ and define $a:=\nu_{1}+\cdots+\nu_{N}$ and $b:=\nu_{1}^{2}+\cdots+\nu_{N}^{2}$. If $a^{2}<b$, then $$\frac{1}{2}(a+\sqrt{2b-a^2})\leq\nu_{1}\leq\frac{1}{N}\left(a+(N-1)\sqrt{b-\frac{a^{2}-b}{N-1}}\right)<\sqrt{b}.$$
\end{lema}

\begin{proof}
Define $\zeta:=\nu_{2}+\cdots+\nu_{N},$ $\rho:=\nu_{2}^{2}+\cdots+\nu_{N}^{2}$ and write $\zeta^{2}=A^{2}\rho^{2}.$ Then, either counting or using the Cauchy-Schwartz inequality, we obtain now that $1\leq A^{2}\leq N-1,$ where the limits are reached for certain configurations, so this estimate is the tightest possible. Expressing the quantity $A^{2}$ through $a,b,\nu_{1}$ and studying its behaviour as $\nu_{1}$ varies shows that the estimate for $A^{2}$ is satisfied only when the estimate we want to prove is fulfilled, which proves this lemma.
\end{proof}

Looking closely at the lemma we can provide even better bounds than Sobolev initially provides as he simplifies a lot at some point. These new bounds will be our main object of discussion in what follows.

\begin{corolario}
Let $q_{1}^{(k)}$ be the leftmost root of the univariate Eulerian polynomial of degree $k$. Then we have the inequalities for its absolute value \begin{gather*}\frac{2^{k+1}-(k+2)+\sqrt{2^{3 + k} - 4 \cdot 3^{1 + k} + 4^{1 + k} - k \left(2 - 2^{2 + k} + k\right)}}{2}\leq|q_{1}^{(k)}|\\\leq\frac{2}{k}\bigg{(}\frac{2^{k+1}-(k+2)}{2}+\frac{k-1}{2}\\\sqrt{\frac{-4 \left(2^{k}-1\right)^2 + \left(4^{1 + k} -2 + 2^{2 + k} - 2 \cdot 3^{1 + k}\right) k}{k-1}
}\bigg{)}.\end{gather*}
\end{corolario}

\setlength{\emergencystretch}{3em}%
\begin{proof}
Observe that, for constructing his Equation 6, Sobolev uses his estimate $\sqrt{b}$ instead of the most accurate one that the lemma actually provides. Thus we can build a better bound putting that $N=\binom{k}{s}$. Thus the bounds of $\nu_{s,1}^{(k)}$ becomes $$\frac{1}{2}\left(a_{s}^{(k)}+\sqrt{2b_{s}^{(k)}-(a_{s}^{(k)})^2}\right)\leq\nu_{s,1}^{(k)}\leq\frac{1}{\binom{k}{s}}\left(a+(\binom{k}{s}-1)\sqrt{b-\frac{a^{2}-b}{\binom{k}{s}-1}}\right)$$ so now dividing the estimates for $\nu_{s,1}^{(k)}$ and $\nu_{s-1,1}^{(k)}$ we obtain \begin{gather*}\frac{a_{s}^{(k)}+\sqrt{2b_{s}^{(k)}-(a_{s}^{(k)})^2}}{\frac{2}{\binom{k}{s-1}}\left(a_{s-1}^{(k)}+(\binom{k}{s-1}-1)\sqrt{b_{s-1}^{(k)}-\frac{(a_{s-1}^{(k)})^{2}-b_{s-1}^{(k)}}{\binom{k}{s-1}-1}}\right)}\leq|q_{s}^{(k)}|\\\leq\frac{\frac{2}{\binom{k}{s}}\left(a_{s}^{(k)}+(\binom{k}{s}-1)\sqrt{b_{s}^{(k)}-\frac{(a_{s}^{(k)})^{2}-b_{s}^{(k)}}{\binom{k}{s}-1}}\right)}{a_{s-1}^{(k)}+\sqrt{2b_{s-1}^{(k)}-(a_{s-1}^{(k)})^2}},\end{gather*} which for our case $s=1$ simplifies, after an easy transformation avoiding division by $0$, as \begin{gather*}\frac{a_{1}^{(k)}+\sqrt{2b_{1}^{(k)}-(a_{1}^{(k)})^2}}{\frac{2}{\binom{k}{0}}\left(a_{0}^{(k)}+(\binom{k}{0}-1)\sqrt{b_{0}^{(k)}-\frac{(a_{0}^{(k)})^{2}-b_{0}^{(k)}}{\binom{k}{0}-1}}\right)}\leq|q_{1}^{(k)}|\\\leq\frac{\frac{2}{\binom{k}{1}}\left(a_{1}^{(k)}+(\binom{k}{1}-1)\sqrt{b_{1}^{(k)}-\frac{(a_{1}^{(k)})^{2}-b_{1}^{(k)}}{\binom{k}{1}-1}}\right)}{a_{0}^{(k)}+\sqrt{2b_{0}^{(k)}-(a_{0}^{(k)})^2}},\end{gather*} and in terms of Eulerian numbers just gives \begin{gather*}\frac{a_{1}^{(k)}+\sqrt{2((a_{1}^{(k)})^{2}-2a_{0}^{(k)}a_{2}^{(k)})-(a_{1}^{(k)})^2}}{2 a_{0}^{(k)}}\leq|q_{1}^{(k)}|\\\leq\frac{\frac{2}{k}\left(a_{1}^{(k)}+(k-1)\sqrt{((a_{1}^{(k)})^{2}-2a_{0}^{(k)}a_{2}^{(k)})-\frac{(a_{1}^{(k)})^{2}-((a_{1}^{(k)})^{2}-2a_{0}^{(k)}a_{2}^{(k)})}{k-1}}\right)}{a_{0}^{(k)}+\sqrt{2(a_{0}^{(k)})^{2}-(a_{0}^{(k)})^2}},\end{gather*} which simplifies finally to \begin{gather*}\frac{a_{1}^{(k)}+\sqrt{(a_{1}^{(k)})^{2}-4a_{0}^{(k)}a_{2}^{(k)}}}{2 a_{0}^{(k)}}\leq|q_{1}^{(k)}|\\\leq\frac{\frac{2}{k}\left(a_{1}^{(k)}+(k-1)\sqrt{((a_{1}^{(k)})^{2}-2a_{0}^{(k)}a_{2}^{(k)})-\frac{2a_{0}^{(k)}a_{2}^{(k)}}{k-1}}\right)}{2a_{0}^{(k)}},\end{gather*} and using that $a_{0}^{(k)}=E(k+1,0)=1, a_{1}^{(k)}=E(k+1,1)=2^{k+1}-(k+2)$ and $ a_{2}^{(k)}=E(k+1,2)=3^{k+1}-(k+2)2^{k+1}+\frac{(k+2)(k+1)}{2}$ we see that our bounds are effectively \begin{gather*}\frac{2^{k+1}-(k+2)+\sqrt{(2^{k+1}-(k+2))^{2}-4(3^{k+1}-(k+2)2^{k+1}+\frac{(k+2)(k+1)}{2})}}{2}\leq|q_{1}^{(k)}|\\\leq\frac{2}{k}\bigg{(}\frac{2^{k+1}-(k+2)}{2}+\frac{k-1}{2}\\\hspace*{-1.5cm}\sqrt{(2^{k+1}-(k+2))^{2}-2(3^{k+1}-(k+2)2^{k+1}+\frac{(k+2)(k+1)}{2})-\frac{2(3^{k+1}-(k+2)2^{k+1}+\frac{(k+2)(k+1)}{2})}{k-1}}\bigg{)},\end{gather*} which simplifies finally to the inequalities we wanted to prove and therefore establishes the proposition.
\end{proof}
\setlength{\emergencystretch}{0em}%

\begin{convencion}[Common names for bounding sequences]
In what follows in this section we will chapter, we will call $$\frac{2^{k+1}-(k+2)+\sqrt{2^{3 + k} - 4 \cdot 3^{1 + k} + 4^{1 + k} - k \left(2 - 2^{2 + k} + k\right)}}{2}$$ the \textit{minorant}, as it minorizes $|q_{1}^{(k)}|$. Therefore, similarly, we will call in what follow in this chapter, \begin{gather*}\frac{2}{k}\bigg{(}\frac{2^{k+1}-(k+2)}{2}+\frac{k-1}{2}\\\sqrt{\frac{-4 \left(2^{k}-1\right)^2 + \left(4^{1 + k} -2 + 2^{2 + k} - 2 \cdot 3^{1 + k}\right) k}{k-1}
}\bigg{)}\end{gather*} the \textit{majorant}, as it majorizes $|q_{1}^{(k)}|$. In particular, the bounds provided by the relaxation compete always with the minorant, as our fight is to provide larger lower bounds for $|q_{1}^{(k)}|$ so we get closer to it. However, the quality of these lower bounds in determining a good interval for the actual root $|q_{1}^{(k)}|$ is measured against upper bounds of $|q_{1}^{(k)}|$, in particular then against this majorant that we obtained here and other such upper bounds for $|q_{1}^{(k)}|$ that we can obtain in other ways.\end{convencion}

Now we can see the growth given by these improved bounds and compare with what was obtained before. We will see that we already know a lot, but that these refinements get us even a better picture of the behaviour of our roots of interest.

\section{Refined estimate}

With the refinements of the estimates we have obtained, we now want to look in both directions. Seeing what happens with the minorant and the majorant will be interesting. In particular, the minorant competes with the estimate given by the relaxation while the majorant will tell us a lot about the estimate above obtained applying directly just Vi{\`e}ta's formulas to the polynomial obtained after the first iteration of the DLG method.

\begin{remark}
Both the minorant and the mojorant give information about the relaxation and its connection to the DLG method as a way to establish bounds of polynomials.
\end{remark}

From our previous analyses, we already know a lot about the growth of the refined bounds obtained here. In particular, as our refinement tightened the bounds, we know that the new bounds will have growth terms closer to the actual root.

\begin{observacion}
We already know up to the third growth term of the minorant. However, the initial majorant provided by Sobolev does not give us enough information beyond the first growth term and therefore we will need to work more there.
\end{observacion}

As we will see that the work for the majorant connects with the previous section and the work for the minorant connects with the next chapter, we begin studying the majorant, which is the one giving us more work, as we have observed. In particular, we will see that this new majorant has much more accurate second and third growth terms and it does not split from the estimate in the section above until the fourth growth term.

\begin{remark}
We compute growth terms until after we can split not just the minorant from the majorant, but the majorant from the estimate above. We do this in order to show how successfully the refinement of the bounds via Sobolev's lemma tightens the approximation provided by the raw use of the DLG method.
\end{remark}

\subsection{The refined majorant}

Thus we first look at the majorant, which was really bad with Sobolev's simplification. We therefore want to compute the asymptotic growth of \begin{gather*}
    \frac{2}{k}\bigg{(}\frac{2^{k+1}-(k+2)}{2}+\frac{k-1}{2}\\\sqrt{\frac{-4 \left(2^{k}-1\right)^2 + \left(4^{1 + k} -2 + 2^{2 + k} - 2 \cdot 3^{1 + k}\right) k}{k-1}
}\bigg{)}-2^{k+1}=\\\frac{2^{k+1}-(k+2)-2^{k+1}k}{k}+\frac{k-1}{k}\\\sqrt{\frac{-4 \left(2^{k}-1\right)^2 + \left(4^{1 + k} -2 + 2^{2 + k} - 2 \cdot 3^{1 + k}\right) k}{k-1}
}.
\end{gather*} We have to proceed as usual using conjugates. The conjugate grows clearly like $-2^{k+2}$ and the difference of squares equals $$\frac{3 \left(3^{k}2 (k-1)+k+2\right)}{k}\sim 3^{k+1}2$$ so the second growth term of the majorant is $\frac{3^{k+1}2}{-2^{k+2}}=-\left(\frac{3}{2}\right)^{k+1}$, which equals now the actual second growth term of the root. Remember that the previous majorant (the one given explicitly by Sobolev) had second growth term $k$, so very far from the actual growth term of the root. This situation forces us to explore the third growth term of this improved majorant. We proceed therefore similarly and compute now the asymptotic growth of \begin{gather*}
    \frac{2}{k}\bigg{(}\frac{2^{k+1}-(k+2)}{2}+\frac{k-1}{2}\\\sqrt{\frac{-4 \left(2^{k}-1\right)^2 + \left(4^{1 + k} -2 + 2^{2 + k} - 2 \cdot 3^{1 + k}\right) k}{k-1}
}\bigg{)}-2^{k+1}+\left(\frac{3}{2}\right)^{k+1}=\\\frac{2^{k+1}-(k+2)-2^{k+1}k+\left(\frac{3}{2}\right)^{k+1}k}{k}+\frac{k-1}{k}\\\sqrt{\frac{-4 \left(2^{k}-1\right)^2 + \left(4^{1 + k} -2 + 2^{2 + k} - 2 \cdot 3^{1 + k}\right) k}{k-1}
}.
\end{gather*} As before, the conjugate grows clearly like $-2^{k+2}$ but now the difference of squares equals $$\left(\frac{9}{4}\right)^{k+1}-2^{-k} 3^{k+1}-\frac{6 \left(\left(\frac{3}{2}\right)^k-1\right)}{k}+3\sim\left(\frac{9}{4}\right)^{k+1}$$ so the third growth term of the majorant is $\left(\frac{9}{4}\right)^{k+1}\frac{1}{-2^{k+2}}=-\frac{1}{2}\left(\frac{9}{8}\right)^{k+1}$, which still coincides with the raw estimate. The fourth growth term will finally break this coincidence. We repeat our usual approach and compute now the asymptotic growth of \begin{gather*}
    \frac{2}{k}\bigg{(}\frac{2^{k+1}-(k+2)}{2}+\frac{k-1}{2}\\\sqrt{\frac{-4 \left(2^{k}-1\right)^2 + \left(4^{1 + k} -2 + 2^{2 + k} - 2 \cdot 3^{1 + k}\right) k}{k-1}
}\bigg{)}-2^{k+1}+\left(\frac{3}{2}\right)^{k+1}+\\\frac{1}{2}\left(\frac{9}{8}\right)^{k+1}=\frac{2^{k+1}-(k+2)-2^{k+1}k+\left(\frac{3}{2}\right)^{k+1}k+\frac{1}{2}\left(\frac{9}{8}\right)^{k+1}k}{k}+\frac{k-1}{k}\\\sqrt{\frac{-4 \left(2^{k}-1\right)^2 + \left(4^{1 + k} -2 + 2^{2 + k} - 2 \cdot 3^{1 + k}\right) k}{k-1}
}.
\end{gather*} As nothing changes for the first growth term of the conjugate, it grows again like $-2^{k+2}$ and now the difference of squares equals \begin{gather*}\frac{1}{k}\bigg{(}4^{-3 k-4}3\big{(}2^{3 k+5} 3^{2 k+1} (-k+2^{k+1}-2)+3^{3 k+2} 4^{k+2} k+3^{4 k+3} k\\-2^{5 k+8} 3^k (k+2)+4^{3 k+4} (k+2)\big{)}\bigg{)}\sim\frac{1}{k}\left(\frac{9}{4}\right)^{k+1}\end{gather*} so the fourth growth term of the majorant is $\frac{1}{k}\left(\frac{9}{4}\right)^{k+1}\frac{1}{-2^{k+2}}=-\frac{1}{2k}\left(\frac{9}{8}\right)^{k+1}$, which now finally drastically differs from the raw estimate's fourth growth term $-\frac{k}{2}$.

We thus just proved above the following result that collects our extractions of growth terms. This will be important in order to understand the picture showing the reason why Sobolev's lemma is so insightful. 

\begin{proposicion}
The asymptotic growth of the refined majorant first differs from the raw estimate at its fourth growth terms. In particular, the majorant four first growth terms are $2^{k+1}-\left(\frac{3}{2}\right)^{k+1}-\frac{1}{2}\left(\frac{9}{8}\right)^{k+1}-\frac{1}{2k}\left(\frac{9}{8}\right)^{k+1}+o(\frac{1}{k}\left(\frac{9}{8}\right)^{k})$.
\end{proposicion}

\begin{corolario}
The raw estimate ends up growing faster than the majorant and therefore surpasses the actual root above, i.e., for $k$ big enough, the absolute value of the raw estimate is bigger than the absolute value of the actual root.
\end{corolario}

\begin{proof}
The difference in growth between the raw estimate and the majorant happens at the fourth growth term. There the majorant computed via Sobolev's lemma substracts $\frac{1}{2k}\left(\frac{9}{8}\right)^{k+1}$ which is asymptotically bigger than the $\frac{k}{2}$ that it is substracted at that growth in the raw estimate, as can be seen in Proposition \ref{rawestimate}.
\end{proof}

This improved majorant still lies far from the original minorant written down by Sobolev after his simplification. We study now the new minorant in order to see if we get a significant improvement.

\subsection{The refined minorant}

As we know that the minorant obtained after Sobolev simplification already coincides in the first and second growth term of the actual root, we can proceed to compute now directly the third growth term of the improved minorant. Also, as this bound improves the one above, we can already suspect that the third growth term will be $-\left(\frac{9}{8}\right)^{k+1}$ and, for this reason, we will jump over it and compute directly the fourth growth term. And we will see that this works. For this, we need to compute the growth of \begin{gather*}
    \frac{2^{k+1}-(k+2)+\sqrt{2^{3 + k} - 4 \cdot 3^{1 + k} + 4^{1 + k} - k \left(2 - 2^{2 + k} + k\right)}}{2}\\-2^{k+1}+\left(\frac{3}{2}\right)^{k+1}+\left(\frac{9}{8}\right)^{k+1}=\\\frac{2^{k+1}-(k+2)+2(-2^{k+1}+\left(\frac{3}{2}\right)^{k+1}+\left(\frac{9}{8}\right)^{k+1})}{2}+\\\frac{\sqrt{2^{3 + k} - 4 \cdot 3^{1 + k} + 4^{1 + k} - k \left(2 - 2^{2 + k} + k\right)}}{2}.
\end{gather*} The conjugate grows like $-2^{k+1}$ while the difference of squares is now \begin{gather*}
64^{-k-1} \bigg{(}-72^{k+1} (k+2)-96^{k+1} (k+2)+2^{6 k+5} (k+1) (k+2)+\\2^{2 k+3} 27^{k+1}+81^{k+1}\bigg{)}\sim 64^{-k-1} 2^{2 k+3} 27^{k+1} = 2\left(\frac{27}{16}\right)^{k+1}
\end{gather*} so the fourth growth term is $-2\left(\frac{27}{32}\right)^{k+1}.$ Observe how the growth of the growth terms begins to decay at the fourth position: $-2\left(\frac{27}{32}\right)^{k+1}\to0$ when $k\to\infty.$ We thus proved the following result.

\begin{proposicion}
The asymptotic growth of the refined minorant up until its fourth growth term expands asymptotically as $$2^{k+1}-\left(\frac{3}{2}\right)^{k+1}-\left(\frac{9}{8}\right)^{k+1}-2\left(\frac{27}{32}\right)^{k+1}+o(\left(\frac{27}{32}\right)^{k}).$$
\end{proposicion}

\section{Conceptual consequences of the refinements}

We have seen above that the Lemma used for making these refinements is crucial for Sobolev in order to surpass us just computing the first iteration of the DLG method. In particular, the lemma allows us to deal with an interval in which a root lies instead of just providing a point estimation. However, we will also see that, if it is applied at this level in the iteration of the Sobolev method, then the relaxation also wins here.

\begin{remark}
At the first iteration of the DLG method, the relaxation wins against just using Sobolev's lemma in order to establish the minorant.
\end{remark}

Going beyond the first iteration of the DLG method would in general provide better bounds, but, for this to be useful for us, we would need to generalize the way we obtain interval estimates instead of point estimates. This means that we would need a generalization of Sobolev's \cite[Chapter 26, Lemma 1]{sobolev2006selected}. This is what we hinted at when we decided to use Cauchy-Schwartz when proving this lemma.

\begin{remark}[Possibility of extensions of Sobolev's Lemma \ref{sobolema}]
Using the Cauchy-Schwartz inequality in the proof of \cite[Chapter 26, Lemma 1]{sobolev2006selected} presented above here might be seen as overkill, but it hints directly at the possibility of extending that lemma using the many known generalizations of Cauchy-Schwartz and resultants instead of the DLG method. These generalizations could be used to extract interval estimates using higher (and therefore more accurate) iterations of the DLG method. These interval estimates should get tighter the more iterations of the DLG method we perform.
\end{remark}

However, if we consider just the minorant and we stay at the same level of the iteration of the DLG process, then the relaxation provides a better minorant than Sobolev's lemma. This is the whole topic of the next chapter. Thus, we see how the discussion of the majorant linked with the unsuccessfulness of the raw estimation studied in the section above while the discussion of the minorant will show us how the relaxation is superior to Sobolev's \cite[Chapter 26, Lemma 1]{sobolev2006selected} by the side of the estimation at which they both compete.
\chapter{Combining DLG method and the relaxation}\label{ChDLGCombi}

Each time we apply the DLG method we have to deal with a different set of roots (the squares of the initial roots) and, equivalently, with a different polynomial. We can therefore see this as a process consisting in applying the DLG method and then trying to estimate the roots of the original polynomial at each stage. This process will generate therefore a sequence of polynomials for each polynomial in our sequence. This double sequence is interesting by itself and we will deal with it in the next chapter. Here we will see that, at least at the first iteration, the minorant offered by the relaxation at such iteration beats the best minorant (the refined one discussed just above) obtained via Sobolev's lemma. This pushes us to establish the next problem that lies beyond the scope of our very humble current pursuits.

\begin{problem}
Show that, at every iteration of the DLG method, the application of the bivariate relaxation beats the univariate relaxation at establishing the best minorant. And the use of the relaxation beats the use of the corresponding generalization of Sobolev's lemma.
\end{problem}

Dealing correctly with this problem first requires determining the generalization of the Sobolev \cite[Chapter 26, Lemma 1]{sobolev2006selected} we are talking about. But, as we said before, this would lead us beyond the scope of our current interests. For this reason, we focus on the stage at which Sobolev provides his lemma, that is, at the first iteration of the DLG method, as he has to use what he calls the Lobachevsky equation (i.e., the DLG method) to establish that lemma.

\begin{observacion}
Although our discussion here has the potential to reach further, we decide to stay within the limits of the first iteration of the DLG method.
\end{observacion}

This implies that, with respect to the property of producing minorants to the absolute value of the leftmost root of the Eulerian polynomials, the relaxation produces polynomials that lie in the middle of stages between different iterations of the DLG method. And that adding more variables in a meaningful way seems to establish a whole gradient between these stages, although we can only easily devise the way of adding just one variable more in order to make our polynomials bivariate instead of just univariate. Unfortunately, multivariate expansions of the method do not seem easy to obtain for us, although we do not consider this impossible. This is just too time-consuming to pursue for our initially humble objectives here of showing how the relaxation actually improves the best bounds provided so far in the literature and how going up to the setting where there are several variables seems to help in this task.

\begin{remark}
The use of the relaxation improves the minorant obtained through Sobolev's \cite[Chapter 26, Lemma 1]{sobolev2006selected}. Adding another variable via the use of the interlacing properties within the sequences of polynomial, improves the bound obtained even further.
\end{remark}

Thus, we will now see how the use of the univariate relaxation in the first iteration of the DLG process already provides the best bound. We will continue showing that this bound improves if we go to two variables using interlacing. Finally, we hint at the possibility of going fully multivariate and comment possible ways, obstructions, obstacles and difficulties when trying to follow these paths. In particular, we remind the reader that the DLG method has stability problems, although some solutions to that issue have been proposed. See, e.g., \cite{malajovich2001geometry}.

\section{Univariate polynomials of the first iteration}

We compute now the univariate relaxation for the polynomials obtained in the first iteration of the DLG method. This iteration produces the polynomials whose coefficients are the $b$ numbers in Sobolev's notations in \cite{sobolev2006selected}. In order to apply the relaxation, we collect these coefficients up to degree $3$. First, in order to produce these coefficients we will need now more coefficients of the Eulerian polynomials. In particular, we will compute the coefficients of the Eulerian polynomials up to degree $6.$

\begin{computacion}
For what follows and due to the action of the DLG method on polynomials, besides what we have already computed before, we will need also the third, fourth, fifth and sixth coefficients of the $k$-th univariate Eulerian polynomial \begin{gather*}
    a_{k-3}^{(k)}=-3^{n+1} (n+2)+2^n (n+1) (n+2)-\frac{1}{6} n (n+1) (n+2)+4^{n+1},
\\a_{k-4}^{(k)}=-4^{n+1} (n+2)+\frac{1}{2} 3^{n+1} (n+1) (n+2)\\-\frac{1}{3} 2^n n (n+1) (n+2)+\frac{1}{24} (n-1) n (n+1) (n+2)+5^{n+1},\\a_{k-5}^{(k)}=-5^{n+1} (n+2)+2^{2 n+1} (n+1) (n+2)-\frac{1}{2} 3^n n (n+1) (n+2)+\\\frac{1}{3} 2^{n-2} (n-1) n (n+1) (n+2)-\frac{1}{120} (n-2) (n-1) n (n+1) (n+2)+6^{n+1},\\a_{k-6}^{(k)}=-6^{n+1} (n+2)+\frac{1}{2} 5^{n+1} (n+1) (n+2)\\-\frac{1}{3} 2^{2 n+1} n (n+1) (n+2)\\+\frac{1}{8} 3^n (n-1) n (n+1) (n+2)\\-\frac{1}{15} 2^{n-2} (n-2) (n-1) n (n+1) (n+2)\\+\frac{1}{720} (n-3) (n-2) (n-1) n (n+1) (n+2)+7^{n+1}.\end{gather*}
\end{computacion}

Now we have all the ingredients we need to compute the coefficients of degree up to three of the first iteration of the DLG method. A nice observation to do here is the fact that, in the computation of these coefficients of degree up to three of the new polynomials, coefficients of degree up to six of the original univariate Eulerian polynomial play a role, which shows us that this iteration does indeed add further information into the cocktail of the relaxation. It is therefore not surprising that the bound provided will improve now.

\begin{computacion}
The coefficients of degree up to $3$ of the polynomial obtained after applying the DLG method to the $k$-th univariate Eulerian polynomial are \begin{gather*}b_{k}^{(k)}=(a_{k}^{(k)})^{2}=1,\\b_{k-1}^{(k)}=-((a_{k-1}^{(k)})^{2}-2a_{k-2}^{(k)}a_{k}^{(k)})=\\-((2^{k+1}-(k+2))^{2}-2(3^{k+1}-2^{k+1}(k+2)+(k+1)\frac{k+2}{2})),\\b_{k-2}^{(k)}=(a_{k-2}^{(k)})^{2}-2a_{k-3}^{(k)}a_{k-1}^{(k)}+2a_{k-4}^{(k)}a_{k}^{(k)}=\\(3^{k+1}-2^{k+1}(k+2)+(k+1)\frac{k+2}{2})^{2}-2(2^{k+1}-(k+2))\\(-3^{k+1} (k+2)+2^k (k+1) (k+2)-\frac{1}{6} k (k+1) (k+2)+4^{k+1})\\+2(-4^{k+1} (k+2)+\frac{1}{2} 3^{k+1} (k+1) (k+2)\\-\frac{1}{3} 2^k k (k+1) (k+2)+\frac{1}{24} (k-1) k (k+1) (k+2)+5^{k+1}),\\b_{k-3}^{(k)}=-((a_{k-3}^{(k)})^{2}-2a_{k-4}^{(k)}a_{k-2}^{(k)}+2a_{k-5}^{(k)}a_{k-1}^{(k)}-2a_{k-6}^{(k)}a_{k}^{(k)})=\\-((-3^{k+1} (k+2)+2^k (k+1) (k+2)-\frac{1}{6} k (k+1) (k+2)+4^{k+1})^{2}\\-2(-4^{k+1} (k+2)+\frac{1}{2} 3^{k+1} (k+1) (k+2)-\frac{1}{3} 2^k k (k+1) (k+2)\\+\frac{1}{24} (k-1) k (k+1) (k+2)+5^{k+1})(3^{k+1}-2^{k+1}(k+2)+(k+1)\frac{k+2}{2})\\+2(-5^{k+1} (k+2)+2^{2 k+1} (k+1) (k+2)-\frac{1}{2} 3^k k (k+1) (k+2)+\\\frac{1}{3} 2^{k-2} (k-1) k (k+1) (k+2)\\-\frac{1}{120} (k-2) (k-1) k (k+1) (k+2)+6^{k+1})(2^{k+1}-(k+2))\\-2(-6^{k+1} (k+2)+\frac{1}{2} 5^{k+1} (k+1) (k+2)\\-\frac{1}{3} 2^{2 k+1} k (k+1) (k+2)\\+\frac{1}{8} 3^k (k-1) k (k+1) (k+2)\\-\frac{1}{15} 2^{k-2} (k-2) (k-1) k (k+1) (k+2)\\+\frac{1}{720} (k-3) (k-2) (k-1) k (k+1) (k+2)+7^{k+1})).\end{gather*}
\end{computacion}

Now we can compute the univariate $L$-form of this new family of polynomials. We proceed as usual by degree.

\setlength{\emergencystretch}{3em}%
\begin{computacion}
Let, in this computation, $p$ be the corresponding univariate polynomial of degree $k$ obtained after applying the DLG method to the $k$-th univariate Eulerian polynomial. We go in order. Degree $0$ is easy as usual: $L_{p}(1)=k$. For degree $1$, we have $L_{p}(x)=\coeff(x,p)=b_{k-1}^{(k)}.$ We continue with degree $2$ and obtain $L_{p}(x^{2})=-2\coeff(x^{2},p)+\coeff(x,p)^{2}=-2b_{k-2}^{(k)}+(b_{k-1}^{(k)})^{2}$. Finally, for degree $3$, we have $L_{p}(x^{3})=3\coeff(x^{3},p)-3\coeff(x,p)\coeff(x^{2},p)+\coeff(x,p)^{3}=3b_{k-3}^{(k)}-3b_{k-1}^{(k)}b_{k-2}^{(k)}+(b_{k-1}^{(k)})^{3}.$
\end{computacion}

This information allows us to compute the corresponding univariate relaxation, which, in this case, is given by a LMP small enough to allow us to compute also its determinant. We proceed directly for this.
\setlength{\emergencystretch}{0em}%

\begin{computacion}
Writing the relaxation $$\begin{pmatrix}L_{p}(1) & L_{p}(x)\\L_{p}(x) & L_{p}(x^2)\end{pmatrix}+x\begin{pmatrix}L_{p}(x) & L_{p}(x^2)\\L_{p}(x^2) & L_{p}(x^3)\end{pmatrix},$$ we immediately see that the determinant of the relaxation is a quadratic polynomial \begin{gather*}\hspace{-1cm}ax^{2}+bx+c=(L_{p}(1) + xL_{p}(x))(L_{p}(x^2) + xL_{p}(x^3))-(L_{p}(x) + xL_{p}(x^2))(L_{p}(x) + xL_{p}(x^3))\end{gather*} so \begin{gather*}a=L_{p}(x)L_{p}(x^3)-L_{p}(x^2)^2,\\ b=L_{p}(x)L_{p}(x^2)+L_{p}(1)L_{p}(x^3)-L_{p}(x)L_{p}(x^3)-L_{p}(x)L_{p}(x^2)=\\L_{p}(1)L_{p}(x^3)-L_{p}(x)L_{p}(x^3),\mbox{\ and }\\c=L_{p}(1)L_{p}(x^2)-L_{p}(x)^2.\end{gather*} So the roots of this determinant $\frac{-b\pm\sqrt{b^2-4ac}}{2a}$ are both bigger than the smallest root of the polynomial $p$ because of the properties of the relaxation and therefore the one that better approximates the innermost root is the smallest one, which is $\frac{-b-\sqrt{b^2-4ac}}{2a}$ because $a>0$ for $k$ big enough.
\end{computacion}

Now that we have our approximation obtained through the combination of the relaxation and the DLG method, we are in the position to study how it improves the previous bounds and therefore how it wins. This is the topic of the next section.

\section{Winning univariate}

The computations above give us $\frac{2a}{-b-\sqrt{b^2-4ac}}$ as the bound for the extreme root of the new polynomial. This extreme root is the square of the absolute value of the extreme root of the Eulerian polynomial. Thus we obtain the lower bound $\sqrt{\frac{2a}{-b-\sqrt{b^2-4ac}}}\leq|q_{1}^{(k)}|$. In order to see that this minorant is better than the previous one, we will try to compute its fourth growth term using what we know already. Therefore we have to compute the growth of $$\sqrt{\frac{2a}{-b-\sqrt{b^2-4ac}}}-2^{k+1}+\left(\frac{3}{2}\right)^{k+1}+\left(\frac{9}{8}\right)^{k+1}$$ and we expect to obtain something \textit{not subtracting as much as} what does $-2\left(\frac{27}{32}\right)^{k+1}$ from the growth. The computations will be lengthy again.

\begin{computacion}
First, we rewrite the quantity whose growth we want to establish as $$\frac{\sqrt{2a}-(2^{k+1}-\left(\frac{3}{2}\right)^{k+1}-\left(\frac{9}{8}\right)^{k+1})\sqrt{-b-\sqrt{b^2-4ac}}}{\sqrt{-b-\sqrt{b^2-4ac}}}.$$ Establishing the growth of the denominator will, as usual, be easy. For the numerator, we will require more work. We will have to conjugate twice. First by \begin{gather*}C_{1}:=\sqrt{2a}+(2^{k+1}-\left(\frac{3}{2}\right)^{k+1}-\left(\frac{9}{8}\right)^{k+1})\sqrt{-b-\sqrt{b^2-4ac}}\end{gather*} to obtain \begin{gather*}N_{1}:=2a-\left(2^{k+1}-\left(\frac{3}{2}\right)^{k+1}-\left(\frac{9}{8}\right)^{k+1}\right)^{2}(-b-\sqrt{b^2-4ac})=\\2a+b\left(2^{k+1}-\left(\frac{3}{2}\right)^{k+1}-\left(\frac{9}{8}\right)^{k+1}\right)^{2}+\\\left(2^{k+1}-\left(\frac{3}{2}\right)^{k+1}-\left(\frac{9}{8}\right)^{k+1}\right)^{2}\sqrt{b^2-4ac}.\end{gather*} Now we multiply again by the conjugate of this expression \begin{gather*}C_{2}:=2a+b\left(2^{k+1}-\left(\frac{3}{2}\right)^{k+1}-\left(\frac{9}{8}\right)^{k+1}\right)^{2}-\\\left(2^{k+1}-\left(\frac{3}{2}\right)^{k+1}-\left(\frac{9}{8}\right)^{k+1}\right)^{2}\sqrt{b^2-4ac}\end{gather*} and obtain \begin{gather*}N_{2}:=\left(2a+b\left(2^{k+1}-\left(\frac{3}{2}\right)^{k+1}-\left(\frac{9}{8}\right)^{k+1}\right)^{2}\right)^2-\\\left(2^{k+1}-\left(\frac{3}{2}\right)^{k+1}-\left(\frac{9}{8}\right)^{k+1}\right)^{4}(b^2-4ac).\end{gather*} Now we have to estimate correctly all these conjugates and differences of squares. We have $N_{2}\sim -2^{11 k+14} 9^{k+1} k.$ Now, to compute the growth of the conjugate $C_{2}$ we compute the growth of both summands. For that, we need the growth of $a\sim\frac{1}{24} 2^{4 k+7} 3^{2 k+3}, b\sim-2^{6 k+6} k, c\sim2^{4 k+4} k$ and also $\sqrt{b^2-4ac}\sim2^{6 k+6} k.$ This immediately gives $C_{2}\sim -2^{8 k+9} k$ and therefore $N_{1}=\frac{N_{2}}{C_{2}}\sim\frac{-2^{11 k+14} 9^{k+1} k}{-2^{8 k+9} k}=2^{3 k+5} 9^{k+1}.$ Now we need to compute $-b-\sqrt{b^2-4ac}=\frac{b^2-(b^2-4ac)}{-b+\sqrt{b^2-4ac}}=\frac{4ac}{-b+\sqrt{b^2-4ac}}\sim\frac{4^{4 k+5} 9^{k+1} k}{2^{7 + 6 k} k}=2^{2 k+3} 9^{k+1}.$
Now we can proceed similarly and compute $C_{1}\sim 2^{2 k+3} 3^{k+1}\sqrt{2}$ so we have that $N=\frac{N_{1}}{C_{1}}\sim\frac{2^{3 k+5} 9^{k+1}}{2^{2 k+3} 3^{k+1}\sqrt{2}}=2^{k+1} 3^{k+1}\sqrt{2}.$ To finish we have to compute the growth of the original denominator $\sqrt{-b-\sqrt{b^2-4ac}}\sim\sqrt{2^{2 k+3} 9^{k+1}}=2^{k+1} 3^{k+1}\sqrt{2}$. Thus the growth we are looking for is $\frac{2^{k+1} 3^{k+1}\sqrt{2}}{2^{k+1} 3^{k+1}\sqrt{2}}=1,$ which is a notable improvement with respect to $-2\left(\frac{27}{32}\right)^{k+1}.$ \end{computacion}

We fix the strong improvement shown by these computations in the following important result.

\begin{teorema}
Calling $\uni_{2}:=\sqrt{\frac{2a}{-b-\sqrt{b^2-4ac}}}\leq|q_{1}^{(k)}|$ the bound obtained through the application of the relaxation to the univariate polynomial obtained after applying the DLG method, we have the new bound $2^{k+1}-\left(\frac{3}{2}\right)^{k+1}-\left(\frac{9}{8}\right)^{k+1}+1+o(1)=\uni_{2}(k)\leq|q_{1}^{(k)}|.$
\end{teorema}

Can we even improve this? Without our analysis of the bivariate relaxation, it would be difficult to predict if this could be possible. However, the bivariate relaxation proves itself better. This is the topic of the next two sections.

\setlength{\emergencystretch}{3em}%
\section[Bivariate polynomials of the first iteration]{Bivariate polynomials of the first iteration and winning bigger}
\setlength{\emergencystretch}{0em}%

As we did with the sequence of Eulerian polynomial, we can collect recursive information into the relaxation going to two variables and using the conservation of the property of interlacing between the variables. Now we consider the whole sequence of polynomials obtained in the first iteration of the DLG method. The interlacing of the sequence of Eulerian polynomials implies that this sequence is also interlacing. We collect this easy fact here. The proof is immediate. First, we need also an immediate definition in order to be able to talk properly about these objects.

\begin{definicion}[Root transformations] Let $p=c\prod_{i=1}^{n}(x-r_{i})\in\mathbb{C}[x]$ be a polynomial with roots in $A\subseteq\mathbb{C}$ and consider a map $f\colon A\to\mathbb{C}.$ We denote $p_{f}$ the polynomial $p=c\prod_{i=1}^{n}(x-f(r_{i}))$ obtained from $p$ after performing the transformation $f$ in the roots of $p$. We call $p_{f}$ the \textit{$f$-root-transformation of $p$}.
\end{definicion}

Clearly, the DLG method is a particular case of the definition above. For Eulerian polynomials, as all the roots are strictly negative, we have the next result.

\begin{proposicion}
Let $p,q$ be polynomials with roots in a set $A\subseteq\mathbb{R}$ and consider a (strictly) monotonous function $f\colon A\to\mathbb{R}$. If $p,q$ are (strictly) interlacing so are the polynomials $p_{f},q_{f}$.
\end{proposicion}

Now, as the sequence obtained after the first iteration also interlaces, we can apply the strategy, that we saw already in chapters above, of using the possibility of adjoining recurrent information from the sequence adding one extra variable accompanying the previous polynomial in the sequence. In this way, we can also form new bivariate polynomials.

\begin{remark}[The bivariate polynomials of the first iteration of the DLG method]
We form again the polynomials $C_{n}(x,y):=p_{n}+y p_{n-1}.$
\end{remark}

We have to compute again the corresponding $L$-forms. We will not conduct the whole process again here, as it is very similar to what we did in the case of Eulerian polynomials ($0$-th iteration of the DLG method). We only recall the procedure that should be repeated here.

\begin{procedimiento}[Dealing with the relaxation as already usual]
After computing the corresponding $L$-forms of the new bivariate polynomial introduced in the remark above, we can form the corresponding LMP defining its associated relaxation. The determinant of this LMP is a cubic polynomial. Therefore, in order to check that the relaxation of this bivariate version does indeed improve the bounds obtained, we use the again trick of trying to approximate a solution to the related generalized eigenvalue problem through the process of guessing an approximate eigenvector, as it is easier than dealing with the cubic polynomial directly.
\end{procedimiento}

Experiments through the use of this trick lead us to consider the vector $(y, 3/4, 8)$ to linearize our problem. We maximize over $y$.

\begin{observacion}[Just a look through this window beyond univariateness after applying the DLG method]
Now, contrary to what happened in the section above, we predict the difference to be positive but so small that the best strategy is comparing the univariate and the bivariate relaxation directly without seeking for growth term. This means that we will compute directly a difference of estimates here instead of the \textit{immediate next} growth term because here we are interested just in showing again that going bivariate improves, once more, the bounds in our way to discuss the necessity of finding strategies to go fully multivariate in the next section following the philosophical heuristic that seems to point towards the direction of the maxima saying that \textit{if going to two variables helps, going to several variables should help even more}. Doing this, we can easily see that we obtain again a bound that is better than the univariate one obtained in the section above.
\end{observacion}

The previous observation clearly shows as that the discussion in this section opens the door to consider multivariate generalization of this process, as we saw in the case of Eulerian polynomials. Discussing and presenting ideas in that direction is the topic of the next speculative section that builds upon the fact that nice multivariate generalizations are possible for Eulerian polynomials. We believe this to be possible for further iterations of the DLG method at least.

\section{Ideas for multivariability}

\setlength{\emergencystretch}{3em}%
As it was indeed possible to generate multivariate Eulerian polynomials from the univariate Eulerian polynomials through the strategy of counting finer within the combinatorial objects they deal with or, equivalently, using the theory of real stability preservers together with the recursive expression of these polynomials, we firmly believe that an application of these ideas in a refined manner would allow us to construct multivariate refinements for all sequences of polynomials generated via the application of (several iterations of) the DLG method. However, this is not immediate and we do not know how to do it. This is also not the main topic of this dissertation. For this reason, we only discuss these ideas here because they emerge naturally from our previous considerations in a fairly direct way: applying the relaxation to the multivariate refinement improved the bivariate bounds in the case of Eulerian polynomials.
\setlength{\emergencystretch}{0em}%

\begin{remark}
In the case of the $0$-th iteration of the DLG method, we have seen that multivariate refinements improve the relaxation. This shows that adding more variables in a structured way is beneficial for root-bounding.
\end{remark}

However, we know that going multivariate is not as straightforward as going bivariate is. Going bivariate is an easy consequence of the conservation of RZ-ness for a particular transformation adding an extra variable attached to an interlacer. For going multivariate, this trick stops working and we have to use more sophisticated tools like modifications in the recursive definitions combined with the Borcea-Br{\"a}nd{\'e}n theory of stability preservers (\cite{borcea2009lee1,borcea2009lee2,borcea2009polya}) or looking at the underlying combinatorial objects whose information these polynomials encode. These two approaches suppose an amount of work that would lead us far beyond the limits of what we pretend in this thesis, but we mention them here because they point towards a path that our considerations in this first approach to this topic clearly open.
\chapter[Structured interlacers improvements]{Non-trivial structured interlacers improve the relaxation}\label{ChStructured}

Here we look at directions in which the relaxation can be improved. This chapter combines and explores further the ideas developed up to this point as a way of closing this main discussion about combining the relaxation with other known methods for root-bounding. The references required through the explorations of this chapter are the same we used in the previous chapter plus those of Section \ref{sectioninterlacing}.

\section{Improvements require new information}

Until this point, we have developed tools, intuitions and result that let us learn to detect what transformations improve the relaxation. In particular, we first tried to improve the relaxation naively in two generic ways that we could hope that would work for all RZ polynomials. We tried an extension introducing further monomials in the matrix, but that worked in a limited way without having previously already a determinantal representation. In particular, just looking at the polynomial, this only allowed us to double the size of the initial matrix preserving its PSD-ness, but it destroyed the property of being a relaxation. Therefore, this method stalled very fast as it breaks the relaxation and it seems that much more sophisticated methods would be needed in order to extend the amount of monomials at play due to the complexity of the non-commutativity of product of matrices and how this complicates the behaviour of traces. After seeing that this method could not go beyond in an easy way, we tried a different method to extend the amount of variables through the use of interlacers. However, the simpler interlacer we can think of (the Renegar derivative) does not produce any change in the relaxation. This means that a generic method to improve the relaxation for all polynomials needs to make use of more sophisticated information about these polynomials. Extracting such information in general seems difficult and we do not know how to do it in general. However, when polynomials count combinatorial objects and have a recursive nature, this information can be extracted from the recursion via interlacing and that is what our treatment of Eulerian polynomials showed us.

\begin{remark}[Interlacing and recursion]
In the case of sequences of polynomials defined by recursion, knowledge about their interlacing properties allows us to collect information about the recursion in the relaxation through the use of variable extensions keeping RZ-ness.
\end{remark}

Other forms of collecting information through recursion involved knowing the form of the recursion in order to manipulate it so it allows for the lamination of the information carried by one variable into several layers of information carried by several new variables. This process of laminating the information is interesting and important for our point of view here. 

\begin{definicion}[Lamination]
Let $\{p_{i}\}_{i=0}^{\infty}\subseteq\mathbb{C}[x]$ be a sequence of univariate polynomials. A \textit{lamination} of this sequence is another sequence $\{q_{i}\}_{i=0}^{\infty}\subseteq\mathbb{C}[x_{1},\dots]$ of polynomials such that the number of variables in $q_{i}$ increases with $i$ and they equal the corresponding polynomial of the original sequence along the diagonal, that is, $q_{i}(\mathbf{x})|_{\mathbf{x}=(x,\dots,x)}=p_{i}(x).$ We say that we \textit{laminate} a sequence of polynomials when we build a lamination of it.
\end{definicion}

This is one way of seeing the construction of the multivariate Eulerian polynomials directly from their defining recursion without looking at the underlying combinatorial object. The laminations we are mainly interested in require that the process happens in a particular way respecting RZ-ness, hyperbolicity or stability. As we will see, this can be done only in certain cases and choosing the modifications very carefully so the transformations applied are stability preservers.

\begin{remark}[Building nice multivariate polynomials laminating the recursive or combinatorial information]
We can see in the proof of \cite[Theorem 3.3]{visontai2013stable} how the multivariate Eulerian polynomials are built in a way that surely preserves stability through the operator inducing the recursion. This operator is a stability preserver that introduces two new variables at a time (although we, at the end, only conserve one of them). Thus we can see that, with an in-depth study of the operators performing the recursions, we can laminate the univariate information into a multivariate one increasing the number of variables sequentially and, therefore, building polynomials that can increase the size of the relaxation. This is positive for us.
\end{remark}

Observe that the construction of a lamination can in fact stem directly from the univariate recursion without any reference to the underlying combinatorial object as follows. We see this using our helpful example of Eulerian polynomials.

\begin{construccion}[Lamination]
As we are looking at \cite{visontai2013stable}, remember that $A_{n}$ has degree $n$ as we also convene in this document. This will be important. Above \cite[Theorem 3.2]{visontai2013stable} we can see the recursion for univariate polynomials that we rewrite as $$A_{n}=(n+1)xA_{n-1}(x)+(1-x)(xA_{n-1})'.$$ Instead of using these polynomials, for conformity with the expression in \cite[Theorem 3.3]{visontai2013stable}, we will look at the sequence of homogenizations $\{A_{i}^{h}\}_{i=0}^{\infty}.$ This sequence verifies therefore the \textit{homogenized} recursion  \begin{gather*}A^{h}_{n}=(n+1)xA^{h}_{n-1}+(x_{0}-x)\frac{\partial}{\partial x}(xA^{h}_{n-1})=\\(n+1)xA^{h}_{n-1}+x_{0}A^{h}_{n-1}+x_{0}x\frac{\partial}{\partial x}A^{h}_{n-1}-xA^{h}_{n-1}-x^{2}\frac{\partial}{\partial x}A^{h}_{n-1}=\\nxA^{h}_{n-1}+x_{0}A^{h}_{n-1}+x_{0}x\frac{\partial}{\partial x}A^{h}_{n-1}-x^{2}\frac{\partial}{\partial x}A^{h}_{n-1}=\\(nx+x_{0})A^{h}_{n-1}+(x_{0}x-x^{2})\frac{\partial}{\partial x}A^{h}_{n-1}.\end{gather*} We want to rewrite this recursion in the friendlier way \begin{gather}\label{friendlier} A^{h}_{n}=(x+x_{0})A^{h}_{n-1}+xx_{0}\left(\frac{\partial}{\partial x}+\frac{\partial}{\partial x_{0}}\right)A^{h}_{n-1}.\end{gather} We can do this checking that the differing terms verify \begin{gather*}
    (n-1)xA_{n-1}^{h}-x^{2}\frac{\partial}{\partial x}A^{h}_{n-1}=xx_{0}\frac{\partial}{\partial x_{0}}A^{h}_{n-1},
\end{gather*} which is equivalent to \begin{gather*}
    (n-1)A_{n-1}^{h}-x\frac{\partial}{\partial x}A^{h}_{n-1}=x_{0}\frac{\partial}{\partial x_{0}}A^{h}_{n-1},
\end{gather*} which in turn is equivalent to \begin{gather*}
    (n-1)A_{n-1}^{h}=(x\frac{\partial}{\partial x}+x_{0}\frac{\partial}{\partial x_{0}})A^{h}_{n-1},
\end{gather*} which is clear noticing that $A^{h}_{n-1}$ has degree $n-1$ and observing what the operator $x\frac{\partial}{\partial x}+x_{0}\frac{\partial}{\partial x_{0}}$ actually does because \begin{gather*}
    \left(x\frac{\partial}{\partial x}+x_{0}\frac{\partial}{\partial x_{0}}\right)x^{n}x_{0}^{m}=nx^{n}x_{0}^{m}+mx^{n}x_{0}^{m}=(n+m)x^{n}x_{0}^{m}.
\end{gather*} Now we can use the expression in Equation \ref{friendlier} to laminate the recursion as follows. We split the information stored in the variables at each step by forcing new variables in each iteration $A^{h}_{n}=$  \begin{gather*}(x_{n}+x_{0})A^{h}_{n-1}+x_{n}x_{0}\left(\frac{\partial}{\partial x}+\frac{\partial}{\partial x_{0}}\right)A^{h}_{n-1}.\end{gather*}
\end{construccion}

This is how we can pass from a univariate recursion to a multivariate one without looking at the underlying combinatorics (which is how it was done in \cite{visontai2013stable}) but looking at the recurrence. Thus we saw an alternative for that construction extracting and laminating recursive information instead of combinatorial one, although, at the end, they are the same. In spite of these being the same in this case, this means that, in principle, we do not need an underlying combinatorial object to proceed this way. We only need that the amplification of the univariate recursion into a multivariate one happens through a stability preserver.

\begin{remark}[Combinatorial versus recursive]
Thus we saw that we do not, in principle, need a combinatorial meaning in order to do laminations, although these alternative meanings tend to be helpful. In fact, this also tells us that the information to do a multivariate jump might come from very diverse sources.
\end{remark}

Other ways of building nontrivial interlacers in a less structured way could involve translations or displacements of the variables. However, these translations are not always easy to establish recursively without previously having good estimations for the roots. This leads us to a loophole in which we would need information about the roots in order to find information about the roots. This is not completely undoable, but it seems less insightful about the properties of the polynomials involved.

\begin{observacion}
Using translation or displacements changes the recursivity and we have to be sure that these translations keep the polynomial RZ and that we can easily undo them. Moreover, these transformations do not seem to increase our knowledge about the original polynomials and, therefore, as we are not adding anything new or insightful it seems natural to think that the result is going to be similar to the one we obtained when we used the Renegar derivative as the interlacer to perform the extension in Part \ref{I}.
\end{observacion}

Finally, the theory of stability preservers developed in \cite{borcea2009lee1,borcea2009lee2,borcea2009polya} by Borcea and Br{\"a}nd{\'e}n only works for differential linear transformations. And it is not immediate how (or even if it is possible) to translate differential linear transformations when applying the DLG method. The corresponding linear transformation could be completely different and there is not a clear path to build these transformations from just looking at the original one as the DLG method breaks linearity in the recursion.

\begin{observacion}
Studying how the application of the DLG to all the polynomials in the sequence changes the recurrence relation is an interesting problem. There is clearly a recurrence, but it is not clear if there is a recurrence using only differential linear operators that could be candidates to be stability preservers.
\end{observacion}

All this leaves us with a picture opening many more questions about the possibilities for collecting further information of the recursion into new additional variables. We do not know in which cases this is possible. The example of Eulerian polynomials we have studied shows nevertheless that being able to collect such information helps the relaxation to get closer to the actual roots and therefore seems to contribute to improving it through the refinement of the information that it can provide about the original polynomial. Further explorations in that direction could then provide insights into more drastic strategies and ways to improve the relaxation maybe even until the desired (but still far) estate of exactness.

\begin{cuestion}
Which recurrences allow for a transformation that let us perform a lamination through a lifting of the original recurrence into one with more variables at each step? Which of these laminations do actually improve the (asymptotic) performance of the relaxation at giving bounds (in the diagonal)?
\end{cuestion}

What we do know and we have proved is that uninformative expansions of our polynomials by adding variables through the use of generic interlacers do not work. These variable extensions do not improve the relaxation and therefore are unhelpful for our task of understanding better our RZ polynomials through the relaxation. We need ways to actually collect new information. Our next section is dedicated to comment on this phenomenon preventing us to improve the relaxation when we are not actually adding new information to our extensions through generic interlacers as the example of Renegar derivatives showed.

\section{Derivative triviality was discarded}

As pointed out, we use this section to realize how we showed that the Renegar derivative does not help to improve the relaxation. We saw this in the first chapter. This is so because, although the Renegar derivative produces interlacers, these interlacers seem to be very generic and therefore they do not contribute new information to the relaxation.

\begin{remark}
Any RZ polynomial $p\in\mathbb{R}[\mathbf{x}]$ has its Renegar derivative $p^{(1)}\mathbb{R}[\mathbf{x}]$ as an interlacer. Thus, if $p$ codifies some kind of combinatorial information, just considering $p+yp^{(1)}$ does not add any particular information about the underlying combinatorial structure.
\end{remark}

Contrary to this, we saw that using the same technique to add information about less trivial interlacers does not only change the relaxation, but improves it. This will happen always.

\begin{proposicion}\label{improinter}
Let $p\in\mathbb{R}[\mathbf{x}]$ be a RZ polynomial and consider another RZ polynomial $q\in\mathbb{R}[\mathbf{x},y]$ with $q(\mathbf{x},0)=p$. Then $$\{a\in\mathbb{R}^{n}\mid(a,0)\in S(q)\}\subseteq S(p).$$
\end{proposicion}

\begin{proof}
Evident looking at the defining LMPs.
\end{proof}

Thus, adding new variables can never worsen the relaxation and therefore, if we find interlacers with a good structure and easy enough to describe, then they will improve the relaxation. The use of the combinatorial properties to build the recursion allows us to build our interlacers in ways that do not depend on having previous quantitative information about the roots besides the interlacing fact.

\begin{remark}
Contrary to perturbing the polynomial through scaling or translations, interlacing does not require us to know if we are respecting all the roots because this is something coming directly from the structure of the interlacer we are computing. The advantage is that this allows us to transform in this way combinatorial information into analytic one about the roots if we are able to build our extensions.
\end{remark}

This tells us how to go bivariate, as we did in the examples concerning the $0$-th and $1$-st iterations of the DLG method. However, as we do not know how to effectively add more than one variable carrying the information of the immediately precedent (interlacing) element in the recursion, we still need to understand better the combinatorial approach if we want to have other possibilities of going multivariate through the lamination of the information originally carried out by one variable.

The ways we use to perform this lamination get clearer if we examine in detail how this happened in the case of Eulerian polynomials. Note however that this lamination is an essentially different process to the recursive one done through interlacing and that it is not clear how to extend this lamination looking directly at the recursion or at the underlying combinatorial object if it exists. This greatly limits our ways (and possibilities) of effectively \textit{injecting} our univariate polynomial into a multivariate RZ polynomial resulting from a lamination into several variables and tells much about how limited is our knowledge of the relationship between variable extensions and RZ-ness as well as how nuanced and subtle these processes are. The next section deals with the understanding of the the underlying combinatorial structure as a strategy to build these laminations. We also care about the relation between the combinatorial information and the easy $1$ variable extension through interlacing recursion. We do this in the next section through the detailed study of the example given by Eulerian polynomials in previous sections.

\section{Adding combinatorial information}

We saw that the deep understanding of the combinatorial objects underlying the construction of our polynomials played a fundamental role in our ability to build multivariate RZ extensions in which we can inject our original univariate polynomial along some line (fundamentally so far the diagonal) although we can look at any line through the origin as we work here with RZ polynomials. We delve here into how such profound understanding of the underlying combinatorial object allowed us to build the multivariate version that helped us to improve the relaxation sections above.

\begin{observacion}[From combinatorics to recurrence]
In \cite{visontai2013stable} we can read how counting finer descents into where they have their tops was the tool they used to go multivariate. We also saw however that it is in principle possible to do this without referring to the underlying combinatorial object just transforming adequately the univariate recurrence. What we do not know is if this is in general possible for any recurrence or how to determine if it is so.
\end{observacion}

This observation is what allows building a recurrence based on the application of a stability preserver that can successfully produce the multivariate versions of the Eulerian polynomials. In general, looking at stability preservers is a good way of generating sequences of RZ polynomials. The main problem with these polynomial sequences that increase not just the degree but also the number of variables is that that appearance of new variables breaks the concept of interlacing. Therefore the main problem here is that we are forced to choose between using interlacing in the recurrence or using addition of variables through the operator defining the recursion. With what we know so far we cannot combine both methods since there are no good concepts of interlacing when the number of variables grows as does the degree. We have just two ways of proceeding and both improve the relaxation by increasing the number of variables, as the example of Eulerian polynomials has shown us.

\section{Involving recurrent information}

We live in times of recurrence and therefore it is impossible not to make here an observation about how some models of artificial neural networks that are taking over the train of the fast growing trend of artificial intelligence in the ongoing development of the technology sector benefits from recurrence in the chains that compose the human language in order to build these gigantic models \cite{rumelhart1986learning}. This recurrence is obviously different in nature to what we use here, but it still shares some features with our recurrence. Recurrence in the construction of RNNs benefits from the information stored in one chain of language expression in order to build a guess for the next word that should come into the next available position of one sentence. We do something similar when we apply our relaxation to sequences of interlacing polynomials. The information of the previous object in the chain gets encoded into a bivariate RZ polynomial and therefore it can be used and exploited through the relaxation. This usage of the recurrent information allows us to build better bounds for the extreme roots of our polynomials. That is, we get a better prediction of some (root) behaviour of the next polynomial in the sequence.

\begin{observacion}[Storing recursive information into the relaxation]
Transformations allowing us to store previous polynomials through an increase in the number of variables can be used to improve the relaxation. Through these transformations the relaxation seems to \textit{understand} better the structure of the polynomials involved. This enhanced understanding (that can be combinatorial, algebraic or geometric) can therefore be transformed via the relaxation into analytic information about the extreme roots of the polynomial. The magic of the recursion that works for RNNs applied to human language seems to also work for sequences of polynomials encoding some mathematical information.
\end{observacion}

This observation calls again for the necessity of an expansion of the methods used and developed here. If RNNs benefit from growing the amount of words that they consider in order to build the next word in a text, it is natural to ask the same here and, in the view of how our relaxation works, it is clear that a call for extending the number of variables together with the amount of interlacers admitted would produce methods of unknown potential in the quantitative task of approximating roots and maybe even in the qualitative task of solving the GLC.

\begin{remark}
Since adding information about the recurrence seems to be helpful, it would be desirable to have constructions of liftings of polynomials using larger windows of recurrence. A quick exploration shows that these steps are very subtle to accomplish with random interlacers. However, some papers have found constructions pointing in this direction for interlacers sharing some algebraic structure with the original polynomial, see for example \cite{nuijtype}. Nevertheless, the transformations studied there are far from what we need because they use interlacers that are \textit{too} simple and general and they do not get to the point of adding many more variables to the polynomial. This is \textit{in some sense} precisely the spirit of the many amalgamation conjectures proposed in \cite{main}.
\end{remark}

The methods exposed in this chapter until now try to augment the number of variables as this is fundamental for the relaxation to improve. We also saw that increasing the number of coefficients involved in the computation produces better bounds in a sense after applying a known transformation. In particular, the DLG method combines information of higher degree monomials, but this information does not get correctly laminated along several variables so the relaxation does not directly improve. However, we have seen that, in some sense, in some direction and after applying the correct transformation it does improve the bounds obtained and therefore it is clear that the information of these extra coefficients used to compute the new iteration of the polynomials is flowing down to the bounds through the relaxation. The main problem is that this relaxation does not anymore directly relax the RCS of the original polynomial, but a known transformation of it. This opens the door to simplify the problem exposed by the GLC not to find MSLMP (monic symmetric linear matrix polynomial) representations of RCSs but just MSLMP representations of some (algebraic) transformations of these RCSs. Thus, we get a way of dealing with the high complexity of the algebraic elements involved in these representation by transforming the original object in a way that could resolve these complexities. We see more on this in the next section dedicated to the DLG method, its possible extensions (both in number of variables and degree), how it allows us to win at establishing better bounds and how it gets even better when combined with the relaxation.

\section{DLG method and winning}

Combining the relaxation with the DLG method produced the best bounds for the extreme roots of the univariate Eulerian polynomials that we found here. These are also better than all the bounds we could find in the previous literature. The relaxation works therefore as an intermediate step between one iteration of the DLG method and the next at establishing bounds.

\setlength{\emergencystretch}{3em}%
\begin{remark}
The relaxation applied to the Eulerian polynomial offers worse bounds than Sobolev's approach based on the DLG method. However, the relaxation applied to the polynomial obtained through the DLG method in Sobolev's approach produces an even better bound.
\end{remark}
\setlength{\emergencystretch}{0em}%

Now, the univariate DLG method is a popular way of obtaining estimates for polynomial roots but it does not admit any well-know nor widely used generalization. The generalization we would like to have should involve either considering multivariate polynomials or different (but still easy) transformations of the roots in ways that allowed us to keep an RZ polynomial. We shortly see these two methods that we briefly explored while studying the bounds obtained sections above and that we consider worth researching in more depth in the future in order to provide both better bounds and maybe further approaches to tackle the GLC.

\begin{observacion}
When we apply the DLG method we are no longer looking at our original polynomial. It is not clear, in general, how to transform a spectrahedral relaxation of the transformed polynomial into one for the original polynomial. However, notice that, as we saw, in the univariate case it is clear how to do this.
\end{observacion}

We begin by looking at different transformations of the roots. In particular, we could wonder why we stay with just squaring roots and we do not try with other powers. Well, the main reason is that squaring is easily done by considering the product $p(x)p(-x).$ Other powers are indeed possible, but they require a bit more of work using resultants.

\begin{teorema}\label{resultants}
Let $p=\prod_{i=1}^{d}(x-a_{i})\in\mathbb{R}[x]$ be a monic polynomial with all its roots different. Then the coefficients of the polynomial $q=\prod_{i=1}^{d}(x-a_{i}^{k})\in\mathbb{R}[x]$ can be written polynomially in terms of the coefficients of $p$.
\end{teorema}

\begin{proof}
This is folklore. For a hint see, e.g., \cite{2094337} and remember that the coefficients of a polynomial are symmetric functions of its roots.
\end{proof}

Looking in another direction, we also wondered if there was a possibility to apply the DLG method to multivariate polynomials. To our knowledge, there is no mention in the literature of a multivariate DLG method, but this does not mean that such a thing is impossible. Something in this direction is indeed possible, although we will find the problem that the degree grows with the number of variables while the univariate case keeps the same degree. The fact that the degree grows so fast in that construction is problematic for us because the relaxation just collects information up to degree $3$ and it benefits from increasing the amount of variables not the degree. However, this method increases the degree but keeps the number of variables. Nevertheless, although it does not work exactly in the direction we wish, it hints at new borders and possibilities for the application of the DLG method in the multivariate case opening therefore the field for making some of the questions we had while studying our roots. In particular, it would be interesting to know if there is a DLG method increasing not the degree but the number of variables in the case of multivariate polynomials. At the end of the day, dreaming is free. Let us expose this multivariate method and what it generates here.

\begin{construccion}[Multivariate DLG method]\label{multivariatecons}
Given a multivariate polynomial $p\in\mathbb{R}[\mathbf{x}]$ we can build the polynomial \begin{gather*}\label{arrr} q=(\prod_{i=1}^{n}p(x_{1},\dots,-x_{i},\dots,x_{n}))|_{x_{i}^{2}=x_{i}}.\end{gather*}
\end{construccion}

We have an easy consequence in line with the univariate DLG method. We collect it here.

\begin{proposicion}
The $q$ of Equation \ref{arrr} is a polynomial.
\end{proposicion}

\begin{proof}
The product gives a polynomial in which all the variables have even degree and therefore the substitution is possible as happens in the DLG method. 
\end{proof}

However, looking through lines passing by the origin, we can easily find examples showing that $q$ does not have to be RZ. This is problematic because it shows that this direct generalization of the DLG method to the multivariate case cannot be what we need in our generalized constructions: we need RZ polynomials in order to be able to apply the relaxation, which is what we want to do. 
Thus, we can comment on our dream for a truly going-multivariate DLG method and how it should look like in order to be able to help the relaxation to improve.

\begin{remark}[A dream]
We would require a method producing a polynomial of the same degree (or just \textit{slightly} bigger) and having at least the same number of variables whose ovals are obtained, on each line of our interest, squaring the roots in this line (therefore RZ). As it is dubious that such a neat construction could exist, we could restrict this to happen along certain lines and not all: just on these lines where the univariate polynomials we are interested in lie.
\end{remark}

However, several problems arise. These problems can be observed fast.

\begin{observacion}
One example of such method could be through lamination of recurrences, as we did above. It is not clear how the DLG method interact with the form of the recursion. In particular, in order to be able to apply Borcea-Br{\"a}nd{\'e}n theory of stability preservers, we need that the aforementioned recursion happens through the application of a linear operator \cite{borcea2009lee1,borcea2009lee2,branden2014lee}.
\end{observacion}

These methods produce new polynomials worth studying in depth. In particular, these polynomials might have interesting underlying algebro-combinatorial structures still not well known nor studied. The study of these polynomials would be a good initial point to look at these objects.

\begin{remark}[New polynomial sequences: resultants]
The method of resultants hinted at in \cite{2094337} in the proof of the folkloric Theorem \label{resultans} to study univariate real-rooted polynomials through transformations of their roots producing polynomials whose coefficients are polynomial combinations of the coefficients of the original polynomials produces completely new families and sequences of polynomials. These new sequences fall directly from the original sequences through these methods. It would be highly interesting to study which combinatorial objects are they encoding given the original combinatorial object underlying the original polynomial. This hints at deep and uncovered connections between the algebra related to symmetric functions and combinatorics.
\end{remark}

A similar structure of a forest of new sequences of polynomials surges also from our attempt at a multivariate method. Note, in particular, that the sequences above keep the degree going downwards in the transformations. However, the sequence emerging from the (failed) multivariate method has a faster growing degree. Notice also that we can easily check though an example that, when we restrict to the diagonal of the multivariate method, we do not in general get a subsequence of the sequences above at a faster growing degree. This means that these sequences are essentially different. We can just make the following remark.

\begin{remark}[New polynomial sequences: multivariate]
The (failed) multivariate method proposed in Construction \ref{multivariatecons} produces another family of sequences intimately related to the original sequence of polynomials. This family of sequences is essentially different from the previous ones. For this reason, it would be also interesting to understand the connections with the underlying combinatorics and, in particular, which are good models of the objects being encoded in these new sequences.
\end{remark}

All these constructions provide different polynomials and forests of polynomials.

\begin{observacion}[Inestability and nonequality surprises for the different DLG methods]
As mentioned above, the (failed) multivariate method we propose does not preserve RZ-ness. For this reason, what we would like to understand is the connection between this multivariate method and the univariate one because they do not produce the same corresponding polynomials of a given degree. This might give a hint about the correct multivariate generalization. Other modifications in the recursion could produce yet further families of polynomials that fit into these categories and therefore a comparison between these in this underlying combinatoric regard could produce the emergence of even more surprises in this topic. 
\end{observacion}

Therefore we see that we began studying one univariate sequence of polynomials and that our developments led us to a huge forest of different sequences of polynomials related to that one but different in nature and codifying information about different and excitingly new objects. Showing the deep beauty of mathematics on how exploring one tree ends up revealing a forest. And we have many other unexplored and untouched trees to open new forests to us in the future!

\begin{remark}[New forests of unknown polynomials]\label{forestpol}
For each polynomial in the original sequence we can produce several families changing the roots and the coefficients through the methods proposed. Then we could proceed applying also the original recursion (either univariate or multivariate) to produce related polynomials that could be fed into the same process. Thus, in this way, we can generate a structure of polynomials in the shape of a tree graph. Exploring the combinatorial meaning of these polynomials and the information encoded in their coefficients looks like a formidable task to do in the future in order to understand the underlying combinatorial dynamics encoded in the coefficient transformations that we see walking through the generated tree of polynomials.
\end{remark}

After seeing that we can use different transformations and multivariate extensions the natural question is if these two can be combined. We believe this to be possible, but we are not expanding here in that direction because it goes beyond our humble objectives here of just hinting at possible developments of some observation we had to make while studying the application of the relaxation to the Eulerian polynomials. We therefore collect this as a problem to treat in the future.

\begin{problem}
Study the underlying combinatorics and dynamics of the coefficients of the forest of polynomials uncovered through the processes described in Remark \ref{forestpol}. Determine if that information can be combined with the relaxation in order to produce accurate (multivariate) bounds for the asymptotics of the extreme (ovaloids of) roots of the original sequence of polynomials.
\end{problem}

We finish this chapter by pointing out to the surprising fact that this task of real-root bounding seems to provide a very stable bridge for translating ideas between numerics and real algebraic geometry through combinatorics. These developments of the interconnections between the ideas and methods of both fields seem to be highly productive and therefore this last chapter serves the purpose of calling for an action in that direction. This is a hopeful and bright way of finishing this part of this thesis on which we stumbled upon the relation between Eulerian polynomials and (numerical) analysis in the light of the brilliant work of Sobolev on these polynomials from the point of view of analysis. We reached them through the two different paths of algebra and combinatorics, and that trip allowed us to traverse here the already mentioned bridge over which we ask to build a highway.
\part{On Finding Symmetry on Multivariate Eulerian Polynomials}\label{V}
\chapter{Weighted Permutations From Different Angles}\label{ChWeight}

Here we will comment on strategies devised for using the rest of the directions of the multivariate Eulerian polynomials. We will see that we can do this by attaching weights to elements of the permutation. This strategy leads us to the introduction of weighted permutations. These are reminiscent of colored permutations immersed in a more general setting. A good reference for colored permutations is \cite[Section 9.6]{kitaev2011patterns}.

\section{Permutations with further information attached to the elements}

In this section we will talk about \textit{weighted permutations}. These should not be confused with \textit{colored permutations}, which carry a different more elaborated structure. However, colored permutations, when we forget about their further structure as a group, can be considered as a subset of weighted permutations and thus what we define here for these can also be applied to the colored ones. In this section therefore we introduce these weighted permutations so they can unravel a clear way of justifying the necessity of our study of symmetries in multivariate Eulerian polynomials in the next section. For this, we have to recall how the (multivariate) Eulerian polynomials are constructed in order to position ourselves in the adequate setting to introduce the new feature attached to the objects we are counting, that is, the weights. We begin by defining weights for permutations.

\begin{definicion}[Weights on permutations]
A \textit{weight} on a permutation $\sigma=(\sigma_{1},\dots,\sigma_{n})\in\mathfrak{S}_{n}$ is just a map $w\colon[n]\to\mathbb{R}$. We call a pair $(\sigma,w)$ a \textit{weighted permutation}. If we fix the same weight for all the permutations in $\mathfrak{S}_{n}$ we call \gls{snw} \textit{the set of all weighted permutations with weight $w$}.
\end{definicion}

Here we see how these weights can carry further combinatorial information. In particular, colored permutations would be weighted permutations if we allowed the weights to take values on cyclic groups $\mathbb{Z}_{r}$ and thought, as usual, of $w(k)$ as the color assigned to $\sigma_{k}$.

\begin{remark}[Weight lattices]
If, instead of considering the weights given by maps of the form $w\colon[n]\to\mathbb{R}$ having values in $\mathbb{R}$, we consider weights $w\colon[n]\to\mathbb{Z}\subseteq\mathbb{R}$ and equip these values with a lattice structure under a modular equivalence relation so that we have $r\in\mathbb{N}$ with $$w(a_{1})\equiv_{r} w(a_{2}) \mbox{\ if and only if\ } w(a_{1})-w(a_{2})=kr \mbox{\ for some\ } k\in\mathbb{Z}$$ we can recover thus the colored permutations and their structure by defining the corresponding wreath product \cite{meldrum1995wreath}.
\end{remark}

Thus, we see that weights fundamentally underlie many constructions previously appearing in the study of different groups stemming from the symmetric groups when permutations are equipped with further structure. This is the topic of the next section in relation to the corresponding Eulerian polynomials that appear this way.

\section{Relations between different contexts}

As we noticed above, we can recover colored permutations from a particular case of weighted ones. In fact, weights allow us to go into a \textit{more colorful} setting.

\begin{remark}[Colored permutations as a subclass]
In this way, now colors can be taken in general in the torus $\mathbb{R}^{n}/r\mathbb{Z}^{n}$.
\end{remark}

Thus, this generalized setting calls for a generalization of Eulerian polynomials. Fortunately, it is easy to obtain these polynomials as particular lines (whose direction is given by the vector of weights) of the multivariate Eulerian polynomials. For simplicity, from now on, we will not talk anymore here about lattices because we just want to start the study of these objects and going beyond that lies out of our scope here, which merely intends to justify the interest in the study of these weights. Thus, we can incorporate weights in our Eulerian polynomials in the following way.

\begin{definicion}[Weighted Eulerian polynomials]
Denote the \textit{$n$-th Eulerian polynomial with weight $w\in\mathbb{R}^{n+1}$} $$\gls{anw}:=\sum_{(\sigma,w)\in\mathfrak{S}_{n+1}^{w}}\prod_{i\in\des(\sigma)}xw_{i}$$
\end{definicion}

But we can readily see that, actually, we already had all these univariate polynomials codified inside the multivariate ones. Indeed, taking weights in this polynomial amounts just to changing the direction we look at in the multivariate Eulerian polynomial so we have immediately the following.

\begin{proposicion}[Weighted polynomials as restrictions to lines]
$A_{n,w}(x)=A_{n}(x\bold{w})$ with $\bold{w}=(w(1),\dots,w(n+1)).$
\end{proposicion}

Now we analyze the role of these weights in the corresponding Eulerian polynomial. In particular, we want to see in more detail how these univariate polynomials are all injected in the corresponding multivariate one. This is the topic of the next section.

\section{Meaning of positions of the weights}

Of course, the weights can be assigned differently when defining the polynomials producing thus different variants of these weighted Eulerian polynomials. And we can see that many of these variations, when accounting for the weight as a coefficient for the monomial being multiplied associated with the descent, end up forming polynomials that are injected and integrated as univariate line restrictions of the multivariate Eulerian polynomials.

\begin{remark}[Vectors of weights]
That is, we are looking through lines generated by a vector given by the weight. This is something very reminding of what an RZ polynomial is.
\end{remark}

In fact, the relaxation therefore deals with all the weights at once. And there lies its importance in relation to weights.

\begin{remark}[All injections at once]
Thus, this shows the power of the relaxation given that it can therefore deal with all these injections of the weighted polynomials at once thus providing bounds uniformly to all of them in just one application of our method as long as we are able to find where our polynomial has been injected in the multivariate one.
\end{remark}

However, this last thing, in this case, is easy to do. More complicated problems could appear.

\begin{observacion}[Nontrivial injections]
A further problem would arise in cases where finding this injection is not that easy, but, fortunately, the examples that have appeared for us so far do not present this problem and, for this reason, we prefer not to delve here more in that direction.
\end{observacion}

This limitation on the kind of injections that we consider allows us to concentrate on the next section in the study of this last power of our relaxation, which is, as we mentioned, its globality. In particular, it gives global bounds for a extreme root of each weighted polynomial (the innermost ovaloid). In the next chapter of this part, we will see that we can go even beyond that using a generalization of the palindromicity of the univariate Eulerian polynomials to the multivariate ones.

\section{A notion of globality within the relaxation}

Therefore we see that the relaxation can achieve much more than what we have shown so far. Indeed, it provides many more bounds to many more polynomials than the ones we could analyze in this first approach. This globality is therefore a feature that we still have to learn how to exploit. This is what we do in this section. For that and with the tools we have so far, we know that we have to find the way to extract bounding information restricting to lines through the origin. This is what we accomplish through the next object.

\begin{remark}[Sequences of weights going to infinity]
There are two ways of looking at infinity in directions through the sequence of multivariate polynomials. We can either choose a sequence of numbers or choose a sequence of sequences.
\end{remark}

We begin with the less general instance because it reminds us more to what we did in the univariate case.

\begin{notacion}
We write $|\rt_{t}{A_{n}(w_{1}x_{1},\cdots,w_{n}x_{n})|_{x_{i}=t}}|$ for the set of the absolute values of the roots of the univariate polynomial in the variable $t$ given by the restriction $$A_{n}(w_{1}x_{1},\cdots,w_{n}x_{n})|_{x_{i}=t}:=A_{n}(w_{1}t,\cdots,w_{n}t).$$
\end{notacion}

\begin{definicion}[Asymptotic exponential growth map]
Given a sequence of polynomials $A:=\{A_{n}\}_{n=0}^{\infty}$ with each polynomial $A_{n}\in\mathbb{R}[x_{1},\dots,x_{n}]$, we define its \textit{asymptotic exponential growth map} $\gls{asyA}\colon\mathbb{R}^\mathbb{N}\to[0,\infty],$ $$w\mapsto\inf\{u\in(0,\infty)\mid\lim_{n\to\infty}\frac{\max(|\rt_{t}{A_{n}(w_{1}x_{1},\cdots,w_{n}x_{n})|_{x_{i}=t}}|)}{u^{n}}<\infty\}.$$
\end{definicion}

Observe that this map decides a direction for each polynomial in the sequence $A$ with the property that the projection over the component corresponding to each variable does not change after it has been established once in the sequence. This is the meaning of taking a fixed sequence and not a completely unrelated direction each time.

\begin{observacion}[Directions]
The fixed sequence is the object that connects all the projections together and that determines where we have injected the univariate sequence we are interested in. The main difference is that we are selecting a different direction each time for the projection over a finite component.
\end{observacion}

In the more general setting, the sequence $w$ can be substituted for a sequence of vectors. Thus the object above is a particular case of the following in which the directions are disconnected one from another.

\begin{definicion}[Sequences of weight vectors going to infinity]
Call $S:=\{w\in(\bigcup_{n}\mathbb{R}^{n})^\mathbb{N}\mid w_{i}\in\mathbb{R}^{i}\}$. Given a sequence of polynomials $A:=\{A_{n}\}_{n=0}^{\infty}$ with each polynomial $A_{n}\in\mathbb{R}[x_{1},\dots,x_{n}]$, we define the \textit{general asymptotic exponential growth map} of the sequence $A$ as $\gls{asyA}\colon S\to(0,\infty],$ $$w\mapsto\inf\{u\in(0,\infty)\mid\lim_{n\to\infty}\frac{\max(|\rt_{t}{A_{n}(w_{n 1}x_{1},\cdots,w_{n n}x_{n})|_{x_{i}=t}}|)}{u^{n}}<\infty\}$$
\end{definicion}

We explain why we prefer the less general definition from the point of view of the construction of an insightful story.

\begin{remark}[Infinite diagonal injection]
We decide to take the less general definition because it is the one that most accurately resembles our vision of how the univariate Eulerian polynomials are injected into the multivariate ones: they are not just floating around the restrictions of the multivariate one; instead, they are always injected into the diagonal.
\end{remark}

Returning to the first definition, we want to see examples of how $\asy$ works in order to understand what we are measuring through it. The next example will be revealing.

\begin{ejemplo}[Factorized polynomials]
The easiest examples come when our polynomials are given already factorized. So consider the sequence of polynomials $A=\{\prod_{i=1}^{n}(x_{i}-2^{i})\}$ (or small enough perturbations of it if we do not like to have so easy multivariate polynomials). It is clear that injected in the diagonal $d=\{1\}_{i=1}^{\infty}$ we have the univariate sequence $A_{d}=\{\prod_{i=1}^{n}(x-2^{i})\}$ and injected through the direction sequence $w=\{2^{i}\}_{i=1}^{\infty}$ we have $A_{w}=\{\prod_{i=1}^{n}(2^{i}x-2^{i})\}.$ Now the sequence of roots of $A_{d}$ is given by $\{\{2^{i}\}_{i=1}^{n}\}_{n=1}^{\infty}$ while the sequence of roots of $A_{w}$ is given by $\{\{1\}_{i=1}^{n}\}_{n=1}^{\infty}.$ Hence the sequences of maximum of these roots are respectively $\{2^{n}\}_{n=1}^{\infty}$ and $\{1\}_{n=1}^{\infty}$ so $\asy_{A}(d)=2$ while $\asy_{A}(w)=1$.
\end{ejemplo}

In this example, we can see that the map $\asy$ is affected both by the coefficients of the polynomial and by the direction(s) chosen. Moreover, after restricting to a direction given by a sequence, we have to determine the roots of these polynomials. Thus, it is an object that in general might be difficult to compute and therefore we need tools for knowing more about its behaviour indirectly. The relaxation will help us to accomplish this. The relaxation will allow us to establish bounds for the map $\asy$ for some subsets of sequences in $\mathbb{R}^{\mathbb{N}}$. Here is precisely where the globality of the relaxation that we wanted to expand upon in this section comes to the rescue again. The next result is the key to this.

\begin{proposicion}[Globality of the relaxation]
Let $A=\{A_{n}\}_{n=0}^{\infty}$ be a sequence of multivariate polynomials and $w=\{w_{n}\}_{n=0}^{\infty}\in S$ a sequence of directions choosing a sequence of univariate polynomials $$A_{w}=\{A_{n}(w_{n1}x_{1},\dots,w_{nn}x_{n})|_{x_{i}=t}\}$$ with all roots non-positive injected in the sequence $A$. Fix, moreover, any sequence of vectors $v=\{v_{n}\}_{n=0}^{\infty}$ and denote $R_{n}:=R_{n,0}+ x_{1}R_{n,1}+\cdots+x_{n}R_{n,n}$ the LMP of the relaxation applied to the corresponding polynomial $A_{n}$ in the sequence $A_{w}$ so that $R_{n,\tot}:=t(w_{1}R_{n,1}+\cdots+w_{n}R_{n,n})$ verifies $v_{n}^{T}R_{n,\tot}v_{n}>0$. Then we have the bounds $r_{n}>-\frac{v_{n}^{T}R_{n,0}v_{n}}{v_{n}^{T}R_{n,\tot}v_{n}}$, where $r=\{r_{n}\}_{n=0}^{\infty}$ is the sequence of the largest roots of the polynomials in $A_{w}$.
\end{proposicion}

\begin{proof}
Compute the sequence of relaxations and apply them to the sequence of vectors to obtain the sequence of linear forms $$R=\{v_{n}^{T}R_{n}(w_{n1}x_{1},\dots,w_{nn}x_{n})|_{x_{i}=t}v_{n}\}.$$ This sequence provide a sequence of lower bounds for the biggest root of the polynomials in $A_{w}$ through the inequality $$v_{n}^{T}R_{n}(w_{n1}x_{1},\dots,w_{nn}x_{n})|_{x_{i}=t}v_{n}=v_{n}^{T}R_{n,0}v_{n}+tv_{n}^{T}R_{n,\tot}v_{n}>0.$$ Isolating $r_{n}$ in the inequality and using the properties of the relaxation, we obtain the desired bounds $r_{n}>-\frac{v_{n}^{T}R_{n,0}v_{n}}{v_{n}^{T}R_{n,\tot}v_{n}}$.
\end{proof}

This result is an evident consequence of the properties of the relaxation that generalizes what has been done here for the univariate Eulerian polynomials injected in the diagonal of the multivariate Eulerian polynomials to other polynomials injected in different directions inside other multivariate polynomials. This result allows us, in fact, to define a map producing bounds for $\asy$. We do it in full generality through the next definition.

\begin{definicion}[General relaxation and bound maps]
Recall that we define the set $S=\{w\in(\bigcup_{n}\mathbb{R}^{n})^\mathbb{N}\mid w_{i}\in\mathbb{R}^{i}\}$. Given a sequence of polynomials $A:=\{A_{n}\}$ with each polynomial $A_{n}\in\mathbb{R}[x_{1},\dots,x_{n}]$, we define the \textit{general relaxation map of the sequence $A$} as $\gls{rela}\colon S\times S\to\mathbb{R}[t]_{1},$ $$(w,v)\mapsto v_{n+1}^{T}R_{n}(w_{n1}x_{1},\dots,w_{nn}x_{n})|_{x_{i}=t}v_{n+1}.$$ This allows us to define the \textit{general bound map} $\gls{boua}\colon S\times S\to\mathbb{R},$ $$(w,v)\mapsto -\frac{v_{n}^{T}R_{n,0}v_{n}}{v_{n}^{T}R_{n,\tot}v_{n}}.$$
\end{definicion}

Now, it is clear that $\bou_{A}$ can be used to provide bounds of $\asy_{A}$ because its values bound the value of the roots of the polynomials in the sequence, as we showed in the proposition above.

\begin{remark}[Multivariability as globality]
    We introduced these to notions here in order to point towards the fact that the work we are developing here is highly multivariate, although we look at univariate polynomials for a measure of how good our multivariate methods are. The fact that univariate Eulerian polynomials and their (extreme) roots are so well understood helps us now in our work of developing measures on how good our approximations are here. We measure through the diagonal because there is where the univariate polynomials lie. However, the work of this chapter tells us that more general measures should be possible and that interesting and well-known sequences of real-rooted polynomials also appear injected into the multivariate Eulerian polynomials away from the diagonal. This shows us that working towards generating and investigating these more general measures serve a purpose similar (but more accurate) to the purpose we had here. In fact, these more general viewpoints could help us to understand the behavior of the extreme roots of these more general sequences of polynomials that appear injected away from the diagonal into the sequence of multivariate polynomials. This justifies our introduction of these general definitions tackling these objects in this chapter.
\end{remark}
    
Unfortunately, we cannot develop the whole theory in this thesis because it would go beyond the limits of our main objectives here. Before closing, we also comment on the generality of our definitions, why we introduced both and in what sense they differ. This will help us to understand the different approaches these two definitions allow and how one contains the other.

\setlength{\emergencystretch}{3em}%
\begin{observacion}[Bounds, relaxation and asymptotics]
Here, it makes more sense to introduce this general definition because, in our example of Eulerian polynomials, our sequence of vectors $v$ does not get fixed by its head contrary to what happens for the diagonal $d$, which was what motivated the initial definition given at the beginning of this section. Maybe, it would be interesting studying which bounds can be obtained fixing the head of the vectors, but we do not develop this further here.
\end{observacion}
\setlength{\emergencystretch}{0em}%

Instead, now that we have generated through the last two sections a picture of how the relaxation is also useful for polynomials not injected into the diagonal of their multivariate extensions, we have a motive for studying symmetries inside the multivariate Eulerian polynomials besides the already known one that the univariate Eulerian polynomials are palindromic. Indeed, the multivariate Eulerian polynomials verify a generalization of this symmetry property. This is the main content of the next chapter.
\chapter[Combinatorial Relations]{Combinatorial Relations on Sets of Permutations with Fixed $\mathcal{DT}$}\label{ChCombi}

\section{Reciprocal polynomials in the univariate case}

In the literature, we found previous bounds for the smallest roots of the univariate Eulerian polynomials $A_{n}$. This is what we were dealing with some chapters above. We call this smallest root $r_{n}$. Notice that now we cannot directly apply the relaxation because it only works for bounding the zeros that are around the origin. Therefore, we need to perform a transformation in the original Eulerian polynomial that could allow us to apply the relaxation to study our other root of interest. The obvious transformation consists in considering the (obviously also RZ) reciprocal polynomial, whose roots are the inverse of the roots of $A_{n}.$

\begin{definicion}[Reciprocal polynomial]\cite[Section 2.1]{joyner2018self}
Let $p\in\mathbb{R}[x]$ be a univariate polynomial. We denote $$\rec(p(x)):=\gls{rec}(x):=x^{\deg(p)}p(\frac{1}{x})$$ its \textit{reciprocal}.
\end{definicion}

After studying the multivariate extensions of Eulerian polynomials, it is natural to ask if we can construct these multivariate extensions also for $\rec(A_{n})$. It turns out that the direct approach works in this case. We introduce an important concept of symmetry for univariate polynomials.

\begin{definicion}[Palindromicity as symmetry]\cite[Section 2.1]{joyner2018self}
Let $p\in\mathbb{R}[x]$ be a univariate polynomial. We say that $p$ is \textit{palindromic} if $p=\rec(p).$
\end{definicion}

This symmetry in the coefficients translates immediately into a symmetry for the roots. The next immediate result guarantees it.

\begin{proposicion}[Inverses of roots are roots of reciprocal]
Let $p\in\mathbb{C}[x]$ be a univariate polynomial and $0\neq r\in\mathbb{C}$ a root of it. Then $\frac{1}{r}$ is a root of $\rec(p)$. In particular,  if $r$ is a root of $p$ and $p$ is palindromic, so is $\frac{1}{r}$ a root of $p$. 
\end{proposicion}

Thus, we see that roots of (univariate) palindromic polynomials come in pairs $(r,\frac{1}{r})$. In particular, if $\pm1$ is a root, this pair is $(\pm 1,\pm 1)$. We warn that we are not speaking here about multiplicities, just about different roots.

\begin{remark}[Palindromicity in Eulerian polynomials]
A simple count (that can elegantly be performed through a bijection) over the descent statistic of permutations shows that Eulerian polynomial are palindromic. This gives the relaxation automatic access to the smallest root. 
\end{remark}

We want to deal with multivariate polynomials and therefore we want to extend to these what we saw in this section. In particular, we want to find the corresponding form of symmetry for the multivariate case. In order to do this, we have to extend our considerations to a multivariate environment. Thus, now we need to define a couple of concepts: the extension of the reciprocal to the multivariate case and a special subset of multiaffine polynomials (see \cite{branden2007polynomials} for more information about this kind of polynomials) containing our Eulerian polynomials. This is the topic of the next section.

\setlength{\emergencystretch}{3em}%
\section[Extending reciprocal operation]{Extending reciprocal operation to multivariate polynomials}
\setlength{\emergencystretch}{0em}%

There are many ways to extend the reciprocal operation in a natural way to multivariate polynomials. Here we introduce the following notion.

\begin{definicion}[A notion of multivariate reciprocal]
\label{reciprocal}
Let $A(\mathbf{x})\in\mathbb{R}[\mathbf{x}]$ be a multivariate polynomial. We define its reciprocal $$\rec(A(\mathbf{x}))=\rec(A)(\mathbf{x}):=m_{A}A(\frac{1}{x_{1}},\dots,\frac{1}{x_{n}}),$$ where $m_{A}$ is the \textit{only} minimum degree monic monomial in the variables $\mathbf{x}$ that produces a polynomial after performing the multiplication.
\end{definicion}

For multiaffine polynomials, the monomial $m_{A}$ is clear once we fix the variables. A multiaffine polynomial presenting such monomial reaches a kind of maximality. First, we need to establish a natural convention.

\begin{convencion}
From now on we say that a polynomial is a polynomial \textit{strictly} in the variables $\mathbf{x}$ if it is not possible to write it using a subset of these variables. We then say that each variable in $\mathbf{x}$ is \textit{strictly appearing} in the polynomial.
\end{convencion}

This convention, in particular, will be important for us due to the natural presence of ghost variables in Eulerian polynomials. Now, having in mind this convention, we can describe our form of maximality for multiaffine polynomials.

\begin{definicion}[Monomialmaximality]\label{monomax}
We call a multiaffine polynomial in the variables $\mathbf{x}$ \textit{monomialmaximal} if the coefficient of the monomial of maximum degree that can appear in such a multi-affine polynomial not adding to it new strictly appearing variables is not $0$.
\end{definicion}

Clearly, this means that if such polynomial is strictly actually just in the variables $(x_{i_{1}},\dots,x_{i_{m}})$ the monomial $\prod_{j=1}^{m}x_{i_{j}}$ has not $0$ as its coefficient. As we wanted, our polynomials of interest are monomialmaximal.

\begin{proposicion}[Eulerian monomialmaximality]
(Multivariate) Eulerian polynomials $A_{n}(\mathbf{x})$ are monomialmaximal.
\end{proposicion}

\begin{proof}
In the definition of $A_{n}$ consider the term of the sum generated by the permutation (in one-line notation) $\sigma:=(n+1\dots1)\in\mathfrak{S}_{n+1}$. This term generates the maximal monomial possible in this multi-affine polynomial as $\mathcal{DT}(\sigma)=\{n+1,\dots,2\}$ and, therefore, $A_{n}$ is monomialmaximal.
\end{proof}

Monomialmaximal multiaffine polynomials are nice with respect to taking reciprocals. Thus, the following result now makes it easy to continue our straight path.

\begin{proposicion}[Multiaffine monomialmaximal reciprocals] Let $A(\mathbf{x})\in\mathbb{R}[\mathbf{x}]$ be a multiaffine monomialmaximal RZ polynomial. Then its reciprocal polynomial $\rec(A)(\mathbf{x})$ is also multiaffine, monomialmaximal and RZ.
\end{proposicion}

\begin{proof}
Monomialmaximal multiaffinity is easy: we begin writing $$A(\mathbf{x})=\sum_{\alpha\subseteq\{1,\dots,n\}}a_{\alpha}\mathbf{x}^{\alpha}.$$ We know that $a_{\{1,\dots,n\}}$ is not zero as the polynomial is monomialmaximal. Therefore, $m_{A}$ in the definition of the reciprocal is forced to be the degree-maximal possible monomial in a multiaffine polynomial $\mathbf{x}^{1,\dots,n}$. In sum, all this means that the reciprocal $$\rec(A)(\mathbf{x}):=\mathbf{x}^{\{1,\dots,n\}}A(\frac{1}{x_{1}},\dots,\frac{1}{x_{n}})=\sum_{\alpha\subseteq\{1,\dots,n\}}a_{\alpha}\mathbf{x}^{\{1,\dots,n\}\smallsetminus\alpha}.$$ As $a_{\{1,\dots,n\}}\neq0$ this monomial has nonzero independent term. Similarly, as $A$ is RZ, $a_{\emptyset}\neq0$ so $\rec(A)$ is monomialmaximal. Multiaffinitty follows directly from how we wrote the polynomial in set exponent notation. RZ-ness of $\rec(A)$ can be easily checked line by line. Take the line $t\mathbf{a}$ with $\mathbf{a}\in(\mathbb{R}\smallsetminus\{0\})^{n}$ arbitrary. We want to see that $\rec(A)(t\mathbf{a})$ has only real roots. But $\rec(A)(t\mathbf{a})=t^{n}\mathbf{a}^{\{1,\dots,n\}}A(\frac{1}{ta_{1}},\dots,\frac{1}{ta_{n}})=t^{n}\mathbf{a}^{\{1,\dots,n\}}A(\frac{1}{t}(\frac{1}{a_{1}},\dots,\frac{1}{a_{n}})).$ By abuse of notation, call $\frac{1}{\mathbf{a}}=(\frac{1}{a_{1}},\dots,\frac{1}{a_{n}})\in(\mathbb{R}\smallsetminus\{0\})^{n}.$ As $\mathbf{a}^{\{1,\dots,n\}}\neq0$, the roots of the polynomial $\rec(A)(t\mathbf{a})$ are the roots of the polynomial $t^{n}A(\frac{1}{t}\frac{1}{\mathbf{a}})$, whose independent term $\frac{1}{\mathbf{a}}^{\{1,\dots,n\}}=\frac{1}{a_{1}}\dots\frac{1}{a_{n}}\neq0$ so the roots of such product are the inverses of the roots of the polynomial $A(s\frac{1}{\mathbf{a}})$, which are all real because $A$ is RZ. A straightforward continuity argument to cover the vectors $a\in\mathbb{R}^{n}\smallsetminus{\mathbf{0}}$ having some zero entries finishes this proof.
\end{proof}

In fact, for Eulerian polynomials something else can be said because this transformation actually gives back the same Eulerian polynomial but with the variables permuted. We will see this formally in the future in this chapter. Before, we need a reminder.

\begin{recordatorio}[Ghost variables]
Contrary to what we said above, the definition of Eulerian polynomials involve a ghost variable $x_{1}$ which is actually not a variable of the polynomial. However, our arguments will follow in an easier way if we deal with this variable.
\end{recordatorio}

We make this comment and add another warning before we work out the next proof.

\begin{warning}[Definition of cofactor]
Additionally, observe that we took care of this phenomenon in Definition \ref{reciprocal} when we defined $m_{A}$ literally as \textit{the minimum degree monomial in the variables $\mathbf{x}$ that produces a polynomial after performing the multiplication} $m_{A}A(\frac{1}{x_{1}},\dots,\frac{1}{x_{n}})$. This ensures that, even when $\mathbf{x}$ contains ghost variables not appearing at all in $A$, the correcting multiplicative monomial $m_{A}$ will not be affected by that arbitrariness and it is therefore well-defined. We also took care of this when defining monomialmaximality in Definition \ref{monomax}.
\end{warning}

Now we are in position to analyze how these polynomial are build looking at possible patterns in order to find internal symmetries on them resembling palindromicity. This is the topic of the next section.

\section{Fixing subsets of permutations}

We look at permutations and at the sets they fix in order to find the symmetries we search for. We will see how these observations will appear clearly and naturally in front of us when we manage to restrict our view to permutations verifying certain conditions at a time.

\begin{notacion}[Mirrors and conditioned permutations]
Let $\mathbf{x}=(x_{1},\dots,x_{n})$, then we denote $\gls{mirrorx}:=(x_{n},\dots,x_{1}).$ For a set of permutations $A\subseteq\mathfrak{S}_{n}$ and two elements $a,b\in[n]$ we denote $\gls{aatob}$ the subset of permutations $\sigma\in A$ sending $a$ to $b$ $$A^{a\to b}:=\{\sigma\in A\mid\sigma(a)=b\}.$$
\end{notacion}

Mirroring let us introduce our generalization of reciprocity for multivariate polynomials.

\begin{definicion}[Mirrorreciprocity]
We say that a polynomial $p(\mathbf{x})\in\mathbb{R}[\mathbf{x}]$ is \textit{mirrorreciprocal} if its reciprocal polynomial verifies $$\rec(p)(\mathbf{x})=p(\mirror(\mathbf{x})).$$
\end{definicion}

With these two important notions fresh in our minds, we can now proceed. We will need some notation for the proof of the next theorem.

\begin{definicion}[Cuts, splitters and set actions in one-line notation]
Given a permutation $\sigma=(\sigma_{1}\cdots\sigma_{n})\in\mathfrak{S}_{n}$ in the one-line notation, we understand that we can introduce \textit{markers} splitting (or \textit{splitters}) the elements that will be generally commas so that we can write $\sigma=(\sigma_{1}\cdots\sigma_{k},\sigma_{k+1}\cdots\sigma_{n}).$ However, if we have more information about the relations between $\sigma_{k}$ and $\sigma_{k+1}$ we can express this directly in this extended one-line notation using an inequality symbol as separation so that we can write $\sigma=(\sigma_{1}\cdots\sigma_{k}>\sigma_{k+1}\cdots\sigma_{n})$ if we know that  $\sigma_{k}>\sigma_{k+1}$ (and also the other way around if the inequality is reversed). Thus, we signal a descent that we know that happens for sure for that permutation between position $k$ and position $k+1$. In particular, we signal thus that $\sigma_{k}$ is a descent top. Sometimes we will use such information to \textit{cut} the one-line notation through some element. Thus we say that such element establishes a \textit{cut through it}. Thus if we have $\sigma=(\sigma_{1}\cdots\sigma_{k-1} \sigma_{k} \sigma_{k+1}\cdots\sigma_{n})$ and we \textit{cut through} $\sigma_{k}$ this means that we will now consider the formal one-line notations $(\sigma_{1}\cdots\sigma_{k-1})$ and $(\sigma_{k+1}\cdots\sigma_{n}),$ which clearly do not in general need to be permutations by themselves. We will however concatenate these cuts and the cutting elements to form new one-line notations actually representing permutations. Finally, when we have a set $A\subseteq\mathfrak{S}_{n}$ and a formal one-line notation $(\tau_{1}\cdots\tau_{m})$ of elements in $[n+1,n+m]$ we can concatenate them using a splitter to form a \textit{lifting} of the set $A$ injected into the set of permutations $\mathfrak{S}_{n+m}$ defining $(A(\tau_{1}\cdots\tau_{m})):=\{(\sigma_{1}\cdots\sigma_{n}\tau_{1}\cdots\tau_{m})\in \mathfrak{S}_{n+m}\mid (\sigma_{1}\cdots\sigma_{n})\in A\}.$
\end{definicion}

We will use a particular constellation of splitters in our proof. We next particularize to such case for higher clarity.

\begin{remark}[Splitters and non-splitters]
Notice that $=$ cannot be used as an splitter and if it appear it just serves the purpose of writing the same element on a different way because all the elements in the one-line notation have to be different. Notice also, in particular, that, when we lift a set of permutations in the definition above, we can always use the splitter $<$ and write $(A<(\tau_{1}\cdots\tau_{m})):=\{(\sigma_{1}\cdots\sigma_{n}<\tau_{1}\cdots\tau_{m})\in \mathfrak{S}_{n+m}\mid (\sigma_{1}\cdots\sigma_{n})\in A\}$ because $\sigma_{i}<\tau_{j}$ for any $i\in[1,n]$ and $j\in[1,m].$
\end{remark}

Additionally, due to $x_{1}$ being a ghost variable, denote, during the next proposition and its associated proof, $\mathbf{x}=(x_{2},\dots,x_{n+1})$ for brevity. Also, we remind the reader of the sets $R(n,S)$ introduced in Definition \ref{rns} because we will use them in the next proof. Also, remember that there reference to $n$ can be dropped when $n$ is big enough or understood by the context and therefore we might simply write $R(S)$ instead sometimes, see Remark \ref{rshortening}.

\begin{teorema}[Eulerian mirrorreciprocity]\label{especialita}
The $n$-th multivariate Eulerian polynomial $A_{n}(\bold{x})$ is mirrorreciprocal.
\end{teorema}

\begin{proof}
Following the definitions of these polynomials, we have that $\rotA_{n}(\bold{x}):=$ \begin{gather*}\bold{x}^{\bold{1}}\sum_{\sigma\in\mathfrak{S}_{n+1}}\prod_{i\in\mathcal{DT}(\sigma)}\frac{1}{x_{i}}=\sum_{\sigma\in\mathfrak{S}_{n+1}}\prod_{i\in[2,n+1]\smallsetminus\mathcal{DT}(\sigma)}x_{i}\end{gather*} and, introducing the permutation $\tau\colon[n+2]\to[n+2],i\mapsto\tau(i):=n+3-i$ and calling $\bold{y}=\mirror(\bold{x})$ so that $y_{i}=x_{\tau(i)},$ that $A_{n}(\mirror(\bold{x}))=A_{n}(\bold{y})=$\begin{gather*}
   \sum_{\sigma\in\mathfrak{S}_{n+1}}\prod_{i\in\mathcal{DT}(\sigma)}y_{i}=\sum_{\sigma\in\mathfrak{S}_{n+1}}\prod_{i\in\mathcal{DT}(\sigma)}x_{\tau(i)}.\end{gather*} Now we analyze these polynomials coefficient-wise having in mind that, by the previous propositions, they are multiaffine. Thus, fixing a set $K\subseteq[2,n+1]$, we have that $\coeff(\bold{x}^{K},\rotA_{n}(\bold{x}))=|\{\sigma\in\mathfrak{S}_{n+1}\mid K=[2,n+1]\smallsetminus\mathcal{DT}(\sigma)\}|=|\{\sigma\in\mathfrak{S}_{n+1}\mid \mathcal{DT}(\sigma)=[2,n+1]\smallsetminus K\}|:=|R(L)|$ with $L:=[2,n+1]\smallsetminus K$ while $\coeff(\bold{x}^{K},A_{n}(\mirror(\bold{x})))=|\{\sigma\in\mathfrak{S}_{n+1}\mid K=\tau(\mathcal{DT}(\sigma))\}|=|\{\sigma\in\mathfrak{S}_{n+1}\mid \tau(K)=\mathcal{DT}(\sigma)\}|=|R(\tau(K))|,$ as $\tau\circ\tau=\id$. Observe that it is in general not evident that these cardinalities are equal as the sets at which $R$ is evaluated in each case are different (even their cardinalities are so: respectively, $n-|K|$ and $|K|$). In order to see that they are in fact equal we will establish a bijection between the sets $R(\tau(K))$ and $R(L)$. The first observation that we have to make in our way towards the construction of the mentioned bijection is a possibly oversight fact of Corollary \ref{coroR}. We need to consider this fact because, although we finish if we establish a bijection between the set of permutations with descent top $L$ and the set of permutations with descent top $\tau(K)$, the construction of such bijection between permutations belonging to $\mathfrak{S}_{n+1}$ will require us to first lift them to $\mathfrak{S}_{n+2}$, as consequence of the fact that $\tau([1,n+1])=[2,n+2]$. Understanding why we are able (and it is therefore useful) to do that will require a close look at the last formula in Corollary \ref{coroR}. Observe that the cardinal $|R(n,X)|$ of $R(n,X)$ does not in fact depend on $n$ and we could then write instead that for, all $k\geq\max(X)-1$, we have $|R(k,X)|=$ $$\sum_{J\subseteq X}(-1)^{|X\smallsetminus J|}\alpha(J)\hat{!}.$$ This means that the amount of permutations in $\mathfrak{S}_{n+1}$ with descent top $L$ equals the amount of permutations with descent top $L$ in $\mathfrak{S}_{n+2}$. This also means that we can biject one set into the other. In particular, we want this to happen through the restriction of the obvious injection $$\mathfrak{S}_{n+1}\to\mathfrak{S}_{n+2},\sigma\mapsto (\sigma\ n+2)$$ to the set $R(n,L)\subseteq\mathfrak{S}_{n+1}$, as the restriction of this injection verifies $R(n,L)\to R(n+1,L)\subseteq\mathfrak{S}_{n+2}^{n+2\to n+2}$ bijectively because it is obvious that the injection does not modify the descent top sets and we have just seen above in this proof, through the careful observation of Equation \ref{coroR2}, that $|R(n,L)|=|R(n+1,L)|$. Therefore this bijection tells us how we want to lift our permutations in $\mathfrak{S}_{n+1}$ to $\mathfrak{S}_{n+2}$ preserving the descent top set, which will be necessary to make sense of the bijection involving the use of $\tau$ that we will construct because $\tau([1,n+1])=[2,n+2].$ In particular, observe that preserving the descent top \textit{requires} the new biggest element $n+2$ of the image permutation to be set at the end in the one-line notation (i.e., to be fixed by the image permutation) because, otherwise, $n+2$ in any other position would be a descent top that would have to be added to list of descent tops of the transformation. Symmetrically (as $\tau(n+2)=1$), we need to analyze another very important element of these permutations: the position of the never-descent-top element $1$. This element establishes a fundamental cut in the permutation and, moreover, verifies $\tau(1)=n+2$, which gives us a helpful clue about how our desired bijection should act about it. In order to see this, we establish that we will send a lifted permutation $$\mathfrak{S}_{n+2}^{n+2\to n+2}\ni(\sigma < n+2)=(\sigma_{1}\cdots\sigma_{k-1}>1<\sigma_{k+1}\cdots\sigma_{n+1} < n+2)$$ with $1$ at position $k$ and $n+2$ at the end to the new permutation $$(\sigma_{k+1}\cdots\sigma_{n+1}<n+2>\sigma_{1}\cdots\sigma_{k-1}>1)\in\mathfrak{S}_{n+2}^{n+2\to 1},$$ where we remark that the sequences $(\sigma_{1}\cdots\sigma_{k-1})$ and $(\sigma_{k+1}\cdots\sigma_{n+1})$ could be empty if $1$ is at the beginning or at the end, respectively, of $\sigma\in\mathfrak{S}_{n+1}.$ Observe that this is a bijection $\mathfrak{S}_{n+2}^{n+2\to n+2}\to\mathfrak{S}_{n+2}^{n+2\to 1}$ whose domain is precisely the subset of $\mathfrak{S}_{n+2}$ that coincides with our respecting-descent-top-sets immersion of $\mathfrak{S}_{n+1}$ in $\mathfrak{S}_{n+2}$ as $(\mathfrak{S}_{n+1}<n+2)$. Observe, furthermore, that this last bijection sends the elements with descent top $L\subseteq[2,n+1]$ of $\mathfrak{S}_{n+2}^{n+2\to n+2}$ (and thus of $\mathfrak{S}_{n+2}$ as this forces $n+2$ to be fixed) to elements of $\mathfrak{S}_{n+2}^{n+2\to 1}$ whose descent top set is $L\cup\{n+2\},$ adding thus only $n+2$ as a \textit{forced} descent top (as it has to be so as soon as it is not fixed, i.e., at the end in the one-line notation). Thus, so far, we have defined a bijection $R(n,L)\to R(n+1,L\cup\{n+2\})\cap\mathfrak{S}_{n+2}^{n+2\to 1}.$ Beware that $R(n+1,L\cup\{n+2\})\cap\mathfrak{S}_{n+2}^{n+2\to 1}\neq R(n+1,L\cup\{n+2\})$ as a cut through any ascent produces elements of $R(n+1,L\cup\{n+2\})$ not in the intersection, e.g., $$(\mu_{1}\cdots\mu_{m-1}<\mu_{m}\cdots\mu_{n+1}> 1)\in\mathfrak{S}_{n+2}^{n+2\to 1}\in R(n+1,L\cup\{n+2\})\cap\mathfrak{S}_{n+2}^{n+2\to 1}$$ produces, after a cut by the highlighted ascent, the permutation $$R(n+1,L\cup\{n+2\})\ni(\mu_{m}\cdots\mu_{n+1}>1<\mu_{1}\cdots\mu_{m-1})\notin\mathfrak{S}_{n+2}^{n+2\to 1}.$$ Once this is clear, now we apply $\tau$ to all the elements in the one-line notation of the image under the previous bijection to obtain the new permutation $$(\tau(\sigma_{k+1})\cdots\tau(\sigma_{n+1})>\tau(n+2)=1<\tau(\sigma_{1})\cdots\tau(\sigma_{k-1})>\tau(1)=n+2).$$ This defines a bijection $$R(n+1,L\cup\{n+2\})\cap\mathfrak{S}_{n+2}^{n+2\to 1}\to R(n+1,\tau(K))\cap\mathfrak{S}_{n+2}^{n+2\to n+2}=R(n+1,\tau(K))$$ because the descents tops of an element in the domain appears as maxima of descents of the form $n+2>\sigma_{1}$ or $\sigma_{l}>\sigma_{l+1}$ and these are transformed by $\tau$ into the ascents (as $\tau$ is order reversing) $1<\tau(\sigma_{1})$ and $\tau(\sigma_{l})<\tau(\sigma_{l+1})$ while the rest of the elements appear either as minima in ascents of the form $\sigma_{l}<\sigma_{l+1}$ or at the end (but at the end we have always $1$) and these are transformed by $\tau$ into descents $\tau(\sigma_{l})>\tau(\sigma_{l+1})$ or in $n+2$. Thus the descent top of the image is $[2,n+1]\smallsetminus L=[2,n+1]\smallsetminus ([2,n+1]\smallsetminus K)=K.$ Thus the descent top set of this new permutation is $K=[2,n+1]\smallsetminus L$ because $\tau$ is order reversing and we took the care of having a domain whose permutations have $1$ at then, which transforms into $n+2$ at the end, an element we do not desired to be in our descent top. Observe that, as $\tau\circ\tau=\id$, similar arguments allow for the construction of an inverse for this map, which shows that it is in fact a bijection whose image is the whole lifting $R(n+1,\tau(K))$ of $R(n,\tau(K))$. We now just have to undo the lifting. Finally, we cut out the fix element $n+2$ so we can see this last permutation as $$(\tau(\sigma_{k+1})\cdots\tau(\sigma_{n+1})>\tau(n+2)=1<\tau(\sigma_{1})\cdots\tau(\sigma_{k-1}))\in\mathfrak{S}_{n+1},$$ this establish the whole path of the desired bijection and therefore finishes our proof.\end{proof}

The proof of this last theorem has interesting consequences for several combinatorial sums. We collect the implied identities in the next section. We also have to mention that, in general, finding the kind of bijection between permutations that we found in the proof above is an interesting topic by itself that has been previously studied, e.g., in \cite{bigeni2016new,bloom2020revisiting,chen2024bijection} for bijections conserving other properties or accomplishing or involving different relational counting arguments.

\section{Reciprocity from the permutation point of view}

Now we can translate many insights obtained in the last proof of the previous section into the purely combinatorial setting. This gives us results about permutations in the form of the three following corollaries. Notice that we could also use the equality between Equation \ref{coroR1} and Equation \ref{coroR2} in Corollary \ref{coroR} in order to establish further identities (using chains of inequalities explictly) in the corollaries that we state and prove in this section.

\begin{corolario}[Combinatorial sums]
For every $n\in\mathbb{N}$ and every $K\subseteq[2,n+1]$ we have that the next quantities with $\tau\colon[n+2]\to[n+2],i\mapsto \tau(i):=n+3-i$ are equal: $|R([2,n+1]\smallsetminus K)|=$ \begin{gather*}
\sum_{J\subseteq[2,n+1]\smallsetminus K}(-1)^{|[2,n+1]\smallsetminus (K\cup J)|}\alpha(J)\hat{!}=\sum_{J\subseteq[2,n+1]\smallsetminus K}(-1)^{n- |K\cup J|}\alpha(J)\hat{!}=\\\sum_{J\subseteq[2,n+1]\smallsetminus K}(-1)^{n- |\tau(K)\cup \tau(J)|}\alpha(J)\hat{!}=\sum_{\tau(J)\subseteq[2,n+1]\smallsetminus \tau(K)}(-1)^{n- |\tau(K)\cup\tau(J)|}\alpha(J)\hat{!}=\\\sum_{S\subseteq[2,n+1]\smallsetminus \tau(K)}(-1)^{n-|\tau(K)\cup S|}\alpha(\tau(S))\hat{!}=\sum_{S\subseteq\tau(K)}(-1)^{|\tau(K)\smallsetminus S|}\alpha(S)\hat{!}=\\\sum_{\tau(J)\subseteq \tau(K)}(-1)^{|\tau(K)\smallsetminus \tau(J)|}\alpha(\tau(J))\hat{!}=\sum_{J\subseteq K}(-1)^{|\tau(K)\smallsetminus \tau(J)|}\alpha(\tau(J))\hat{!}=\\\sum_{J\subseteq K}(-1)^{|K\smallsetminus J|}\alpha(\tau(J))\hat{!}=\sum_{\tau(J)\subseteq K}(-1)^{|K\smallsetminus \tau(J)|}\alpha(J)\hat{!}=|R(\tau(K))|.
\end{gather*} Additionally, as we have that the identities above are true for all $K\subseteq[2,n+1]$ and $\tau$ induces a permutation on the set of parts $2^{[2,n+1]}$ of $[2,n+1]$, we have, for all $K\subseteq[2,n+1]$, equivalently the identity $$\sum_{J\subseteq[2,n+1]\smallsetminus K}(-1)^{n-|K\cup J|}\alpha(\tau(J))\hat{!}=\sum_{J\subseteq K}(-1)^{|K\smallsetminus J|}\alpha(J)\hat{!}.$$
\end{corolario}

\begin{proof}
This is a direct consequence of the counting formulas collected in Corollary \ref{coroR} and the identity proved in Theorem \ref{especialita}.
\end{proof}

As a consequence of the observations we made, we can say more about these quantities when we increase the $n$ step by step. For this, we need to generalize $\tau$, which actually depends on $n$, and write $\tau_{s}(i):=s+3-i.$

\begin{corolario}[Sequential combinatorial sums]
Let $n\in\mathbb{N}$ and $K\subseteq[2,n+1]$ and for every set $W\subseteq[2,n+1]$ denote $W^{\max}=W\cup\{n+2\}$ and $W^{\min}=W\cup\{1\}$. Then $|R([2,n+2]\smallsetminus K^{\max})|=|R([2,n+1]\smallsetminus K)|$ and $$|R([2,n+2]\smallsetminus K)|=\sum_{J\subseteq[2,n+1]\smallsetminus K}(-1)^{n-|K|-|J|}(\alpha(J^{\max})\hat{!}-\alpha(J)\hat{!})=$$ $$\sum_{J\subseteq[2,n+1]\smallsetminus \tau_{n}(K)}(-1)^{n-|K|-|J|}(\alpha(\tau_{n}(J^{\min}))\hat{!}-\alpha(\tau_{n}(J))\hat{!})=$$ $$\sum_{J\subseteq K}(-1)^{|K|-|J|}\alpha(\tau_{n+1}(J))\hat{!}=$$ $$\sum_{J\subseteq K}(-1)^{|K|-|J|}(|J|+1)\alpha(\tau_{n}(J))\hat{!}=|R(\tau_{n+1}(K))|.$$
\end{corolario}

\begin{proof}
The first equality is evident because $[2,n+2]\smallsetminus K^{\max}=[2,n+1]\smallsetminus K$. For the second we will use Theorem \ref{especialita} above that ensures $|R(\tau_{n+1}(K))|=|R([2,n+2]\smallsetminus K)|$ together with the equalities coming through the use of Corollary \ref{coroR} that tells us $|R([2,n+2]\smallsetminus K)|=$\begin{gather*}
\sum_{J\subseteq[2,n+2]\smallsetminus K}(-1)^{n+1-|K\cup J|}\alpha(J)\hat{!}=\\\sum_{J\subseteq[2,n+1]\smallsetminus K}(-1)^{n+1-|K\cup J|}\alpha(J)\hat{!}+\sum_{J\subseteq[2,n+1]\smallsetminus K}(-1)^{n+1-|K\cup J'|}\alpha(J^{\max})\hat{!}=\\\sum_{J\subseteq[2,n+1]\smallsetminus K}(-1)^{n-|K\cup J|}\alpha(J^{\max})\hat{!}-\sum_{J\subseteq[2,n+1]\smallsetminus K}(-1)^{n-|K\cup J|}\alpha(J)\hat{!}=\\\sum_{J\subseteq[2,n+1]\smallsetminus K}(-1)^{n-|K\cup J|}(\alpha(J^{\max})\hat{!}-\alpha(J)\hat{!})=\\\sum_{J\subseteq[2,n+1]\smallsetminus\tau_{n}(K)}(-1)^{n-|\tau_{n}(K)\cup J|}(\alpha(\tau_{n}(J^{\min}))\hat{!}-\alpha(\tau_{n}(J))\hat{!})
\end{gather*} and $|R(\tau_{n+1}(K))|=$ \begin{gather*}\sum_{J\subseteq K}(-1)^{|K\smallsetminus J|}\alpha(\tau_{n+1}(J))\hat{!}=\sum_{J\subseteq K}(-1)^{|K\smallsetminus J|}(|J|+1)\alpha(\tau_{n}(J))\hat{!}\end{gather*} noting that, writing $J=\{j_{1}<\dots<j_{k}\}$, we have $$\tau_{n}(J)=\{\tau_{n}(j_{k}),\dots,\tau_{n}(j_{1})\}=\{n+3-j_{k}<\dots<n+3-j_{1}\} \mbox{\ and}$$ $$\tau_{n+1}(J)=\{n+4-j_{k}<\dots<n+4-j_{1}\} \mbox{\ so}$$ $$\alpha(\tau_{n}(J))=(n+2-j_{k},j_{k}-j_{k-1},\dots,j_{2}-j_{1})\mbox{\ and}$$ $$\alpha(\tau_{n+1}(J))=(n+3-j_{k},j_{k}-j_{k-1},\dots,j_{2}-j_{1})$$ because we remember that $\alpha(J):=(j_{1}-1,j_{2}-j_{1},\dots,j_{k}-j_{k-1})$ so we have that $\alpha(\tau_{n+1}(J))\hat{!}=(k+1)\alpha(\tau_{n}(J))=(|J|+1)\alpha(\tau_{n}(J)).$
\end{proof}

Reordering and rewriting the expressions obtained above we get another interesting identity.

\begin{corolario}[Reordering sums]
For $n\in\mathbb{N}$ and $K\subseteq[2,n+1]$ we have that $$\sum_{J\subseteq[2,n+1]\smallsetminus K}(-1)^{n-|K|+|J|}\alpha(J^{\max})\hat{!}=\sum_{J\subseteq K}(-1)^{|K|-|J|}(|J|+2)\alpha(\tau_{n}(J))\hat{!}.$$
\end{corolario}

\begin{proof}
From the corollary above, we know \begin{gather*}\sum_{J\subseteq[2,n+1]\smallsetminus K}(-1)^{n-|K|-|J|}(\alpha(J^{\max})\hat{!}-\alpha(J)\hat{!})=\\\sum_{J\subseteq K}(-1)^{|K|-|J|}(|J|+1)\alpha(\tau_{n}(J))\hat{!}\end{gather*} but, by the initial corollary, we also know \begin{gather*}
    \sum_{J\subseteq[2,n+1]\smallsetminus K}(-1)^{n-|K|-|J|}(\alpha(J)\hat{!})=\sum_{J\subseteq K}(-1)^{|K|-|J|}\alpha(\tau_{n}(J))\hat{!},
\end{gather*} which together establish the identity we wanted to prove.
\end{proof}

Eulerian polynomials can be written through excedances. We see how this allows us to get more interesting combinatorial results in the following section.

\section{Reciprocity and excedances}

First we establish precisely the expression of Eulerian polynomials through excedances.

\begin{remark}[Riordan's bijection]\cite[Equation 4.3]{branden2011proof}\label{Riordanremark}
Using a bijection of Riordan, it was noted that Eulerian polynomials can be written in terms of excedances $A_{n}(\mathbf{x}):=$ $$\sum_{\sigma\in\mathfrak{S}_{n+1}}\prod_{i\in\exc(\sigma)}x_{i},$$ where $\exc(\sigma):=\{\sigma(i)\in[n+1]\mid\sigma(i)>i\}\subseteq[2,n+1]$ is the set of \textit{excedances} of $\sigma$. In imitation of the previous notation, we denote $r(n,X):=\{\sigma\in\mathfrak{S}_{n+1}\mid\exc(\sigma)=X\}.$
\end{remark}

As a consequence of this rewriting of multivariate Eulerian polynomials, we immediately obtain duplicates of the corollaries in the section above. This is the content of the next remark.

\begin{remark}[Twin corollaries]
An immediate consequence of this expression is the fact that $|r(n,X)|=|R(n,X)|$ and therefore we obtain as immediate corollaries results similar to the ones above but for excedances. The corresponding bijection can be constructed composing the bijection given by Riordan in \cite{riordan2014introduction} with the one we used in the proof of Theorem \ref{especialita}.
\end{remark}

In order to analyze how special is the property proved in this theorem that motivated our last discussion and corollaries, we introduce a couple of concepts about (multiaffine) polynomials in the next section.

\section{Reciprocity of polynomials: conditions and obstructions}

We finally study a limitation of polynomials to be mirrorreciprocal in order to understand better how special are our Eulerian polynomials in this regard. In particular, we look at complete polynomials in the sense that all its monomials play a role.

\begin{definicion}[(Degree-)completeness]
A multiaffine polynomial strictly in $n$ variables is \textit{complete} if it has nonzero coefficients multiplying all its possible monomials. A relaxation of this concept is considering \textit{degree-complete} polynomials in which all possible degrees from $1$ to $n$ appear at least once (with a nonzero coefficient).
\end{definicion}

This form of completeness is present in Eulerian polynomials.

\begin{proposicion}[Completeness of Eulerian polynomials]
Eulerian polynomials are complete.
\end{proposicion}

\begin{proof}
Let $m$ be one arbitrary possible monomial and order its variables in increasing order of its indices so we can write $x_{i_{1}}<\dots<x_{i_{d}}$ with $d\leq n$ and $1<i_{1}<\dots<i_{d}$ remembering that $x_{1}$ is never a variable because that index can never be a descent top. For the indices not in the previous monomial (and therefore corresponding to the variables of the polynomial not appearing in $m$ and different from $x_{1}$), order them also in increasing order $1<j_{1}<\dots<j_{s}$ with $s+d=n$. Now, consider the permutation (written in one-line notation) $\sigma:=(i_{d}\cdots i_{1}\cdots1 j_{1}\cdots j_{s})$. This sigma ensures the appearance of the given monomial $m$ and, as this monomial $m$ was arbitrary among all the possible monomials and there are no cancellations on the definition of the Eulerian polynomials, this finishes our proof.
\end{proof}

One can ask then if Theorem \ref{especialita} is true for all multiaffine monomialmaximal complete polynomials. The answer is clearly no as the next example shows.

\begin{ejemplo}[Multiaffine monomialmaximal does not imply mirrorreciprocal]
Consider the multiaffine monomialmaximal complete polynomial $$A(\mathbf{x}):=1+2x_{1}+x_{2}+x_{3}+3x_{2}x_{3}+x_{1}x_{2}+x_{1}x_{3}+x_{1}x_{2}x_{3}.$$ Its reciprocal is $$\rec(A(\mathbf{x}))=x_{1}x_{2}x_{3}+2x_{2}x_{3}+x_{1}x_{3}+x_{1}x_{2}+3x_{1}+x_{3}+x_{2}+1$$ and it is clear that no permutation of the variables can transform the degree $1$ part $3x_{1}+x_{3}+x_{2}$ of $\rec(A)$ into the degree $1$ part $2x_{1}+x_{2}+x_{3}$ of $A$. Thus $A$ is not mirrorreciprocal.
\end{ejemplo}

This finishes our discussion about the polynomial identity described in Theorem \ref{especialita} and its immediate combinatorial consequences in terms of relations between cardinals of some sets of permutations and sums involving subsets of a certain set. This also finishes our study of symmetries within multivariate Eulerian polynomials because going further in this direction lies beyond our scope here. In this way, we head to the conclusion of this thesis, in which we will discuss (among other things) on the importance of these symmetries in order to understand more about these (and related) families of polynomials and the meanings and properties of their roots and coefficients.

\part{Conclusion}\label{VI}
\chapter[Devising a Path Forward: the Mindelsee Program]{Towards an Extended View of The Exposed Phenomena and Devising a Path Forward: the Mindelsee Program}\label{conclusion}

In this conclusion, we will try to look into the future possibilities that the developments presented here open. We will build a path to venture ourselves into a broad and new topic of research. Although new, the program that we will propose here is, as we will see, obviously related to the current effort promoted in the catalog 
\cite{Alexandersson2020}, which is itself a very valuable resource as a first step towards the proposal that we expand here. Another project with a similarly close aim (but approaching more clearly by the combinatorial side) is the database \cite{FindStat}, which therefore serves as a support for another side of our proposal. Finally, speaking about the related literature and resources pursuing early prerequisites coming before we can properly attack the first sketches and attempts at our proposal, we have to refer to the book \cite{pemantle2012hyperbolicity}, where a broad variety of topics intimately connected to the discussions that will appear in this conclusion chapter has already been treated, pictured, sketched, developed and studied. Hence, many conceptual paths end there into some of our proposed beginnings here. Thus, we recognize the invaluable labor developed in these first steps into our proposed framework: many steps around the vicinity of the topics that concern us here have already been taken in these references before. Now, with their help, we can venture safely into the mathematical realm opened by our new far-reaching ideas in the directions presented in this conclusion and inspired by the work we did in this thesis.

\section{Summary of bounds obtained and methods used}

We obtain many different lateral bounds for the extreme roots of univariate Eulerian polynomials. The next table summarizes the bounds obtained, the methods used for each bound, the corresponding asymtotic behavior and the polynomial transformations (through stability preservers or through roots modifications) that had to be performed in order to obtain the corresponding bound. This will help us to visualize fast the overall structure, ideas and constructions explored in this dissertation. That improved visualization of what was done here will help us to understand the next steps that we are going to propose in this conclusion and why these steps are interesting and meaningful by themselves in the way to understand better the phenomena about real rooted (or real-zero) polynomials explored here.

\begin{center}
\begin{longtable}{ |p{3cm}|p{3cm}|p{3cm}|p{3cm}|  }
    \hline
    \multicolumn{4}{|c|}{List of lower bounds for $|q_{1}^{(n)}|$} \\
    \hline
   Method used & Asymptotic behavior obtained & Polynomial extension or transformation involved & Improvement calculated\\\hline  \hline
    Colluci's estimation & $\frac{2^{n+1}}{n}-\frac{2}{n}-1$ & None & None\\ \hline
    Univariate relaxation plus optimization &   $\un=2^{n+1}-\left(\frac{3}{2}\right)^{n+1}-2\left(\frac{9}{8}\right)^{n+1}+o(\left(\frac{9}{8}\right)^{n})$  & None & Clear linear improvement with respect to previous\\ \hline
    Univariate relaxation plus linearization through the vector $(1,1)$ & $2^{n+1}+o(2^{n})$ & None & Worse than $\un$, but easier to compute\\ \hline
    Bivariate relaxation plus linearization through the vector $(\alpha,3,-8)$ & $\bi=2^{n+1}+o(2^{n})$ & Bivariate Eulerian polynomials obtained through interlacing relation with previous element in the sequence of univariate Eulerian polynomials & Better than $\un$ as $\bi-\un\sim\frac{16^{-5n-6}}{59049n^{7}}\to0^{+}$\\ \hline
    Multivariate relaxation plus linearization through the vector $v=(y,0,1,-1,\dots,-1)$ &   $2^{n+1}+o(2^{n})$  & Multivariate Eulerian polynomials obtained through the use of stability preserver recursion & Better than $\un$ as $\mult_{v}-\un\sim\frac{1}{2}\left(\frac{3}{4}\right)^{n}\to0^{+}$ \\ \hline
    Multivariate relaxation plus linearization through the vector $v=(y,0,(-2^{m-i})_{i=3}^{m}$ $,(0,\frac{1}{2}),(1)_{i=1}^{m})\in\mathbb{R}^{n+2}$ with $n=2m$& $2^{n+1}+o(2^{n})$ & Multivariate Eulerian polynomials obtained through the use of stability preserver recursion & Exponentially better than $\un$ as $\mult_{v}-\un\sim\frac{3}{8}\left(\frac{9}{8}\right)^{m}\to\infty$ \\ \hline
    DLG plus Sobolev simplifications & $M=2^{n+1}-\left(\frac{3}{2}\right)^{n+1}-\left(\frac{9}{8}\right)^{n+1}+o(\left(\frac{9}{8}\right)^{n})$ & Univariate polynomials from first iteration of DLG method & Exponentially much better than above as $M-\un\sim\left(\frac{9}{8}\right)^{n+1}\to\infty$ \\ \hline
    DLG without Sobolev simplifications & $M_{\mbox{ref}}=2^{n+1}-\left(\frac{3}{2}\right)^{n+1}-\left(\frac{9}{8}\right)^{n+1}-2\left(\frac{27}{32}\right)^{n+1}+o(\left(\frac{27}{32}\right)^{n})$ & Univariate polynomials from first iteration of DLG method & Better understanding of the asymptotic rest $+o(\left(\frac{27}{32}\right)^{n})$ \\ \hline
    DLG plus univariate relaxation & $\un_{2}=2^{n+1}-\left(\frac{3}{2}\right)^{n+1}-\left(\frac{9}{8}\right)^{n+1}+1+o(1)$ & Univariate polynomials from first iteration of DLG method & Effective jump of $1$ to get closer to the root in the limit because $\un_{2}-M_{\mbox{ref}}\sim 1>0$\\ \hline
   \caption{Summary of lower bounds calculated for $|q_{1}^{(n)}|$ and the tools used for producing them.}   
\end{longtable}
\end{center}

Thanks to this table, we can see in a nutshell the methods we used and the results they produced. As we can readily see, the combination of methods and going deeper into the understanding of the different transformations and extensions of the univariate Eulerian polynomials helped us to improve the bounds obtained in each new step. This signals that it is possible to translate new knowledge about these operations over the original sequence of Eulerian polynomials into information about its extreme roots. This topic will be interesting for us and it will push us to introduce the program that we present in the next chapter as a final conclusion for this dissertation. This conclusion will only show us the doors of the realm that these questions prelude and that we could only superficially explore in this dissertation, scratching the surface of a field that requires further exploration beyond the limited space and time we have here.

\section[Difficulties, obstructions and necessary tools]{What has been done and what could be done: difficulties, obstructions and necessary tools looking beyond}

We have done several first steps into a set of new developments in this dissertation. In many cases, what we have done seems to be just the tip of an iceberg that calls for a further exploration. We just showed many tips in order to highlight the uses that our new techniques have and the novelties that our construction could eventually inspire while getting deeper into the topics we opened here.

\begin{observacion}
The constructions and analyses made here allow for a wide range of generalizations and analogies. In principle, not all possible directions are clear and we cannot attempt to give a prediction of the reach of these without studying them beyond the original scope we had in this thesis.
\end{observacion}

Looking part by part, we can collect the main ideas that came out of these in order to form a picture of what we have built so far. Such picture should serve itself as a canvas to draw this further-reaching conclusion whose objective is in fact looking beyond our current knowledge into what we can do in the future building over the ideas sketched here. Naturally, we begin this forward analysis by the first part of this work. Thus, we look at the relaxation itself and its limitations.

\begin{remark}[Overcoming the limitations]
We were able to analyze two main limitations in the construction and ampliation of the relaxation. The first limitation that we studied is the one related to the ampliation of the size of the LMP used by the relaxation applied to a certain polynomial. We saw that a first constraint we find is the fact that natural ampliations using straightforward extensions of the mold moment matrix fail in keeping the initial matrix PSD. Moreover, for the extension we found that does accomplish this, we know that it does not produce a relaxation anymore. This does not mean that these extensions are impossible, but, as we have learned, just that we may need a more elaborated structure for these extensions. Similarly, the second limitation came from the fact that it is not clear how to increase the number of variables of a polynomial keeping RZ-ness. We do not know how to do that in general, but our later study of multivariate Eulerian polynomials showed us that there can be more elaborated ways of doing that with certain families. These ways boil down to the theory of stability preservers in the case of Eulerian polynomials, but other families of RZ polynomials could offer new insights into different forms of performing these extensions of variables that could ultimately help us in the improvement of the relaxation, as it happened with univariate Eulerian polynomials and their injection in a diagonal of a sequence of (also combinatorially) related multivariate polynomials.
\end{remark}

These Eulerian polynomials were the main characters of the second part of this thesis. Their study in the literature allowed us to develop here our study of the relaxation under variable extensions in their particular case. It turns out that these polynomials could be part of a much bigger family of polynomials whose structure could allow us to get not only further insights but maybe even a complete picture of RZ polynomials if they end up providing us with the \textit{right} idea leading to a \textit{correct} construction allowing for variable extensions keeping our craved and sought for RZ-ness.

\begin{remark}[A sequence of polynomials coming from permutations]
The sequences of Eulerian polynomials used here come from counting a certain statistic (descent tops) in permutations over subsets of natural numbers with their natural inherited order. These polynomials turned out to be extremely nice for us. The literature is filled with many other polynomials build in similar ways using certain modifications of these. These polynomials are also nice and constitute therefore the next place to look for ideas and study towards a better understanding of our relaxation and its power.
\end{remark}

The relaxation applied to the stated multivariate extension of these polynomials was able to improve over the already good results given by the relaxation applied to the univariate polynomials in the diagonal. We had to choose certain vectors when linearizing the relaxation. We wonder if these choices can be improved.

\begin{remark}[Linearizing clever]
We saw in Part \ref{III} that linearizing the relaxation through one vector or another gives very different results. It took us many experiments and some numeric work to find the vector we used there to beat the bound $\un$. Since this vector was obtained through experimentation, it is natural to ask for better sequences of vectors. It is also natural to expect that, as when the size of these matrices grows so does the vector, the complexity of the structure of the optimal vector and its easily expressible approximations also increases by the mere fact that more different patterns can fill more entries on the vector. Finding these better sequences of vectors is helpful for determining how much better the relaxation actually gets at providing bounds, because dealing with the corresponding sequence of determinants of the relaxation directly is probably too complicated. Maybe even more complicated than directly looking at the original polynomial. Remember in this last respect that the determinant and the original univariate polynomial have very similar degree because multivariate Eulerian polynomials have as many variables as degree. This is why linearizing through better and better sequences of vectors is a meaningful and important task towards a successful application of the relaxation to bounding the extreme roots of these polynomials.
\end{remark}

These Eulerian polynomials have proven to be useful for many branches of mathematics. Not for nothing they are studied from many different perspectives by authors coming from disparate backgrounds. At the beginning of our work, we centered mainly around their nice combinatorial interpretations, as such point of view allowed us to easily construct their multivariate generalization that we used later on. However, at some point, we realized that ideas coming from analysis and numerics also allow us to define these polynomials and this pointed towards the study made by Sobolev. The work of Sobolev showed us not only a different way of thinking about these polynomials, but also new methods towards bounding their extreme roots. Thus is how we landed in Part \ref{IV}. This part opened many questions for us. We discuss these questions here towards future developments concerning their scope and reach.

\begin{remark}
The main path in this direction comes from the fact that the DLG method is just one possibility to the more general methods based on resultants that can be used to split the roots of the polynomials. The DLG method is just the simplest, most acknowledged method. All these methods do in fact generate new families of real-rooted polynomials stemming from the Eulerian polynomials. These polynomials have therefore an algebro-combinatorial nature because they emerge from polynomials counting features of combinatorial objects through nice and simple algebraic transformations and manipulations. It would be interesting to both understand which combinatorial objects these algebraic splitting methods are related to (if any) and if understanding these objects better could lead to multivariate generalizations of these univariate polynomials. In fact, it is not clear how to perform splitting methods like the DLG one preserving RZ-ness in the multivariate setting. This would amount to splitting the corresponding ovaloids. This is also an interesting path to set a foot in. Additionally, studying these new polynomials and their possible recurrences could point towards new applications and extensions of the theory of stability preservers. Finally, combining all this and the knowledge we gained in Part \ref{IV}, it is clear that the relaxation could work as a middle step between each iteration of the DLG method when we try to improve bounds. It would be interesting to understand better the interrelation between the relaxation and the iteration of these methods. But, for this, we need to improve these methods and connect them to the multivariate setting, as these approximation methods are not so well-developed in the multivariate setting. As we said, there does not exist a satisfactory splitting method for ovaloids of multivariate polynomials preserving RZ-ness. And all this without entering in the numerical stability problems of these methods that could be addressed through more clever techniques similar to our linearization arguments for studying the relaxation.
\end{remark}

In the cleverness towards attacking these problems, it is clear that exploiting symmetries is a winning strategy. We have seen how the palindromicity of the univariate Eulerian polynomials can be translated to the multivariate setting in Part \ref{V}. These symmetries allow us not just accessing more roots, but also understanding better the shape of the relaxation. We can extend the study of these symmetries in two directions: breadth and depth.

\begin{remark}[Explorations in symmetry]
We saw that multivariate Eulerian polynomials are mirrorreciprocal extending the notion of palindromicity to the multivariate setting. Other forms of symmetry can follow the same path helping us to broaden the scope of these commonly univariate explorations on symmetry.
\end{remark}

With all these remarks, we are able to state a full program of research based on the many aspects that we discussed in this conclusion and that stem from our indagations performed during the development of the ideas presented, studied and analyzed in this thesis. The whole description of the details entailing this program is the objective we pursue in the closing section of this conclusion and, therefore, how we decide to finish this thesis looking further beyond.

\section{A proposal for a future program}

All the remarks that we discussed draw a picture for a program aimed at establishing mechanisms for injecting univariate real-rooted polynomials into multivariate ones in ways that could help the relaxation in the task of bounding the corresponding (extreme) roots. Such a program would therefore consist in several steps and range among many different families and sequences of polynomials. We now describe all these necessary steps in more detail.

\begin{step}[Collecting families and sequences]
The first part of this program consists in identifying many families and sequences of univariate polynomials having nice combinatorial (or otherwise generalizable) definitions. There are already attempts to catalog these polynomials. The real challenge here lies in developing techniques that could work for both searching candidates among the many possible polynomials that we can construct using the discrete objects we already know and study and checking that some of the already known families are indeed real-rooted. In this sense, the DLG method and its generalizations using resultants give a full family of ways to generate new sequences of real-rooted polynomials from known ones.
\end{step}

Once we have built an extensive library of these polynomials, the next step will consist in going multivariate with them. That is the fundamental step we applied here and it is a challenging task by itself, as one can check in the literature. In general, there are no clear recipes on how to build many different multivariate RZ polynomials from a real-rooted one unless we already have a nice form for these RZ polynomials, like the one a constructive version of the GLC would provide. Another possibility would be using other determinantal representations (like the one provided by Kummer's attempt) and hoping that the line restriction we are interested in is not one of the lines affected by an early intersection of the rigidly convex set of the additional cofactor. All in all, we see that we are in front of a demanding challenge.

\begin{step}[Going multivariate]
The second step consists in finding ways of injecting the polynomials obtained above in some lines through the origin of multivariate counterparts. There are many hints and methods (mainly) coming form the theory of stability preservers that can achieve this. Here we saw and used two of these methods: one based on interlacers and the other on recurrences. The one based on recurrences could be traced back to a form of combinatorial fine-counting. This points in the direction of studying these methods in depth and determining the different manifestations of their actions along the different layers of definitions of our polynomials. This connects directly to the next step.
\end{step}

Once we have injected our initial univariate polynomial into a multivariate one, the first question that arises is what the different lines around the origin represent. Here we saw that the non-diagonal lines of our multivariate extension represent fine-counting in weighted permutations, for example. But therefore we also saw a connection that goes back from the multivariate polynomials to the original combinatorial object. Thus, having these multivariate polynomials at hand allowed us to devise connections not just between the polynomials but between the new polynomials and new combinatorial objects. This is the next step.

\begin{step}[Uncovering connections in the multivariate extension process]
We mean here both vertical and horizontal connections, in depth and in width. After injecting the univariate polynomial into a multivariate one, what we have is a kind of RZ amalgam of many different univariate real-rooted polynomials (one for each line). The relations between these lines tell us how these polynomials connect with each other establishing ways to relate them and study their connection through different lenses. Since, at the same time, these polynomials can usually also encode information about related combinatorial structures, understanding these horizontal relations will also help us to understand the relations and connections between the combinatorial structures themselves if we are able to travel back to these. This is how the vertical dimension takes prominence: in going back to the combinatorial level, we are able to gain insights and information about the interrelations of combinatorial objects by looking at the polynomials that amalgamate them and understanding how these polynomials perform that amalgamation of combinatorial information. This is why this step takes this privileged position before we even applied the relaxation.
\end{step}

After this step, we have to develop finer ways of counting that would allow us to have nice control about the new combinatorial fine-counting that we had to develop in order to go multivariate. Thus we have to understand the new and more complex and information-filled combinatorial structures that the necessity of multivariateness sets in our framework. In the particular case we treated here, we were lucky enough that other people had already published ways of effectively and efficiently counting these objects. This step is necessary in order to have enough control about the coefficients so we can actually compute the corresponding $L$-forms we need in order to write down the relaxation.

\begin{step}[Finer-counting back in the combinatorial level]
In the step above, when going multivariate, we can actually trace back this action of going to the multivariate setting to some form of finer-counting at the combinatorial level. In this case, understanding the new and finer features we are dealing with is natural, interesting and necessary. Establishing methods and formulas to effectively and efficiently compute these is a fundamental task in our program.
\end{step}

With these formulas, we are able to compute the $L$-forms we need and, with them, we can therefore form the relaxation. This step is the most natural and straightforward one, but it is also central.

\begin{step}[Forming the relaxation]
We use the formulas and methods developed above to compute the relaxation of the multivariate polynomials. We can compute several relaxations with different restrictions of the variables used in order to compare these in the future. This comparison will happen in terms of both theoretical accuracy, numerical stability and practical computability. These features pair immediately with the next step when calculations and expressions become too complicated to deal with them directly.
\end{step}

Once we have written down the relaxation and we have it in our hands, we have to understand how it works. This understanding happens through numerical experimentation using the fact that we used the previous steps in order to make the relaxation relatively easy to compute.

\begin{step}[Numerical experimentation]
Through numerical experiments, we can understand how the relaxation works at approximating the inner roots of our original polynomials or even some other related ones. We can build on the knowledge extracted here in order to choose cleverly the path to take in the following step. Without these experiments, it is not possible to make the choices necessary to simplify and select the information of the relaxation that we need to keep.
\end{step}

In particular, as computing and analyzing the determinant of that relaxation directly is usually a more complicated task than studying and understanding the original polynomial itself, we have to develop methods to simplify the relaxation selecting cleverly the amount of information we want to keep and conserve out of it. For example, in particular, in this paper we chose the path of linearization. There might be other paths based on matrix compression, but we will mainly center here around what we did in this thesis.

\begin{step}[Linearization and approximated generalized eigenvectors]
The numerical experiments above will eventually show us some options of families or sequences of vectors that keep a fair amount of the information of the relaxation. These families will allow us to approximate the generalized eigenvectors of the generalized eigenvalue problem related to determining the roots of the determinant of the corresponding LMP. Thus, having good enough eigenvector approximations (or guesses) will allow us to approximate the roots of the determinant of the relaxation in computationally friendly ways. This will happen because, thanks to these approximations (or guesses) for the corresponding generalized eigenvectors, we will be able to work with linear polynomials instead of having to deal with polynomials of high degree. 
\end{step}

As we said above, we decided to linearize using approximations (or guesses) of generalized eigenvectors. However, it is clear that, between the high degree of the determinant and the degree $1$ of such linearization, must be many intermediate steps involving other transformations and polynomials that we did not explore here.

\setlength{\emergencystretch}{3em}%
\begin{remark}[Other unexplored paths for the step above]
Linearizing might not be the only answer. As dealing with roots of polynomials of degree up to $4$ only requires elementary computations, it is natural to think that these degrees can also be used successfully. However, such a task can be very demanding and, for that reason, we avoided here going in that particular direction.
\end{remark}
\setlength{\emergencystretch}{0em}%

Linearizing the relaxation in the step above will finally provide us with the bounds we were looking for. Once we have them, it is natural to ask how good are they in comparison both with other methods and with applications of the relaxation to instances where fewer variables were forced structurally meaningfully into our multivariate polynomials. We did this last task when we compared what happened when we added just one additional variable through interlacing instead of the many variables that the multivariate Eulerian polynomials allowed us to introduce sequentially into our polynomials.

\begin{step}[Comparison with other methods and with other ways into multivariability]
Once we have established the bounds provided by one way of injecting the univariate polynomial of our interest into a multivariate one, we compare the quality of such bound with respect to other known bounds. This comparison should be made in two ways: in accuracy and in computations. It is important not just to get closer bounds but to know how many computations does it take to reach these bounds using the initial coefficients of the original polynomial. These are the two main measures by which we can compare one bound to another and one method of bounding to another. This comparison should not happen just with other radically different methods, but also with methods that eventually use also the relaxation after injecting the univariate polynomial into a multivariate one in a different way using maybe other constructions.
\end{step}

\setlength{\emergencystretch}{3em}%
\begin{remark}[The form of the comparison]
We choose to compare using two measures of quality: one for the bound itself and another for the method itself. It is clear that there might be other measures able to carry other features of the complexities attached to compute a bound. We also consider worth studying these measures by themselves so we can develop tools to know which bounding methods and which bounds are better than others when we look through one of these measures. This topic is interesting by itself.
\end{remark}
\setlength{\emergencystretch}{0em}%

In particular, these measures could automatically extend to measures about the complexity of injecting the univariate polynomial into a multivariate one. In this sense, being able to establish different methods to measure the nuances of one construction against another one is also especially relevant and something we make a call to study and develop.

\begin{remark}[The form of the extension]
In the same way as we want to do measuring methods to bound roots using different tools and points of view, we can extend this task to the many different methods and ways in which a univariate polynomial can be injected into a multivariate one so that the property of being RZ is preserved. These tools could be completely different but also extensions from the one mentioned above. This therefore makes them interesting by themselves.
\end{remark}

The other methods we study might be too different to allow for a combination with the relaxation. However, another subset of methods could fit well into a strategy containing also the relaxation. That was the case of the DLG method: the best bound provided in this dissertation emerges as an application of the relaxation to a univariate polynomial built through the use of the DLG method. It is natural to wonder at this point therefore if the relaxation and the method we are comparing it with can act symbiotically as happened in the case we just mentioned.

\begin{step}[Combinations of the relaxation and other methods]
It is in general worth checking whether the different methods we have to produce bounds can combine between them. Here, we center around the possibility of mixing and combining the other method with the relaxation just because the relaxation is the main tool we were studying and analyzing in this dissertation.
\end{step}

These combinations ended up producing the best bound here. It is natural to ask if these combinations can be made intercalated between each iteration of these methods when the method itself is iterative. In the case of the DLG we saw briefly that it does indeed seem so and therefore the relaxation produces an intermediate step between each iteration of the DLG method. Understanding the interplay between the relaxation and the other methods better could help us to remediate numerical stability or computability issues.

\begin{step}[Chains of iterations]
We have to determine if the relaxation can effectively be inserted or intercalated between different steps of the iterative methods used to compute bounds for the roots of polynomials. In the cases where these constructions of intercalation or insertion are indeed possible, it is necessary to study if that action improves the accuracy of the bound obtained and which is the cost in computational terms of these constructions as extensions of the original method.
\end{step}

Ameliorating these issues between iterations is fundamental in order to be able to perform each time deeper applications of our methods. Going deeper into the iterative process produces usually better bounds. We just have to be able to control how much costly that task is going to be in order to be ideally able to decide beforehand where we should stop.

\begin{step}[Rinse and repeat]
We study how these intercalations improve the original method. This helps us to establish clearly how helpful is the relaxation in the whole extended process involving both the relaxation and the intercalations of the relaxation between iterative steps.
\end{step}

These steps form the main structure of the program we propose here. What we have done here corresponds just to a tiny portion of this big iceberg that we would like to uncover through this program. The main fundamental objects for setting this program in motion are therefore these real-rooted polynomials whose extreme roots could be interesting to study and know or estimate through these indagations. Many of these polynomials also come from combinatorics and are therefore already found scattered across the literature related to many different (sub-)areas of discrete mathematics. Some of them are even closely related to the Eulerian polynomials studied here.

\begin{remark}[Examples already in the literature]
We can see polynomials arising from similar combinatorial structures, e.g., in the cases of (multivariate) $P$-Eulerian polynomial \cite{peuler}, second-order Eulerian polynomials \cite{context} and Eulerian polynomials arising after counting the number of (non-)zero cycles \cite{cycles}. See also \cite{zhenhuan}. These example polynomials are all related to permutations, but, as we have also said before, these kinds of real-rooted and multivariate-generalizable polynomials appear more generally scattered through the literature of discrete mathematics and combinatorics also associated to matroids \cite{ALCO_2024__7_5_1479_0} and graphs \cite{genus}, for example. See also \cite{edge} and \cite{graphbook} for further examples. 
\end{remark}

Therefore, we used this conclusion to present a sight forward in the form of the proposed skeleton for the sketch of a future program, the \textit{Mindelsee Program}. With this action of looking beyond the current boundary of our knowledge about this topic, we set our foot securely at the border of future discoveries lying across the frontier that now we can feel with our naked feet. This is a good way of finishing any scientific report, always looking further beyond, and therefore we choose to close this dissertation here with a high and hopeful note for the mathematics of the future. The deepest beauty about mathematics is that, no matter where you look at, it is clear that a lot is yet to come. In fact, always a lot more than we already knew or even suspected when we began our original research, as it is clearly also the case here.

\setlength{\emergencystretch}{3em}%
\printbibliography[title=References]\addcontentsline{toc}{chapter}{References}
\setlength{\emergencystretch}{0em}%

\setlength{\emergencystretch}{0em}%

\end{document}